\documentclass[oneside,letterpaper]{memoir}

\usepackage[dvipsnames]{xcolor}
\definecolor{nicered}{rgb}{.647,.129,.149}
\definecolor{PUorange}{RGB}{245,128,37}
\definecolor{MITred}{RGB}{163, 31, 52}
\definecolor{nicegreen}{RGB}{245,128,37}
\definecolor{niceblue}{RGB}{245,128,37}
\usepackage{graphicx}
\usepackage[unicode,colorlinks,linkcolor=MITred,citecolor=MITred]{hyperref}

\usepackage{calc,soul,helvet}
\usepackage{amsmath,amsfonts,amssymb,amsthm}
\usepackage{mathrsfs}
\usepackage{colortbl, enumitem}
\usepackage{rotating,graphics,psfrag,epsfig, tikz}
\usepackage[abbreviations]{foreign}

\newcommand{\DS}{\displaystyle}

\newcommand{\cA}{\mathcal{A}}
\newcommand{\cB}{\mathcal{B}}
\newcommand{\cC}{\mathcal{C}}
\newcommand{\cD}{\mathcal{D}}

\newcommand{\cF}{\mathcal{F}}
\newcommand{\cG}{\mathcal{G}}
\newcommand{\cH}{\mathcal{H}}

\newcommand{\cL}{\mathcal{L}}

\newcommand{\cN}{\mathcal{N}}

\newcommand{\cP}{\mathcal{P}}

\newcommand{\cS}{\mathcal{S}}

\newcommand{\cU}{\mathcal{U}}
\newcommand{\cV}{\mathcal{V}}
\newcommand{\cX}{\mathcal{X}}
\newcommand{\cY}{\mathcal{Y}}

\newcommand{\sg}{\mathsf{subG}}
\newcommand{\subE}{\mathsf{subE}}

\newcommand{\bA}{\mathbf{A}}
\newcommand{\bB}{\mathbf{B}}

\newcommand{\bH}{\mathbf{H}}

\newcommand{\bQ}{\mathbf{Q}}

\newcommand{\bX}{\mathbf{X}}
\newcommand{\bY}{\mathbf{Y}}

\newcommand{\R}{\mathrm{ I}\kern-0.18em\mathrm{ R}}
\newcommand{\h}{\mathrm{ I}\kern-0.18em\mathrm{ H}}
\newcommand{\K}{\mathrm{ I}\kern-0.18em\mathrm{ K}}
\newcommand{\p}{\mathrm{ I}\kern-0.18em\mathrm{ P}}

\newcommand{\E}{\mathrm{ I}\kern-0.18em\mathrm{ E}}
\newcommand{\1}{\mathrm{ 1}\kern-0.24em\mathrm{ I}}
\newcommand{\N}{\mathrm{ I}\kern-0.18em\mathrm{ N}}

\newcommand{\field}[1]{\mathbb{#1}}
\newcommand{\q}{\field{Q}}
\newcommand{\Z}{\field{Z}}
\newcommand{\X}{\field{X}}
\newcommand{\Y}{\field{Y}}

\newcommand{\x}{\mathcal{X}}

\newcommand{\Bern}{\mathsf{Ber}}
\newcommand{\Bin}{\mathsf{Bin}}
\newcommand{\Lap}{\mathsf{Lap}}

\newcommand{\tr}{\mathsf{Tr}}

\newcommand{\pn}{\p_{\kern-0.25em n}}
\newcommand{\pnm}{\p_{\kern-0.25em n,m}}
\newcommand{\psubm}{\p_{\kern-0.25em m}}
\newcommand{\psubp}{\p_{\kern-0.25em p}}
\newcommand{\cfi}{\cF_{\kern-0.25em \infty}}
\newcommand{\e}{\mathrm{e}}

\newcommand{\Var}{\mathrm{Var}}

\newcommand{\indep}{\perp\kern-0.95em{\perp}}

\newcommand{\supp}{\mathop{\mathrm{supp}}}
\newcommand{\rank}{\mathop{\mathrm{rank}}}
\newcommand{\vol}{\mathop{\mathrm{vol}}}
\newcommand{\conv}{\mathop{\mathrm{conv}}}
\newcommand{\card}{\mathop{\mathrm{card}}}
\newcommand{\argmin}{\mathop{\mathrm{argmin}}}
\newcommand{\argmax}{\mathop{\mathrm{argmax}}}
\newcommand{\ud}{\mathrm{d}}
\newcommand{\var}{\mathrm{var}}

\newcommand{\MSE}{\mathsf{MSE}}

\newcommand{\eps}{\varepsilon}

\newcommand{\hint}[1]{\texttt{[Hint:#1]}}

\newlength{\minipagewidth}
\makeatletter
\newcommand\iid{i.i.d\@ifnextchar.{}{.\@\xspace} } 
\makeatother

\newcommand\MoveEqLeft[1][2]{\kern #1em & \kern -#1em} 
\newcommand{\leadeq}[2][4]{\MoveEqLeft[#1] #2 \nonumber}

\newcommand\independent{\protect\mathpalette{\protect\independenT}{\perp}}
\def\independenT#1#2{\mathrel{\rlap{$#1#2$}\mkern2mu{#1#2}}}

\newcommand{\MIT}[1]{{\color{MITred} #1}}
\newtheorem{thm}{Theorem}[chapter]

\newtheorem{propdef}[thm]{Definition-Proposition}

\newtheorem{cor}[thm]{Corollary}

\newtheorem{lem}[thm]{Lemma}

\theoremstyle{definition}
\newtheorem{defn}[thm]{Definition}
\newtheorem{prop}[thm]{Proposition}
\newtheorem{rem}[thm]{Remark}
\newtheorem{example}[thm]{Example}
\newtheorem{exercise}{Problem}[chapter]
\numberwithin{equation}{chapter}

\newenvironment{assumption}[1]{\begin{trivlist}
\item[\hskip \labelsep {\bfseries Assumption #1}]}{\end{trivlist}}

\makeatletter
\newlength\dlf@normtxtw
\newsavebox{\feline@chapter}
\newcommand\feline@chapter@marker[1][4cm]{%
  \sbox\feline@chapter{%
    \resizebox{!}{#1}{\fboxsep=1pt%
      \colorbox{MITred}{\color{white}\bfseries\sffamily\thechapter}%
    }}%
  \rotatebox{90}{%
    \resizebox{%
      \heightof{\usebox{\feline@chapter}}+\depthof{\usebox{\feline@chapter}}}%
    {!}{\scshape\so\@chapapp}}\quad%
  \raisebox{\depthof{\usebox{\feline@chapter}}}{\usebox{\feline@chapter}}%
}
\newcommand\feline@chm[1][4cm]{%
  \sbox\feline@chapter{\feline@chapter@marker[#1]}%
  \makebox[0pt][l]{
    \makebox[1cm][r]{\usebox\feline@chapter}%
  }}
\makechapterstyle{daleif1}{

  \renewcommand\printchapternum{\null\hfill\feline@chm[2.5cm]\par}

}
 \renewcommand\section{\@startsection {section}{1}{\z@}%
                                     {\bigskipamount}%
                                     {\bigskipamount}%
                                     {\centering \normalsize\sffamily\fontseries{bx}\selectfont\mathversion{bold}\MakeUppercase}}

\newcommand{\dHyp}{\{-1,1\}^d}

\makeatother

\chapterstyle{daleif1}
\captiondelim{. }
\captionnamefont{\bfseries\sffamily\small}
\captiontitlefont{\normalfont\sffamily\small}
\precaption{\rule{\linewidth}{0.2pt}\par}

\begin{document}

\title{{\normalfont\huge\sffamily\bfseries\color{MITred}
High-Dimensional Statistics} \\ ~ \\
\textbf{Lecture Notes}
}
\author{Philippe Rigollet and Jan-Christian H\"utter}
\date{October 29, 2023}

\begin{titlingpage}
\thispagestyle{empty}
\maketitle
\vfill
\begin{center}
\includegraphics[width=1.0in]{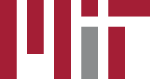}
\end{center}
\end{titlingpage}

\frontmatter

\chapter*[Preface]{Preface}
\addcontentsline{toc}{chapter}{Preface}

These lecture notes were written for the course  18.657, \textit{High Dimensional Statistics} at MIT. They build on a set of notes that was prepared at Princeton University in 2013-14 that was modified (and hopefully improved) over the years.

Over the past decade, statistics have undergone drastic changes with the development of high-dimensional statistical inference. Indeed, on each individual, more and more features are measured to a point that their number usually far exceeds the number of observations. This is the case in biology and specifically genetics where millions of (combinations of) genes are measured for a single individual. High-resolution imaging, finance, online advertising, climate studies \dots the list of intensive data-producing fields is too long to be established exhaustively. Clearly not all measured features are relevant for a given task and most of them are simply noise. But which ones? What can be done with so little data and so much noise? Surprisingly, the situation is not that bad and on some simple models we can assess to which extent meaningful statistical methods can be applied. Regression is one such simple model.

Regression analysis can be traced back to 1632 when Galileo Galilei used a procedure to infer a linear relationship from noisy data. It was not until the early 19th century that Gauss and Legendre developed a systematic procedure: the least-squares method. Since then, regression has been studied in so many forms that much insight has been gained and recent advances on high-dimensional statistics would not have been possible without standing on the shoulders of giants. In these notes, we will explore one, obviously subjective giant on whose shoulders high-dimensional statistics stand: nonparametric statistics.

The works of Ibragimov and  Has'minskii in the seventies followed by many researchers from the Russian school have contributed to developing a large toolkit to understand regression with an infinite number of parameters. Much insight from this work can be gained to understand high-dimensional or sparse regression and it comes as no surprise that Donoho and Johnstone have made the first contributions on this topic in the early nineties. 

Therefore, while not obviously connected to high dimensional statistics, we will talk about nonparametric estimation.
I borrowed this disclaimer (and the template) from my colleague Ramon van Handel. It does apply here.
\begin{quote}
I have no illusions about the state of these notes---they were written
rather quickly, sometimes at the rate of a chapter a week.  I have no
doubt that many errors remain in the text; at the very least many of the
proofs are extremely compact, and should be made a little clearer as is
befitting of a pedagogical (?) treatment.  If I have another opportunity 
to teach such a course, I will go over the notes again in detail and 
attempt the necessary modifications.  For the time being, however, the
notes are available as-is.
\end{quote}

As any good set of notes, they should be perpetually improved and updated but a two or three year horizon is more realistic. Therefore, if you have any comments, questions, 
suggestions, omissions,  and of course mistakes, please let me know.  I can be contacted by e-mail at 
\texttt{rigollet@math.mit.edu}.

\vskip.2cm

\textit{Acknowledgements.} These notes were improved thanks to the careful reading and comments of Mark Cerenzia, Youssef El Moujahid, Georgina Hall, Gautam Kamath,  Hengrui Luo, Kevin Lin, Ali Makhdoumi, Yaroslav Mukhin, Mehtaab Sawhney, Ludwig Schmidt, Bastian Schroeter, Vira Semenova, Mathias Vetter, Yuyan Wang, Jonathan Weed,  Chiyuan Zhang and Jianhao Zhang.

These notes were written under the partial support of the National Science Foundation, CAREER award DMS-1053987.

\vskip.2cm

\textit{Required background.} I assume that the reader has had  basic courses in probability and  mathematical statistics.  Some 
elementary background in analysis and measure theory is helpful but not required. Some basic notions of linear algebra, especially spectral decomposition of matrices is required for the later chapters.

Since the first version of these notes was posted a couple of manuscripts on high-dimensional probability by Ramon van Handel~\cite{Han17} and Roman Vershynin~\cite{Ver18} were published. Both are of outstanding quality---much higher than the present notes---and very related to this material. I strongly recommend the reader to learn about this fascinating topic in parallel with high-dimensional statistics.

\chapter*[Preface]{Notation}
\addcontentsline{toc}{chapter}{Notation}

\renewcommand\arraystretch{1.4}
\medskip

\noindent\textsc{Functions, sets, vectors}

\begin{tabular}{ p{2cm} l}
$[n]$ & Set of integers $[n]=\{1, \ldots, n\}$\\ 
$\cS^{d-1}$ & Unit sphere in dimension $d$\\
$\1(\,\cdot\,)$ & Indicator function\\
$|x|_q$& $\ell_q$ norm of $x$ defined by $|x|_q=\big(\sum_{i}|x_i|^q\big)^{\frac1q}$ for $q>0$\\
$|x|_0$& $\ell_0$ norm of $x$ defined to be the number of nonzero coordinates of $x$\\
$f^{(k)}$& $k$-th derivative of $f$\\
$e_j$& $j$-th vector of the canonical basis\\
$A^c$& complement of set $A$\\
$\conv(S)$ & Convex hull of set $S$.\\
$a_n \lesssim b_n$ & $a_n \le C b_n$ for a numerical constant $C>0$\\
$S_n$& symmetric group on \( n \) elements\\
\end{tabular}

\medskip

\noindent\textsc{Matrices}

\medskip

\begin{tabular}{ p{2cm} l}
$I_p$& Identity matrix of $\R^p$\\
$\tr(A)$& trace of a square matrix $A$\\
\( \cS_d \)& Symmetric matrices in \( \R^{d \times d} \)\\
\( \cS_d^+ \)& Symmetric positive semi-definite matrices in \( \R^{d \times d} \)\\
\( \cS_d^{++} \)& Symmetric positive definite matrices in \( \R^{d \times d} \)\\
\( \bA \preceq \bB \)& Order relation given by \( \bB - \bA \in \cS^+ \)\\
\( \bA \prec \bB \)& Order relation given by \( \bB - \bA \in \cS^{++} \)\\
$M^\dagger$ & Moore-Penrose pseudoinverse of $M$\\
$\nabla_xf(x)$& Gradient of $f$ at $x$\\
$\nabla_xf(x)|_{x=x_0}$& Gradient of $f$ at $x_0$\\
\end{tabular}

\medskip

\noindent\textsc{Distributions}

\medskip

\begin{tabular}{ p{2cm} l}
$\cN(\mu, \sigma^2)$ & Univariate Gaussian distribution with mean $\mu \in \R$ and variance $\sigma^2>0$\\
$\cN_d(\mu, \Sigma)$ & $d$-variate distribution with mean $\mu \in \R^d$ and covariance matrix $\Sigma \in \R^{d \times d}$\\
$\sg(\sigma^2)$ & Univariate sub-Gaussian distributions with variance proxy $\sigma^2>0$\\
$\sg_d(\sigma^2)$ & $d$-variate sub-Gaussian distributions with   variance proxy $\sigma^2>0$\\
$\subE(\sigma^2)$ & sub-Exponential distributions with variance proxy $\sigma^2>0$\\
$\Bern(p)$& Bernoulli distribution with parameter $p \in [0,1]$\\
$\Bin(n,p)$& Binomial distribution with parameters $n \ge 1, p \in [0,1]$\\
$\Lap(\lambda)$& Double exponential (or Laplace) distribution with parameter $\lambda>0$\\
$P_X$ & Marginal distribution of $X$
\end{tabular}

\medskip

\noindent\textsc{Function spaces}

\medskip

\begin{tabular}{ p{2cm} l}
$W(\beta, L)$ & Sobolev class of functions\\
$\Theta(\beta, Q)$ & Sobolev ellipsoid of $\ell_2(\N)$\\
\end{tabular}
\renewcommand\arraystretch{1}

\cleardoublepage
\tableofcontents

\mainmatter
\renewcommand{\theequation}{\arabic{equation}}
\renewcommand{\thefigure}{\arabic{figure}}
\chapter*[Introduction]{Introduction}
\addcontentsline{toc}{chapter}{Introduction}

This course is mainly about learning a regression function from a collection of observations. In this chapter, after defining this task formally, we give an overview of the course and the questions around regression. We adopt the statistical learning point of view where the task of \emph{prediction} prevails. Nevertheless, many interesting questions will remain unanswered when the last page comes: testing, model selection, implementation,\dots

\section*{Regression analysis and prediction risk}

\subsection*{Model and definitions}
Let $(X, Y) \in \cX \times \cY$ where $X$ is  called \emph{feature} and lives in a topological  space $\cX$ and $Y \in \cY \subset \R$ is called \emph{response} or sometimes \emph{label} when $\cY$ is a discrete set, e.g., $\cY=\{0,1\}$. Often $\cX \subset \R^d$, in which case $X$ is called \emph{vector of covariates} or simply \emph{covariate}.  Our goal will be to predict $Y$ given $X$ and for our problem to be meaningful, we need $Y$ to depend non-trivially on $X$. Our task would be done if we had access to the conditional distribution of $Y$ given $X$. This is the world of the probabilist. The statistician does not have access to this valuable information but rather has to estimate it, at least partially. The regression function gives a simple summary of this conditional distribution, namely, the conditional expectation.

Formally, the \emph{regression function of $Y$ onto $X$} is defined by:
$$
f(x)=\E[Y|X=x]\,, \quad x \in \cX\,.
$$
As we will see, it arises naturally in the context of prediction.

\subsection*{Best prediction and prediction risk}

Suppose for a moment that you know the conditional distribution of $Y$ given $X$. Given the realization of $X=x$, your goal is to predict the realization of $Y$. Intuitively, we would like to find a measurable\footnote{all topological spaces are equipped with their Borel $\sigma$-algebra} function $g\,:\,\cX \to \cY$ such that $g(X)$ is close to $Y$, in other words, such that $|Y-g(X)|$ is small. But $|Y-g(X)|$ is a random variable, so it is not clear what ``small" means in this context. A somewhat arbitrary answer can be given by declaring a random variable $Z$ small if $\E [Z^2]=[\E Z]^2 + \var[ Z]$ is small. Indeed, in this case, the expectation of $Z$ is small and the fluctuations of $Z$ around this value are also small. The function $R(g)=\E[Y-g(X)]^2$ is called the \emph{$L_2$ risk of $g$} defined for $\E Y^2<\infty$.

For any measurable function $g\,:\, \cX \to \R$, the $L_2$ risk of $g$ can be decomposed as
\begin{align*}
\E[Y&-g(X)]^2=\E[Y-f(X)+f(X)-g(X)]^2\\
&=\E[Y-f(X)]^2+\E[f(X)-g(X)]^2+2\E[Y-f(X)][f(X)-g(X)]
\end{align*}
The cross-product term satisfies
\begin{align*}
\E[Y-f(X)][f(X)-g(X)]&=\E\big[\E\big([Y-f(X)][f(X)-g(X)]\big|X\big)\big]\\
&=\E\big[[\E(Y|X)-f(X)][f(X)-g(X)]\big]\\
&=\E\big[[f(X)-f(X)][f(X)-g(X)]\big]=0\,.
\end{align*}
The above two equations yield
$$
\E[Y-g(X)]^2=\E[Y-f(X)]^2+\E[f(X)-g(X)]^2\ge \E[Y-f(X)]^2\,,
$$
with equality iff $f(X)=g(X)$ almost surely. 

We have proved that  the regression function $f(x)=\E[Y|X=x], x \in \cX$, enjoys the   \emph{best prediction} property, that is
$$
\E[Y-f(X)]^2=\inf_{g}\E[Y-g(X)]^2\,,
$$
where the infimum is taken over all measurable functions $g\,:\, \cX \to \R$.

\subsection*{Prediction and estimation}

As we said before, in a statistical problem, we do not have access to the conditional distribution of $Y$ given $X$ or even to the regression function $f$  of $Y$ onto $X$. Instead, we observe a \emph{sample} $\cD_n=\{(X_1, Y_1), \ldots, (X_n, Y_n)\}$ that consists of independent copies of $(X,Y)$. The  goal of regression function estimation is to use this data to construct an estimator $\hat f_n\,:\,\cX \to \cY$ that has small $L_2$ risk $R(\hat f_n)$. 

Let $P_X$ denote the marginal distribution of $X$ and for any  $h\,:\, \cX \to \R$, define 
$$
\|h\|_2^2=\int_\cX h^2 \ud P_X\,.
$$
Note that $\|h\|_2^2$ is the Hilbert norm associated to the inner product
$$
\langle h, h'\rangle_2=\int_\cX hh'\ud P_X\,.
$$
When the reference measure is clear from the context, we will simply write $\|h\|_2 =\|h\|_{L_2(P_X)}$ and $\langle h, h'\rangle_{2}:=\langle h, h'\rangle_{L_2(P_X)}$.

It follows from the proof of the best prediction property above that
\begin{align*}
R(\hat f_n)&=\E[Y-f(X)]^2+\|\hat f_n-f\|_2^2\\
&=\inf_{g}\E[Y-g(X)]^2+\|\hat f_n-f\|_2^2
\end{align*}
In particular, the prediction risk will always be at least equal to the positive constant $\E[Y-f(X)]^2$. Since we tend to prefer a measure of accuracy to be able to go to zero (as the sample size increases), it is equivalent to study the \emph{estimation error} $\|\hat f_n-f\|_2^2$. 
Note that if $\hat f_n$ is random, then  $\|\hat f_n-f\|_2^2$ and $R(\hat f_n)$ are \emph{random} quantities and we need  deterministic summaries to quantify their size. It is customary to use one of the two following options. Let $
\{\phi_n\}_n$ be a sequence of positive numbers that tends to zero as $n$ goes to infinity.
\begin{enumerate}
\item \textbf{Bounds in expectation}. They are of the form:
$$
\E\|\hat f_n-f\|_2^2 \le \phi_n\,,
$$
where the expectation is taken with respect to the sample $\cD_n$. They indicate the \emph{average behavior} of the estimator over multiple realizations of the sample. Such bounds have been established in nonparametric statistics where typically $\phi_n=O(n^{-\alpha})$ for some $\alpha \in (1/2,1)$ for example.

Note that such bounds do not characterize the size of the deviation of the random variable $\|\hat f_n-f\|_2^2$ around its expectation. As a result, it may be therefore appropriate to accompany such a bound with the second option below.
\item \textbf{Bounds with high-probability}. They are of the form:
$$
\p\big[\|\hat f_n-f\|_2^2 > \phi_n(\delta)\big]\le \delta\,, \quad \forall \delta \in (0, 1/3)\,.
$$
Here $1/3$ is arbitrary and can be replaced by another positive constant. Such bounds control the tail of the distribution of $\|\hat f_n-f\|_2^2$. They show how large the quantiles of the random variable $\|f-\hat f_n\|_2^2$ can be. Such bounds are favored in learning theory, and are sometimes called PAC-bounds (for Probably Approximately Correct).
\end{enumerate}
Often, bounds with high probability follow from a bound in expectation and a concentration inequality that bounds the following probability
$$
\p\big[\|\hat f_n-f\|_2^2-\E\|\hat f_n-f\|_2^2  > t\big]
$$
by a quantity that decays to zero exponentially fast. Concentration of measure is a fascinating but wide topic and we will only briefly touch it. We recommend the reading of \cite{BouLugMas13} to the interested reader. This book presents many aspects of concentration that are particularly well suited to the applications covered in these notes.

\subsection*{Other measures of error}
We have chosen the $L_2$ risk somewhat arbitrarily. Why not the $L_p$ risk defined by $g \mapsto \E|Y-g(X)|^p$ for some $p \ge 1$? The main reason for choosing the $L_2$ risk is that it greatly simplifies the mathematics of our problem: it is a Hilbert space! In particular, for any estimator $\hat f_n$, we have the remarkable identity:
$$
R(\hat f_n)=\E[Y-f(X)]^2+\|\hat f_n-f\|_2^2\,.
$$
This equality allowed us to consider only the part $\|\hat f_n-f\|_2^2$ as a measure of error. While this decomposition may not hold for other risk measures, it may be desirable to explore other distances (or pseudo-distances). This leads to two distinct ways to measure error. Either by bounding a pseudo-distance $d(\hat f_n, f)$ (\emph{estimation error}) or by bounding the risk $R(\hat f_n)$  for  choices  other than the $L_2$ risk. These two measures coincide up to the additive constant $\E[Y-f(X)]^2$ in the case described above. However, we show below that these two quantities may live independent lives. Bounding the estimation error is more customary in statistics whereas risk bounds are preferred in learning theory. 

Here is a list of choices for the pseudo-distance employed in the estimation error.
\begin{itemize}
\item \textbf{Pointwise error}. Given a point $x_0$, the pointwise error measures only the error at this point. It uses the pseudo-distance:
$$
d_0(\hat f_n, f)=|\hat f_n(x_0)-f(x_0)|\,.
$$
\item \textbf{Sup-norm error}. Also known as the $L_\infty$-error and defined by 
$$
d_\infty(\hat f_n, f)=\sup_{x \in \cX}|\hat f_n(x)-f(x)|\,.
$$
It controls the worst possible pointwise error.
\item \textbf{$L_p$-error}. It generalizes both the $L_2$ distance and the sup-norm error by taking for any $p  \ge 1$, the pseudo distance
$$
d_p(\hat f_n, f)=\int_{\cX}|\hat f_n-f|^p\ud P_X\,.
$$
The choice of $p$ is somewhat arbitrary and mostly employed as a mathematical exercise.
\end{itemize}
Note that these three examples can be split into two families: global   (Sup-norm and $L_p$) and local (pointwise). 

For specific problems, other considerations come into play. For example, if $Y\in \{0,1\}$ is a label, one may be interested in the \emph{classification risk} of a classifier $h\,:\, \cX \to \{0,1\}$. It is defined by
$$
R(h)=\p(Y \neq h(X))\,.
$$
We will not cover this problem in this course.

Finally,  we will devote a large part of these notes to the study of linear models. For such models, $\cX=\R^d$ and $f$ is linear (or affine), i.e., $f(x)=x^\top \theta$ for some unknown $\theta \in \R^d$. In this case, it is traditional to measure error directly on the coefficient $\theta$. For example, if $\hat f_n(x)=x^\top \hat \theta_n$ is a candidate linear estimator, it is customary to measure the distance of $\hat f_n$ to $f$ using a (pseudo-)distance between $\hat \theta_n$ and $\theta$ as long as $\theta$ is identifiable. 

\section*{Models and methods}
\subsection*{Empirical risk minimization}

In our considerations on measuring the performance of an estimator $\hat f_n$, we have carefully avoided the question of how to construct $\hat f_n$. This is of course one of the most important task of statistics. As we will see, it can be carried out in a fairly mechanical way by following one simple principle: \emph{Empirical Risk Minimization} (ERM\footnote{ERM may also mean \emph{Empirical Risk Minimizer}}). Indeed, an overwhelming proportion of statistical methods consist in replacing an (unknown) expected value ($\E$) by a (known) empirical mean ($\frac{1}{n}\sum_{i=1}^n$). For example, it is well known that a good candidate to estimate the expected value $\E X$ of a random variable $X$ from a sequence of i.i.d copies $X_1, \ldots, X_n$ of $X$, is their empirical average 
$$
\bar X=\frac{1}{n}\sum_{i=1}^n X_i\,.
$$
In many instances, it corresponds to the maximum likelihood estimator of $\E X$. Another example is the sample variance, where $\E(X-\E(X))^2$ is estimated by
$$
\frac{1}{n}\sum_{i=1}^n (X_i-\bar X)^2\,.
$$
It turns out that this principle can be extended even if an optimization follows the substitution. Recall that the $L_2$ risk is defined by $R(g)=\E[Y-g(X)]^2$. See the expectation? Well, it can be replaced by an average to form the \emph{empirical risk} of $g$ defined by
$$
R_n(g)=\frac{1}{n}\sum_{i=1}^n \big(Y_i-g(X_i)\big)^2\,.
$$
We can now proceed to minimizing this risk. However, we have to be careful. Indeed, $R_n(g)\ge0$ for all $g$. Therefore any function $g$ such that $Y_i=g(X_i)$ for all $i=1, \ldots, n$ is a minimizer of the empirical risk. Yet, it may not be the best choice (Cf. Figure~\ref{fig:sin}). To overcome this limitation, we need to leverage some prior knowledge on $f$: either it may belong to a certain class $\cG$ of functions (e.g., linear functions) or it is smooth (e.g., the $L_2$-norm of its second derivative is small). In both cases, this extra knowledge can be incorporated to ERM using either a \emph{constraint}:
$$
\min_{g \in \cG}R_n(g)
$$
or a \emph{penalty}:
$$
\min_{g}\Big\{R_n(g) + \mathrm{pen}(g)\Big\}\,,
$$
or both
$$
\min_{g \in \cG}\Big\{R_n(g) + \mathrm{pen}(g)\Big\}\,,
$$
These schemes belong to the general idea of \emph{regularization}. We will see many variants of regularization throughout the course.

Unlike traditional (low dimensional) statistics, \emph{computation} plays a key role in high-dimensional statistics. Indeed, what is the point of describing an estimator with good prediction properties if it takes years to compute it on large datasets? As a result of this observation, much of the modern estimators, such as the Lasso estimator for sparse linear regression can be computed efficiently using simple tools from convex optimization. We will not describe such algorithms for this problem but will comment on the computability of estimators when relevant. 

In particular computational considerations have driven the field of \emph{compressed sensing} that is closely connected to the problem of sparse linear regression studied in these notes. We will only briefly mention some of the results and refer the interested reader to the book~\cite{FouRau13} for a comprehensive treatment.

\begin{figure}[t]
\centering
\includegraphics[width=\textwidth]{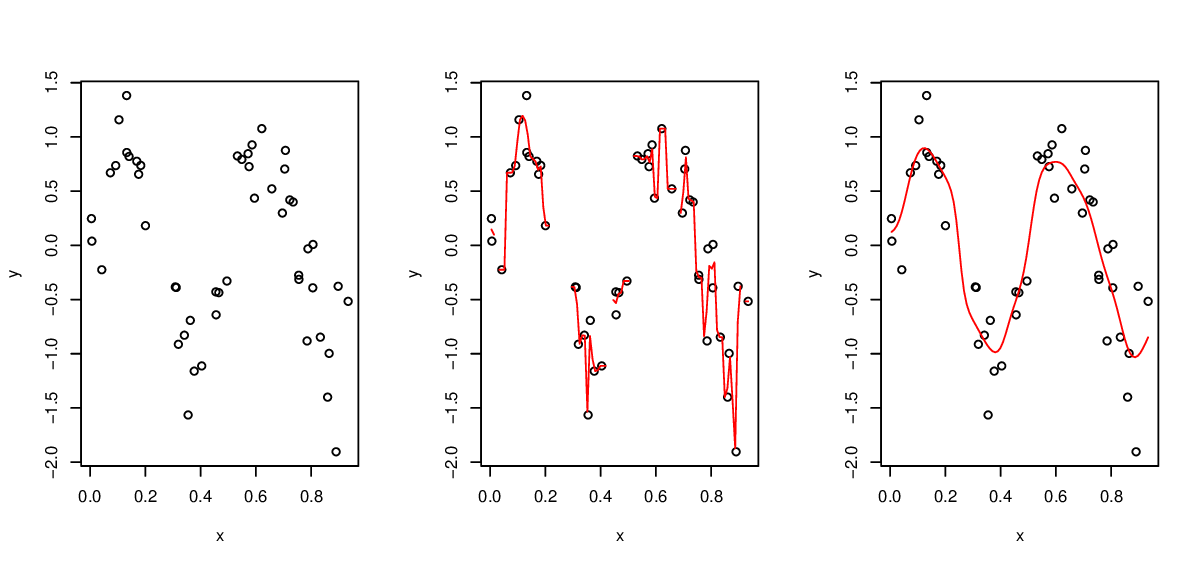}
\caption{It may not be the best choice idea to have $\hat f_n(X_i)=Y_i$ for all $i=1, \ldots, n$.} 
\label{fig:sin}
\end{figure}

\subsection*{Linear models}
When $\cX=\R^d$, an all time favorite constraint $\cG$ is the class of linear functions that are of the form $g(x)=x^\top\theta$, that is parametrized by $\theta \in \R^d$. Under this constraint, the estimator obtained by ERM is usually called \emph{least squares estimator} and is defined by $\hat f_n(x)=x^\top\hat \theta$, where 
$$
\hat \theta\in \argmin_{\theta \in \R^d} \frac{1}{n}\sum_{i=1}^n(Y_i-X_i^\top\theta)^2\,.
$$
Note that $\hat \theta$ may not be unique.  In the case of a \emph{linear model}, where we assume that the regression function is of the form $f(x)=x^\top \theta^*$ for some unknown $\theta^* \in \R^d$, we will need assumptions to ensure \emph{identifiability} if we want to prove bounds on $d(\hat \theta, \theta^*)$ for some specific pseudo-distance $d(\cdot\,,\cdot)$. Nevertheless, in other instances such as regression with fixed design, we can prove bounds on the prediction error that are valid for any $\hat \theta$ in the argmin. In the latter case, we will not even require that $f$ satisfies the linear model but our bound will be meaningful only if $f$ can be well approximated by a linear function. In this case, we talk about a \emph{misspecified model}, i.e., we try to fit a linear model to data that may not come from a linear model. Since linear models can have good approximation properties especially when the dimension $d$ is large, our hope is that the linear model is never too far from the truth.

In the case of a misspecified model, there is no hope to drive the estimation error $d(\hat f_n,f)$ down to zero even with a sample size that tends to infinity. Rather, we will pay a systematic approximation error. When $\cG$ is a linear subspace as above, and the pseudo distance is given by the squared $L_2$ norm $d(\hat f_n,f)=\|\hat f_n-f\|_2^2$, it follows from the Pythagorean theorem that
$$
\|\hat f_n-f\|_2^2=\|\hat f_n-\bar f\|_2^2+ \|\bar f-f\|_2^2\,,
$$
where $\bar f$ is the projection of $f$ onto the linear subspace $\cG$. The systematic approximation error is entirely contained in the \emph{deterministic} term  $\|\bar f-f\|_2^2$ and one can proceed to bound $\|\hat f_n-\bar f\|_2^2$ by a quantity that goes to zero as $n$ goes to infinity. In this case, bounds (e.g., in expectation) on the estimation error take the form
$$
\E\|\hat f_n-f\|_2^2\le \|\bar f-f\|_2^2+  \phi_{n}\,.
$$

The above inequality is called an \emph{oracle inequality}. Indeed, it says that if $\phi_{n}$ is small enough, then $\hat f_n$ the estimator mimics the \emph{oracle} $\bar f$. It is called ``oracle" because it cannot be constructed without the knowledge of the unknown $f$. It is clearly the best we can do when we restrict our attention to estimator in the class $\cG$. Going back to the gap in knowledge between a probabilist who knows the whole joint distribution of $(X,Y)$ and a statistician who only sees the data, the oracle sits somewhere in-between: it can only see the whole distribution through the lens provided by the statistician. In the case above, the lens is that of linear regression functions. Different oracles are more or less powerful and there is a tradeoff to be achieved. On the one hand, if the oracle is weak, then it's easy for the statistician to mimic it but it may be very far from the true regression function; on the other hand, if the oracle is strong, then it is harder to mimic but it is much closer to the truth. 

Oracle inequalities were originally developed as analytic tools to prove adaptation of some nonparametric estimators. With the development of aggregation \cite{Nem00, Tsy03, Rig06} and high dimensional statistics \cite{CanTao07, BicRitTsy09, RigTsy11}, they have become important finite sample results that characterize the interplay between the important parameters of the problem.

In some favorable instances, that is when the $X_i$s enjoy specific properties, it is even possible to estimate the vector $\theta$ accurately, as is done in parametric statistics. The techniques employed for this goal will essentially be the same as the ones employed to minimize the prediction risk. The extra assumptions on the $X_i$s will then translate in interesting properties on $\hat \theta$ itself, including uniqueness on top of the prediction properties of the function $\hat f_n(x)=x^\top \hat \theta$. 

\subsection*{High dimension and sparsity}

These lecture notes are about high-dimensional statistics and it is time they enter the picture. By high dimension, we informally mean that the model has more ``parameters" than there are observations. The word ``parameter" is used here loosely  and a more accurate description is perhaps \emph{degrees of freedom}. For example, the linear model $f(x)=x^\top\theta^*$ has one parameter $\theta^*$ but effectively $d$ degrees of freedom when $\theta^* \in \R^d$. The notion of degrees of freedom is actually well defined in the statistical literature but the formal definition does not help our informal discussion here. 

As we will see in Chapter~\ref{chap:GSM}, if the regression function is linear $f(x)=x^\top \theta^*$, $\theta^* \in \R^d$, and under some assumptions on the marginal distribution of $X$, then the least squares estimator $\hat f_n(x)=x^\top \hat \theta_n$ satisfies
\begin{equation}
\label{EQ:intro_oi_sparse1}
\E\|\hat f_n-f\|_2^2\le  C\frac{d}{n}\,,
\end{equation}
where $C>0$ is a constant and in Chapter~\ref{chap:minimax}, we will show that this cannot be improved apart perhaps for a smaller multiplicative constant. Clearly, such a bound is uninformative if $d\gg n$ and actually, in view of its optimality, we can even conclude that the problem is too difficult statistically. However, the situation is not hopeless if we assume that the problem has actually fewer degrees of freedom than it seems. In particular, it is now standard to resort to the \emph{sparsity} assumption to overcome this limitation.

A vector $\theta \in \R^d$ is said to be $k$-sparse for some $k \in \{0, \dots, d\}$ if it has at most $k$ non-zero coordinates. We denote by $|\theta|_0$ the number of nonzero coordinates of $\theta$, which is also known as sparsity or ``$\ell_0$-norm" though it is clearly not a norm (see footnote~\ref{foot:ellqnorms}). Formally, it is defined as
$$
|\theta|_0=\sum_{j=1}^d\1(\theta_j \neq 0)\,.
$$

Sparsity is just one of many ways to limit the size of the set of potential $\theta$ vectors to consider. One could consider vectors $\theta$ that have the following structure for example (see Figure~\ref{FIG:struct_theta}):
\begin{itemize}
\item Monotonic: $\theta_1 \ge \theta_2 \ge \dots \ge \theta_d$
\item Smooth: $|\theta_i -\theta_j|\le C|i-j|^\alpha$ for some $\alpha>0$
\item Piecewise constant: $\sum_{j=1}^{d-1}\1(\theta_{j+1}\neq\theta_j) \le k$
\item Structured in another basis: $\theta=\Psi \mu$, for some orthogonal matrix and $\mu$ is in one of the structured classes described above.
\end{itemize}
Sparsity plays a significant role in statistics because, often, structure translates into sparsity in a  certain basis. For example, a smooth function is sparse in the trigonometric basis and a piecewise constant function has sparse increments. Moreover, real images are approximately sparse in certain bases such as wavelet or Fourier bases. This is precisely the feature exploited in compression schemes such as JPEG or JPEG-2000: only a few coefficients in these images are necessary to retain the main features of the image. 

We say that $\theta$ is \emph{approximately sparse} if $|\theta|_0$ may be as large as $d$ but many coefficients $|\theta_j|$ are small rather than exactly equal to zero. There are several mathematical ways to capture this phenomenon, including  $\ell_q$-``balls" for $q\le 1$. For $q>0$, the unit $\ell_q$-ball of $\R^d$ is defined as
$$
\cB_q(R)=\big\{\theta \in \R^d\,:\, |\theta|_q^q=\sum_{j=1}^d |\theta_j|^q \le 1\big\}
$$
where $|\theta|_q$ is often called \emph{$\ell_q$}-norm\footnote{Strictly speaking, $|\theta|_q$ is a norm and the $\ell_q$ ball is a ball only for $q\ge 1$.\label{foot:ellqnorms}}. As we will see, the smaller $q$ is, the better vectors in the unit $\ell_q$ ball can be approximated by sparse vectors.
\begin{figure}
  \centering
  \psfrag{y1}[ct][br]{\footnotesize \begin{turn}{-90}$\theta_j$\end{turn}}
    \psfrag{y2}[ct][b]{\footnotesize \begin{turn}{-90}$\theta_j$\end{turn}}
      \psfrag{y3}[ct][b]{\footnotesize \begin{turn}{-90}$\theta_j$\end{turn}}
        \psfrag{y4}[ct][b]{\footnotesize \begin{turn}{-90}$\theta_j$\end{turn}}
\psfrag{x}[cb][t]{\footnotesize $j$}

\psfrag{Monotone}[l]{\footnotesize Monotone}
\psfrag{Smooth}[l]{\footnotesize Smooth}
\psfrag{Constant}[c]{\footnotesize \qquad \qquad Piecewise constant}
\psfrag{Basis}[c]{\footnotesize \qquad Smooth in a different basis}

 \includegraphics[width=\textwidth]{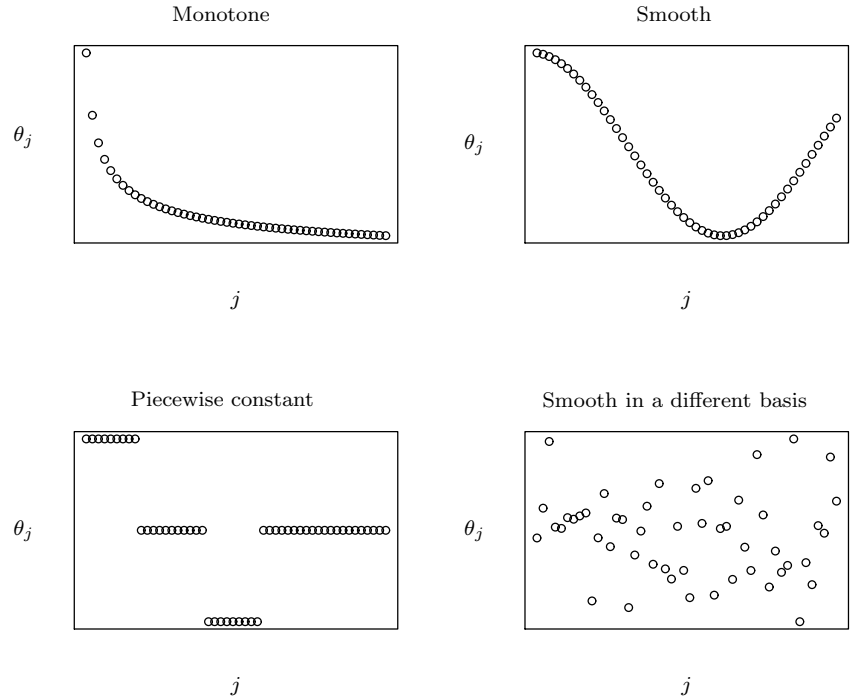}
\caption{Examples of structures vectors $\theta \in \R^{50}$}
\label{FIG:struct_theta}
\end{figure}

Note that the set of $k$-sparse vectors of $\R^d$ is a union of $\sum_{j=0}^k\binom{d}{j}$ linear subspaces with dimension at most $k$ and that are spanned by at most $k$ vectors in the canonical basis of $\R^d$. If we knew that $\theta^*$ belongs to one of these subspaces, we could simply drop irrelevant coordinates and obtain an oracle inequality such as~\eqref{EQ:intro_oi_sparse1}, with $d$ replaced by $k$. Since we do not know what subspace $\theta^*$ lives exactly, we will have to pay an extra term to \emph{find} in which subspace  $\theta^*$ lives. It turns out that this term is exactly of the order of
$$
\frac{\log\Big(\sum_{j=0}^k\binom{d}{j}\Big)}{n}\simeq C \frac{k \log\big(\frac{ed}{k}\big)}{n}
$$
Therefore, the price to pay for not knowing which subspace to look at is only a logarithmic factor.

\subsection*{Nonparametric regression}

Nonparametric does not mean that there is no parameter to estimate (the regression function is a parameter) but rather that the parameter to estimate is infinite-dimensional (this is the case of a function). In some instances, this parameter can be identified with an infinite sequence of real numbers, so that we are still in the realm of countable infinity. Indeed, observe that since $L_2(P_X)$ equipped with the inner product $\langle \cdot\,, \cdot\rangle_2$ is a separable Hilbert space,  it admits an orthonormal basis $\{\varphi_k\}_{k \in \Z}$ and any function $f \in  L_2(P_X)$ can be decomposed as 
$$
f=\sum_{k \in \Z} \alpha_k \varphi_k\,,
$$
where $\alpha_k=\langle f, \varphi_k\rangle_2$.

Therefore estimating a regression function $f$ amounts to estimating the infinite sequence $\{\alpha_k\}_{k \in \Z} \in \ell_2$. You may argue (correctly) that the basis $\{\varphi_k\}_{k \in \Z}$ is also unknown as it depends on the unknown $P_X$. This is absolutely correct but we will make the convenient assumption that $P_X$ is (essentially) known whenever this is needed.

Even if infinity is countable, we still have to estimate an infinite number of coefficients using a finite number of observations. It does not require much statistical intuition to realize that this task is impossible in general. What if we know  something about the sequence $\{\alpha_k\}_{k}$? For example, if we know that $\alpha_k=0$ for $|k| > k_0$, then there are only $2k_0+1$ parameters to estimate (in general, one would also have to ``estimate" $k_0$). In practice, we will not exactly see $\alpha_k=0$ for $|k| > k_0$, but rather that the sequence $\{\alpha_k\}_{k}$ decays to 0 at a certain polynomial rate. For example $|\alpha_k| \le C|k|^{-\gamma}$ for some $\gamma>1/2$ (we need this sequence to be in $\ell_2$). It corresponds to a smoothness assumption on the function $f$. In this case, the sequence $\{\alpha_k\}_{k}$ can be well approximated by a sequence with only a finite number of non-zero terms.

We can view this problem as a misspecified model. Indeed, for any cut-off $k_0$, define the oracle 
$$
\bar f_{k_0} = \sum_{|k| \le k_0} \alpha_k \varphi_k\,.
$$ 
Note that it depends on the unknown $\alpha_k$ and define the estimator 
$$
\hat f_n= \sum_{|k| \le k_0} \hat \alpha_k \varphi_k\,,
$$
where $\hat \alpha_k$ are some data-driven coefficients (obtained by least-squares for example). Then by the Pythagorean theorem and Parseval's identity,  we have
\begin{align*}
\|\hat f_n-f\|_2^2&=\|\bar f-f\|_2^2+ \|\hat f_n-\bar f\|_2^2\\
&= \sum_{|k|> k_0}\alpha_k^2+ \sum_{|k|\le k_0}(\hat \alpha_k-\alpha_k)^2
\end{align*}
We can even work further on this oracle inequality using the fact that $|\alpha_k| \le C|k|^{-\gamma}$. Indeed, we have\footnote{Here we illustrate a convenient notational convention that we will be using throughout these notes: a constant $C$ may be different from line to line. This will not affect the interpretation of our results since we are interested in the order of magnitude of the error bounds. Nevertheless, we will, as much as possible, try to make such constants explicit. As an exercise, try to find an expression of the second $C$ as a function of the first one and of $\gamma$.}
$$
\sum_{|k|> k_0}\alpha_k^2\le C^2\sum_{|k|> k_0}k^{-2\gamma} \le Ck_0^{1-2\gamma}\,.
$$
The so called \emph{stochastic term} $\E\sum_{|k|\le k_0}(\hat \alpha_k-\alpha_k)^2$  clearly increases with $k_0$ (more parameters to estimate) whereas the \emph{approximation term} $Ck_0^{1-2\gamma}$ decreases with $k_0$ (less terms discarded). We will see that  we can strike a compromise called \emph{bias-variance tradeoff}. 

The main difference between this and oracle inequalities is that we make assumptions on the regression function (here in terms of smoothness) in order to control the approximation error. Therefore oracle inequalities are more general but can be seen on the one hand as less quantitative. On the other hand, if one is willing to accept the fact that approximation error is inevitable then there is no reason to focus on it.  This is not the final answer to this rather philosophical question. Indeed, choosing the right $k_0$ can only be done with a control of the approximation error. Indeed, the best $k_0$ will depend on $\gamma$. We will see that even if the smoothness index $\gamma$ is unknown, we can select $k_0$ in a data-driven way that achieves almost the same performance as if $\gamma$ were known. This phenomenon is called \emph{adaptation} (to $\gamma$).

It is important to notice the main difference between the approach taken in nonparametric regression and the one in sparse linear regression. It is not so much about linear vs. non-linear models as we can always first take nonlinear transformations of the $x_j$'s in linear regression. Instead, sparsity or approximate sparsity is a much weaker notion than the decay of coefficients $\{\alpha_k\}_k$ presented above. In a way, sparsity only imposes that \emph{after ordering} the coefficients present a certain decay, whereas in nonparametric statistics, the order is set ahead of time: we assume that we have found a basis that is ordered in such a way that coefficients decay at a certain rate.

\subsection*{Matrix models}
In the previous examples, the response variable is always assumed to be a scalar. What if it is a higher dimensional signal? In Chapter~\ref{chap:matrix}, we consider various problems of this form: matrix completion a.k.a. the Netflix problem, structured graph estimation and covariance matrix estimation. All these problems can be described as follows. 

Let $M, S$ and $N$ be three matrices, respectively called \emph{observation}, \emph{signal} and \emph{noise}, and that satisfy
$$
M=S+N\,.
$$
Here $N$ is a random matrix such that $\E[N]=0$, the all-zero matrix. The goal is to estimate the signal matrix $S$ from the observation of $M$.

The structure of $S$ can also be chosen in various ways. We will consider the case where $S$ is sparse in the sense that it has many zero coefficients. In a way, this assumption does not leverage much of the matrix structure and essentially treats matrices as vectors arranged in the form of an array. This is not the case of \emph{low rank} structures where one assumes that the matrix $S$ has either low rank or can be well approximated by a low rank matrix. This assumption makes sense in the case where $S$ represents user preferences as in the Netflix example. In this example, the $(i,j)$th coefficient $S_{ij}$  of $S$ corresponds to the rating (on a scale  from 1 to 5) that user $i$ gave to movie $j$. The low rank assumption simply materializes the idea that there are a few canonical profiles of users and that each user can be represented as a linear combination of these users. 

At first glance, this problem seems much more difficult than sparse linear regression. Indeed, one needs to learn not only the sparse coefficients in a given basis, but also the basis of eigenvectors. Fortunately, it turns out that the latter task is much easier and is dominated by the former in terms of statistical price.

Another important example of matrix estimation is high-dimensional covariance estimation, where the goal is to estimate the covariance matrix of a random vector $X \in \R^d$, or its leading eigenvectors, based on  $n$ observations. Such a problem has many applications including principal component analysis, linear discriminant analysis and portfolio optimization. The main difficulty is that $n$ may be much smaller than the number of degrees of freedom in the covariance matrix, which can be of order $d^2$. To overcome this limitation, assumptions on the rank or the sparsity of the matrix can be leveraged.

\subsection*{Optimality and minimax lower bounds}

So far, we have only talked about upper bounds. For a linear model, where $f(x)=x^\top \theta^*$, we will prove in Chapter~\ref{chap:GSM} the following bound for a modified least squares estimator $\hat f_n=x^\top\hat \theta$ 
$$
\E\|\hat f_n-f\|_2^2\le C\frac{d}{n}\,.
$$
Is this the right dependence in $d$ and $n$? Would it be possible to obtain as an upper bound something like $C(\log d)/n$, $C/n$ or $\sqrt{d}/n^2$, by either improving our proof technique or using another estimator altogether? It turns out that the answer to this question is negative. More precisely, we can prove that for any estimator $\tilde f_n$, there exists a function $f$ of the form $f(x)=x^\top \theta^*$ such that 
$$
\E\|\hat f_n-f\|_2^2> c\frac{d}{n}
$$
for some positive constant $c$. Here we used a different notation for the constant to emphasize the fact that lower bounds guarantee optimality only \emph{up to a constant} factor. Such a lower bound on the risk is called \emph{minimax lower bound} for reasons that will become clearer in chapter~\ref{chap:minimax}.

How is this possible? How can we make a statement \emph{for all} estimators? We will see that these statements borrow from the theory of tests where we know that it is impossible to drive both the type I and the type II error to zero simultaneously (with a fixed sample size). Intuitively this phenomenon is related to the following observation: Given $n$ observations $X_1, \ldots, X_n$, it is hard to tell if they are distributed according to $\cN(\theta, 1)$ or to $\cN(\theta',1)$ for a Euclidean distance $|\theta-\theta'|_2$ is small enough. We will see that it is the case for example if $|\theta-\theta'|_2\le C\sqrt{d/n}$, which will yield our lower bound.

\renewcommand{\thefigure}{\arabic{chapter}.\arabic{figure}}
\renewcommand{\theequation}{\arabic{chapter}.\arabic{equation}}

\chapter{Sub-Gaussian Random Variables}
\label{chap:subGauss}
\section{Gaussian tails and MGF}
Recall that a random variable $X \in \R$ has Gaussian distribution iff it has a density $p$ with respect to the Lebesgue measure on $\R$ given by
$$
p(x)=\frac{1}{\sqrt{2\pi\sigma^2}}\exp\Big(-\frac{(x-\mu)^2}{2\sigma^2}\Big)\,, \quad x \in \R\,,
$$
where $\mu=\E(X) \in \R$ and $\sigma^2=\var(X)>0$ are the \emph{mean} and \emph{variance} of $X$. We write $X \sim \cN(\mu, \sigma^2)$. Note that $X=\sigma Z+\mu$ for $Z \sim \cN(0,1)$ (called standard Gaussian) and where the equality holds in distribution. Clearly, this distribution has unbounded support but it is well known that it has \emph{almost} bounded support in the following sense: $\p(|X-\mu|\le 3\sigma)\simeq 0.997$. This is due to the fast decay of the tails of $p$ as $|x|\to \infty$ (see Figure~\ref{fig:gaussdens}). This decay can be quantified using the following proposition (Mill's inequality).
\begin{figure}[t]
\centering
\includegraphics[width=0.7\textwidth]{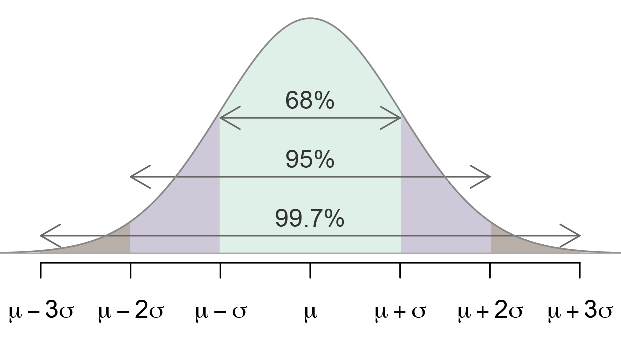}
\caption{Probabilities of falling within 1, 2, and 3 standard deviations close to the mean in a Gaussian distribution. Source \url{http://www.openintro.org/}}
\label{fig:gaussdens}
\end{figure}

\begin{prop}
\label{prop:tail_gauss}
Let $X$ be a Gaussian random variable with mean $\mu$ and variance $\sigma^2$ then for any $t>0$, it holds
$$
\p(X-\mu>t)\le \frac{\sigma}{\sqrt{2\pi}}\frac{\e^{-\frac{t^2}{2\sigma^2}}}{t}\,.
$$
By symmetry we also have
$$
\p(X-\mu<-t)\le \frac{\sigma}{\sqrt{2\pi}}\frac{\e^{-\frac{t^2}{2\sigma^2}}}{t}\,
$$
and
$$
\p(|X-\mu|>t)\le \sigma \sqrt{\frac{2}{\pi}}\frac{\e^{-\frac{t^2}{2\sigma^2}}}{t}\,.
$$
\end{prop}
\begin{proof}
Note that it is sufficient to prove the theorem for $\mu=0$ and $\sigma^2=1$ by simple translation and rescaling. We get for $Z\sim \cN(0,1)$,
\begin{align*}
\p(Z>t)&= \frac{1}{\sqrt{2\pi}}\int_t^\infty \exp\Big(-\frac{x^2}{2}\Big) \ud x\\
&\le\frac{1}{\sqrt{2\pi}}\int_t^\infty \frac{x}{t}\exp\Big(-\frac{x^2}{2}\Big)\ud x\\
&=\frac{1}{t\sqrt{2\pi}}\int_t^\infty -\frac{\partial}{\partial x}\exp\Big(-\frac{x^2}{2}\Big) \ud x\\
&=\frac{1}{t\sqrt{2\pi}}\exp(-t^2/2)\,.
\end{align*}
The second inequality follows from symmetry and the last one using the union bound:
$$
\p(|Z|>t)=\p(\{Z>t\}\cup\{Z<-t\})\le \p(Z>t)+\p(Z<-t)=2\p(Z>t)\,.
$$
\end{proof}
The fact that a Gaussian random variable $Z$ has tails that decay to zero exponentially fast can also be seen in the \emph{moment generating function} (MGF) 
$$
M\,:\,s \mapsto M(s)=\E[\exp(sZ)]\,.
$$
Indeed, in the case of a standard Gaussian random variable, we have
\begin{align*}
 M(s)=\E[\exp(sZ)]&=\frac{1}{\sqrt{2\pi}}\int e^{sz}e^{-\frac{z^2}{2}}\ud z\\
 &=\frac{1}{\sqrt{2\pi}}\int e^{-\frac{(z-s)^2}{2}+\frac{s^2}{2}}\ud z\\
 &=e^{\frac{s^2}{2}}\,.
\end{align*}
It follows that if $X\sim \cN(\mu,\sigma^2)$, then $\E[\exp(sX)]=\exp\big(s\mu+\frac{\sigma^2s^2}{2}\big)$.

\section{Sub-Gaussian random variables and Chernoff bounds}

\subsection{Definition and first properties}

Gaussian tails are practical when controlling the tail of an average of independent random variables. Indeed, recall that if $X_1, \ldots, X_n$ are \\iid $\cN(\mu,\sigma^2)$, then $\bar X=\frac{1}{n}\sum_{i=1}^nX_i\sim\cN(\mu,\sigma^2/n)$. Using Lemma~\ref{lem:subgauss} below for example, we get
$$
\p(|\bar X -\mu|>t)\le 2\exp\big(-\frac{nt^2}{2\sigma^2}\big)\,.
$$
Equating the right-hand side with some confidence level $\delta>0$, we find that with probability at least\footnote{We will often commit the statement ``at least" for brevity} $1-\delta$,
\begin{equation}
\label{EQ:subG:CI}
\mu \in \Big[\bar X - \sigma\sqrt{\frac{2 \log (2/\delta)}{n}}, \bar X + \sigma\sqrt{\frac{2 \log (2/\delta)}{n}}\Big]\,,
\end{equation}
This is almost the confidence interval that you used in introductory statistics. The only difference is that we used an approximation for the Gaussian tail whereas statistical tables or software use a much more accurate computation. Figure~\ref{FIG:conf_int_gauss} shows the ratio of the width of the confidence interval to that of the confidence interval computed by the software R. It turns out that intervals of the same form can be also derived for non-Gaussian random variables as long as they have sub-Gaussian tails.

\begin{figure}
  \centering
  \psfrag{y}[c]{\footnotesize  width }
    \psfrag{x}[cb]{\footnotesize $\delta$ in \%}
 \includegraphics[width=0.7\textwidth]{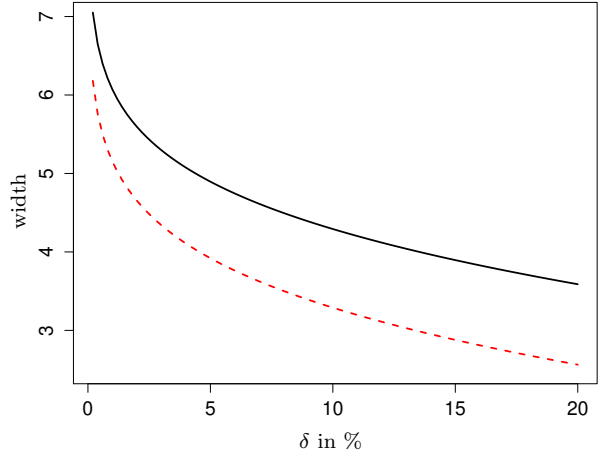}
\caption{Width of confidence intervals from exact computation in R (red dashed) and from the approximation \eqref{EQ:subG:CI} (solid black). Here $x=\delta$ and $y$ is the width of the confidence intervals.}
\label{FIG:conf_int_gauss}
\end{figure}

\begin{defn}
\label{def:subgauss}
A random variable $X \in \R$ is said to be \emph{sub-Gaussian} with variance proxy $\sigma^2$ if $\E[X]=0$ and its moment generating function satisfies
\begin{equation}
\label{EQ:defsubgauss}
\E[\exp(sX)]\le \exp\Big(\frac{\sigma^2s^2}{2}\Big)\,, \quad \forall\, s \in \R\,.
\end{equation}
In this case we write $X \sim \sg(\sigma^2)$. Note that $\sg(\sigma^2)$ denotes a class of distributions rather than a distribution. Therefore, we abuse notation when writing $X \sim \sg(\sigma^2)$.

More generally, we can talk about sub-Gaussian random vectors and matrices.
A random vector $X \in \R^d$ is said to be \emph{sub-Gaussian} with variance proxy $\sigma^2$ if $\E[X]=0$ and   $u^\top X$ is sub-Gaussian with  variance proxy $\sigma^2$ for any vector $u \in \cS^{d-1}$. In this case we write $X \sim \sg_d(\sigma^2)$. Note that if $X \sim \sg_d(\sigma^2)$, then for any $v$ such that $|v|_2\le 1$, we have $v^\top X\sim \sg(\sigma^2)$. Indeed, denoting $u=v/|v|_2 \in  \cS^{d-1}$, we have
$$
\E[e^{sv^\top X}]=\E[e^{s|v|_2u^\top X}]\le e^{\frac{\sigma^2 s^2|v|_2^2}{2}}\le e^{\frac{\sigma^2 s^2}{2}}\,.
$$
 A random matrix $X \in \R^{d\times T}$ is  said to be \emph{sub-Gaussian} with variance proxy $\sigma^2$ if $\E[X]=0$ and   $u^\top X v$ is sub-Gaussian with  variance proxy $\sigma^2$ for any unit vectors $u \in \cS^{d-1}, v \in \cS^{T-1}$. In this case we write $X \sim \sg_{d\times T}(\sigma^2)$.

\end{defn}
This property can equivalently be expressed in terms of bounds on the tail of the random variable $X$.
\begin{lem}
\label{lem:subgauss}
Let $X \sim \sg( \sigma^2)$. Then for any $t>0$, it holds
\begin{equation}
\label{EQ:LEM:subgausstails}
\p[X>t]\le \exp\Big(-\frac{t^2}{2\sigma^2}\Big)\,, \quad \text{and} \quad \p[X<-t]\le \exp\Big(-\frac{t^2}{2\sigma^2}\Big)\,.
\end{equation}
\end{lem}
\begin{proof}
Assume first that  $X \sim \sg( \sigma^2)$. We will employ a very useful technique called \textbf{Chernoff bound} that allows to translate a bound on the moment generating function into a tail bound. Using Markov's inequality, we have for any $s>0$,
$$
\p(X>t)\le \p\big(e^{sX}>e^{st}\big)\le \frac{\E\big[e^{sX}\big]}{e^{st}}\,.
$$
Next we use the fact that $X$ is sub-Gaussian to get
$$
\p(X>t)\le e^{\frac{\sigma^2 s^2}{2}-st}\,.
$$
The above inequality holds for any $s>0$ so to make it the tightest possible, we minimize with respect to $s>0$. Solving $\phi'(s)=0$, where $\phi(s)=\frac{\sigma^2 s^2}{2}-st$, we find that $\inf_{s>0}\phi(s)=-\frac{t^2}{2\sigma^2}$. This proves the first part of~\eqref{EQ:LEM:subgausstails}. The second inequality in this equation follows in the same manner (recall that~\eqref{EQ:defsubgauss} holds for any $s \in \R$).

\end{proof}

\subsection{Moments}
Recall that the absolute moments of $Z \sim \cN(0, \sigma^2)$ are given by
$$
\E[|Z|^k]=\frac{1}{\sqrt\pi}(2\sigma^2)^{k/2}\Gamma\Big(\frac{k+1}{2}\Big)
$$
where $\Gamma(\cdot)$ denotes the Gamma function defined by
$$
\Gamma(t)=\int_0^\infty x^{t-1}e^{-x}\ud x \,, \quad t>0\,.
$$
The next lemma shows that the tail bounds of Lemma~\ref{lem:subgauss} are sufficient to show that the absolute moments of $X \sim \sg( \sigma^2)$ can be bounded by those of $Z\sim \cN(0,\sigma^2)$ up to multiplicative constants.
\begin{lem}
\label{LEM:subgauss_moments}
Let $X$ be a  random variable such that 
$$
\p[|X|>t]\le 2\exp\Big(-\frac{t^2}{2\sigma^2}\Big)\,,
$$
then for any positive integer $k\ge 1$,
$$
\E [|X|^k] \le (2\sigma^2)^{k/2}k\Gamma(k/2)\,.
$$
In particular,
$$
\big(\E [|X|^k])^{1/k} \le \sigma e^{1/e} \sqrt{k}\,, \quad k\ge 2\,.
$$
and $\E[|X|] \le \sigma\sqrt{2\pi}$\,.
\end{lem}
\begin{proof}
\begin{align*}
\E [|X|^k]&=\int_0^\infty \p(|X|^k>t)\ud t&\\
&=\int_0^\infty \p(|X|>t^{1/k})\ud t&\\
&\le 2\int_0^\infty e^{-\frac{t^{2/k}}{2\sigma^2}}\ud t&\\
&=(2\sigma^2)^{k/2}k\int_0^\infty e^{-u}u^{k/2-1}\ud u\,,& u=\frac{t^{2/k}}{2\sigma^2}\\
&=(2\sigma^2)^{k/2}k\Gamma(k/2)&
\end{align*}
The second statement follows from $\Gamma(k/2)\le (k/2)^{k/2}$ and $k^{1/k} \le e^{1/e}$ for any $k \ge 2$. It yields
$$
\big((2\sigma^2)^{k/2}k\Gamma(k/2)\big)^{1/k}\le k^{1/k}\sqrt{\frac{2\sigma^2k}{2}}\le e^{1/e}\sigma\sqrt{k}\,.
$$
Moreover, for $k=1$, we have $\sqrt{2}\Gamma(1/2)=\sqrt{2\pi}$.
\end{proof}
Using moments, we can prove the following reciprocal to Lemma~\ref{lem:subgauss}.
\begin{lem}
	If~\eqref{EQ:LEM:subgausstails} holds and \( \E[X] = 0 \), then for any $s>0$, it holds
$$
\E[\exp(sX)]\le e^{4\sigma^2 s^2}\,.
$$
As a result, we will sometimes write $X\sim \sg(\sigma^2)$ when it satisfies~\eqref{EQ:LEM:subgausstails} and \( \E[X] = 0 \).
\end{lem}
\begin{proof}
We use the Taylor expansion of the exponential function as follows. Observe that by the dominated convergence theorem
\begin{align*}
\E\big[e^{sX}\big]&\le 1+\sum_{k=2}^\infty \frac{s^k\E[|X|^k]}{k!}&\\
&\le 1+\sum_{k=2}^\infty \frac{(2\sigma^2s^2)^{k/2}k\Gamma(k/2)}{k!}&\\
&= 1+\sum_{k=1}^\infty \frac{(2\sigma^2s^2)^{k}2k\Gamma(k)}{(2k)!}+\sum_{k=1}^\infty \frac{(2\sigma^2s^2)^{k+1/2}(2k+1)\Gamma(k+1/2)}{(2k+1)!}&\\
&\le1+\big(2+\sqrt{2\sigma^2s^2}\big)\sum_{k=1}^\infty \frac{(2\sigma^2s^2)^{k}k!}{(2k)!}&\\
&\le1+\big(1+\sqrt{\frac{\sigma^2s^2}{2}}\big)\sum_{k=1}^\infty \frac{(2\sigma^2s^2)^{k}}{k!}&\hspace{-2cm} 2(k!)^2\le(2k)!\\
&=e^{2\sigma^2s^2}+\sqrt{\frac{\sigma^2s^2}{2}}(e^{2\sigma^2s^2}-1)\\
&\le e^{4\sigma^2 s^2}\,.&
\end{align*}
\end{proof}
From the above Lemma, we see that sub-Gaussian random variables can be equivalently defined from their tail bounds and their moment generating functions, up to constants. 

\subsection{Sums of independent sub-Gaussian random variables}

Recall that if $X_1, \ldots, X_n$ are \\iid $\cN(0,\sigma^2)$, then for any $a \in \R^n$, 
$$
\sum_{i=1}^na_iX_i \sim \cN(0,|a|_2^2 \sigma^2).
$$
If we only care about the tails, this property is preserved for sub-Gaussian random variables.
\begin{thm}
\label{TH:proj_subG}
Let $X=(X_1, \ldots, X_n)$ be a vector of independent sub-Gaussian random variables that have variance proxy $\sigma^2$. Then, the random vector $X$ is  sub-Gaussian with variance proxy $\sigma^2$.
\end{thm}
\begin{proof}
Let $u \in \cS^{n-1}$ be a unit vector, then
$$
\E[e^{su^\top X}]=\prod_{i=1}^n \E[e^{su_i X_i}]\le \prod_{i=1}^ne^{\frac{\sigma^2s^2u_i^2}{2}}=e^{\frac{\sigma^2s^2|u|_2^2}{2}}=e^{\frac{\sigma^2s^2}{2}}\,.
$$
\end{proof}
Using a Chernoff bound, we immediately get the following corollary
\begin{cor}
\label{cor:chernoff}
Let $X_1, \ldots, X_n$ be $n$ independent random variables such that $X_i \sim \sg( \sigma^2)$. Then for any $a \in \R^n$, we have
$$
\p\Big[\sum_{i=1}^n a_i X_i>t\Big]\le \exp\Big(-\frac{t^2}{2\sigma^2|a|_2^2}\Big)\,,
$$
and
$$
\p\Big[\sum_{i=1}^n a_i X_i<-t\Big]\le \exp\Big(-\frac{t^2}{2\sigma^2|a|_2^2}\Big)
$$
\end{cor}
Of special interest is the case where $a_i=1/n$ for all $i$. Then, we get that the average $\bar X=\frac{1}{n}\sum_{i=1}^n X_i$, satisfies
$$
\p(\bar X > t) \le e^{-\frac{nt^2}{2\sigma^2}} \quad \text{and}  \quad \p(\bar X < -t) \le e^{-\frac{nt^2}{2\sigma^2}}
$$
just like for the Gaussian average.

\subsection{Hoeffding's inequality}

The class of sub-Gaussian random variables is actually quite large. Indeed, Hoeffding's lemma below implies that all random variables that are bounded uniformly are actually sub-Gaussian with a variance proxy that depends on the size of their support.

\begin{lem}[Hoeffding's lemma (1963)]
Let $X$ be a random variable such that $\E(X)=0$ and $X \in [a,b]$ almost surely. Then, for any $s \in \R$, it holds
$$
\E[e^{sX}]\le e^{\frac{s^2(b-a)^2}{8}}\,.
$$
In particular, $X \sim  \sg(\frac{(b-a)^2}{4})$\,.
\end{lem}
\begin{proof}
Define $\psi(s)=\log \E[e^{s X}]$, and observe that
and we can readily compute
$$
	\psi'(s)=\frac{
	\E[Xe^{s X}]
	}{\E[e^{s X}]},\qquad
	\psi''(s)=
	\frac{
	\E[X^2e^{s X}]
	}{\E[e^{s X}]}-
	\bigg[
	\frac{
	\E[Xe^{s X}]
	}{\E[e^{s X}]}
	\bigg]^2.
$$
Thus $\psi''(s)$ can be interpreted as the variance of the random 
variable $X$ under the probability measure
$d\mathbb{Q}=\frac{e^{s X}}{\E[e^{s X}]}d\p$. But since $X \in [a,b]$ almost surely, we have, under any probability,
$$
\var(X)=\var\big(X-\frac{a+b}{2}\big)\le \E\Big[\Big(X-\frac{a+b}{2}\Big)^2\Big] \le \frac{(b-a)^2}{4}\,.
$$
The fundamental theorem of calculus yields
$$
	\psi(s) = \int_0^s\int_0^\mu \psi''(\rho)\,d\rho\,d\mu
	\le \frac{s^2(b-a)^2}{8}
$$
using $\psi(0)=\log 1=0$ and $\psi'(0)=\E X=0$.
\end{proof}

Using a Chernoff bound, we get the following (extremely useful) result.
\begin{thm}[Hoeffding's inequality]
\label{TH:hoeffding}
Let $X_1, \ldots, X_n$ be $n$ independent random variables such that almost surely,
$$
X_i \in [a_i, b_i]\,, \qquad \forall\, i.
$$
Let $\bar X=\frac{1}{n}\sum_{i=1}^n X_i$, then for any $t>0$,
    $$ \p (\bar X - \E(\bar X) > t) \leq \exp \Big(-\frac{2n^2t^2}{\sum_{i=1}^n (b_i - a_i)^2} \Big),$$
and $$ \p(\bar X - \E(\bar X) <-t) \leq \exp \Big(-\frac{2n^2t^2}{\sum_{i=1}^n (b_i - a_i)^2} \Big)\,.$$
\end{thm}
Note that Hoeffding's lemma holds for \emph{any} bounded random variable. For example, if one knows that $X$ is a Rademacher random variable,
$$
 \E(e^{sX})=\frac{e^s+e^{-s}}{2}=\cosh(s)\le e^{\frac{s^2}{2}}.
$$
Note that 2 is the best possible constant in the above approximation. For such variables, $a=-1, b=1$, and $\E(X)=0$, so Hoeffding's lemma yields 
$$
 \E(e^{sX})\le e^{\frac{s^2}{2}}\,.
$$
Hoeffding's inequality is  very general but there is a price to pay for this generality. Indeed, if the random variables have small variance, we would like to see it reflected in the exponential tail bound (as for the Gaussian case) but the variance does not appear in Hoeffding's inequality. We need a more refined inequality.

\section{Sub-exponential random variables}
What can we say when a centered random variable is not sub-Gaussian? A typical example is the double exponential (or Laplace) distribution with parameter 1, denoted by $\Lap(1)$. Let $X \sim \Lap(1)$ and observe that 
$$
\p(|X|> t)=e^{-t}, \quad t \ge 0\,.
$$
In particular, the tails of this distribution do not decay as fast as the Gaussian ones (that decays as $e^{-t^2/2}$). Such tails are said to be \emph{heavier} than Gaussian. This tail behavior is also captured by the moment generating function of $X$. Indeed, we have
$$
\E\big[e^{sX}\big]=\frac{1}{1-s^2} \quad \text{if} \ |s| <1\,,
$$
and it is not defined for $s \ge 1$. It turns out that a rather weak condition on the moment generating function is enough to partially reproduce some of the bounds that we have proved for sub-Gaussian random variables. Observe that for $X \sim \Lap(1)$
$$
\E\big[e^{sX}\big]\le e^{2s^2}\quad \text{if} \ |s| <1/2\,,
$$
In particular, the moment generating function of the Laplace distribution is bounded by that of a Gaussian in a neighborhood of 0 but does not even exist away from zero. It turns out that all distributions that have tails at most as heavy as that of a Laplace distribution satisfy such a property.
\begin{lem}
Let $X$ be a centered random variable such that $\p(|X|>t)\le 2e^{-2t/\lambda}$ for some $\lambda>0$. Then, for any positive integer $k\ge 1$,
$$
\E [|X|^k] \le \lambda^{k}k!\,.
$$
Moreover, 
$$
\big(\E [|X|^k])^{1/k} \le 2\lambda k\,,
$$
and the moment generating function of $X$ satisfies
$$
\E\big[e^{sX}\big]\le e^{2s^2\lambda^2}\,, \qquad \forall |s| \le \frac{1}{2\lambda}\,.
$$
\end{lem}
\begin{proof}
\begin{align*}
\E [|X|^k]&=\int_0^\infty \p(|X|^k>t)\ud t&\\
&=\int_0^\infty \p(|X|>t^{1/k})\ud t&\\
&\le \int_0^\infty 2e^{-\frac{2t^{1/k}}{\lambda}}\ud t&\\
&=2(\lambda/2)^{k}k\int_0^\infty e^{-u}u^{k-1}\ud u\,,& u=\frac{2t^{1/k}}{\lambda}\\
&\le \lambda^{k}k\Gamma(k)=\lambda^{k}k!&
\end{align*}
The second statement follows from $\Gamma(k)\le k^k$ and $k^{1/k} \le e^{1/e}\le 2$ for any $k \ge 1$. It yields
$$
\big(\lambda^{k}k\Gamma(k)\big)^{1/k}\le 2\lambda k\,.
$$
To control the MGF of $X$, we use the Taylor expansion of the exponential function as follows. Observe that by the dominated convergence theorem, for any $s$ such that $|s|\le 1/2\lambda$
\begin{align*}
\E\big[e^{sX}\big]&\le 1+\sum_{k=2}^\infty \frac{|s|^k\E[|X|^k]}{k!}&\\
&\le 1+\sum_{k=2}^\infty(|s|\lambda)^{k}&\\
&= 1+s^2\lambda^2\sum_{k=0}^\infty(|s|\lambda)^{k}&\\
&\le 1+ 2s^2\lambda^2& |s|\le \frac{1}{2\lambda}\\
&\le e^{2s^2\lambda^2}
\end{align*}
\end{proof}
This leads to the following definition
\begin{defn}
A random variable $X$ is said to be sub-exponential with parameter $\lambda$ (denoted $X \sim \subE(\lambda)$) if $\E[X]=0$ and  its moment generating function satisfies
$$
\E\big[e^{sX}\big]\le e^{s^2\lambda^2/2}\,, \qquad \forall |s| \le \frac{1}{\lambda}\,.
$$
\end{defn}

A simple and useful example of of a sub-exponential random variable is given in the next lemma.

\begin{lem}
\label{LEM:squaredsubG}
Let $X\sim\sg(\sigma^2)$ then the random variable $Z=X^2-\E[X^2]$ is sub-exponential: $Z\sim \subE(16\sigma^2)$.
\end{lem}
\begin{proof}
We have, by the dominated convergence theorem,
\begin{align*}
\E[e^{sZ}]&=1+\sum_{k=2}^\infty\frac{s^k\E\big[X^2-\E[X^2]\big]^k}{k!}\\
&\le 1+\sum_{k=2}^\infty\frac{s^k2^{k-1}\big(\E[X^{2k}]+(\E[X^2])^{k}\big)}{k!}&\text{(Jensen)}\\
&\le 1+\sum_{k=2}^\infty\frac{s^k4^{k}\E[X^{2k}]}{2(k!)}&\text{(Jensen again)}\\
&\le 1+\sum_{k=2}^\infty\frac{s^k4^{k}2(2\sigma^2)^kk!}{2(k!)}&\text{(Lemma~\ref{LEM:subgauss_moments})}\\
&=1+(8s\sigma^2)^2\sum_{k=0}^\infty(8s\sigma^2)^k&\\
&=1+128s^2\sigma^4&\text{for} \quad |s|\le \frac{1}{16\sigma^2}\\
&\le e^{128s^2\sigma^4}\,.
\end{align*}
\end{proof}

Sub-exponential random variables also give rise to exponential deviation inequalities such as Corollary~\ref{cor:chernoff} (Chernoff bound) or Theorem~\ref{TH:hoeffding} (Hoeffding's inequality) for weighted sums of independent sub-exponential random variables. The significant difference here is that the larger deviations are controlled in by a weaker bound.


\subsection{Bernstein's inequality}
\begin{thm}[Bernstein's inequality] \label{TH:Bernstein}
Let $X_1,\ldots,X_n$ be independent random variables such that $\E(X_i) = 0$ and $X_i \sim \subE(\lambda)$. Define
    $$ \bar X = \frac{1}{n}\sum_{i=1}^n X_i\,,$$
 Then for any $t>0$ we have
    $$ 
    \p(\bar X >t) \vee  \p(\bar X <-t)\leq \exp \left[- \frac{n}{2}(\frac{t^2}{\lambda^2}\wedge \frac{t}{\lambda})\right].
    $$
\end{thm}
\begin{proof}
	Without loss of generality, assume that $\lambda=1$ (we can always replace $X_i$ by $X_i/\lambda$ and $t$ by $t/\lambda$). Next, using a Chernoff bound, we get for any $s> 0$
$$
\p(\bar X >t)\le \prod_{i=1}^n\E\big[e^{sX_i}\big]e^{-snt}\,.
$$
Next, if $|s|\le 1$, then $\E\big[e^{sX_i}\big] \le e^{s^2/2}$ by definition of sub-exponential distributions. It yields
$$
\p(\bar X >t)\le e^{\frac{ns^2}{2}-snt}
$$
Choosing $s=1\wedge t$ yields
$$
\p(\bar X >t)\le e^{-\frac{n}{2}(t^2\wedge t)}
$$
We obtain the same bound for $\p(\bar X <-t)$ which concludes the proof.
\end{proof}
Note that usually, Bernstein's inequality refers to a slightly more precise result that is qualitatively the same as the one above: it exhibits a Gaussian tail $e^{-nt^2/(2\lambda^2)}$ and an exponential tail $e^{-nt/(2\lambda)}$. See for example Theorem~2.10 in \cite{BouLugMas13}.

\section{Maximal inequalities}

The exponential inequalities of the previous section are valid for linear combinations of independent random variables, and, in particular, for the average $\bar X$. In many instances, we will be interested in controlling the \emph{maximum} over the parameters of such linear combinations (this is because of empirical risk minimization). The purpose of this section is to present such results. 

\subsection{Maximum over a finite set}

We begin by the simplest case possible: the maximum over a finite set.

\begin{thm}
\label{TH:finitemax}
Let $X_1, \ldots, X_N$ be $N$ random variables such that  
$$
X_i \sim \sg( \sigma^2).
$$
Then,
$$
\E[\max_{1\le i \le N}X_i]\le \sigma\sqrt{2\log (N)}\,, \qquad \text{and}\qquad \E[\max_{1\le i \le N}|X_i|]\le \sigma\sqrt{2\log (2N)}
$$
Moreover, for any $t>0$,
$$
\p\big(\max_{1\le i \le N}X_i>t\big)\le Ne^{-\frac{t^2}{2\sigma^2 }}\,, \qquad \text{and}\qquad \p\big(\max_{1\le i \le N}|X_i|>t\big)\le 2Ne^{-\frac{ t^2}{2\sigma^2}}
$$
\end{thm}
Note that the random variables in this theorem need not be independent.
\begin{proof}
For any $s>0$, 
\begin{align*}
\E[\max_{1\le i \le N}X_i]&=\frac{1}{s}\E\big[\log e^{s\max_{1\le i \le N}X_i}\big]&\\
&\le \frac{1}{s}\log \E\big[e^{s\max_{1\le i \le N}X_i}\big]&\text{(by Jensen)}\\
&= \frac{1}{s}\log\E\big[\max_{1\le i \le N}e^{sX_i}\big]&\\
&\le  \frac{1}{s}\log\sum_{1\le i \le N}\E\big[e^{sX_i}\big]&\\
&\le  \frac{1}{s}\log\sum_{1\le i \le N}e^{\frac{\sigma^2s^2}{2}}&\\
&=\frac{\log N}{s}+ \frac{\sigma^2s}{2}.\\
\end{align*}
Taking $s=\sqrt{2(\log N)/\sigma^2}$ yields the first inequality in  expectation.

The first inequality in probability is obtained by a simple union bound:
\begin{align*}
\p\big(\max_{1\le i \le N}X_i>t\big)&=\p\Big(\bigcup_{1\le i\le N} \{X_i>t\}\Big)\\
&\le \sum_{1\le i \le N}\p(X_i>t)\\
& \le Ne^{-\frac{t^2}{2\sigma^2 }}\,,
\end{align*}
where we used Lemma~\ref{lem:subgauss} in the last inequality.

The remaining two inequalities follow trivially by noting that
$$
\max_{1\le i \le N}|X_i|=\max_{1\le i \le 2N}X_i\,,
$$
where $X_{N+i}=-X_i$ for $i=1, \ldots, N$. 
\end{proof}

Extending these results to a maximum over an infinite set may be impossible. For example, if one is given an infinite sequence of \\iid $\cN(0, \sigma^2)$ random variables $X_1, X_2, \ldots$, then for any $N \ge 1$, we have for any $t>0$,
$$
\p(\max_{1\le i \le N} X_i <t)=[\p(X_1<t)]^N\to 0\,, \quad N \to \infty\,.
$$
On the opposite side of the picture, if all the $X_i$s are equal to the same random variable $X$, we have for any $t>0$,
$$
\p(\max_{1\le i \le N} X_i <t)=\p(X_1<t)>0 \quad \forall\: N \ge 1\,.
$$
In the Gaussian case, lower bounds are also available. They illustrate the effect of the correlation between the $X_i$s. 

Examples from statistics have structure and we encounter many examples where a maximum of random variables over an infinite set is in fact finite. This is due to the fact that the random variables that we are considering are not independent of each other. In the rest of this section, we review some of these examples.

\subsection{Maximum over a convex polytope}
We use the definition of a polytope from \cite{Gru03}: a convex polytope $\mathsf{P}$ is a compact set with a finite number of vertices $\cV(\mathsf{P})$ called extreme points. It satisfies $\mathsf{P}=\conv(\cV(\mathsf{P}))$, where $\conv(\cV(\mathsf{P}))$ denotes the convex hull of the vertices of $\mathsf{P}$.

Let $X \in \R^d$ be a random vector and consider the (infinite) family of random variables 
$$
\cF=\{\theta^\top X\,:\,\theta \in \mathsf{P}\}\,,
$$ 
where $\mathsf{P} \subset \R^d$ is a polytope with $N$ vertices. While the family $\cF$ is infinite, the maximum over $\cF$ can be reduced to the a finite maximum using the following useful lemma.
\begin{lem}
\label{lem:poly}
Consider a linear form $x \mapsto c^\top x$, $x, c \in \R^d$. Then for any convex polytope $\mathsf{P} \subset \R^d$,
$$
\max_{x \in \mathsf{P}}c^\top x=\max_{x \in \cV(\mathsf{P})}c^\top x
$$
where $\cV(\mathsf{P})$ denotes the set of vertices of $\mathsf{P}$.
\end{lem}
\begin{proof}
Assume that $\cV(\mathsf{P})=\{v_1, \ldots, v_N\}$.   For any $x \in \mathsf{P}=\conv(\cV(\mathsf{P}))$, there exist nonnegative numbers $\lambda_1, \ldots \lambda_N$ that sum up to 1 and such that $x=\lambda_1 v_1 + \cdots + \lambda_N v_N$. Thus
$$
c^\top x=c^\top\Big(\sum_{i=1}^N \lambda_i v_i\Big)=\sum_{i=1}^N \lambda_ic^\top v_i\le \sum_{i=1}^N \lambda_i\max_{x \in \cV(\mathsf{P})}c^\top x=\max_{x \in  \cV(\mathsf{P})}c^\top x\,.
$$
It yields
$$
\max_{x \in \mathsf{P}}c^\top x\le \max_{x \in  \cV(\mathsf{P})}c^\top x\le \max_{x \in  \mathsf{P}}c^\top x
$$
so the two quantities are equal.
\end{proof}
It  immediately yields the following theorem
\begin{thm}
\label{TH:polytope}
Let $\mathsf{P}$ be a polytope with $N$ vertices $v^{(1)}, \ldots, v^{(N)} \in \R^d$ and let $X \in \R^d$ be a random vector such that $[v^{(i)}]^\top X$, $i=1, \ldots, N$, are  sub-Gaussian random variables with  variance proxy $\sigma^2$. Then
$$
\E[\max_{\theta \in \mathsf{P}} \theta^\top X]\le \sigma \sqrt{2\log (N)}\,, \qquad \text{and}\qquad \E[\max_{\theta \in \mathsf{P}} |\theta^\top X|]\le \sigma \sqrt{2\log (2N)}\,.
$$
Moreover, for any $t>0$,
$$
\p\big(\max_{\theta \in \mathsf{P}} \theta^\top X>t\big)\le Ne^{-\frac{ t^2}{2\sigma^2}}\,, \qquad \text{and}\qquad \p\big(\max_{\theta \in \mathsf{P}}| \theta^\top X|>t\big)\le 2Ne^{-\frac{t^2}{2\sigma^2}}
$$
\end{thm}
Of particular interest are polytopes that have a small number of vertices. A primary example is the $\ell_1$ ball of $\R^d$ defined  by
$$
\cB_1=\Big\{x \in \R^d\,:\, \sum_{i=1}^d|x_i|\le 1\Big\}\,.
$$
Indeed, it has exactly $2d$ vertices. 

\subsection{Maximum over the $\ell_2$ ball}

Recall that the unit $\ell_2$ ball of $\R^d$ is defined by the set of vectors $u$ that have Euclidean norm $|u|_2$ at most 1. Formally, it is defined by
$$
\cB_2=\Big\{x \in \R^d\,:\, \sum_{i=1}^d x_i^2\le 1\Big\}\,.
$$
Clearly, this ball is not a polytope and yet, we can control the maximum of random variables indexed by $\cB_2$. This is due to the fact that there exists a finite subset of $\cB_2$ such that the maximum over this finite set is of the same order as the maximum over the entire ball. 

\begin{defn}
Fix $K\subset \R^d$ and $\eps>0$. A set $\cN$ is called an $\eps$-net of $K$ with respect to a distance $d(\cdot, \cdot)$ on $\R^d$, if  $\cN \subset K$ and for any $z \in K$, there exists $x \in \cN$ such that $d(x,z)\le \eps$.
\end{defn}
Therefore, if $\cN$ is an $\eps$-net of $K$ with respect to a norm $\|\cdot\|$, then every point of $K$ is at distance at most $\eps$ from a point in $\cN$. Clearly, every compact set admits a finite $\eps$-net. The following lemma gives an upper bound on the size of the smallest $\eps$-net of $\cB_2$.
\begin{lem}
\label{lem:coveringell2}
Fix $\eps\in(0,1)$. Then the unit Euclidean ball $\cB_2$ has an $\eps$-net $\cN$ with respect to the Euclidean distance of cardinality $|\cN|\le (3/\eps)^d$
\end{lem}
\begin{proof}
Consider the following iterative construction if the $\eps$-net. Choose $x_1=0$. For any $i\ge 2$, take $x_i$ to be any $x \in \cB_2$ such that $|x -x_j|_2>\eps$ for all $j<i$. If no such $x$ exists, stop the procedure. Clearly, this will create an $\eps$-net. We now control its size.

Observe that since $|x -y|_2>\eps$ for all $x,y \in \cN$, the Euclidean balls centered at $x \in \cN$ and with radius $\eps/2$ are disjoint. Moreover, 
$$
 \bigcup_{z \in \cN} \{z+\frac{\eps}{2} \cB_2\} \subset (1+\frac{\eps}{2}) \cB_2
$$
where $\{z+\eps\cB_2\}=\{z+\eps x\,, x \in \cB_2\}$. Thus, measuring volumes, we get
$$
\vol\big((1+\frac{\eps}{2}) \cB_2\big)
\ge \vol\Big( \bigcup_{z \in \cN} \{z+\frac{\eps}{2} \cB_2\}\Big)=\sum_{z \in \cN}\vol\big(\{z+\frac{\eps}{2} \cB_2\}\big)
$$
This is equivalent to
$$
\Big(1+\frac{\eps}{2}\Big)^d \ge |\cN|\Big(\frac{\eps}{2}\Big)^d\,.
$$
Therefore, we get the following bound
$$
|\cN|\le \Big(1+\frac{2}{\eps}\Big)^d\le  \Big(\frac{3}{\eps}\Big)^d\,.
$$
\end{proof}

\begin{thm}
\label{TH:supell2}
Let $X \in \R^d$ be a sub-Gaussian random vector with variance proxy $\sigma^2$. Then
$$
\E[\max_{\theta \in \cB_2} \theta^\top X]=\E[\max_{\theta \in \cB_2} |\theta^\top X|]\le  4\sigma\sqrt{d}\,.
$$
Moreover, for any $\delta>0$, with probability $1-\delta$, it holds
$$
\max_{\theta \in \cB_2} \theta^\top X=\max_{\theta \in \cB_2}| \theta^\top X|\le 4\sigma\sqrt{d}+2\sigma\sqrt{2\log(1/\delta)}\,.
$$
\end{thm}
\begin{proof}
Let $\cN$ be a $1/2$-net of $\cB_2$ with respect to the Euclidean norm that satisfies $|\cN|\le 6^d$. Next, observe that for every $\theta \in \cB_2$, there exists $z \in \cN$ and $x$ such that $|x|_2\le 1/2$ and $\theta=z+x$. Therefore,
$$
\max_{\theta \in \cB_2} \theta^\top X\le \max_{z \in \cN}z^\top X+ \max_{x \in \frac12 \cB_2} x^\top X
$$
But 
$$
\max_{x \in \frac12 \cB_2} x^\top X=\frac{1}{2}\max_{x \in  \cB_2} x^\top X
$$
Therefore, using Theorem~\ref{TH:finitemax}, we get
$$
\E[\max_{\theta \in \cB_2} \theta^\top X]\le 2\E[\max_{z \in \cN}z^\top X] \le 2\sigma \sqrt{2\log (|\cN|)}\le 2\sigma \sqrt{2(\log 6)d}\le 4\sigma\sqrt{d}\,.
$$

\medskip

The bound with high probability then follows because
$$
\p\big(\max_{\theta \in \cB_2} \theta^\top X>t\big)\le \p\big(2\max_{z \in \cN}z^\top X>t\big)\le |\cN|e^{-\frac{t^2}{8\sigma^2 }}\le 6^de^{-\frac{t^2}{8\sigma^2 }}\,.
$$
To conclude the proof, we find $t$ such that
$$
e^{-\frac{t^2}{8\sigma^2 }+d\log(6)}\le \delta \ \Leftrightarrow \ t^2 \ge 8\log(6)\sigma^2d+8\sigma^2\log(1/\delta)\,.
$$
Therefore, it is sufficient to take $t=\sqrt{8\log(6)}\sigma\sqrt{d}+2\sigma\sqrt{2\log(1/\delta)}$\,.
\end{proof}

\section{Sums of independent random matrices}

In this section, we are going to explore how concentration statements can be extended to sums of matrices.
As an example, we are going to show a version of Bernstein's inequality, \ref{TH:Bernstein}, for sums of independent matrices, closely following the presentation in \cite{Tro12}, which builds on previous work in \cite{AhlWin02}.
Results of this type have been crucial in providing guarantees for low rank recovery, see for example \cite{Gro11}.

In particular, we want to control the maximum eigenvalue of a sum of independent random symmetric matrices \( \p(\lambda_{\mathrm{max}}(\sum_i \bX_i) > t)\).
The tools involved will closely mimic those used for scalar random variables, but with the caveat that care must be taken when handling exponentials of matrices because \( e^{A + B} \neq e^{A} e^{B} \) if \( A, B \in \R^d \) and \( AB \neq BA \).

\subsection{Preliminaries}

We denote the set of symmetric, positive semi-definite, and positive definite matrices in \( \R^{d \times d} \) by \( \cS_d, \cS_d^{+} \), and \( \cS_d^{++} \), respectively, and will omit the subscript when the dimensionality is clear.
Here, positive semi-definite means that for a matrix \( \bX \in \cS^+ \), \( v^\top \bX v \geq 0 \) for all \( v \in \R^d \), \( |v|_2 = 1 \), and positive definite means that the equality holds strictly.
This is equivalent to all eigenvalues of \( \bX \) being larger or equal (strictly larger, respectively) than 0, \( \lambda_j(\bX) \geq 0 \) for all \( j = 1, \dots, d \).

The cone of positive definite matrices induces an order on matrices by setting \( \bA \preceq \bB \) if \( \bB - \bA \in \cS^+ \).

Since we want to extend the notion of exponentiating random variables that was essential to our derivation of the Chernoff bound to matrices, we will make use of the \emph{matrix exponential} and the \emph{matrix logarithm}.
For a symmetric matrix \( \bX \), we can define a functional calculus by applying a function \( f : \R \to \R \) to the diagonal elements of its spectral decomposition, \ie, if \( \bX = \bQ \Lambda \bQ^\top \), then
\begin{equation}
	\label{eq:funccalc}
	f(\bX) := \bQ f(\Lambda) \bQ^\top,
\end{equation}
where
\begin{equation}
	\label{eq:al}
	[f(\Lambda)]_{i, j} = f(\Lambda_{i, j}), \quad i, j \in [d],
\end{equation}
is only non-zero on the diagonal and \( f(\bX) \) is well-defined because the spectral decomposition is unique up to the ordering of the eigenvalues and the basis of the eigenspaces, with respect to both of which \eqref{eq:funccalc} is invariant, as well.
From the definition, it is clear that for the spectrum of the resulting matrix, 
\begin{equation}
	\label{eq:specthm}
	\sigma(f(\bX)) = f(\sigma(\bX)).
\end{equation}
While inequalities between functions do not generally carry over to matrices, we have the following \emph{transfer rule}:
\begin{equation}
	\label{eq:g}
	f(a) \leq g(a) \text{ for } a \in \sigma(\bX) \implies f(\bX) \preceq g(\bX).
\end{equation}

We can now define the matrix exponential by \eqref{eq:funccalc} which is equivalent to defining it via a power series,
\begin{equation*}
	\exp(\bX) = \sum_{k=1}^{\infty} \frac{1}{k!} \bX^k.
\end{equation*}
Similarly, we can define the matrix logarithm as the inverse function of \( \exp \) on \( \cS \), \( \log(e^{A}) = A \), which defines it on \( \cS^+ \).

Some nice properties of these functions on matrices are their monotonicity:
For \( \bA \preceq \bB \),
\begin{equation}
	\label{eq:h}
	\tr \exp \bA \leq \tr \exp \bB,
\end{equation}
and for \( 0 \prec \bA \preceq \bB \),
\begin{equation}
	\label{eq:i}
	\log \bA \preceq \log \bB.
\end{equation}
Note that the analog of \eqref{eq:i} is in general not true for the matrix exponential.

\subsection{The Laplace transform method}

For the remainder of this section, let \( \bX_1, \dots, \bX_n \in \R^{d \times d} \) be independent random symmetric matrices.

As a first step that can lead to different types of matrix concentration inequalities, we give a generalization of the Chernoff bound for the maximum eigenvalue of a matrix.

\begin{prop}
	\label{prp:laptrafmat}
	Let \( \bY \) be a random symmetric matrix. Then, for all \( t \in \R \),
	\begin{equation*}
		\p ( \lambda_{\mathrm{max}}(\bY) \geq t) \leq \inf_{\theta > 0} \{ e^{- \theta t} \E[ \tr \, e^{\theta \bY} ]\}
	\end{equation*}
\end{prop}

\begin{proof}
	We multiply both sides of the inequality \( \lambda_{\mathrm{max}}(\bY) \geq t \) by \( \theta \), take exponentials, apply the spectral theorem \eqref{eq:specthm} and then estimate the maximum eigenvalue by the sum over all eigenvalues, the trace.
	\begin{align*}
		\label{eq:f}
		\p(  \lambda_{\mathrm{max}}(\bY) \geq t )
		= {} & \p( \lambda_{\mathrm{max}}(\theta \bY) \geq \theta t )\\
		= {} & \p( e^{\lambda_{\mathrm{max}}(\theta \bY)} \geq e^{\theta t} )\\
		\leq {} & e^{-\theta t} \E[ e^{\lambda_{\mathrm{max}}(\theta \bY)} ]\\
		= {} & e^{-\theta t} \E[  \lambda_{\mathrm{max}}(e^{\theta \bY}) ]\\
		\leq {} & e^{-\theta t} \E[ \tr (e^{\theta \bY}) ].
	\end{align*}
\end{proof}

Recall that the crucial step in proving Bernstein's and Hoeffding's inequality was to exploit the independence of the summands by the fact that the exponential function turns products into sums.
\begin{equation*}
	\E[ e^{ \sum_{i} \theta X_i } ] = \E[ \prod_{i} e^{\theta X_i} ] = \prod_i \E [ e^{\theta X_i}].
\end{equation*}
This property, \( e^{A + B} = e^{A}e^{B} \), no longer holds true for matrices, unless they commute.

We could try to replace it with a similar property, the \emph{Golden-Thompson inequality},
\begin{equation*}
	\tr [e^{\theta (\bX_1 + \bX_2)}] \leq \tr [e^{\theta \bX_1} e^{\theta \bX_2}].
\end{equation*}
Unfortunately, this does not generalize to more than two matrices, and when trying to peel off factors, we would have to pull a maximum eigenvalue out of the trace,
\begin{equation*}
	\tr [e^{\theta \bX_1} e^{\theta \bX_2}] \leq  \lambda_{\mathrm{max}}(e^{\theta \bX_2}) \tr [e^{\theta \bX_1}].
\end{equation*}
This is the approach followed by Ahlswede-Winter \cite{AhlWin02}, which leads to worse constants for concentration than the ones we obtain below.

Instead, we are going to use the following deep theorem due to Lieb \cite{Lie73}. A sketch of the proof can be found in the appendix of \cite{Rus02}.

\begin{thm}[Lieb]
	\label{thm:lieb}
	Let \( \bH \in \R^{d \times d} \) be symmetric.
	Then,
	\begin{equation*}
		\cS^{+}_d \to \R, \quad \bA \mapsto \tr e^{\bH + \log \bA}
	\end{equation*}
	is a concave function.
\end{thm}

\begin{cor}
	\label{cor:subaddsingle}
	Let \( \bX, \bH \in \mathcal{S}_d \) be two symmetric matrices such that \( \bX \) is random and \( \bH \) is fixed.
	Then,
	\begin{equation*}
		\E [\tr e^{\bH + \bX}] \leq \tr e^{\bH + \log \E[e^{\bX}]}.
	\end{equation*}
\end{cor}

\begin{proof}
  Write \( \bY = e^{\bX} \) and use Jensen's inequality:
	\begin{align*}
		\label{eq:e}
		\E[\tr e^{\bH+\bX}]
		= {} & \E[ \tr e^{\bH + \log \bY}]\\
		\leq {} & \tr e^{\bH + \log(\E[\bY])}\\
		= {} & \tr e^{\bH + \log \E [e^\bX]}
	\end{align*}
\end{proof}

With this, we can establish a better bound on the moment generating function of sums of independent matrices.

\begin{lem}
	\label{lem:subaddcgf}
	Let \( \bX_1, \dots, \bX_n \) be \( n \) independent, random symmetric matrices.
	Then,
	\begin{equation*}
		\E[ \tr \exp(\sum_i \theta \bX_i)] \leq \tr \exp(\sum_i \log \E e^{\theta \bX_i})
	\end{equation*}
\end{lem}

\begin{proof}
  Without loss of generality, we can assume \( \theta = 1 \).
	Write \( \E_k \) for the expectation conditioned on \( \bX_1, \dots, \bX_k \), \( \E_k[\,.\,] = \E[\,.\,|\bX_1, \dots, \bX_k] \).
	By the tower property,
	\begin{align*}
		\E[\tr \exp (\sum_{i = 1}^{n} \bX_i)]
		= {} & \E_0 \dots \E_{n-1} \tr \exp (\sum_{i = 1}^{n} \bX_i)\\
		\leq {} & \E_0 \dots \E_{n-2} \tr \exp \big( \sum_{i = 1}^{n-1} \bX_i + \log \underbrace{\E_{n-1} e^{\bX_n}}_{= \E e^{\bX_n}} \big)\\
		\vdots\\
		\leq {} & \tr \exp ( \sum_{i = 1}^{n}\log  \E[e^{\bX_i}]),
	\end{align*}
	where we applied Lieb's theorem, Theorem \ref{thm:lieb} on the conditional expectation and then used the independence of \( \bX_1, \dots, \bX_n \).
\end{proof}

\subsection{Tail bounds for sums of independent matrices}

\begin{thm}[Master tail bound]
	\label{thm:mastertailbound}
	For all \( t \in \R \),
	\begin{equation*}
		\p ( \lambda_{\mathrm{max}}( \sum_{i} \bX_i ) \geq t ) \leq \inf_{\theta > 0}\{ e^{-\theta t} \tr \exp ( \sum_{i} \log \E e^{\theta \bX_i})\}
	\end{equation*}
\end{thm}

\begin{proof}
  Combine Lemma \ref{lem:subaddcgf} with Proposition \ref{prp:laptrafmat}.
\end{proof}

\begin{cor}
	\label{cor:tailboundwmaj}
	Assume that there is a function \( g : (0, \infty) \to [0, \infty] \) and fixed symmetric matrices \( \bA_i \) such that \( \E[e^{\theta \bX_i}] \preceq e^{g(\theta) \bA_i} \) for all \( i \).
	Set
	\begin{equation*}
		\rho = \lambda_{\mathrm{max}} ( \sum_{i} \bA_k ).
	\end{equation*}
	Then, for all \( t \in \R \),
	\begin{equation*}
		\p  ( \lambda_{\mathrm{max}}( \sum_{i} \bX_i	) \geq t) \leq d \inf_{\theta > 0}\{ e^{- \theta t + g(\theta) \rho}\}.
	\end{equation*}
\end{cor}

\begin{proof}
  Using the operator monoticity of \( \log \), \eqref{eq:i}, and the monotonicity of \( \tr \exp \), \eqref{eq:h}, we can plug the estimates for the matrix mgfs into the master inequality, Theorem \ref{thm:mastertailbound}, to obtain
	\begin{align*}
		\p( \lambda_{\mathrm{max}}( \sum_{i} \bX_i ) \geq t )
		\leq {} & e^{- \theta t} \tr \exp ( g(\theta) \sum_{i} \bA_i)\\
		\leq {} & d e^{-\theta t} \lambda_{\mathrm{max}} (\exp(g(\theta) \sum_{i}  \bA_i))\\
		= {} & d e^{-\theta t}  \exp(g(\theta) \lambda_{\mathrm{max}} (\sum_{i}  \bA_i)),
	\end{align*}
	where we estimated \( \tr(\bX) = \sum_{j=1}^{d} \lambda_j \leq d \lambda_{\mathrm{max}} \) and used the spectral theorem.
\end{proof}

Now we are ready to prove Bernstein's inequality We just need to come up with a dominating function for Corollary \ref{cor:tailboundwmaj}.

\begin{lem}
	\label{lem:bennettbound}
	If \( \E[\bX] = 0 \) and \( \lambda_{\mathrm{max}}(\bX) \leq 1 \), then
	\begin{equation*}
		\E[e^{\theta \bX}] \preceq \exp((e^\theta - \theta - 1) \E[\bX^2]).
	\end{equation*}
\end{lem}

\begin{proof}
  Define
	\begin{equation*}
		f(x) = \left\{
		\begin{aligned}
			\frac{e^{\theta x}-\theta x - 1}{x^2}, {} \quad & x \neq 0,\\
			\frac{\theta^2}{2}, \quad & x = 0,
		\end{aligned}
		\right.
	\end{equation*}
	and verify that this is a smooth and increasing function on \( \R \).
	Hence, \( f(x) \leq f(1) \) for \( x \leq 1 \).
	By the transfer rule, \eqref{eq:g}, \( f(\bX) \preceq f(I) = f(1) I \).
	Therefore,
	\begin{equation*}
		e^{\theta \bX} = I + \theta \bX + \bX f(\bX) \bX \preceq I + \theta \bX + f(1) \bX^2,
	\end{equation*}
	and by taking expectations,
	\begin{equation*}
		\E [e^{\theta \bX}] \preceq I + f(1) \E[\bX^2] \preceq \exp( f(1) \E[\bX^2]) = \exp((e^{\theta} - \theta - 1) \E[\bX^2]).
	\end{equation*}
\end{proof}

\begin{thm}[Matrix Bernstein]
	\label{thm:matbernstein}
	Assume \( \E \bX_i = 0 \) and \(  \lambda_{\mathrm{max}}(\bX_i) \leq R \) almost surely for all \( i \) and set
	\begin{equation*}
		\sigma^2 = \| \sum_{i} \E[\bX_i^2] \|_{\mathrm{op}}.
	\end{equation*}
	Then, for all \( t \geq 0 \),
	\begin{align*}
		\p ( \lambda_{\mathrm{max}}( \sum_{i} \bX_i) \geq t )
		\leq {} & d \exp(- \frac{\sigma^2}{R^2} h(Rt/\sigma^2))\\
		\leq {} & d \exp \left( \frac{-t^2/2}{\sigma^2 + Rt/3} \right)\\
		\leq {} & \left\{
		\begin{aligned}
			d \exp\left(-\frac{3t^2}{8 \sigma^2}\right), \quad & t \leq \frac{\sigma^2}{R},\\
			d \exp\left(-\frac{3t}{8R}\right), \quad & t \geq \frac{\sigma^2}{R},
		\end{aligned}
		\right.
	\end{align*}
	where \( h(u) = (1+u) \log (1+u) - u \), \( u > 0 \).
\end{thm}

\begin{proof}
	Without loss of generality, take \( R = 1 \).
	By Lemma \ref{lem:bennettbound},
	\begin{equation*}
		\E[e^{\theta \bX_i}] \preceq \exp(g(\theta) \E[\bX_i^2]), \quad \text{with } g(\theta) = e^\theta - \theta - 1, \quad \theta > 0.
	\end{equation*}
	By Corollary \ref{cor:tailboundwmaj},
	\begin{equation*}
		\p( \lambda_{\mathrm{max}}( \sum_{i} \bX_i) \geq t) \leq d \exp(-\theta t + g(\theta) \sigma^2).
	\end{equation*}
	We can verify that the minimal value is attained at \( \theta = \log(1+t/\sigma^2) \).

	The second inequality then follows from the fact that \( h(u) \geq h_2(u) = \frac{u^2/2}{1+u/3} \) for \( u \geq 0 \) which can be verified by comparing derivatives.

	The third inequality follows from exploiting the properties of \( h_2 \).
\end{proof}

With similar techniques, one can also prove a version of Hoeffding's inequality for matrices, see \cite[Theorem 1.3]{Tro12}.

\newpage
\section{Problem set}

\medskip

\begin{exercise}
\label{EXO:bernstein}
Let $X_1,\ldots,X_n$ be independent random variables such that $\E(X_i) = 0$ and $X_i \sim \subE(\lambda)$. For any vector $a=(a_1, \ldots, a_n)^\top \in \R^n$, define the weighted sum
    $$S(a) = \sum_{i=1}^n a_iX_i\,,$$
Show that for any $t>0$ we have
    $$ 
    \p(|S(a)| >t) \leq 2\exp \left[- C\big(\frac{t^2}{\lambda^2|a|_2^2}\wedge \frac{t}{\lambda|a|_\infty}\big)\right].
    $$
    for some positive constant $C$. 
\end{exercise}

\medskip

\begin{exercise}
\label{EXO:chi2}
A random variable $X$ has $\chi^2_n$ (chi-squared with $n$ degrees of freedom) if it has the same distribution as $Z_1^2+ \ldots +Z_n^2$, where $Z_1, \ldots, Z_n$ are \iid $\cN(0,1)$.
\begin{enumerate}[label=(\alph*)]
\item Let $Z \sim \cN(0,1)$. Show that the moment generating function of $Y=Z^2-1$ satisfies
$$
\phi(s):=E\big[e^{sY}\big]=\left\{
\begin{array}{ll}
\displaystyle\frac{e^{-s}}{\sqrt{1-2s}}& \text{if } s<1/2\\
 \infty & \text{otherwise}
\end{array}\right.
$$
\item Show that for all $0<s<1/2$,
$$
\phi(s)\le \exp\Big(\frac{s^2}{1-2s}\Big)\,.
$$
\item Conclude that 
$$
\p(Y>2t+2\sqrt{t})\le e^{-t}
$$
\texttt{[Hint: you can use the convexity inequality $\sqrt{1+u}\le 1+u/2$]}.
\item Show that if $X \sim \chi^2_n$, then, with probability at least $1-\delta$, it holds
$$
X \le n+ 2 \sqrt{n\log(1/\delta)}+ 2\log(1/\delta) \,.
$$
\end{enumerate}
\end{exercise}
\medskip

\begin{exercise}
\label{EXO:hetero}
Let $X_1, X_2 \ldots$ be an infinite sequence of sub-Gaussian random variables with variance proxy $\sigma_i^2=C(\log i)^{-1}$. Show that for $C$ large enough, we get
$$
\E\big[\max_{i\ge 2} X_i \big]<\infty\,.
$$
\end{exercise}

\medskip

\begin{exercise}
\label{EXO:randmat}
Let $A=\{A_{i,j}\}_{\substack{1\le i\le n \\ 1\le j \le m}}$ be a random matrix such that its entries are \iid sub-Gaussian random variables with variance proxy $\sigma^2$.
\begin{enumerate}[label=(\alph*)]
\item Show that the matrix $A$ is sub-Gaussian. What is its variance proxy?
\item Let $\|A\|$ denote the operator norm of $A$ defined by				
$$
\max_{x \in \R^m}\frac{|Ax|_2}{|x|_2}\,.
$$
Show that there exits a constant $C>0$ such that
$$
\E\|A\|\le C(\sqrt{m}+\sqrt{n})\,.
$$
\end{enumerate}
\end{exercise}

\medskip

\begin{exercise}
\label{EXO:maxellp}
Recall that for any $q \ge 1$, the $\ell_q$ norm of a vector $x \in \R^n$ is defined by
$$
|x|_q=\Big(\sum_{i=1}^n |x_i|^q\Big)^{\frac1q}\,.
$$
Let $X=(X_1, \ldots, X_n)$ be a vector with independent entries such that $X_i$ is sub-Gaussian with variance proxy $\sigma^2$ and $\E(X_i)=0$. 
\begin{enumerate}[label=(\alph*)]
\item Show that for any $q\ge2$, and any $x \in \R^d$, 
$$
|x|_2\le |x|_qn^{\frac12-\frac1q}\,,
$$
and prove that the above inequality cannot be improved
\item Show that for for any $q >1$, 
$$
\E|X|_q\le 4\sigma  n^{\frac{1}{q}}\sqrt{q}
$$
\item Recover from this bound that
$$
\E\max_{1\le i\le n} |X_i|\le 4e\sigma\sqrt{\log n}\,.
$$
\end{enumerate}
\end{exercise}

\medskip

\begin{exercise}
\label{EXO:epsnet}
Let $K$ be a compact subset of the unit sphere of $\R^p$ that admits an $\eps$-net $\cN_\eps$ with respect to the Euclidean distance of $\R^p$ that satisfies
$|\cN_\eps|\le   (C/\eps)^d$ for all $\eps \in (0,1)$. Here $C \ge 1$ and $d\le p$ are positive constants. Let $X \sim \sg_p(\sigma^2)$ be a centered random vector.

Show that there exists positive constants $c_1$ and $c_2$ to be made explicit such that for any $\delta \in (0,1)$, it holds
$$
\max_{\theta \in K} \theta^\top X \le c_1\sigma\sqrt{d \log (2p/d)} +c_2\sigma\sqrt{ \log(1/\delta)}
$$
with probability at least $1-\delta$. Comment on the result in light of Theorem~\ref{TH:supell2}\,.
\end{exercise}

\medskip

\begin{exercise}
\label{EXO:packingduality}
For any $K \subset \R^d$, distance $d$ on $\R^d$ and $\eps>0$, the $\eps$-covering number $C(\eps)$ of $K$ is the cardinality of the smallest $\eps$-net of $K$. The $\eps$-packing number $P(\eps)$ of $K$ is the cardinality of the largest set $\cP \subset K$ such that $d(z, z')> \eps$ for all $z,z' \in \cP$, $z \neq z'$. Show that
$$
C(2\eps)\le P(2\eps)\le C(\eps)\,.
$$
\end{exercise}

\medskip

\begin{exercise}
\label{EXO:medofmeans}
Let $X_1, \ldots, X_n$ be $n$ independent and  random variables such that $\E[X_i]=\mu$ and $\var(X_i)\le \sigma^2$.
Fix $\delta \in (0,1)$ and assume without loss of generality that $n$ can be factored into $n=K\cdot G$ where $G=8\log(1/\delta)$ is a positive integers. 

For $g=1,\ldots, G$, let $\bar X_g$ denote the average over the $g$th group of $k$ variables. Formally  
$$
\bar X_g=\frac{1}{k}\sum_{i=(g-1)k+1}^{gk}X_i\,.
$$ 
\begin{enumerate}
\item Show that for any $g= 1, \ldots, G$,
$$
\p\big[\bar X_g - \mu >\frac{2\sigma}{\sqrt{k}}\big] \le \frac{1}{4}\,.
$$
\item Let $\hat \mu$ be defined as the median of $\{\bar X_1, \ldots, \bar X_G\}$. Show that
$$
\p\big[\hat \mu -\mu > \frac{2\sigma}{\sqrt{k}}\big] \le \p\big[\cB \ge \frac{G}{2}\big]\,,
$$
where $\cB\sim \Bin(G, 1/4)$.
\item Conclude that
$$
\p\big[\hat \mu -\mu > 4\sigma\sqrt{\frac{2\log (1/\delta)}{n}}\big] \le \delta
$$
\item Compare this result with Corollary~\ref{cor:chernoff} and Lemma~\ref{lem:subgauss}. Can you conclude that $\hat \mu -\mu \sim \sg(\bar\sigma^2/n)$ for some $\bar\sigma^2$? Conclude.
\end{enumerate}
\end{exercise}

\begin{exercise}
	\label{EXO:jl}
	The goal of this problem is to prove the following theorem:
	\begin{thm}[Johnson-Lindenstrauss Lemma]
		\label{thm:jl}
		Given \( n \) points denoted by \( X = \{x_1, \dots, x_n\} \) in \( \R^d \), let \( Q \in \R^{k \times d} \) be a random projection operator and set \( P := \sqrt{\frac{d}{k}} Q \).
		There is a constant \( C > 0 \) such that if
		\begin{equation*}
			k \geq \frac{C}{\varepsilon^2} \log n,
		\end{equation*}
		\( P \) is an \emph{\( \varepsilon \)-isometry} for \( X \), \ie
		\begin{equation*}
			(1 - \varepsilon) \| x_i - x_j \|_2^2 \leq \| P x_i - P x_j \|_2^2 \leq (1 + \varepsilon) \|x_i - x_j\|_2^2, \quad \text{for all } i, j
		\end{equation*}
		with propability at least \( 1 - 2 \exp(-c \varepsilon^2 k) \).
	\end{thm}

	\begin{enumerate}
	\item Convince yourself that if \( d > n \), there is a projection \( P \in \R^{n \times d} \) to an \( n \) dimensional subspace  such that \( \| P x_i - P x_j \|_2 = \| x_i - x_j \|_2 \), \ie pairwise distances are exactly preserved.
	\end{enumerate}

	Let \( k \leq d  \) be two integers, \( Y = (y_1, \dots, y_d) \sim \cN(0, I_{d \times d} ) \) independent and identically distributed Gaussians and \( Q \in \R^{d \times k} \) the projection onto the first \( k \) coordinates, \ie \( Qy = (y_1, \dots, y_k) \).
	Define \( Z = \frac{1}{\|Y\|}QY = \frac{1}{\|Y\|} (y_1, \dots, y_k) \) and \( L = \|Z\|^2 \).

	\begin{enumerate}[resume]
	\item Show that \( \E[L] = \frac{k}{d} \).

	\item Show that for all \( t > 0 \) such that \( 1 - 2t(k \beta - d) > 0 \) and \( 1 - 2t \beta k > 0 \),
	\begin{equation*}
		\label{eq:a}
		\p \left( \sum_{i = 1}^{k} y_i^2 \leq \beta \frac{k}{d} \sum_{i = 1}^{d} y_i^2 \right) \leq (1 - 2t(k \beta - d))^{-k/2} (1 - 2t \beta k)^{-(d-k)/2}
	\end{equation*}
	(Hint: Show that \( \E[\e^{\lambda X^2}] = \frac{1}{\sqrt{1 - 2 \lambda}} \) for \( \lambda < \frac{1}{2} \) if \( X \sim \cN(0,1) \).)

	\item Conclude that for \( \beta < 1 \),
	\begin{equation*}
		\label{eq:b}
		\p\left(L \leq \beta \frac{k}{d}\right) \leq \exp \left( \frac{k}{2} (1 - \beta + \log \beta) \right).
	\end{equation*}

	\item Similarly, show that for \( \beta > 1 \),
	\begin{equation*}
		\label{eq:c}
		\p\left(L \geq \beta \frac{k}{d}\right) \leq \exp \left( \frac{k}{2} (1 - \beta + \log \beta) \right).
	\end{equation*}

	\item Show that for a random projection operator \( Q \in \R^{k \times d} \) and a fixed vector \( x \in \R^d \),
		\begin{enumerate}
			\item \( \E[\|Qx\|^2] = \frac{k}{d} \|x\|^2 \).
			\item For \( \varepsilon \in (0, 1) \), there is a constant \( c > 0 \) such that with probability at least \( 1 - 2 \exp(-c k \varepsilon^2) \),
	\begin{equation*}
		\label{eq:d}
		(1 - \varepsilon) \frac{k}{d} \| x \|^2 \leq \| Q x \|_2^2 \leq (1 + \varepsilon) \frac{k}{d} \|x\|_2^2.
	\end{equation*}
	(Hint: Think about how to apply the previous results in this case and use the inequalities \( \log (1-\varepsilon) \leq -\varepsilon - \varepsilon^2/2 \) and \( \log (1+\varepsilon) \leq \varepsilon - \varepsilon^2/2 + \varepsilon^3/3 \).)
\end{enumerate}

	\item Prove Theorem \ref{thm:jl}.
\end{enumerate}
\end{exercise}

\chapter{Linear Regression Model}
\label{chap:GSM}

\newcommand{\thetahard}{\hat \theta^{\textsc{hrd}}}
\newcommand{\thetasoft}{\hat \theta^{\textsc{sft}}}
\newcommand{\thetabic}{\hat \theta^{\textsc{bic}}}
\newcommand{\thetalasso}{\hat \theta^{\cL}}
\newcommand{\thetaslope}{\hat \theta^{\cS}}

\newcommand{\thetals}{\hat \theta^{\textsc{ls}}}
\newcommand{\thetalsm}{\tilde \theta^{\textsc{ls}}_X}
\newcommand{\thetaridge}{\hat\theta^{\mathrm{ridge}}_\tau}
\newcommand{\thetalsK}{\hat \theta^{\textsc{ls}}_K}
\newcommand{\muls}{\hat \mu^{\textsc{ls}}}

In this chapter, we consider the following regression model:
\begin{equation}
\label{EQ:regmod}
Y_i=f(X_i)+\eps_i,\quad i=1, \ldots, n\,,
\end{equation}
where $\eps=(\eps_1, \ldots, \eps_n)^\top $ is sub-Gaussian with variance proxy $\sigma^2$ and such that $\E[\eps]=0$. Our goal is to estimate the function $f$ under a linear assumption. Namely, we assume that $x \in \R^d$ and $f(x)=x^\top\theta^*$ for some unknown $\theta^* \in \R^d$.

\section{Fixed design linear regression}
\label{SEC:fixed_Vs_random}
Depending on the nature of the \emph{design} points $X_1, \ldots, X_n$, we will favor a different measure of risk. In particular, we will focus either on \emph{fixed} or \emph{random} design.

\subsection{Random design}
The case of random design corresponds to the statistical learning setup. Let $(X_1,Y_1), \ldots, (X_{n+1}, Y_{n+1})$ be $n+1$ \iid random couples. Given the pairs $(X_1,Y_1), \ldots, (X_{n}, Y_{n})$, the goal is construct a function $\hat f_n$ such that $\hat f_n(X_{n+1})$ is a good predictor of $Y_{n+1}$. Note that when $\hat f_n$ is constructed, $X_{n+1}$ is still unknown and we have to account for what value it is likely to take.

 Consider the following example from~\cite[Section~3.2]{HasTibFri01}. The response variable $Y$ is the log-volume of a cancerous tumor, and the goal is to predict it based on $X \in \R^6$, a collection of variables that are easier to measure (age of patient, log-weight of prostate, \ldots). Here the goal is clearly to construct $f$ for \emph{prediction} purposes. Indeed, we want to find an automatic mechanism that outputs a good prediction of the log-weight of the tumor given certain inputs for a new (unseen) patient. 
 
A natural measure of performance here is the $L_2$-risk employed in the introduction:
$$
R(\hat f_n)=\E[Y_{n+1}-\hat f_n(X_{n+1})]^2=\E[Y_{n+1}-f(X_{n+1})]^2+\|\hat f_n -f\|^2_{L^2(P_X)}\,,
$$
where $P_X$ denotes the marginal distribution of $X_{n+1}$. It measures how good the prediction of $Y_{n+1}$ is in average over realizations of $X_{n+1}$. In particular, it does not put much emphasis on values of $X_{n+1}$ that are not very likely to occur.

Note that if the $\eps_i$ are random variables with variance $\sigma^2$, then one simply has $R(\hat f_n)=\sigma^2+\|\hat f_n -f\|^2_{L^2(P_X)}$. Therefore, for random design, we will focus on the squared $L_2$ norm $\|\hat f_n -f\|^2_{L^2(P_X)}$ as a measure of accuracy. It measures how close $\hat f_n$ is to the unknown $f$ \emph{in average} over realizations of $X_{n+1}$.

\subsection{Fixed design}
In fixed design, the points (or vectors) $X_1, \ldots, X_n$ are \emph{deterministic}. To emphasize this fact, we use lowercase letters $x_1, \ldots, x_n$ to denote fixed design. Of course, we can always think of them as realizations of a random variable but the distinction between fixed and random design is deeper and significantly affects our measure of performance. Indeed, recall that for random design, we look at the performance \emph{in average} over realizations of $X_{n+1}$. Here, there is no such thing as a marginal distribution of $X_{n+1}$.
Rather, since the design points $x_1, \ldots, x_n$ are considered deterministic, our goal is to estimate $f$ \emph{only} at these points. This problem is sometimes called \emph{denoising} since our goal is to recover $f(x_1), \ldots, f(x_n)$ given noisy observations of these values. 

In many instances, fixed designs exhibit particular structures.
A typical example is the \emph{regular design} on $[0,1]$, given by $x_i=i/n$, $i=1, \ldots, n$. Interpolation between these points is possible under smoothness assumptions.

Note that in fixed design, we observe $\mu^*+\eps$, where $\mu^*=\big(f(x_1), \ldots, f(x_n)\big)^\top \in \R^n$ and $\eps=(\eps_1, \ldots, \eps_n)^\top \in \R^n$  is sub-Gaussian with variance proxy $\sigma^2$. Instead of a functional estimation problem, it is often simpler to view this problem as a vector problem in $\R^n$. This point of view will allow us to leverage the Euclidean geometry of $\R^n$.

In the case of fixed design, we will focus on the \emph{Mean Squared Error} (MSE) as a measure of performance. It is defined by
$$
\MSE(\hat f_n)=\frac{1}{n}\sum_{i=1}^n \big(\hat f_n(x_i)-f(x_i)\big)^2\,.
$$
Equivalently, if we view our problem as a vector problem, it is defined by 
$$
\MSE(\hat \mu)=\frac{1}{n}\sum_{i=1}^n \big(\hat \mu_i -\mu^*_i\big)^2=\frac{1}{n} |\hat \mu -\mu^*|_2^2\,.
$$
Often, the design vectors $x_1, \ldots, x_n \in \R^d$ are stored in an $n\times d$ design matrix $\X$, whose $j$th row is given by $x_j^\top$. With this notation, the linear regression model can be written as
\begin{equation}
\label{EQ:regmod_matrix}
Y = \X \theta^* +\eps\,,
\end{equation}
where $Y=(Y_1, \ldots, Y_n)^\top$ and $\eps=(\eps_1, \ldots, \eps_n)^\top$. Moreover,
\begin{equation}
\label{EQ:mse_matrix}
\MSE(\X\hat \theta)=\frac{1}{n} |\X(\hat \theta -\theta^*)|_2^2=(\hat \theta -\theta^*)^\top \frac{\X^\top \X}{n} (\hat \theta -\theta^*)\,.
\end{equation}

A natural example of fixed design regression is image denoising. Assume that $\mu^*_i, i \in 1, \ldots, n$ is the grayscale value of pixel $i$ of an image. We do not get to observe the image $\mu^*$ but rather a noisy version of it $Y=\mu^*+\eps$. Given a library of $d$ images $\{x_1, \ldots, x_d\}, x_j \in \R^{n}$, our goal is to recover the original image $\mu^*$ using linear combinations of the images $x_1, \ldots, x_d$. This can be done fairly accurately (see Figure~\ref{FIG:number6}).

\begin{figure}[h] \centering
\includegraphics[width=0.3\textwidth]{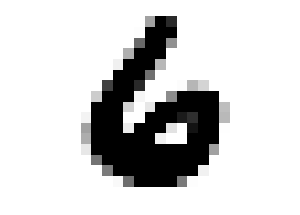}%
\includegraphics[width=0.3\textwidth]{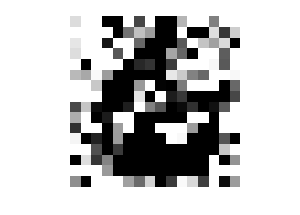}%
\includegraphics[width=0.3\textwidth]{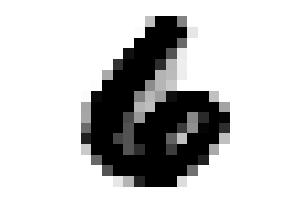}%
\caption{Reconstruction of the digit ``6": Original (left), Noisy (middle) and Reconstruction (right). Here $n=16\times 16=256$ pixels. Source~\cite{RigTsy11}.} \label{FIG:number6}
\end{figure}

As we will see in Remark~\ref{rem:lam_min}, properly choosing the design also ensures that if $\MSE(\hat f)$ is small for some linear estimator $\hat f(x)=x^\top \hat \theta$, then $|\hat \theta -\theta^*|_2^2$ is also small.

\begin{center}
\MIT{\framebox{\textbf{In this chapter we only consider the fixed design case.}}}
\end{center}

\section{Least squares estimators}

Throughout this section, we consider the regression model~\eqref{EQ:regmod_matrix} with fixed design.

\subsection{Unconstrained least squares estimator}

Define the (unconstrained) \emph{least squares estimator} $\thetals$ to be any vector such that
$$
\thetals \in \argmin_{\theta \in \R^d}|Y-\X\theta|_2^2\,.
$$
Note that we are interested in estimating $\X\theta^*$ and not $\theta^*$ itself, so by extension, we also call $\muls=\X\thetals$ least squares estimator. Observe that $\muls$ is the projection of $Y$ onto the column span of $\X$.

It is not hard to see that least squares estimators of $\theta^*$ and $\mu^*=\X\theta^*$ are maximum likelihood estimators when $\eps \sim \cN(0,\sigma^2I_n)$. 

\begin{prop}
The least squares estimator $\muls=\X\thetals \in \R^n$ satisfies
$$
\X^\top \muls=\X^\top Y\,.
$$
Moreover, $\thetals$ can be chosen to be
$$
\thetals = (\X^\top \X)^\dagger\X^\top Y\,,
$$
where $(\X^\top \X)^\dagger$ denotes the Moore-Penrose pseudoinverse of $\X^\top \X$.
\end{prop}
\begin{proof}
The function $\theta \mapsto |Y-\X\theta|_2^2$ is convex so any of its minima satisfies
$$
\nabla_\theta |Y-\X\theta|_2^2=0,
$$
where $\nabla_\theta$ is the gradient operator. Using matrix calculus, we find
$$
\nabla_\theta |Y-\X\theta|_2^2=\nabla_\theta\big\{|Y|_2^2 -2Y^\top \X \theta + \theta^\top \X^\top \X\theta\big\}=-2(Y^\top\X-\theta^\top \X^\top \X)^\top\,.
$$
Therefore, solving $\nabla_\theta |Y-\X\theta|_2^2=0$ yields
$$
\X^\top\X \theta=\X^\top Y\,.
$$
It concludes the proof of the first statement. The second statement follows from the definition of the Moore-Penrose pseudoinverse.
\end{proof}
We are now going to prove our first result on the finite sample performance of the least squares estimator for fixed design.
\begin{thm}
\label{TH:lsOI}
Assume that the linear model~\eqref{EQ:regmod_matrix} holds where $\eps\sim \sg_n(\sigma^2)$. Then the least squares estimator $\thetals$ satisfies
$$
\E\big[\MSE(\X\thetals)\big]=\frac{1}{n}\E|\X\thetals - \X\theta^*|_2^2 \lesssim \sigma^2 \frac{r}{n}\,,
$$
where $r=\rank(\X^\top \X)$. Moreover, for any $\delta>0$, with probability at least $1-\delta$, it holds
$$
\MSE(\X\thetals)\lesssim \sigma^2 \frac{r+\log(1/\delta)}{n}\,.
$$
\end{thm}
\begin{proof}
Note that by definition
\begin{equation}
\label{EQ:fund_ineq_ls}
|Y-\X\thetals|_2^2 \le |Y-\X\theta^*|_2^2=|\eps|^2_2\,.
\end{equation}
Moreover, 
$$
|Y-\X\thetals|_2^2 =|\X\theta^*+\eps-\X\thetals|_2^2= |\X\thetals-\X\theta^*|_2^2-2\eps^\top\X(\thetals -\theta^*) + |\eps|_2^2\,.
$$
Therefore, we get
\begin{equation}
\label{EQ:bound_ls_1}
|\X\thetals-\X\theta^*|_2^2\le 2\eps^\top\X(\thetals - \theta^*)=2|\X\thetals-\X\theta^*|_2\frac{\eps^\top\X(\thetals - \theta^*)}{|\X(\thetals - \theta^*)|_2}
\end{equation}
Note that it is difficult to control 
$$
\frac{\eps^\top\X(\thetals - \theta^*)}{|\X(\thetals - \theta^*)|_2}
$$
as $\thetals$ depends on $\eps$ and the dependence structure of this term may be complicated. To remove this dependency, a traditional technique is to ``sup-out" $\thetals$. This is typically where maximal inequalities are needed. Here we have to be a bit careful. 

Let $\Phi=[\phi_1, \ldots, \phi_r] \in \R^{n \times r}$ be an orthonormal basis of the column span of $\X$. In particular, there exists $\nu \in \R^r$ such that $\X(\thetals - \theta^*)=\Phi \nu$. It yields
$$
\frac{\eps^\top\X(\thetals - \theta^*)}{|\X(\thetals - \theta^*)|_2}=\frac{\eps^\top \Phi \nu}{|\Phi\nu|_2} =\frac{\eps^\top \Phi \nu}{|\nu|_2}=
\tilde \eps^\top \frac{\nu}{|\nu|_2}\le \sup_{u \in \cB_2}\tilde \eps^\top u\,,
$$
where $\cB_2$ is the unit ball of $\R^r$ and   $\tilde \eps=\Phi^\top\eps$. Thus
$$
|\X\thetals-\X\theta^*|_2^2 \le 4\sup_{u \in \cB_2}(\tilde \eps^\top u)^2\,,
$$
Next, let \( u \in \cS^{r-1} \) and note that $|\Phi u|_2^2=u^\top \Phi^\top \Phi u=u^\top u=1$, so that for any $s \in \R$, we have
$$
\E[e^{s \tilde \eps^\top  u}]=\E[e^{s \eps^\top \Phi u}]\le \e^{\frac{s^2\sigma^2}{2}}\,.
$$
Therefore, $\tilde \eps\sim \sg_r(\sigma^2)$.

To conclude the bound in expectation, observe that Lemma~\ref{LEM:subgauss_moments} yields 
$$
 4\E\big[\sup_{u \in \cB_2}(\tilde \eps^\top  u  )^2\big]=4\sum_{i=1}^r \E[\tilde \eps_i^2]\le 16\sigma^2r\,.
$$
Moreover, with probability $1-\delta$, it follows from the last step in the proof\footnote{we could use Theorem~\ref{TH:supell2} directly here but at the cost of a factor 2 in the constant.} of Theorem~\ref{TH:supell2} that
$$
\sup_{u \in \cB_2}(\tilde \eps^\top u)^2 \le 8\log(6)\sigma^2 r  + 8\sigma^2\log(1/\delta)\,.
$$
\end{proof}
\begin{rem}
\label{rem:lam_min}
If $d \le n$ and  $B:=\frac{\X^\top \X}{n}$ has rank $d$, then we have
$$
|\thetals -\theta^*|_2^2 \le \frac{\MSE(\X\thetals)}{\lambda_{\mathrm{min}} (B)}\,,
$$
and we can use Theorem~\ref{TH:lsOI} to bound $|\thetals -\theta^*|_2^2$ directly. 
By contrast, in the high-dimensional case, we will need more structural assumptions to come to a similar conclusion.
\end{rem}

\subsection{Constrained least squares estimator}

Let $K\subset \R^d$ be a symmetric convex set. If we know \emph{a priori} that $\theta^* \in K$, we may prefer a \emph{constrained least squares} estimator $\thetalsK$ defined by 
$$
\thetalsK \in \argmin_{\theta \in K} |Y-\X\theta|_2^2\,.
$$
The fundamental inequality~\eqref{EQ:fund_ineq_ls} would still hold and the bounds on the MSE may be smaller. Indeed, ~\eqref{EQ:bound_ls_1} can be replaced by
$$
|\X\thetalsK-\X\theta^*|_2^2\le 2\eps^\top\X(\thetalsK - \theta^*)\le 2\sup_{\theta \in K-K}(\eps^\top\X\theta)\,,
$$
where $K-K=\{x-y\,:\, x, y \in K\}$. It is easy to see that if $K$ is symmetric ($x \in K \Leftrightarrow -x \in K$) and convex, then $K-K=2K$ so that
$$
2\sup_{\theta \in K-K}(\eps^\top\X\theta)=4\sup_{v \in \X K}(\eps^\top v)
$$
where $\X K=\{\X\theta\,:\, \theta \in K\} \subset \R^n$. This is a measure of the size (width) of $\X K$. If $\eps\sim \cN(0,I_d)$, the expected value of the above supremum is actually called \emph{Gaussian width} of $\X K$. While $\eps$ is not Gaussian but sub-Gaussian here, similar properties will still hold.

\subsubsection{$\ell_1$ constrained least squares}

Assume here that $K=\cB_1$ is the unit $\ell_1$ ball of $\R^d$. Recall that it is defined by
$$
\cB_1=\Big\{x \in \R^d\,:\, \sum_{i=1}^d|x_i|\le 1\big\}\,,
$$
and it has exactly $2d$ vertices $\cV=\{e_1, -e_1, \ldots, e_d, -e_d\}$, where $e_j$ is the $j$-th vector of the canonical basis of $\R^d$ and is defined by 
$$
e_j=(0, \ldots, 0,\underbrace{1}_{j\text{th position}},0, \ldots, 0)^\top\,.
$$
It implies that the set $\X K=\{\X\theta, \theta \in K\} \subset \R^n$ is also a polytope with at most $2d$ vertices that are in the set $\X\cV=\{\X_1, -\X_1, \ldots, \X_d, -\X_d\}$ where $\X_j$ is the $j$-th column of $\X$. Indeed, $\X K$ is obtained by rescaling and embedding (resp. projecting) the polytope $K$ when $d \le n$ (resp., $d \ge n$).  Note that some columns of $\X$ might not be vertices of $\X K$ so that $\X \cV$ might be a strict superset of the set of vertices of $\X K$.

\begin{thm}
\label{TH:ell1const}
Let $K=\cB_1$ be the unit $\ell_1$ ball of $\R^d, d\ge 2$ and assume that  $\theta^* \in \cB_1$. Moreover, assume the conditions of Theorem~\ref{TH:lsOI} and that the columns of $\X$ are normalized in such a way that $\max_j|\X_j|_2\le \sqrt{n}$. Then the constrained least squares estimator $\thetals_{\cB_1}$ satisfies
$$
\E\big[\MSE(\X\thetals_{\cB_1})\big]=\frac{1}{n}\E|\X\thetals_{\cB_1} - \X\theta^*|_2^2 \lesssim \sigma\sqrt{ \frac{\log d}{n}}\,,
$$
Moreover, for any $\delta>0$, with probability $1-\delta$, it holds
$$
\MSE(\X\thetals_{\cB_1})\lesssim \sigma \sqrt{\frac{\log (d/\delta)}{n}}\,.
$$
\end{thm}
\begin{proof}
From the considerations preceding the theorem, we get that
$$
|\X\thetals_{\cB_1}-\X\theta^*|_2^2\le4\sup_{v \in \X K}(\eps^\top v).
$$
Observe now that since $\eps\sim \sg_n(\sigma^2)$, for any column $\X_j$ such that $|\X_j|_2\le \sqrt{n}$, the random variable $\eps^\top \X_j\sim \sg(n\sigma^2)$. Therefore, applying Theorem~\ref{TH:polytope}, we get the bound on $\E\big[\MSE(\X\thetalsK)\big]$ and for any $t\ge 0$, 
$$
\p\big[\MSE(\X\thetalsK)>t  \big]\le \p\big[\sup_{v \in \X K}(\eps^\top v)>nt/4  \big]\le 2de^{-\frac{nt^2}{32\sigma^2}}
$$
To conclude the proof, we find $t$ such that 
$$
2de^{-\frac{nt^2}{32\sigma^2}}\le \delta \ \Leftrightarrow \ t^2 \ge 32\sigma^2 \frac{ \log(2d)}{n} + 32 \sigma^2 \frac{ \log(1/\delta)}{n}\,.
$$

\end{proof}
Note that the proof of Theorem~\ref{TH:lsOI} also applies to $\thetals_{\cB_1}$ (exercise!) so that $\X \thetals_{\cB_1}$ benefits from the best of both rates,
$$
\E\big[\MSE(\X\thetals_{\cB_1})\big]\lesssim \min\Big(\sigma^2\frac{r}{n}, \sigma\sqrt{\frac{\log d}{n}}\Big)\,.
$$
This is called an \emph{elbow effect}. The elbow takes place around $r\simeq\sqrt{n}$ (up to logarithmic terms).

\subsubsection{$\ell_0$ constrained least squares}

By an abuse of notation, we call the number of non-zero coefficients of a vector \( \theta \in \R^d \) its \emph{$\ell_0$ norm} (it is not actually a norm). It is denoted by
$$
|\theta|_0=\sum_{j=1}^d \1(\theta_j \neq 0)\,.
$$
We call a vector $\theta$ with ``small" $\ell_0$ norm a \emph{sparse} vector. More precisely, if $|\theta|_0\le k$, we say that $\theta$ is a $k$-sparse vector. We also call \emph{support} of $\theta$ the set
$$
\supp(\theta)=\big\{j \in \{1, \ldots, d\}\,:\, \theta_j \neq 0\big\}
$$
so that $|\theta|_0=\card(\supp(\theta))=:|\supp(\theta)|$\,.
\begin{rem}
The $\ell_0$ terminology and notation comes from the fact that
$$
\lim_{q \to 0^{\tiny +}}\sum_{j=1}^d|\theta_j|^q=|\theta|_0
$$
Therefore it is really $\lim_{q \to 0^{\tiny +}}|\theta|_q^q$ but the notation $|\theta|_0^0$ suggests too much that it is always equal to 1.
\end{rem}
By extension, denote by $\cB_0(k)$ the $\ell_0$ ball of $\R^d$, i.e., the set of $k$-sparse vectors, defined by
$$
\cB_0(k)=\{\theta \in \R^d\,:\, |\theta|_0 \le k\}\,.
$$
In this section, our goal is to control the $\MSE$ of $\thetalsK$ when $K=\cB_0(k)$. Note that computing $\thetals_{\cB_0(k)}$ essentially requires computing $\binom{d}{k}$ least squares estimators, which is an exponential number in $k$. In practice this will be hard (or even impossible) but it is interesting to understand the statistical properties of this estimator and to use them as a benchmark.
\begin{thm}
\label{TH:bss1}
Fix a positive integer $k \le d/2$. Let $K=\cB_0(k)$ be  set of $k$-sparse vectors of $\R^d$ and assume that  $\theta^*\in \cB_0(k)$.  Moreover, assume the conditions of Theorem~\ref{TH:lsOI}. Then, for any $\delta>0$, with probability $1-\delta$, it holds
$$
\MSE(\X\thetals_{\cB_0(k)})\lesssim \frac{\sigma^2}{n}\log\binom{d}{2k} + \frac{\sigma^2k}{n} +\frac{\sigma^2}{n}\log(1/\delta)\,.
$$
\end{thm}
\begin{proof}
We begin as in the proof of Theorem~\ref{TH:lsOI} to get~\eqref{EQ:bound_ls_1}:
$$
|\X\thetalsK-\X\theta^*|_2^2\le 2\eps^\top\X(\thetalsK - \theta^*)=2|\X\thetalsK-\X\theta^*|_2\frac{\eps^\top\X(\thetalsK - \theta^*)}{|\X(\thetalsK - \theta^*)|_2}\,.
$$
We know that both $\thetalsK$ and $\theta^*$ are in $\cB_0(k)$ so that $\thetalsK-\theta^* \in \cB_0(2k)$. For any $S \subset \{1, \ldots, d\}$, let $\X_S$ denote the $n\times |S|$ submatrix of $\X$ that is obtained from the columns of $\X_j, j \in S$ of $\X$. Denote by $r_S \le |S|$ the rank of $\X_S$ and let $\Phi_S=[\phi_1, \ldots, \phi_{r_S}] \in \R^{n\times r_S}$ be an orthonormal basis of the column span of $\X_S$. Moreover, for any $\theta \in \R^d$, define $\theta(S) \in \R^{|S|}$ to be the vector with coordinates $\theta_j, j \in S$. If we denote by $\hat S=\supp(\thetalsK - \theta^*)$, we have $|\hat S|\le 2k$ and there exists $\nu \in \R^{r_{\hat S}}$ such that 
$$
\X(\thetalsK - \theta^*)=\X_{\hat S}(\thetalsK(\hat S) - \theta^*(\hat S))=\Phi_{\hat S} \nu\,.
$$
Therefore,
$$
\frac{\eps^\top\X(\thetalsK - \theta^*)}{|\X(\thetalsK - \theta^*)|_2} =\frac{\eps^\top \Phi_{\hat S} \nu}{|\nu|_2}\le \max_{|S|= 2k}\sup_{u \in \cB_2^{r_{S}}}[\eps^\top \Phi_S] u 
$$
where $\cB_2^{r_{S}}$ is the unit ball of $\R^{r_S}$. 
It yields
$$
|\X\thetalsK-\X\theta^*|_2^2 \le 4 \max_{|S|= 2k}\sup_{u \in \cB_2^{r_{S}}}(\tilde \eps_S^\top u)^2\,,
$$
with $\tilde \eps_S=\Phi_S^\top\eps \sim \sg_{r_S}(\sigma^2)$. 

Using a union bound, we get for any $t>0$, 
$$
\p\big(\max_{|S|= 2k}\sup_{u \in \cB_2^{r_S}}(\tilde \eps^\top u)^2>t\big) \le \sum_{|S|= 2k} \p\big( \sup_{u \in \cB_2^{r_S}}(\tilde \eps^\top u)^2>t\big)
$$
It follows from the proof of Theorem~\ref{TH:supell2} that for any $|S| \le 2k$,
$$
\p\big( \sup_{u \in \cB_2^{r_S}}(\tilde \eps^\top u)^2 >t\big)\le 6^{|S|}e^{-\frac{t}{8\sigma^2}} \le 6^{2k}e^{-\frac{t}{8\sigma^2}}\,.
$$
Together, the above three displays yield
\begin{equation}
\label{EQ:prMSEl0HP}
\p(|\X\thetalsK-\X\theta^*|_2^2>4t) \le \binom{d}{2k} 6^{2k} e^{-\frac{t}{8\sigma^2}}\,.
\end{equation}
To ensure that the right-hand side of the above inequality is bounded by $\delta$, we need
$$
t\ge C\sigma^2\Big\{\log \binom{d}{2k} + k\log(6) + \log(1/\delta)\Big\}\,.
$$
\end{proof}
How large is $\log\binom{d}{2k}$? It turns out that it is not much larger than $k$.
\begin{lem}
\label{lem:nchoosek}
For any integers $1\le k \le n$, it holds
$$
\binom{n}{k} \le \Big(\frac{en}{k}\Big)^k.
$$
\end{lem}
\begin{proof}
Observe first that if $k=1$, since $n \ge 1$, it holds,
$$
\binom{n}{1}=n\le en=\Big(\frac{en}{1}\Big)^1
$$
Next, we proceed by induction and assume that it holds for some $k \le n-1$ that
$$
\binom{n}{k} \le \Big(\frac{en}{k}\Big)^k.
$$
Observe that
$$
\binom{n}{k+1}=\binom{n}{k}\frac{n-k}{k+1}\le  \Big(\frac{en}{k}\Big)^k\frac{n-k}{k+1} \leq \frac{e^k n^{k+1}}{(k+1)^{k+1}}\Big(1+\frac1k\Big)^k \,,
$$
where we used the induction hypothesis in the first inequality. To conclude, it suffices to observe that
\begin{equation*}
\left(1+\frac1k\right)^k\le e. \qedhere
\end{equation*}
\end{proof}

It immediately leads to the following corollary:
\begin{cor}
\label{COR:bss1}
Under the assumptions of Theorem~\ref{TH:bss1}, for any $\delta>0$, with probability at least $1-\delta$, it holds
$$
\MSE(\X\thetals_{\cB_0(k)})\lesssim \frac{\sigma^2k}{n}\log \Big(\frac{ed}{2k}\Big) + \frac{\sigma^2k}{n}\log(6) +  \frac{\sigma^2}{n}\log(1/\delta)\,.
$$
\end{cor}
Note that for any fixed $\delta$, there exits a constant $C_\delta>0$ such that for any $n \ge 2k$, with high probability,
$$
\MSE(\X\thetals_{\cB_0(k)})\le C_\delta\frac{\sigma^2k}{n}\log \Big(\frac{ed}{2k}\Big)\,.
$$
Comparing this result with Theorem~\ref{TH:lsOI} with $r=k$, we see that the price to pay for not knowing the support of $\theta^*$ but only its size, is  a logarithmic factor in the dimension $d$.

This result immediately leads to the following bound in expectation.
\begin{cor}
\label{COR:bss2}
Under the assumptions of Theorem~\ref{TH:bss1}, 
$$
\E\big[\MSE(\X\thetals_{\cB_0(k)})\big]\lesssim\frac{\sigma^2k}{n}\log \Big(\frac{ed}{k}\Big)\,.
$$
\end{cor}
\begin{proof}
It follows from~\eqref{EQ:prMSEl0HP} that for any $H\ge 0$,
\begin{align*}
\E\big[\MSE(\X\thetals_{\cB_0(k)})\big]&=\int_0^\infty\p(|\X\thetalsK-\X\theta^*|_2^2>nu) \ud u\\
&\le H+\int_0^\infty\p(|\X\thetalsK-\X\theta^*|_2^2>n(u+H)) \ud u\\
&\le H+\sum_{j=1}^{2k}\binom{d}{j} 6^{2k} \int_0^\infty e^{-\frac{n(u+H)}{32\sigma^2}} \ud u\\
&=H+\sum_{j=1}^{2k}\binom{d}{j} 6^{2k}e^{-\frac{nH}{32\sigma^2}} \frac{32\sigma^2}{n}\,.
\end{align*}
Next,  take $H$ to be such that
$$
\sum_{j=1}^{2k}\binom{d}{j} 6^{2k}e^{-\frac{nH}{32\sigma^2}}=1\,.
$$
This yields
$$
H \lesssim \frac{\sigma^2k}{n}\log \Big(\frac{ed}{k}\Big)\,,
$$
which completes the proof.
\end{proof}

\section{The Gaussian Sequence Model}

The Gaussian Sequence Model is a toy model that has received a lot of attention, mostly in the eighties. The main reason for its popularity is that it carries already most of the insight of nonparametric estimation. While the model looks very simple it allows to carry deep ideas that extend beyond its framework and in particular to the linear regression model that we are interested in. Unfortunately, we will only cover a small part of these ideas and the interested reader should definitely look at the excellent books by A. Tsybakov~\cite[Chapter~3]{Tsy09} and I. Johnstone~\cite{Joh11}.

The model is as follows:
\begin{equation}
\label{EQ:gsm}
Y_i=\theta^*_i +\eps_i\,,\qquad  i=1, \ldots, d\,,
\end{equation}
where $\eps_1, \ldots, \eps_d$ are \iid $\cN(0,\sigma^2)$ random variables. Note that often, $d$ is taken equal to $\infty$ in this sequence model and we will also discuss this case. Its links to nonparametric estimation will become clearer in Chapter~\ref{chap:misspecified}. The goal here is to estimate the unknown vector $\theta^*$.

\subsection{The sub-Gaussian Sequence Model}

Note first that the model~\eqref{EQ:gsm} is a special case of the linear model with fixed design~\eqref{EQ:regmod} with $n=d$ and $f(x_i)=x_i^\top \theta^*$, where $x_1, \ldots, x_n$ form the canonical basis \( e_1, \dots, e_n \) of $\R^n$ and $\eps$ has a Gaussian distribution. Therefore, $n=d$ is both the dimension of the parameter $\theta$ and the number of observations and it looks like we have chosen to index this problem by $d$ rather than $n$ somewhat arbitrarily. We can bring $n$ back into the picture, by observing that this model encompasses slightly more general choices for the design matrix $\X$ as long as it satisfies the following assumption.
\begin{assumption}{\textsf{ORT}}
The design matrix satisfies
$$
\frac{\X^\top \X}{n}=I_d\,,
$$
where $I_d$ denotes the identity matrix of $\R^d$. 
\end{assumption}
Assumption~\textsf{ORT} allows for cases where $d\le n$ but not $d>n$ (high dimensional case) because of obvious rank constraints. In particular, it means that the $d$ columns of $\X$ are orthogonal in $\R^n$ and all have norm $\sqrt{n}$.

Under this assumption, it follows from  the linear regression model~\eqref{EQ:regmod_matrix} that
\begin{align*}
y:=\frac{1}{n}\X^\top Y&=\frac{\X^\top \X}{n}\theta^*+ \frac{1}{n}\X^\top\eps\\
&=\theta^*+\xi\,,
\end{align*}
where $\xi=(\xi_1, \ldots, \xi_d)\sim \sg_d(\sigma^2/n)$ (If $\eps$ is Gaussian, we even have $\xi \sim \cN_d(0, \sigma^2/n)$). As a result, under the assumption~\textsf{ORT}, when $\xi$ is Gaussian, the linear regression model~\eqref{EQ:regmod_matrix} is equivalent to the Gaussian Sequence Model~\eqref{EQ:gsm} up to a transformation of the data $Y$ and a rescaling of the variance. Moreover, for any estimator $\hat \theta \in \R^d$, under \textsf{ORT}, it follows from~\eqref{EQ:mse_matrix} that
$$
\MSE(\X\hat \theta)=(\hat \theta -\theta^*)^\top \frac{\X^\top \X}{n}(\hat \theta -\theta^*)=|\hat \theta - \theta^*|_2^2\,.
$$
Furthermore, for any $\theta \in \R^d$, the assumption~\textsf{ORT} yields,
\begin{align}
|y-\theta|_2^2&=|\frac{1}{n}\X^\top Y-\theta|_2^2 \nonumber\\
                &= |\theta|_2^2 - \frac{2}{n} \theta^\top \X^\top Y + \frac{1}{n^2} Y^\top \X \X^\top Y\nonumber \\
                &= \frac{1}{n} |\X \theta|_2^2 - \frac{2}{n} ( \X\theta)^\top Y + \frac{1}{n} |Y|_2^2 +Q \nonumber \\
                &= \frac{1}{n} |Y - \X \theta|_2^2 +Q\label{EQ:ERM_sGSM}\,,
\end{align}
where $Q$ is a constant that does not depend on $\theta$ and is defined by
$$
Q=\frac{1}{n^2} Y^\top \X \X^\top Y- \frac{1}{n} |Y|_2^2
$$
This implies in particular that the least squares estimator $\thetals$ is equal to $y$.

\bigskip

We introduce a sightly more general model called \emph{sub-Gaussian sequence model}:

\begin{equation}
\label{EQ:sGSM}
y=\theta^* +\xi \in \R^d\,,
\end{equation}
where $\xi\sim \sg_d(\sigma^2/n)$. 

In this section, we can actually completely ``forget" about our original model~\eqref{EQ:regmod_matrix}. In particular we can define this model independently of Assumption \textsf{ORT} and thus for any values of $n$ and $d$. 

The sub-Gaussian sequence model and the Gaussian sequence model are called \emph{direct} (observation) problems as opposed to \emph{inverse problems} where the goal is to estimate the parameter $\theta^*$ only from noisy observations of its image through an operator. The linear regression model is one such inverse problem where the matrix $\X$ plays the role of a linear operator. However, in these notes, we never try to invert the operator. See \cite{Cav11} for an interesting survey on the statistical theory of inverse problems.

\subsection{Sparsity adaptive thresholding estimators}

If we knew a priori that $\theta$ was $k$-sparse, we could directly employ Corollary~\ref{COR:bss1} to obtain that with probability at least $1-\delta$, we have
$$
\MSE(\X\thetals_{\cB_0(k)})\le C_\delta\frac{\sigma^2k}{n}\log\Big(\frac{ed}{2k}\Big)\,.
$$
As we will see, the assumption~\textsf{ORT} gives us the luxury to not know $k$ and yet \emph{adapt} to its value. Adaptation means that we can construct an estimator that does not require the knowledge of $k$ (the smallest such that $|\theta^*|_0\le k$) and yet perform as well as $\thetals_{\cB_0(k)}$, up to a multiplicative constant.

Let us begin with some heuristic considerations to gain some intuition. Assume the  sub-Gaussian sequence model~\eqref{EQ:sGSM}. If nothing is known about $\theta^*$ it is natural to estimate it using the least squares estimator $\thetals=y$. In this case,
$$
\MSE(\X\thetals)=|y-\theta^*|_2^2=|\xi|_2^2 \le C_\delta\frac{\sigma^2d}{n}\,,
$$
where the last inequality holds with probability at least $1-\delta$. This is actually what we are looking for if $k=Cd$ for some positive constant $C\le 1$. The problem with this approach is that it does not use the fact that $k$ may be much smaller than $d$, which happens when $\theta^*$ has many zero coordinates. 

If $\theta^*_j=0$, then, $y_j=\xi_j$, which is a sub-Gaussian random variable with variance proxy $\sigma^2/n$. In particular, we know from Lemma~\ref{lem:subgauss} that with probability at least $1-\delta$, 
\begin{equation}
\label{EQ:tau1}
|\xi_j|\le  \sigma\sqrt{\frac{2\log (2/\delta)}{n}}=\tau\,.
\end{equation}
The consequences of this inequality are interesting. One the one hand, if we observe $|y_j|\gg \tau$ , then it must correspond to $\theta_j^*\neq 0$. On the other hand, if $|y_j| \le \tau$ is smaller, then, $\theta_j^*$ cannot be very large. In particular, by the triangle inequality, $|\theta^*_j| \le |y_j|+|\xi_j| \le 2\tau$. Therefore, we loose at most $2\tau$ by choosing $\hat \theta_j=0$. It leads us to consider the following estimator.

\begin{defn}
\label{def:hard}
The \textbf{hard thresholding} estimator with threshold $2\tau>0$ is denoted by $\thetahard$ and has coordinates
$$
\thetahard_j=\left\{
\begin{array}{ll}
y_j& \text{if} \ |y_j|>2\tau\,,\\
0& \text{if} \ |y_j|\le2\tau\,,
\end{array}
\right.
$$
for $j=1, \ldots, d$. In short, we can write $\thetahard_j=y_j\1(|y_j|>2\tau)$.
\end{defn}

From our above consideration, we are tempted to choose $\tau$ as in~\eqref{EQ:tau1}. Yet, this threshold is not large enough. Indeed, we need to choose $\tau$ such that $|\xi_j|\le \tau$ \emph{simultaneously} for all $j$.  This can be done using a maximal inequality. Namely, Theorem~\ref{TH:finitemax} ensures that with probability at least $1-\delta$,
$$
\max_{1\le j\le d}|\xi_j|\le\sigma\sqrt{\frac{2\log (2d/\delta)}{n}}.
$$
It yields the following theorem.
\begin{thm}
\label{TH:hard}
Consider the linear regression model~\eqref{EQ:regmod_matrix} under the assumption~\textsf{ORT} or, equivalenty, the sub-Gaussian sequence model~\eqref{EQ:sGSM}. Then the hard thresholding estimator $\thetahard$ with threshold 
\begin{equation}
\label{EQ:tau}
2\tau=2\sigma\sqrt{\frac{2\log (2d/\delta)}{n}}
\end{equation}
enjoys the following two properties on the same event $\cA$ such that $\p(\cA)\ge 1-\delta$:
\begin{itemize}
\item[(i)] If $|\theta^*|_0=k$, 
$$
\MSE(\X\thetahard)=|\thetahard-\theta^*|_2^2\lesssim \sigma^2\frac{k\log (2d/\delta)}{n}\,.
$$
\item[(ii)] if $\min_{j \in \supp(\theta^*)}|\theta^*_j| > 3\tau$, then
$$
\supp(\thetahard)=\supp(\theta^*)\,.
$$
\end{itemize}
\end{thm}
\begin{proof}
Define the event 
$$
\cA=\Big\{\max_j|\xi_j|\le \tau\big\}\,,
$$
and recall that Theorem~\ref{TH:finitemax} yields $\p(\cA)\ge 1-\delta$. On the event $\cA$, the following holds for any $j=1, \ldots, d$.

First, observe that 
\begin{equation}
\label{EQ:prthhard1}
|y_j|>2\tau \quad \Rightarrow \quad |\theta^*_j|\ge |y_j|-|\xi_j|>\tau
\end{equation}
and
\begin{equation}
\label{EQ:prthhard2}
|y_j|\le 2\tau \quad \Rightarrow \quad |\theta^*_j|\le |y_j| + |\xi_j|\le 3\tau.
\end{equation}
It yields
\begin{align*}
|\thetahard_j-\theta^*_j|&=|y_j-\theta^*_j|\1(|y_j|>2\tau)+ |\theta^*_j|\1(|y_j|\le2\tau)&\\
&\le \tau\1(|y_j|>2\tau)+|\theta^*_j|\1(|y_j|\le2\tau)& \\\
&\le \tau\1(|\theta^*_j|>\tau)+|\theta^*_j|\1(|\theta^*_j|\le 3\tau)& \text{by}~\eqref{EQ:prthhard1} \text{ and}~\eqref{EQ:prthhard2}\\
&\le 4\min(|\theta^*_j|,\tau)
\end{align*}
It yields
\begin{align*}
|\thetahard-\theta^*|_2^2=\sum_{j=1}^d|\thetahard_j-\theta^*_j|^2\le 16\sum_{j=1}^d\min(|\theta^*_j|^2,\tau^2)\le 16|\theta^*|_0\tau^2\,.
\end{align*}
This completes the proof of (i).

To prove (ii), note that if $\theta^*_j\neq 0$, then $|\theta^*_j|>3\tau$ so that 
$$
|y_j|=|\theta^*_j+\xi_j|> 3\tau-\tau=2\tau\,.
$$
Therefore, $\thetahard_j\neq 0$ so that $\supp(\theta^*)\subset \supp(\thetahard)$. 

Next, if $\thetahard_j\neq 0$, then $|\thetahard_j|=|y_j|>2\tau$. It yields
$$
|\theta^*_j|\ge |y_j|-\tau > \tau.
$$
Therefore, $|\theta^*_j|\neq 0$ and $\supp(\thetahard)\subset \supp(\theta^*)$.
\end{proof}
Similar results can be obtained for the \textbf{soft thresholding} estimator $\thetasoft$ defined by
$$
\thetasoft_j=\left\{
\begin{array}{ll}
y_j-2\tau& \text{if} \ y_j>2\tau\,,\\
y_j+2\tau& \text{if} \ y_j<-2\tau\,,\\
0& \text{if} \ |y_j|\le2\tau\,,
\end{array}
\right.
$$
In short, we can write 
$$
\thetasoft_j=\Big(1-\frac{2\tau}{|y_j|}\Big)_+y_j
$$
\begin{center}
\begin{figure}
    \includegraphics[width=\textwidth]{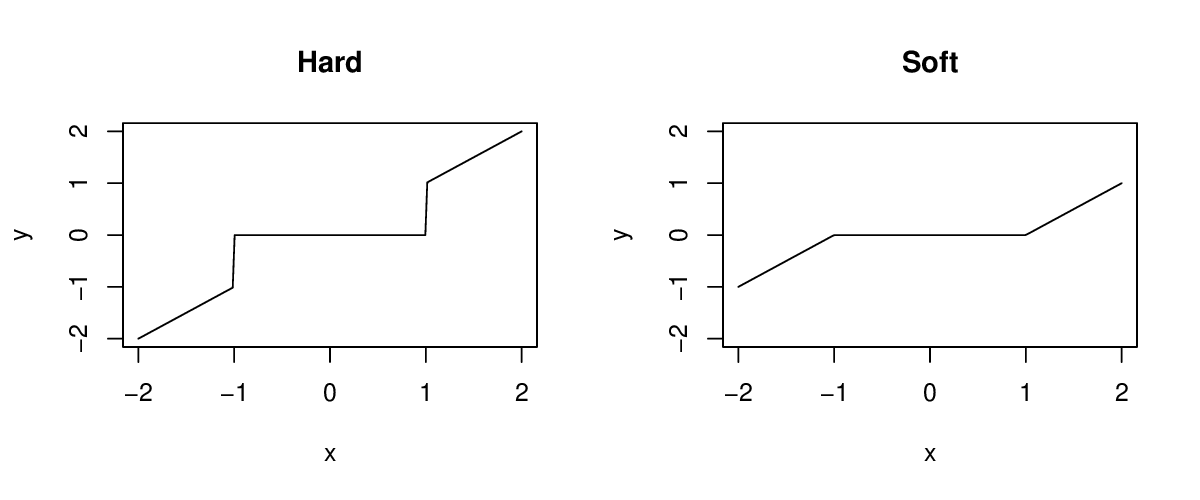} \\
\caption{Transformation applied to $y_j$ with $2\tau=1$ to obtain the hard (left) and soft (right) thresholding estimators}\label{FIG:threshold}
    \end{figure}
\end{center}

\section{High-dimensional linear regression}
\subsection{The BIC and Lasso estimators}

It can be shown (see Problem~\ref{EXO:sparse:variation}) that the hard and soft thresholding estimators are solutions of the following penalized empirical risk minimization problems:
\begin{align*}
\thetahard&=\argmin_{\theta \in \R^d} \Big\{|y-\theta|_2^2 + 4\tau^2|\theta|_0\Big\}\\
\thetasoft&=\argmin_{\theta \in \R^d} \Big\{|y-\theta|_2^2 + 4\tau|\theta|_1\Big\}
\end{align*}
In view of~\eqref{EQ:ERM_sGSM}, under the assumption~\textsf{ORT}, the above variational definitions can be written as
\begin{align*}
\thetahard&=\argmin_{\theta \in \R^d} \Big\{\frac1n|Y-\X\theta|_2^2 + 4\tau^2|\theta|_0\Big\}\\
\thetasoft&=\argmin_{\theta \in \R^d} \Big\{\frac1n|Y-\X\theta|_2^2 + 4\tau|\theta|_1\Big\}\\
\end{align*}
When the assumption~\textsf{ORT} is not satisfied, they no longer correspond to thresholding estimators but can still be defined as above. We change the constant in the threshold parameters for future convenience.\begin{defn}
Fix $\tau>0$ and assume the linear regression model~\eqref{EQ:regmod_matrix}. The BIC\footnote{Note that it minimizes the Bayes Information Criterion (BIC) employed in the traditional literature of asymptotic statistics if $\tau=\sqrt{\log(d)/n}$. We will use the same value below, up to multiplicative constants (it's the price to pay to get non-asymptotic results).} estimator of $\theta^*$ is defined by any $\thetabic$ such that
$$
\thetabic \in \argmin_{\theta \in \R^d} \Big\{\frac1n|Y-\X\theta|_2^2 + \tau^2|\theta|_0\Big\}.\\
$$
Moreover, the Lasso estimator of $\theta^*$ is defined by any $\thetalasso$ such that
$$
\thetalasso \in \argmin_{\theta \in \R^d} \Big\{\frac1n|Y-\X\theta|_2^2 + 2\tau|\theta|_1\Big\}.\\
$$
\end{defn}
\begin{rem}[Numerical considerations] Computing the BIC estimator can be proved to be NP-hard in the worst case. In particular, no computational method is known to be significantly faster than the brute force search among all $2^d$ sparsity patterns. Indeed, we can rewrite:
$$
\min_{\theta \in \R^d} \Big\{\frac1n|Y-\X\theta|_2^2 + \tau^2|\theta|_0\Big\}=\min_{0\le k \le d} \Big\{\min_{\theta\,:\, |\theta|_0=k}\frac1n|Y-\X\theta|_2^2 + \tau^2k\Big\}
$$
To compute $\min_{\theta\,:\, |\theta|_0=k}\frac1n|Y-\X\theta|_2^2$, we need to compute $\binom{d}{k}$ least squares estimators on a space of size $k$. Each costs $O(k^3)$ (matrix inversion). Therefore the total cost of the brute force search is
$$
C\sum_{k=0}^d \binom{d}{k}k^3=Cd^32^d\,.
$$

By contrast, computing the Lasso estimator is a convex problem and there exist many efficient algorithms to solve it. We will not describe this optimization problem in detail but only highlight a few of the best known algorithms:
\begin{enumerate}
\item Probably the most popular method among statisticians relies on coordinate  gradient descent. It is implemented in the \textsf{glmnet} package in \textsf{R} \cite{FriHasTib10}, 
\item An interesting method called LARS \cite{EfrHasJoh04} computes the entire \emph{regularization path}, i.e., the solution of the convex problem for all values of $\tau$. It relies on the fact that, as a function of $\tau$, the solution $\thetalasso$ is a piecewise linear function (with values in $\R^d$). Yet this method proved to be too slow for very large problems and has been replaced by \texttt{glmnet} which computes solutions for values of $\tau$ on a grid much faster.
\item The optimization community has made interesting contributions to this field by using proximal methods to solve this problem. These methods exploit the structure of the objective function which is of the form smooth (sum of squares) + simple ($\ell_1$ norm). A good entry point to this literature is perhaps the FISTA algorithm \cite{BecTeb09}.
\item Recently there has been a lot of interest around this objective for very large $d$ and very large $n$. In this case, even computing $|Y-\X\theta|_2^2 $ may be computationally expensive and solutions based on stochastic gradient descent are flourishing.
\end{enumerate}
\end{rem}

Note that by Lagrange duality, computing $\thetalasso$ is equivalent to solving an $\ell_1$ \emph{constrained} least squares. Nevertheless, the radius of the $\ell_1$ constraint is unknown. In general, it is hard to relate Lagrange multipliers to the size constraints. The name ``Lasso" was originally given to the constrained version this estimator in the original paper of Robert Tibshirani \cite{Tib96}.

\subsection{Analysis of the BIC estimator}

While computationally hard to implement, the BIC estimator gives us a good benchmark for sparse estimation. Its performance is similar to that of $\thetahard$ but without requiring the assumption \textsf{ORT}.

\begin{thm}
\label{TH:BIC}
Assume that the linear model~\eqref{EQ:regmod_matrix} holds, where $\eps \sim \sg_n(\sigma^2)$ and that $|\theta^*|_0\ge 1$. Then,  the BIC estimator $\thetabic$ with regularization parameter
\begin{equation}
\label{EQ:taubic}
\tau^2=16\log(6)\frac{\sigma^2}{n} + 32\frac{\sigma^2 \log (ed)}{n}\,.
\end{equation}
satisfies
\begin{equation}
\MSE(\X\thetabic)=\frac{1}{n}|\X\thetabic-\X\theta^*|_2^2\lesssim |\theta^*|_0 \sigma^2\frac{ \log(ed/\delta)}{n}
\end{equation}
with probability at least $1-\delta$.
\end{thm}
\begin{proof}
We begin as usual by noting that
$$
\frac1n|Y-\X\thetabic|_2^2 + \tau^2|\thetabic|_0 \le \frac1n|Y-\X\theta^*|_2^2 + \tau^2|\theta^*|_0\,.
$$
It implies
$$
|\X\thetabic-\X\theta^*|_2^2  \le n\tau^2|\theta^*|_0 + 2\eps^\top\X(\thetabic-\theta^*)- n\tau^2|\thetabic|_0\,.
$$
First, note that
\begin{align*}
2\eps^\top\X(\thetabic-\theta^*)&=2\eps^\top\Big(\frac{\X\thetabic-\X\theta^*}{|\X\thetabic-\X\theta^*|_2}\Big)|\X\thetabic-\X\theta^*|_2\\
&\le 2\Big[\eps^\top\Big(\frac{\X\thetabic-\X\theta^*}{|\X\thetabic-\X\theta^*|_2}\Big)\Big]^2 + \frac12 |\X\thetabic-\X\theta^*|_2^2\,,
\end{align*}
where we use the inequality $2ab\le 2a^2+\frac12b^2$. Together with the previous display, it yields
\begin{equation}
\label{EQ:pr_bic_1}
|\X\thetabic-\X\theta^*|_2^2  \le 2n\tau^2|\theta^*|_0 + 4\big[\eps^\top\cU(\thetabic-\theta^*)\big]^2- 2n\tau^2|\thetabic|_0\,
\end{equation}
where 
$$
\cU(z)=\frac{z}{|z|_2}
$$
Next, we need to ``sup out" $\thetabic$. To that end, we decompose the sup into a max over cardinalities as follows:
$$
\sup_{\theta \in \R^d}=\max_{1\le k \le d} \max_{|S|=k}\sup_{\supp(\theta)=S}\,.
$$
Applied to the above inequality, it yields
\begin{align*}
4\big[\eps^\top&\cU(\thetabic-\theta^*)\big]^2- 2n\tau^2|\thetabic|_0\\
& \le \max_{1\le k \le d}\big\{\max_{|S|=k} \sup_{\supp(\theta)=S}4\big[\eps^\top\cU(\theta-\theta^*)\big]^2- 2n\tau^2k\big\}\\
&\le  \max_{1\le k \le d}\big\{\max_{|S|=k} \sup_{u\in \cB_2^{r_{S,*}}}4\big[\eps^\top\Phi_{S,*}u\big]^2- 2n\tau^2k\big\}\,,
\end{align*}
where  $\Phi_{S,*}=[\phi_1, \ldots, \phi_{r_{S,*}}]$ is an orthonormal basis of the set $\{\X_j, j \in S \cup \supp(\theta^*)\}$ of columns of $\X$ and $r_{S,*} \le |S|+|\theta^*|_0$ is the dimension of this column span.

Using union bounds, we get for any $t>0$, 
\begin{align*}
\p&\Big( \max_{1\le k \le d}\big\{\max_{|S|=k} \sup_{u\in \cB_2^{r_{S,*}}}4\big[\eps^\top\Phi_{S,*}u\big]^2- 2n\tau^2k\big\}\ge t\Big)\\
&\le \sum_{k=1}^d\sum_{|S|=k}\p\Big(   \sup_{u\in \cB_2^{r_{S,*}}}\big[\eps^\top\Phi_{S,*}u\big]^2\ge \frac{t}4+ \frac12n\tau^2k\Big)
\end{align*}
Moreover, using the $\eps$-net argument from Theorem~\ref{TH:supell2}, we get for $|S|=k$,
\begin{align*}
\p\Big(   \sup_{u\in \cB_2^{r_{S,*}}}&\big[\eps^\top\Phi_{S,*}u\big]^2\ge \frac{t}4+ \frac12n\tau^2k\Big)\le 2\cdot 6^{r_{S,*}}\exp\Big(-\frac{\frac{t}4+ \frac12n\tau^2k}{8\sigma^2}\Big)\\
&\le 2 \exp\Big(-\frac{t}{32\sigma^2}- \frac{n\tau^2k}{16\sigma^2}+(k+|\theta^*|_0)\log(6)\Big)
\\
&\le \exp\Big(-\frac{t}{32\sigma^2}- 2k\log(ed)+|\theta^*|_0\log(12)\Big)
\end{align*}
where, in the last inequality, we used the definition~\eqref{EQ:taubic} of $\tau$.

Putting everything together, we get
\begin{align*}
\p&\Big(|\X\thetabic-\X\theta^*|_2^2 \ge  2n\tau^2|\theta^*|_0+t\Big) \le\\
& \sum_{k=1}^d\sum_{|S|=k} \exp\Big(-\frac{t}{32\sigma^2}- 2k\log(ed)+|\theta^*|_0\log(12)\Big)\\
&= \sum_{k=1}^d \binom{d}{k} \exp\Big(-\frac{t}{32\sigma^2}- 2k\log(ed)+|\theta^*|_0\log(12)\Big)\\
&\le \sum_{k=1}^d  \exp\Big(-\frac{t}{32\sigma^2}-k\log(ed)+|\theta^*|_0\log(12)\Big)& \text{by Lemma~\ref{lem:nchoosek}}\\
&= \sum_{k=1}^d(ed)^{-k}  \exp\Big(-\frac{t}{32\sigma^2} + |\theta^*|_0\log(12)\Big)\\
&\le  \exp\Big(-\frac{t}{32\sigma^2} + |\theta^*|_0\log(12)\Big)\,.
\end{align*}
To conclude the proof, choose $t=32\sigma^2 |\theta^*|_0\log(12) + 32\sigma^2\log(1/\delta)$ and observe that combined with~\eqref{EQ:pr_bic_1}, it yields with probability $1-\delta$, 
\begin{align*}
|\X\thetabic-\X\theta^*|_2^2 & \le 2n\tau^2|\theta^*|_0 + t \\
&= 64\sigma^2 \log(ed) |\theta^*|_0 +64\log(12)\sigma^2|\theta^*|_0  + 32\sigma^2\log(1/\delta)\\
&\le 224|\theta^*|_0 \sigma^2 \log(ed)+ 32\sigma^2\log(1/\delta)\,.
\end{align*}
\end{proof}
It follows from Theorem~\ref{TH:BIC} that $\thetabic$ \emph{adapts} to the unknown sparsity of $\theta^*$, just like $\thetahard$. Moreover, this holds under no assumption on the design matrix $\X$.

\subsection{Analysis of the Lasso estimator}

\subsubsection{Slow rate for the Lasso estimator}

The properties of the BIC estimator are quite impressive. It shows that under no assumption on $\X$, one can mimic two oracles: (i) the oracle that knows the support of $\theta^*$ (and computes least squares on this support), up to a $\log(ed)$ term and (ii) the oracle that knows the sparsity $|\theta^*|_0$ of $\theta^*$, up to a smaller logarithmic term $\log(ed/|\theta^*|_0)$, which is subsumed by $\log(ed)$. Actually, the \( \log(ed) \) can even be improved to \( \log(ed/|\theta^\ast|_0) \) by using a modified BIC estimator (see Problem~\ref{EXO:sparse:betterbic}).

The Lasso estimator is a bit more difficult to analyze because, by construction, it should more naturally adapt to the unknown $\ell_1$-norm of $\theta^*$. This can be easily shown as in the next theorem, analogous to Theorem~\ref{TH:ell1const}.

\begin{thm}
\label{TH:lasso_slow}
Assume that the linear model~\eqref{EQ:regmod_matrix} holds where $\eps\sim\sg_n(\sigma^2)$. Moreover, assume   that the columns of $\X$ are normalized in such a way that $\max_j|\X_j|_2\le \sqrt{n}$. Then, the Lasso estimator $\thetalasso$ with regularization parameter
\begin{equation}
\label{EQ:taulasso}
2\tau=2\sigma\sqrt{\frac{2\log(2d)}{n}}+2\sigma\sqrt{\frac{2\log(1/\delta)}{n}}
\end{equation}
satisfies
$$
\MSE(\X\thetalasso)=\frac{1}{n}|\X\thetalasso-\X\theta^*|_2^2\le 4|\theta^*|_1\sigma\sqrt{\frac{2\log(2d)}{n}}+4|\theta^*|_1\sigma\sqrt{\frac{2\log(1/\delta)}{n}}
$$
with probability at least $1-\delta$.
\end{thm}
\begin{proof}
From the definition of $\thetalasso$, it holds
$$
\frac1n|Y-\X\thetalasso|_2^2 + 2\tau|\thetalasso|_1 \le \frac1n|Y-\X\theta^*|_2^2 + 2\tau|\theta^*|_1\,.
$$
Using H\"older's inequality, it implies
\begin{align*}
|\X\thetalasso-\X\theta^*|_2^2  &\le  2\eps^\top\X(\thetalasso-\theta^*)+ 2n\tau\big(|\theta^*|_1-|\thetalasso|_1\big) \\
&\le 2|\X^\top\eps|_\infty|\thetalasso|_1-2n\tau|\thetalasso|_1 +  2|\X^\top\eps|_\infty|\theta^*|_1+2n\tau|\theta^*|_1\\
&= 2(|\X^\top\eps|_\infty-n\tau)|\thetalasso|_1 + 2(|\X^\top\eps|_\infty+n\tau)|\theta^*|_1
\end{align*}
Observe now that for any $t>0$, 
$$
\p(|\X^\top\eps|_\infty\ge t)=\p(\max_{1\le j \le d}|\X_j^\top \eps|>t) \le 2de^{-\frac{t^2}{2n\sigma^2}}
$$
Therefore, taking $t=\sigma\sqrt{2n\log(2d)}+\sigma\sqrt{2n\log(1/\delta)}=n\tau$, we get that with probability at least $1-\delta$, 
\begin{equation*}
|\X\thetalasso-\X\theta^*|_2^2\le 4n\tau|\theta^*|_1\,. \qedhere
\end{equation*}
\end{proof}
Notice that the regularization parameter~\eqref{EQ:taulasso} depends on the confidence level $\delta$. This is not the case for the BIC estimator (see~\eqref{EQ:taubic}).

\medskip

The rate in Theorem~\ref{TH:lasso_slow} is of order $\sqrt{(\log d)/n}$ (\emph{slow rate}), which is much slower than the rate of order $(\log d)/n$ (\emph{fast rate}) for the BIC estimator. Hereafter, we show that fast rates can also be achieved by the computationally efficient Lasso estimator, but at the cost of a much stronger condition on the design matrix $\X$.

\subsubsection{Incoherence}

\begin{assumption}{$\mathsf{INC}(k)$}
We say that the design matrix $\X$ has incoherence $k$ for some integer $k>0$ if
$$
\big|\frac{\X^\top \X}{n}-I_d\big|_\infty\le  \frac{1}{32k}
$$
where the $|A|_\infty$ denotes the largest element of $A$ in absolute value. 
Equivalently,
\begin{enumerate}
\item For all $j=1, \ldots, d$, 
$$
\left|\frac{|\X_j|_2^2}{n}-1\right|\le \frac{1}{32k}\,.
$$
\item For all $1\le i,j\le d$, $i\neq j$, we have 
$$
\frac{|\X_i^\top \X_j|}{n} \le \frac{1}{32k}\,.
$$
\end{enumerate}
\end{assumption}
Note that Assumption~\textsf{ORT} arises as the limiting case of $\mathsf{INC}(k)$ as $k \to \infty$. However, while Assumption~\textsf{ORT} requires $d \le n$, here we may have $d \gg n$ as illustrated in Proposition~\ref{prop:INCRad} below. To that end, we simply have to show that there exists a matrix that satisfies $\mathsf{INC}(k)$ even for $d >n$. We resort to the \emph{probabilistic method} \cite{AloSpe08}. The idea of this method is that if we can find a probability measure that assigns positive probability to objects that satisfy a certain property, then there must exist objects that satisfy said property.

In our case, we consider the following probability distribution on random matrices with entries in $\{\pm 1\}$. Let the design matrix $\X$ have entries that are \iid Rademacher $(\pm 1)$ random variables. We are going to show that most realizations of this random matrix satisfy Assumption~$\mathsf{INC}(k)$ for large enough $n$.

\begin{prop}
\label{prop:INCRad}
Let $\X \in \R^{n\times d}$ be a random matrix with entries $X_{ij}, i=1,\ldots, n, j=1, \ldots, d$, that are \iid Rademacher  $(\pm 1)$ random variables. Then,   $\X$ has incoherence $k$ with probability $1-\delta$ as soon as 
$$
n\ge 2^{11}k^2\log(1/\delta)+2^{13}k^2\log(d)\,.
$$
It implies that there exist matrices that satisfy Assumption~$\mathsf{INC}(k)$ for $$n \geq C k^2\log(d)\,,$$ for some numerical constant $C$.
\end{prop}
\begin{proof}
Let $\eps_{ij} \in \{-1,1\}$ denote the Rademacher random variable that is on the $i$th row and $j$th column of $\X$.

Note first that the $j$th diagonal entries of $\X^\top\X/n$ are given by
$$
\frac{1}{n}\sum_{i=1}^n \eps_{i,j}^2=1.
$$
Moreover, for $j\neq k$, the $(j,k)$th entry of the $d\times d$ matrix $\X^\top\X/n$ is given by
$$
\frac{1}{n}\sum_{i=1}^n \eps_{i,j}\eps_{i,k}=\frac{1}{n}\sum_{i=1}^n \xi_{i}^{(j,k)}\,,
$$
where for each pair $(j,k)$, $\xi_i^{(j,k)}=\eps_{i,j}\eps_{i,k}$, so that $\xi_1^{(j,k)}, \ldots, \xi_n^{(j,k)}$ are \iid Rademacher random variables.

Therefore,  we get that for any $t>0$, 
\begin{align*}
\p\left(\left|\frac{\X^\top \X}{n}-I_d\right|_\infty>t\right) &=\p\Big(\max_{j\neq k}\Big|\frac{1}{n}\sum_{i=1}^n \xi_{i}^{(j,k)}\Big|>t\Big)&\\
&\le \sum_{j\neq k}\p\Big(\Big|\frac{1}{n}\sum_{i=1}^n \xi_{i}^{(j,k)}\Big|>t\Big)&\text{(Union bound)}\\
&\le \sum_{j\neq k}2e^{-\frac{nt^2}{2}}&\text{(Hoeffding: Theorem~\ref{TH:hoeffding})}\\
&\le 2 d^2e^{-\frac{nt^2}{2}}.
\end{align*}
Taking now $t=1/(32k)$ yields
$$
\p\big(\big|\frac{\X^\top \X}{n}-I_d\big|_\infty>\frac{1}{32k}\big)\le 2 d^2e^{-\frac{n}{2^{11}k^2}} \le \delta
$$
for 
\begin{equation*}
n\ge 2^{11}k^2\log(1/\delta)+2^{13}k^2\log(d)\,. \qedhere
\end{equation*}

\end{proof}

\medskip

For any $\theta \in \R^d$, $S\subset\{1, \ldots, d\}$, define $\theta_S$ to be the vector with coordinates
$$
\theta_{S,j}=\left\{\begin{array}{ll}
\theta_{j} & \text{if} \ j \in S\,,\\
0&\text{otherwise}\,.
\end{array}\right.
$$
In particular $|\theta|_1=|\theta_S|_1+|\theta_{S^c}|_1$.

The following lemma holds
\begin{lem}
\label{lem:inc}
Fix a positive integer $k \le d$ and assume that $\X$ satisfies assumption $\mathsf{INC}(k)$. Then, for any  $S \in \{1, \ldots, d\}$ such that $|S|\le k$ and any $\theta \in \R^d$ that satisfies the \emph{cone condition}
\begin{equation}
\label{EQ:conecond}
|\theta_{S^c}|_1 \le 3|\theta_S|_1\,,
\end{equation}
it holds
$$
|\theta|^2_2 \le 2\frac{|\X\theta|_2^2}n.
$$
\end{lem}
\begin{proof}
We have
$$
\frac{|\X\theta|_2^2}n = \frac{|\X\theta_S|_2^2}{n} +\frac{|\X\theta_{S^c}|_2^2}{n}+2\theta_S^\top \frac{\X^\top\X}{n}\theta_{S^c}.
$$
We bound each of the three terms separately.
\begin{itemize}
\item[(i)] First, if follows from the incoherence condition that
$$\frac{|\X\theta_S|_2^2}{n}=\theta_S^\top \frac{\X^\top\X}{n}\theta_S =|\theta_S|_2^2+\theta_S^\top \left(\frac{\X^\top\X}{n} - I_d\right)\theta_S \ge |\theta_S|_2^2-\frac{|\theta_S|_1^2}{32k}\,,
$$
\item[(ii)] Similarly,
$$
\frac{|\X\theta_{S^c}|_2^2}{n} \ge |\theta_{S^c}|_2^2-\frac{|\theta_{S^c}|_1^2}{32k}\ge  |\theta_{S^c}|_2^2-\frac{9|\theta_{S}|_1^2}{32k}\,,
$$
where, in the last inequality, we used the cone condition~\eqref{EQ:conecond}.
\item[(iii)] Finally,
$$
2\Big|\theta_S^\top \frac{ \X^\top\X}{n}\theta_{S^c}\Big|\le  \frac{2}{32k}|\theta_S|_1|\theta_{S^c}|_1\le  \frac{6}{32k}|\theta_S|_1^2.
$$
where, in the last inequality, we used the cone condition~\eqref{EQ:conecond}.\
\end{itemize}
Observe now that it follows from the Cauchy-Schwarz inequality that
$$
|\theta_S|_1^2 \le |S||\theta_S|_2^2.
$$
Thus for $|S|\le k$,
\begin{equation*}
\frac{|\X\theta|_2^2}{n} \ge |\theta_S|_2^2 +|\theta_{S^c}|_2^2-  \frac{16|S|}{32k}|\theta_S|_2^2 \ge \frac12|\theta|_2^2. \qedhere
\end{equation*}
\end{proof}

\subsubsection{Fast rate for the Lasso}

\begin{thm}
\label{TH:lasso_fast}
Fix $n \ge 2$. Assume that the linear model~\eqref{EQ:regmod_matrix} holds where $\eps\sim \sg_n(\sigma^2)$. Moreover, assume that $|\theta^*|_0\le k$ and that $\X$ satisfies assumption \textsf{INC$(k)$}. Then the Lasso estimator $\thetalasso$ with regularization parameter defined by 
$$
2\tau=8\sigma\sqrt{\frac{\log(2d)}{n}}+8\sigma\sqrt{\frac{\log(1/\delta)}{n}}
$$
satisfies
\begin{equation}
\label{EQ:fastMSELasso}
\MSE(\X\thetalasso)=\frac{1}{n}|\X\thetalasso-\X\theta^*|_2^2\lesssim k\sigma^2\frac{\log(2d/\delta)}{n}
\end{equation}
and
\begin{equation}
\label{EQ:bornel2lasso}
|\thetalasso-\theta^*|_2^2\lesssim k\sigma^2\frac{\log(2d/\delta)}{n}\,.
\end{equation}
with probability at least $1-\delta$.
\end{thm}
\begin{proof}
From the definition of $\thetalasso$, it holds
$$
\frac1n|Y-\X\thetalasso|_2^2 \le \frac1n|Y-\X\theta^*|_2^2 + 2\tau|\theta^*|_1- 2\tau|\thetalasso|_1 \,.
$$
Adding $\tau|\thetalasso-\theta^*|_1$ on each side and multiplying by $n$, we get
$$
|\X\thetalasso-\X\theta^*|_2^2+n\tau|\thetalasso-\theta^*|_1 \le2\eps^\top\X(\thetalasso-\theta^*) +n\tau|\thetalasso-\theta^*|_1+ 2n\tau|\theta^*|_1- 2n\tau|\thetalasso|_1 \,.
$$
Applying H\"older's inequality and using the same steps as in the proof of Theorem~\ref{TH:lasso_slow}, we get that with probability $1-\delta$, we get
\begin{align*}
\eps^\top\X(\thetalasso-\theta^*)&\le|\eps^\top\X|_\infty|\thetalasso-\theta^*|_1\\
&\le \frac{n\tau}2|\thetalasso-\theta^*|_1\,,
\end{align*}
where we used the fact that $|\X_j|_2^2 \le n + 1/(32k) \le 2n$.
Therefore, taking $S=\supp(\theta^*)$ to be the support of $\theta^*$, we get
\begin{align}
|\X\thetalasso-\X\theta^*|_2^2 +n\tau|\thetalasso-\theta^*|_1  &\le  2n\tau|\thetalasso-\theta^*|_1+2n\tau|\theta^*|_1- 2n\tau|\thetalasso|_1\nonumber \\
&= 2n\tau|\thetalasso_S-\theta^*|_1+2n\tau|\theta^*|_1- 2n\tau|\thetalasso_S|_1\nonumber\\
&\le 4n\tau|\thetalasso_S-\theta^*|_1. \label{EQ:pr:lassofast1}
\end{align}

In particular, it implies that
\begin{equation}
\label{EQ:conefordelta}
|\thetalasso_{S^c}-\theta_{S^c}^*|_1\le 3 |\thetalasso_{S}-\theta_{S}^*|_1\,.
\end{equation}
so that $\theta=\thetalasso-\theta^*$ satisfies the cone condition~\eqref{EQ:conecond}.
Using now  the Cauchy-Schwarz inequality  and Lemma~\ref{lem:inc} respectively, we get, since $|S|\le k$,
$$
|\thetalasso_S-\theta^*|_1 \le \sqrt{|S|}|\thetalasso_S-\theta^*|_2 \le \sqrt{|S|}|\thetalasso-\theta^*|_2 \le\sqrt{\frac{2k}{n}}|\X\thetalasso-\X\theta^*|_2\,.
$$
Combining this result with~\eqref{EQ:pr:lassofast1}, we find
$$
|\X\thetalasso-\X\theta^*|_2^2 \le 32nk\tau^2\,.
$$
This concludes the proof of the bound on the MSE.
To prove~\eqref{EQ:bornel2lasso}, we use Lemma~\ref{lem:inc} once again to get
\begin{equation*}
|\thetalasso -\theta^*|_2^2\le 2\MSE(\X\thetalasso)\le  64k\tau^2. \qedhere
\end{equation*}
\end{proof}

Note that all we required for the proof was not really incoherence but the conclusion of Lemma~\ref{lem:inc}:
\begin{equation}
\label{EQ:REcond}
\inf_{|S|\le k}\inf_{\theta \in \cC_S}\frac{|\X\theta|_2^2}{n|\theta|^2_2}\ge \kappa,
\end{equation}
where $\kappa=1/2$ and $\cC_S$ is the cone defined by
$$
\cC_S=\big\{ |\theta_{S^c}|_1 \le 3|\theta_S|_1 \big\}\,.
$$
Condition~\eqref{EQ:REcond} is sometimes called \emph{restricted eigenvalue (RE) condition}. Its name comes from the following observation. Note that all $k$-sparse vectors $\theta$ are in a cone $\cC_{S}$ with $|S|\le k$ so that the RE condition implies that the smallest eigenvalue of $\X_S$ satisfies
$\lambda_\textrm{min}(\X_S)\ge n\kappa$ for all $S$ such that $|S|\le k$. Clearly, the RE condition is weaker than incoherence and it can actually be shown that  a design matrix $\X$  of \iid Rademacher random variables satisfies the RE conditions as soon as $n\ge Ck\log(d)$ with positive probability.

\subsection{The Slope estimator}

We noticed that the BIC estimator can actually obtain a rate of \( k\log(ed/k) \), where \( k = |\theta^\ast|_0 \), while the LASSO only achieves \( k \log(ed) \).
This begs the question whether the same rate can be achieved by an efficiently computable estimator.
The answer is yes, and the estimator achieving this is a slight modification of the LASSO, where the entries in the \( \ell_1 \)-norm are weighted differently according to their size in order to get rid of a slight inefficiency in our bounds.

\begin{defn}[Slope estimator]
	Let \( \lambda = (\lambda_1, \dots, \lambda_d) \) be a non-increasing sequence of positive real numbers, \( \lambda_1 \geq \lambda_2 \geq \dots \geq \lambda_d > 0 \).
	For \( \theta = (\theta_1, \dots, \theta_d) \in \R^d \), let \( (\theta^\ast_1, \dots, \theta^\ast_d) \) be a non-increasing rearrangement of the modulus of the entries, \( |\theta_1|, \dots, |\theta_d| \).

	We define the \emph{sorted \( \ell_1 \) norm of \( \theta \)} as
	\begin{equation}
		|\theta|_\ast = \sum_{j=1}^{d} \lambda_j \theta^\ast_j,	
	\end{equation}
	or equivalently as
	\begin{equation}
		|\theta|_\ast = \max_{\phi \in S_d} \sum_{j=1}^{d} \lambda_j | \theta_{\phi(j)} |.
	\end{equation}
	The \emph{Slope estimator} is then given by
	\begin{equation}
		\label{eq:ad}
		\thetaslope \in \argmin_{\theta \in \R^d} \{ \frac{1}{n} | Y - \X \theta|_2^2 + 2 \tau | \theta |_\ast \}
	\end{equation}
	for a choice of tuning parameters \( \lambda \) and \( \tau > 0 \).

\end{defn}

Slope stands for Sorted L-One Penalized Estimation and was introduced in \cite{BogvanSab15}, motivated by the quest for a penalized estimation procedure that could offer a control of false discovery rate for the hypotheses \( H_{0,j}: \theta^\ast_j = 0 \).
We can check that \( |\,.\,|_\ast \) is indeed a norm and that \( \thetaslope \) can be computed efficiently, for example by proximal gradient algorithms, see \cite{BogvanSab15}.

In the following, we will consider
\begin{equation}
	\label{eq:ab}
	\lambda_j = \sqrt{\log (2d/j)}, \quad j = 1, \dots, d.
\end{equation}
With this choice, we will exhibit a scaling in \( \tau \) that leads to the desired high probability bounds, following the proofs in \cite{BelLecTsy16}.

We begin with refined bounds on the suprema of Gaussians.

\begin{lem}
	\label{lem:a}
	Let \( g_1, \dots, g_d \) be zero-mean Gaussian variables with variance at most \( \sigma^2 \).	
	Then, for any $k\le d$,
	\begin{equation}
		\label{eq:j}
		\p\left(\frac{1}{k \sigma^2} \sum_{j = 1}^{k} (g_j^\ast)^2 > t \log \left( \frac{2d}{k} \right) \right) \leq \left( \frac{2d}{k} \right)^{1-3t/8}.
	\end{equation}
\end{lem}

\begin{proof}
	We consider the moment generating function to apply a Chernoff bound.

  Use Jensen's inequality on the sum to get
	\begin{align*}
		\label{eq:j}
		\E \exp \left( \frac{3}{8 k \sigma^2} \sum_{j=1}^{k} (g^\ast_j)^2 \right)
		\leq {} & \frac{1}{k} \sum_{j=1}^{k} \E \exp \left( \frac{3 (g_j^\ast)^2}{8 \sigma^2} \right)\\
		\leq {} & \frac{1}{k} \sum_{j=1}^{d} \E \exp \left( \frac{3 g_j^2}{8 \sigma^2} \right)\\
		\leq {} & \frac{2d}{k},
	\end{align*}
	where we used the bound on the moment generating function of the \( \chi^2 \) distribution from Problem \ref{EXO:chi2} to get that \( \E[\exp(3g^2/8)] = 2 \) for \( g \sim \cN(0,1) \).

	The rest follows from a Chernoff bound.
\end{proof}

\begin{lem}
	\label{lem:b}
Define $[d]:=\{1, \ldots, d\}$. Under the same assumptions as in Lemma \ref{lem:a},
	\begin{equation}
		\sup_{k \in [d]} \frac{(g^\ast_k)}{\sigma \lambda_k} \leq 4 \sqrt{\log(1/\delta)},
	\end{equation}
	with probability at least \( 1 - \delta \), for \( \delta < \frac{1}{2} \).
\end{lem}

\begin{proof}
	We can estimate \( (g^\ast_k)^2 \) by the average of all larger variables,
  \begin{equation}
  	\label{eq:m}
		(g^\ast_k)^2 \leq \frac{1}{k} \sum_{j=1}^{k} (g^\ast_j)^2.
  \end{equation}
	Applying Lemma \ref{lem:a} yields
	\begin{equation}
		\label{eq:n}
		\p\left(\frac{(g^\ast_k)^2}{\sigma^2 \lambda_k^2} > t\right) \leq \left( \frac{2d}{k} \right)^{1-3t/8}
	\end{equation}

	For \( t > 8 \),
	\begin{align}
		\label{eq:o}
		\p \left( \sup_{k \in [d]} \frac{(g^\ast_k)^2}{\sigma^2 \lambda_k^2} > t \right)
		\leq {} & \sum_{j=1}^{d} \left( \frac{2d}{j} \right)^{1-3t/8}\\
		\leq {} & (2d)^{1-3t/8} \sum_{j=1}^{d} \frac{1}{j^2}\\
		\leq {} & 4 \cdot 2^{-3t/8}.
	\end{align}

	In turn, this means that
	\begin{equation}
		\label{eq:p}
		\sup_{k \in [d]} \frac{(g^\ast_k)}{\sigma \lambda_k} \leq 4 \sqrt{\log(1/\delta)},
	\end{equation}
	for \( \delta \leq 1/2 \).
\end{proof}

\begin{thm}
	Fix $n \ge 2$. Assume that the linear model~\eqref{EQ:regmod_matrix} holds where $\eps\sim \cN_n(0,\sigma^2 I_n)$. Moreover, assume that $|\theta^*|_0\le k$ and that $\X$ satisfies assumption \textsf{INC$(k')$} with \( k' \geq 4k \log(2de/k) \). Then the Slope estimator $\thetaslope$ with regularization parameter defined by 
	\begin{equation}
		\tau = 8 \sqrt{2} \sigma \sqrt{\frac{\log(1/\delta)}{n}} 
	\end{equation}
satisfies
\begin{equation}
\label{EQ:fastMSESlope}
\MSE(\X\thetaslope)=\frac{1}{n}|\X\thetaslope-\X\theta^*|_2^2
\lesssim \sigma^2\frac{k\log(2d/k) \log(1/\delta)}{n}
\end{equation}
and
\begin{equation}
\label{EQ:fastL2slope}
|\thetaslope-\theta^*|_2^2\lesssim \sigma^2\frac{k\log(2d/k) \log(1/\delta)}{n}\,.
\end{equation}
with probability at least $1-\delta$.
\end{thm}

\begin{rem}
	The multplicative depedence on \( \log(1/\delta) \) in \eqref{EQ:fastMSESlope} is an artifact of the simplified proof we give here and can be improved to an additive one similar to the lasso case, \eqref{EQ:fastMSELasso}, by appealing to stronger and more general concentration inequalities.
	For more details, see \cite{BelLecTsy16}.
\end{rem}

\begin{proof}
	By minimality of \( \thetaslope \) in \eqref{eq:ad}, adding \( n \tau |\thetaslope - \theta^\ast|_\ast \) on both sides,
\begin{equation}
	\label{eq:r}
	|\X\thetaslope-\X\theta^*|_2^2+n \tau |\thetaslope-\theta^*|_\ast
	\leq 2\eps^\top\X(\thetaslope-\theta^*) +n \tau |\thetaslope-\theta^*|_\ast+ 2 \tau n|\theta^*|_\ast- 2 \tau n|\thetaslope|_\ast \,.
\end{equation}
Set
\begin{equation}
	\label{eq:l}
	u := \thetaslope - \theta^\ast,  \quad g_j = (\X^\top \varepsilon)_j,
\end{equation}

By Lemma \ref{lem:b}, we can estimate
\begin{align}
	\label{eq:k}
	\varepsilon^\top \X u = {} & \sum_{j = 1}^{d} (\X^\top \varepsilon)_j u_j
	\leq \sum_{j = 1}^{d} g^\ast_j u^\ast_j\\
	= {} & \sum_{j=1}^{d} \frac{g^\ast_j}{\lambda_j} (\lambda_j u^\ast_j)\\
	\leq {} & \sup_{j} \left\{ \frac{g^\ast_j}{\lambda_j} \right\} | u |_\ast\\
	\leq {} & 4 \sqrt{2} \sigma \sqrt{n \log(1/\delta)} | u |_\ast
	= \frac{n \tau}{2} |u|_\ast,
\end{align}
where we used that \( |\X_j|_2^2 \leq 2n \).

Pick a permutation \( \phi \) such that \( | \theta^\ast |_\ast = \sum_{j = 1}^{k} \lambda_j | \theta_{\phi(j)} | \) and \( | u_{\phi(k+1)} | \geq \dots \geq | u_{\phi(d)}| \).
Then, noting that \( \lambda_j \) is monotonically decreasing,
\begin{align}
	\label{eq:q}
	| \theta^\ast |_\ast - | \thetaslope |_\ast
	= {} & \sum_{j=1}^{k} \lambda_j | \theta^\ast_{\phi(j)}| - \sum_{j = 1}^{d} \lambda_j (\thetaslope)^\ast_j\\
	\leq {} & \sum_{j=1}^{k} \lambda_j ( | \theta^\ast_{\phi(j)} | - | \thetaslope_{\phi(j)} |) - \sum_{j=k+1}^{d} \lambda_j | \thetaslope_{\phi(j)} |\\
	\leq {} & \sum_{j=1}^{k} \lambda_j | u_{\phi(j)} | - \sum_{j=k+1}^{d} \lambda_j  u^\ast_j \\
	\leq {} & \sum_{j=1}^{k} \lambda_j  u^\ast_{j}  - \sum_{j=k+1}^{d} \lambda_j  u^\ast_j.
\end{align}

Combined with \( \tau = 8 \sqrt{2} \sigma \sqrt{\log(1/\delta)/n} \) and the basic inequality \eqref{eq:r}, we have that
\begin{align}
	\label{eq:s}
	|\X\thetaslope-\X\theta^*|_2^2+n \tau |u|_\ast
	\leq {} & 2n \tau |u|_\ast + 2 \tau n |\theta^*|_\ast- 2 \tau n|\thetaslope|_\ast\\
	\leq {} & 4 n \tau \sum_{j=1}^{k} \lambda_j u^\ast_j, \label{eq:w}
\end{align}
whence
\begin{equation}
	\label{eq:t}
	\sum_{j=k+1}^{d} \lambda_j u^\ast_j \leq 3 \sum_{j=1}^{k} \lambda_j u^\ast_j.
\end{equation}

We can now repeat the incoherence arguments from Lemma \ref{lem:inc}, with \( S \) being the \( k \) largest entries of \( |u| \), to get the same conclusion under the restriction \( \textsf{INC}(k')  \).
First, by exactly the same argument as in Lemma \ref{lem:inc}, we have
\begin{equation}
	\label{eq:ah}
	\frac{|\X u_{S}|_2^2}{n} \ge | u_S |_2^2 - \frac{| u_{S} |_1^2}{32 k'}
	\ge | u_S |_2^2 - \frac{k}{32 k'} | u_S |_2^2.
\end{equation}
Next, for the cross terms, we have
\begin{align}
	\label{eq:u}
	2\Big|u_S^\top \frac{ \X^\top\X}{n}u_{S^c}\Big|
	\leq {} & \frac{2}{32k'} \sum_{i=1}^{k} u^\ast_i \sum_{j=k+1}^{d} u^\ast_j\\
	= {} & \frac{2}{32k'} \sum_{i=1}^{k} u^\ast_i \sum_{j=k+1}^{d} \frac{\lambda_j}{\lambda_j} u^\ast_j\\
	\leq {} & \frac{2}{32k' \lambda_d} \sum_{i=1}^{k} u^\ast_i \sum_{j=k+1}^{d} \lambda_j u^\ast_j\\
	\leq {} & \frac{6}{32k' \lambda_d} \sum_{i=1}^{k} u^\ast_i \sum_{j=1}^{k} \lambda_j u^\ast_j\\
	\leq {} & \frac{6 \sqrt{k}}{32k' \lambda_d} \left(\sum_{i=1}^{k} \lambda_j^2\right)^{1/2} \sum_{i=1}^{k} (u^\ast_i)^2\\
	\leq {} & \frac{6 k}{32k' \lambda_d} \sqrt{\log (2de/k)} | u_S |_2^2.
\end{align}
where we estimated the sum by
\begin{align}
	\sum_{j=1}^{k} \log \left( \frac{2d}{j} \right)
	= {} & k \log(2d) - \log(k!)
	\leq k \log(2d) - k \log(k/e)\\
	= {} & k \log(2de/k),
\end{align}
using Stirling's inequality, \( k! \geq (k/e)^k \).

Finally, from a similar calculation, using the cone condition twice,
\begin{equation}
	\label{eq:v}
	\frac{|\X u_{S^c}|_2^2}{n} \ge |u_{S^c}|_2^2-\frac{9k}{32k' \lambda_d^2} \log(2de/k) |u_S|_2^2.
\end{equation}

Concluding, if \( \textsf{INC}( k') \) holds with \( k' \geq 4k \log(2de/k) \), taking into account that \( \lambda_d \ge 1/2 \), we have
\begin{equation}
	\label{eq:ak}
	\frac{| \X u |_2^2}{n} \ge | u_S |_2^2 + | u_{S^c} |_2^2 - \frac{36 + 12 + 1}{128} | u_S |_2^2 \ge \frac{1}{2} | u |_2^2.
\end{equation}
Hence, from \eqref{eq:w},
\begin{align}
	\label{eq:x}
		|\X u|_2^2+n \tau |u|_\ast
		\leq {} & 4 n \tau \left( \sum_{j=1}^{k} \lambda_j^2 \right)^{1/2} \left( \sum_{j=1}^{k} (u^\ast_j)^2 \right)^{1/2}\\
		\leq {} & 4 \sqrt{2} \tau \sqrt{n k \log (2de/k)} | \X u |_2\\
		= {} & 2^6 \sigma \sqrt{\log(1/\delta)} \sqrt{k \log(2de/k)} | \X u |_2,
\end{align}
which combined yields
\begin{equation}
	\label{eq:y}
	\frac{1}{n} | \X u |_2^2 \leq 2^{12} \sigma^2 \frac{k}{n} \log(2de/k) \log(1/\delta),
\end{equation}
and
\begin{equation}
	\label{eq:z}
	|u|_2^2 \leq 2^{13} \frac{k}{n} \sigma^2 \log(2de/k) \log(1/\delta). \qedhere
\end{equation}
\end{proof}

\newpage
\section{Problem set}

\begin{exercise}\label{EXO:GSM:rigde}
Consider the linear regression model with fixed design with $d \le n$. The \emph{ridge} regression estimator is employed when $\rank(\X^\top \X)<d$ but we are interested in estimating $\theta^*$. It is defined for a given parameter $\tau>0$ by
$$
\thetaridge=\argmin_{\theta \in \R^d}\Big\{\frac1n|Y-\X\theta|_2^2 + \tau|\theta|_2^2\Big\}\,.
$$
\begin{enumerate}[label=(\alph*)]
\item Show that for any $\tau$, $\thetaridge$ is uniquely defined and give its closed form expression.
\item Compute the bias of $\thetaridge$ and show that it is bounded in absolute value by $|\theta^*|_2$.

\end{enumerate}
\end{exercise}

\begin{exercise}\label{EXO:GSM:design}
Let $X=(1, Z, \ldots, Z^{d-1})^\top \in \R^d$ be a random vector
where $Z$ is a random variable. Show that the matrix $\E(XX^\top)$ is
positive definite if $Z$ admits a probability density with
respect to the Lebesgue measure on $\R$.
\end{exercise}

\begin{exercise}\label{EXO:sparse:bound_ell0}
Let $\thetahard$ be the hard thresholding estimator defined in~Definition~\ref{def:hard}.
\begin{enumerate}
\item Show that
$$
|\thetahard|_0 \le |\theta^*|_0 + \frac{|\thetahard-\theta^*|_2^2}{4\tau^2}
$$
\item Conclude that if $\tau$ is chosen as in Theorem~\ref{TH:hard}, then
$$
|\thetahard|_0 \le C |\theta^*|_0
$$
with probability $1-\delta$ for some constant $C$ to be made explicit.
\end{enumerate}
\end{exercise}

\begin{exercise}\label{EXO:sparse:4to3}
In the proof of Theorem~\ref{TH:hard}, show that $4\min(|\theta^*_j|,\tau)$ can be replaced by $3\min(|\theta^*_j|,\tau)$, i.e., that on the event $\cA$, it holds
$$
|\thetahard_j-\theta^*_j|\le 3\min(|\theta^*_j|,\tau)\,.
$$
\end{exercise}

\begin{exercise}\label{EXO:sparse:weaklq}
For any $q>0$,  a vector $\theta \in \R^d$ is said to be in a weak  $\ell_q$ ball of radius $R$ if the decreasing rearrangement $|\theta_{[1]}|\ge |\theta_{[2]}| \ge \dots$ satisfies 
$$
|\theta_{[j]}|\le Rj^{-1/q}\,.
$$
Moreover, we define the weak $\ell_q$ norm of $\theta$ by
$$
|\theta|_{w\ell_q}=\max_{1\le j\le d} j^{1/q}|\theta_{[j]}|.
$$
\begin{enumerate}[label=(\alph*)]
\item Give examples of $\theta, \theta' \in \R^d$ such that
$$
|\theta+\theta'|_{w\ell_1}>|\theta|_{w\ell_1}+|\theta'|_{w\ell_1}
$$
What do you conclude?
\item Show that $|\theta|_{w\ell_q} \le |\theta|_{q}$.
\item Given a sequence \( \theta_1, \theta_2, \dots \), show that if $\lim_{d \to \infty}|\theta_{\{1, \dots, d\}}|_{w\ell_q}<\infty$, then $\lim_{d \to \infty}|\theta_{\{1, \dots, d\}}|_{q'}<\infty$ for all $q'>q$.
\item Show that, for any $q \in (0,2)$ if $\lim_{d \to \infty}|\theta_{\{1, \dots, d\}}|_{w\ell_q}=C$, there exists a constant $C_q>0$ that depends on $q$ but not on $d$ such that under the assumptions of Theorem~\ref{TH:hard}, it holds
$$
|\thetahard-\theta^*|_2^2\le C_q\Big(\frac{\sigma^2 \log 2d}{n}\Big)^{1-\frac{q}{2}}
$$
with probability .99.
\end{enumerate}
\end{exercise}

\begin{exercise}\label{EXO:sparse:variation}
Show that
\begin{align*}
\thetahard&=\argmin_{\theta \in \R^d} \Big\{|y-\theta|_2^2 + 4\tau^2|\theta|_0\Big\}\\
\thetasoft&=\argmin_{\theta \in \R^d} \Big\{|y-\theta|_2^2 + 4\tau|\theta|_1\Big\}
\end{align*}
\end{exercise}

\begin{exercise}\label{EXO:sparse:betterbic}
Assume the linear model~\eqref{EQ:regmod_matrix} with $\eps\sim \sg_n(\sigma^2)$ and $\theta^*\neq 0$. Show that the modified BIC estimator $\hat \theta$ defined by
$$
\hat \theta \in \argmin_{\theta \in \R^d} \Big\{\frac1n|Y-\X\theta|_2^2 + \lambda|\theta|_0\log\Big(\frac{ed}{|\theta|_0}\Big)\Big\}\\
$$
satisfies
$$
\MSE(\X\hat \theta)\lesssim |\theta^*|_0 \sigma^2\frac{ \log\big(\frac{ed}{|\theta^*|_0}\big)}{n}\,.
$$
with probability .99, for appropriately chosen $\lambda$. What do you conclude?
\end{exercise}

\begin{exercise}\label{EXO:sparse:lassoell1}
Assume that the linear model~\eqref{EQ:regmod_matrix} holds where $\eps\sim \sg_n(\sigma^2)$. Moreover, assume the conditions of Theorem~\ref{TH:lsOI} and that the columns of $X$ are normalized in such a way that $\max_j|\X_j|_2\le \sqrt{n}$. Then the Lasso estimator $\thetalasso$ with regularization parameter
$$
2\tau=8\sigma\sqrt{\frac{2\log(2d)}{n}}\,,
$$
satisfies
$$
|\thetalasso|_1 \le C|\theta^*|_1
$$
with probability $1-(2d)^{-1}$ for some constant $C$ to be specified.
\end{exercise}

\chapter{Misspecified Linear Models}
\label{chap:misspecified}

Arguably, the strongest assumption that we made in Chapter~\ref{chap:GSM} is that the regression function $f(x)$ is of the form $f(x)=x^\top \theta^*$. What if this assumption is violated? In reality, we do not really believe in the linear model and we hope that good statistical methods should be \emph{robust} to deviations from this model. This is the problem of model misspecification, and in particular misspecified linear models.

Throughout this chapter, we assume the following model:
\begin{equation}
\label{EQ:regmodgen}
Y_i=f(X_i)+\eps_i,\quad i=1, \ldots, n\,,
\end{equation}
where $\eps=(\eps_1, \ldots, \eps_n)^\top $ is sub-Gaussian with variance proxy $\sigma^2$. Here $X_i \in \R^d$. 
When dealing with fixed design, it will be convenient to consider the vector $g \in \R^n$ defined for any function $g \,:\, \R^d \to \R$ by $g=(g(X_1), \ldots, g(X_n))^\top$. In this case, we can write for any estimator $\hat f \in \R^n$ of $f$, 
$$
\MSE(\hat f)=\frac1n|\hat f - f|_2^2\,.
$$
Even though the model may not be linear, we are interested in studying the statistical properties of various linear estimators introduced in the previous chapters: $\thetals, \thetalsK, \thetalsm, \thetabic, \thetalasso$. Clearly, even with an infinite number of observations, we have no chance of finding a consistent estimator of $f$ if we don't know the correct model. Nevertheless, as we will see in this chapter, something can still be said about these estimators using \emph{oracle inequalities}.

\section{Oracle inequalities}
\subsection{Oracle inequalities}
As mentioned in the introduction, an oracle is a quantity that cannot be constructed without the knowledge of the quantity of interest, here: the regression function. Unlike the regression function itself, an oracle is constrained to take a specific form. For all matter of purposes, an oracle can be viewed as an estimator (in a given family) that can be constructed with an infinite amount of data. This is exactly what we should aim for in misspecified models.

When employing the least squares estimator $\thetals$, we constrain ourselves to estimating functions that are of the form $x\mapsto x^\top \theta$, even though $f$ itself may not be of this form. Therefore, the oracle $\hat f$ is the linear function that is the closest to $f$. 

Rather than trying to approximate $f$ by a linear function $f(x)\approx \theta^\top x$, we make the model a bit more general and consider a dictionary $\cH=\{\varphi_1, \ldots, \varphi_M\}$ of functions where $\varphi_j\,:\, \R^d \to \R$. In this case, we can actually remove the assumption that $X \in \R^d$. Indeed, the goal is now to estimate $f$ using a linear combination of the functions in the dictionary:
$$
f \approx \varphi_\theta:=\sum_{j=1}^M \theta_j \varphi_j\,.
$$
\begin{rem}
If $M=d$ and $\varphi_j(X)=X^{(j)}$ returns the $j$th coordinate of $X \in \R^d$ then the goal is to approximate $f(x)$ by $\theta^\top x$. Nevertheless, the use of a dictionary allows for a much more general framework.
\end{rem}

Note that the use of a dictionary does not affect the methods that we have been using so far, namely penalized/constrained least squares. We use the same notation as before and define
\begin{enumerate}
\item The least squares estimator:
\begin{equation}
\label{EQ:defthetaLSmis}
\thetals\in \argmin_{\theta \in \R^M}\frac{1}{n} \sum_{i=1}^n\big(Y_i-\varphi_\theta(X_i)\big)^2
\end{equation}
\item The  least squares estimator constrained to $K \subset \R^M$:
$$
\thetalsK \in \argmin_{\theta \in K}\frac{1}{n} \sum_{i=1}^n\big(Y_i-\varphi_\theta(X_i)\big)^2
$$
\item The BIC estimator:
\begin{equation}
\label{EQ:defBICmis}
\thetabic \in  \argmin_{\theta \in \R^M}\Big\{\frac{1}{n} \sum_{i=1}^n\big(Y_i-\varphi_\theta(X_i)\big)^2+\tau^2|\theta|_0\Big\}
\end{equation}
\item The Lasso estimator:
\begin{equation}
\label{EQ:defLassomis}
\thetalasso \in  \argmin_{\theta \in \R^M}\Big\{\frac{1}{n} \sum_{i=1}^n\big(Y_i-\varphi_\theta(X_i)\big)^2+2\tau|\theta|_1\Big\}
\end{equation}
\end{enumerate}

\begin{defn}
\label{def:oracle}
Let $R(\cdot)$ be a risk function and let $\cH=\{\varphi_1, \ldots, \varphi_M\}$ be a dictionary  of functions from $\R^d$ to $\R$. Let $K$ be a subset of $\R^M$. The \emph{oracle} on $K$ with respect to $R$ is defined by $\varphi_{\bar \theta}$, where    $\bar \theta \in K$ is such that
$$
R(\varphi_{\bar \theta}) \le R(\varphi_\theta) \,, \qquad \forall\, \theta \in K\,.
$$
Moreover, $R_{K}=R(\varphi_{\bar \theta})$ is called \emph{oracle risk} on $K$. An estimator $\hat f$ is said to satisfy an oracle inequality (over $K$) with remainder term $\phi$ in expectation (resp. with high probability) if there exists a constant $C\ge 1$ such that
$$
\E R(\hat f) \le C\inf_{\theta \in K}R(\varphi_\theta) + \phi_{n,M}(K)\,, 
$$
or
$$
\p\big\{R(\hat f) \le C\inf_{\theta \in K}R(\varphi_\theta) + \phi_{n,M,\delta}(K)  \big\} \ge 1-\delta \,, \qquad \forall \ \delta>0 
$$
respectively. If $C=1$, the oracle inequality is sometimes called \emph{exact}.
\end{defn}

Our goal will be to mimic oracles. The finite sample performance of an estimator at this task is captured by an oracle inequality.
\subsection{Oracle inequality for the least squares estimator}

While our ultimate goal is to prove sparse oracle inequalities for the BIC and Lasso estimator in the case of misspecified model, the difficulty of the extension to this case for linear models is essentially already captured by the analysis for the least squares estimator. In this simple case, can even obtain an exact oracle inequality.
\begin{thm}
\label{TH:LS_mis}
Assume the general regression model~\eqref{EQ:regmodgen} with $\eps\sim\sg_n(\sigma^2)$. Then,  the least squares estimator $\thetals$ satisfies for some numerical constant $C>0$,
$$
\MSE(\varphi_{\thetals})\le \inf_{\theta \in \R^M}\MSE(\varphi_\theta)+C\frac{\sigma^2M}{n}\log(1/\delta)
$$
with probability at least $1-\delta$. 
\end{thm}
\begin{proof}
Note that by definition
$$
|Y-\varphi_{\thetals}|_2^2 \le |Y-\varphi_{\bar \theta}|_2^2
$$
where $\varphi_{\bar \theta}$ denotes the orthogonal projection of $f$ onto the linear span of $\varphi_1, \ldots, \varphi_M$. Since $Y=f+\eps$, we get
$$
|f-\varphi_{\thetals}|_2^2 \le |f-\varphi_{\bar \theta}|_2^2+2\eps^\top(\varphi_{\thetals}-\varphi_{\bar \theta})
$$
Moreover, by Pythagoras's theorem, we have
$$
|f-\varphi_{\thetals}|_2^2 - |f-\varphi_{\bar \theta}|_2^2=|\varphi_{\thetals}-\varphi_{\bar \theta}|_2^2\,.
$$
It yields
$$
|\varphi_{\thetals}-\varphi_{\bar \theta}|_2^2\le 2\eps^\top(\varphi_{\thetals}-\varphi_{\bar \theta})\,.
$$
Using the same steps as the ones following equation~\eqref{EQ:bound_ls_1} for the well specified case, we get
$$
|\varphi_{\thetals}-\varphi_{\bar \theta}|_2^2\lesssim \frac{\sigma^2M}{n}\log(1/\delta)
$$
with probability $1-\delta$. The result of the lemma follows.
\end{proof}

\subsection{Sparse oracle inequality for the BIC estimator}

The analysis for more complicated estimators such as the BIC in Chapter \ref{chap:GSM} allows us to derive oracle inequalities for these estimators.
\begin{thm}
\label{TH:BIC_mis}
Assume the general regression model~\eqref{EQ:regmodgen} with $\eps\sim\sg_n(\sigma^2)$. Then,  the BIC estimator $\thetabic$ with regularization parameter
\begin{equation}
\label{EQ:taubicmis}
\tau^2=16\log(6)\frac{\sigma^2}{n} + 32\frac{\sigma^2 \log (ed)}{n}
\end{equation}
satisfies for some numerical constant $C>0$,
\begin{align*}
\MSE(\varphi_{\thetabic})\le \inf_{\theta \in \R^M}\Big\{3\MSE(\varphi_\theta)+&\frac{C\sigma^2}{n}|\theta|_0\log (eM/\delta)\Big\}+ \frac{C\sigma^2}{n}\log(1/\delta)
\end{align*}
with probability at least $1-\delta$. 
\end{thm}
\begin{proof}
Recall that the proof of Theorem~\ref{TH:BIC} for the BIC estimator begins as follows:
$$
\frac1n|Y-\varphi_{\thetabic}|_2^2 + \tau^2|\thetabic|_0 \le \frac1n|Y-\varphi_{\theta}|_2^2 + \tau^2|\theta|_0\,.
$$
This is true for any $\theta \in \R^M$.
It implies
$$
|f-\varphi_{\thetabic}|_2^2 + n\tau^2|\thetabic|_0 \le |f-\varphi_{ \theta}|_2^2+2\eps^\top(\varphi_{\thetabic}-\varphi_{ \theta}) + n\tau^2|\theta|_0\,.
$$
Note that if $\thetabic= \theta$, the result is trivial. Otherwise,
\begin{align*}
2\eps^\top(\varphi_{\thetabic}-\varphi_\theta) &=2\eps^\top\Big(\frac{\varphi_{\thetabic}-\varphi_\theta}{|\varphi_{\thetabic}-\varphi_\theta|_2}\Big)|\varphi_{\thetabic}-\varphi_\theta|_2\\
&\le \frac2\alpha\Big[\eps^\top\Big(\frac{\varphi_{\thetabic}-\varphi_\theta}{|\varphi_{\thetabic}-\varphi_\theta|_2}\Big)\Big]^2 + \frac{\alpha}{2} |\varphi_{\thetabic}-\varphi_\theta|_2^2\,,
\end{align*}
where we used Young's inequality $2ab\le \frac2\alpha a^2+\frac{\alpha}{2} b^2$, which is valid for $a,b \in \R$ and $\alpha>0$. 
Next, since 
$$
\frac{\alpha}{2} |\varphi_{\thetabic}-\varphi_\theta|_2^2\le \alpha|\varphi_{\thetabic}-f|_2^2+\alpha|\varphi_\theta-f|_2^2\,,
$$
we get for  $\alpha=1/2$,
\begin{align*}
\frac12|\varphi_{\thetabic}-f|_2^2  &\le \frac32|\varphi_{ \theta}-f|_2^2+n\tau^2|\theta|_0 +4\big[\eps^\top\cU(\varphi_{\thetabic}-\varphi_{\theta})\big]^2- n\tau^2|\thetabic|_0\\
&\le \frac32|\varphi_{ \theta}-f|_2^2+2n\tau^2|\theta|_0+4\big[\eps^\top\cU(\varphi_{\thetabic -\theta})\big]^2- n\tau^2|\thetabic-\theta|_0.
\end{align*}
We conclude as in the proof of Theorem~\ref{TH:BIC}.
\end{proof}

The interpretation of this theorem is  enlightening. It implies that the BIC estimator will mimic the best tradeoff between the approximation error $\MSE(\varphi_\theta)$ and the complexity of $\theta$ as measured by its sparsity. In particular, this result, which is sometimes called \emph{sparse oracle inequality}, implies the following oracle inequality. Define the oracle $\bar \theta$ to be such that
$$
\MSE(\varphi_{\bar \theta})=\min_{\theta \in \R^M}\MSE(\varphi_\theta).
$$
Then, with probability at least $1-\delta$,
$$
\MSE(\varphi_{\thetabic})\le 	3\MSE(\varphi_{\bar \theta})+\frac{C\sigma^2}{n}\Big[|\bar \theta|_0\log (eM)+ \log(1/\delta)\Big]\Big\}
$$
If the linear model happens to be correct, then we simply have $\MSE(\varphi_{\bar \theta})=0$.

\subsection{Sparse oracle inequality for the Lasso}

To prove an oracle inequality for the Lasso, we need additional assumptions on the design matrix, such as incoherence. Here, the design matrix is given by the $n\times M$ matrix $\Phi$ with elements $\Phi_{i,j}=\varphi_j(X_i)$.
\begin{thm}
\label{TH:lasso_fast_mis}
Assume the general regression model~\eqref{EQ:regmodgen} with $\eps\sim\sg_n(\sigma^2)$. Moreover, assume that there exists an integer $k$ such that the matrix $\Phi$ satisfies assumption \textsf{INC$(k)$}.  Then, the Lasso estimator $\thetalasso$ with regularization parameter given by
\begin{equation}
\label{EQ:taulassomis}
2\tau=8\sigma\sqrt{\frac{2\log(2M)}{n}}+8\sigma\sqrt{\frac{2\log(1/\delta)}{n}}
\end{equation}
satisfies for some numerical constant $C$,
\begin{align*}
\MSE(\varphi_{\thetalasso})\le \inf_{\substack{\theta \in \R^M\\ |\theta|_0\le k}}\Big\{ \MSE(\varphi_\theta) &+\frac{C\sigma^2}{n}|\theta|_0\log (eM/\delta)\Big\} 
\end{align*}
with probability at least $1-\delta$. 
\end{thm}
\begin{proof}
From the definition of $\thetalasso$, for any $\theta \in \R^M$,
\begin{align*}
	\frac1n|Y-\varphi_{\thetalasso}|_2^2
	\le \frac1n|Y-\varphi_{\theta}|_2^2 + 2\tau|\theta|_1- 2\tau|\thetalasso|_1 \,.
\end{align*}
Expanding the squares, adding $\tau|\thetalasso-\theta|_1$ on each side and multiplying by $n$, we get 
\begin{align}
	\leadeq{|\varphi_{\thetalasso}-f|_2^2-|\varphi_{\theta}-f|_2^2+n\tau|\thetalasso-\theta|_1}\\
	\le {} & 2\eps^\top(\varphi_{\thetalasso}-\varphi_{\theta}) +n\tau|\thetalasso-\theta|_1+ 2n\tau|\theta|_1- 2n\tau|\thetalasso|_1 \,.
	\label{EQ:pr:TH:lassofastmis1}
\end{align}
Next, note that  \textsf{INC}$(k)$ for any $k \ge 1$ implies that $|\varphi_j|_2 \le 2\sqrt{n}$ for all $j=1, \ldots, M$. Applying H\"older's inequality, we get that with probability $1-\delta$, it holds that
$$
2\eps^\top(\varphi_{\thetalasso}-\varphi_{\theta}) \le \frac{n\tau}2|\thetalasso-\theta|_1.
$$
Therefore, taking $S=\supp(\theta)$ to be the support of $\theta$, we get that the right-hand side of~\eqref{EQ:pr:TH:lassofastmis1} is bounded by
\begin{align}
	|\varphi_{\thetalasso}-f|_2^2-|\varphi_{\theta}-f|_2^2+n\tau|\thetalasso-\theta|_1
 &\le  2n\tau|\thetalasso-\theta|_1+2n\tau|\theta|_1- 2n\tau|\thetalasso|_1\nonumber \\
&= 2n\tau|\thetalasso_S-\theta|_1+2n\tau|\theta|_1- 2n\tau|\thetalasso_S|_1\nonumber\\
&\le 4n\tau|\thetalasso_S-\theta|_1\label{EQ:pr:TH:lassofastmis2}
\end{align}
with probability $1-\delta$.

It implies that either $\MSE(\varphi_{\thetalasso})\le \MSE(\varphi_\theta)$ or that
$$
|\thetalasso_{S^c}-\theta_{S^c}|_1\le 3 |\thetalasso_{S}-\theta_{S}|_1\,.
$$
so that $\theta=\thetalasso-\theta$ satisfies the cone condition~\eqref{EQ:conecond}.
Using now  the Cauchy-Schwarz inequality  and Lemma~\ref{lem:inc}, respectively, assuming that $|\theta|_0\le k$, we get 
$$
4n\tau|\thetalasso_S-\theta|_1 \le 4n\tau\sqrt{|S|}|\thetalasso_S-\theta|_2 \le 4\tau\sqrt{2n|\theta|_0}|\varphi_{\thetalasso}-\varphi_{\theta}|_2\,.
$$
Using now the inequality $2ab\le \frac2{\alpha} a^2+\frac{\alpha}{2} b^2$, we get
\begin{align*}
4n\tau|\thetalasso_S-\theta|_1& \le \frac{16\tau^2n|\theta|_0}{\alpha}+ \frac{\alpha}{2}|\varphi_{\thetalasso}-\varphi_{\theta}|_2^2\\
&\le  \frac{16\tau^2n|\theta|_0}{\alpha}+ \alpha|\varphi_{\thetalasso}-f|_2^2+\alpha|\varphi_{\theta}-f|_2^2
\end{align*}
Combining this result with~\eqref{EQ:pr:TH:lassofastmis1} and~\eqref{EQ:pr:TH:lassofastmis2}, we find
$$
(1-\alpha)\MSE(\varphi_{\thetalasso}) \le (1+\alpha)\MSE(\varphi_{\theta})+ \frac{16\tau^2|\theta|_0}{\alpha}\,.
$$
To conclude the proof, it only remains to divide by $1-\alpha$ on both sides of the above inequality and take $\alpha=1/2$.
\end{proof}

\subsection{Maurey's argument}

In there is no sparse $\theta$ such that $\MSE(\varphi_{\theta})$ is small, Theorem~\ref{TH:BIC_mis} is useless whereas the Lasso may still enjoy slow rates. In reality, no one really believes in the existence of sparse vectors but rather of approximately sparse vectors. Zipf's law would instead favor the existence of vectors $\theta$ with absolute coefficients that decay polynomially when ordered from largest to smallest in absolute value. This is the case for example if $\theta$ has a small $\ell_1$ norm but is not sparse. For such $\theta$, the Lasso estimator still enjoys slow rates as in Theorem~\ref{TH:lasso_slow}, which can be easily extended to the misspecified case (see Problem~\ref{EXO:lasso_mis_slow}). As a result, it seems that the Lasso estimator is strictly better than the BIC estimator as long as incoherence holds since it enjoys both fast and slow rates, whereas the BIC estimator seems to be tailored to the fast rate.  Fortunately, such vectors can be well approximated by sparse vectors in the following sense: for any vector $\theta \in \R^M$ such that $|\theta|_1\le 1$, there exists a vector $\theta'$ that is sparse and for which $\MSE(\varphi_{\theta'})$ is not much larger than $\MSE(\varphi_\theta)$. The following theorem quantifies exactly the tradeoff between sparsity and $\MSE$. It is often attributed to B. Maurey and was published by Pisier \cite{Pis81}. This is why it is referred to as \emph{Maurey's argument}.

\begin{thm}
Let $\{\varphi_1, \dots, \varphi_M\}$ be a dictionary normalized in such a way that
$$
\max_{1\le j \le M}|\varphi_j|_2\le D\sqrt{n}\,.
$$
Then for any integer $k$ such that $1\le k \le M$ and any positive $R$, we have
$$
\min_{\substack{\theta \in \R^M\\ |\theta|_0\le k}}\MSE(\varphi_\theta) \le \min_{\substack{\theta \in \R^M\\ |\theta|_1\le R}}\MSE(\varphi_\theta) + \frac{D^2 R^2}{k}\,.
$$
\end{thm}
\begin{proof}
Define
$$
\bar \theta \in \argmin_{\substack{\theta \in \R^M\\ |\theta|_1\le R}}|\varphi_\theta -f|_2^2
$$
and assume without loss of generality that $|\bar \theta_1| \ge |\bar \theta_2|\ge \ldots \ge |\bar \theta_M|$.  
Let now $U \in \R^n$ be a random vector with values in $\{0, \pm R\varphi_1, \ldots, \pm R \varphi_M\}$ defined by
\begin{alignat*}{2}
	\p(U=R\,\mathrm{sign}(\bar \theta_j)\varphi_j)&=\frac{|\bar \theta_j|}{R}\,, &\quad &j=1, \ldots, M,\\
\p(U=0)&=1-\frac{|\bar \theta|_1}{R}\,.
\end{alignat*}
Note that $\E[U]=\varphi_{\bar \theta}$ and $|U|_2 \le RD \sqrt{n}$. Let now $U_1, \ldots, U_k$ be $k$ independent copies of $U$ and define their average $$\bar U=\frac{1}{k}\sum_{i=1}^kU_i\,.$$
Note that $\bar U=\varphi_{\tilde \theta}$ for some $\tilde \theta \in\R^M$ such that $|\tilde \theta|_0\le k$. Moreover, using the Pythagorean Theorem,
\begin{align*}
\E|f-\bar U|_2^2&=\E|f-\varphi_{\bar \theta}+\varphi_{\bar \theta}-\bar U|_2^2\\
&=\E|f-\varphi_{\bar \theta}|_2^2+|\varphi_{\bar \theta}-\bar U|_2^2\\
 &= |f-\varphi_{\bar \theta}|_2^2 + \frac{\E|U-\E[U]|_2^2}{k}\\
&\le |f-\varphi_{\bar \theta}|_2^2 + \frac{(RD\sqrt{n})^2}{k}
\end{align*}
To conclude the proof, note that
$$
\E|f-\bar U|_2^2 =\E|f-\varphi_{\tilde \theta}|_2^2 \ge \min_{\substack{\theta \in \R^M\\ |\theta|_0\le k}}|f-\varphi_{\theta}|_2^2
$$
and divide by $n$.
\end{proof}

Maurey's argument implies the following corollary.

\begin{cor}
Assume that the assumptions of Theorem~\ref{TH:BIC_mis} hold and that the dictionary $\{\varphi_1, \dots, \varphi_M\}$ is  normalized in such a way that
$$
\max_{1\le j \le M}|\varphi_j|_2\le \sqrt{n}\,.
$$
Then there exists a constant $C>0$ such that the BIC estimator satisfies
\begin{align*}
\MSE(\varphi_{\thetabic})\le \inf_{\theta \in \R^M}\Big\{2\MSE(\varphi_\theta)+C\Big[&\frac{\sigma^2|\theta|_0\log (eM)}{n}\wedge \sigma|\theta|_1\sqrt{\frac{\log (eM)}{n}}\Big]\Big\}\\
&+ C\frac{\sigma^2\log(1/\delta)}{n}
\end{align*}
with probability at least $1-\delta$.
\end{cor}
\begin{proof}
Choosing $\alpha=1/3$ in  Theorem~\ref{TH:BIC_mis} yields
$$
\MSE(\varphi_{\thetabic})\le 2\inf_{\theta \in \R^M}\Big\{\MSE(\varphi_\theta)+C\frac{\sigma^2|\theta|_0\log (eM)}{n}\Big\}+C\frac{\sigma^2\log(1/\delta)}{n}
$$
Let $\theta' \in \R^M$.
It follows from Maurey's argument that for any \( k \in [M] \), there exists $\theta = \theta(\theta', k) \in \R^M$ such that $|\theta|_0 = k$ and
$$
\MSE(\varphi_{\theta}) \le \MSE(\varphi_{\theta'})+\frac{2|\theta'|_1^2}{k}
$$
It implies that
$$
\MSE(\varphi_{\theta})+C\frac{\sigma^2|\theta|_0\log (eM)}{n} \le \MSE(\varphi_{\theta'})+\frac{2|\theta'|_1^2}{k}+C\frac{\sigma^2 k\log (eM)}{n}
$$
and furthermore that
$$
\inf_{\theta \in \R^M}\MSE(\varphi_{\theta})+C\frac{\sigma^2|\theta|_0\log (eM)}{n} \le \MSE(\varphi_{\theta'})+\frac{2|\theta'|_1^2}{k}+C\frac{\sigma^2 k\log (eM)}{n}
$$
Since the above bound holds for any \( \theta' \in \mathbb{R}^M \) and \( k \in [M] \), we can take an infimum with respect to both $\theta'$ and $k$ on the right-hand side to get
\begin{align*}
 \inf_{\theta \in \R^M}&\Big\{\MSE(\varphi_{\theta})+C\frac{\sigma^2|\theta|_0\log (eM)}{n} \Big\}\\
 &\le \inf_{\theta' \in \R^M}\Big\{\MSE(\varphi_{\theta'})+C\min_k\Big(\frac{|\theta'|_1^2}{k}+C\frac{\sigma^2k\log (eM)}{n}\Big) \Big\}\,.
\end{align*}
To control the minimum over $k$, we need to consider three cases for the quantity
$$
\bar k=\frac{|\theta'|_1}{\sigma}\sqrt{\frac{n}{\log (eM)}}.
$$
\begin{enumerate}
\item If  $1 \le \bar k \le M$, then we get
$$
\min_k\Big(\frac{|\theta'|_1^2}{k}+C\frac{\sigma^2k\log (eM)}{n}\Big) \le C\sigma|\theta'|_1\sqrt{\frac{\log (eM)}{n}}
$$
\item If $\bar k \le 1$, then 
$$
|\theta'|_1^2\le C\frac{\sigma^2\log (eM)}{n}\,,
$$
which yields
$$
\min_k\Big(\frac{|\theta'|_1^2}{k}+C\frac{\sigma^2k\log (eM)}{n}\Big) \le  C\frac{\sigma^2\log (eM)}{n}
$$
\item If $\bar k \ge M$, then 
$$
\frac{\sigma^2M\log (eM)}{n}\le C \frac{|\theta'|_1^2}{M}\,.
$$
Therefore, on the one hand, if $M \ge \frac{|\theta|_1}{\sigma\sqrt{\log(eM)/n}}$, we get
$$
\min_k\Big(\frac{|\theta'|_1^2}{k}+C\frac{\sigma^2k\log (eM)}{n}\Big) \le C \frac{|\theta'|_1^2}{M} \le C\sigma|\theta'|_1\sqrt{\frac{\log (eM)}{n}}\,.
$$
On the other hand, if $M \le \frac{|\theta|_1}{\sigma\sqrt{\log(eM)/n}}$, then for any $\theta \in \R^M$, we have
\begin{equation*}
\frac{\sigma^2|\theta|_0\log (eM)}{n} \le \frac{\sigma^2M\log (eM)}{n}  \le C\sigma|\theta'|_1\sqrt{\frac{\log (eM)}{n}}\,. \qedhere
\end{equation*}
\end{enumerate}

Combined,
\begin{align*}
	\leadeq{\inf_{\theta' \in \R^M}\Big\{\MSE(\varphi_{\theta'})+C\min_k\Big(\frac{|\theta'|_1^2}{k}+C\frac{\sigma^2k\log (eM)}{n}\Big) \Big\}}\\
	\le {} & \inf_{\theta' \in \R^M}\Big\{\MSE(\varphi_{\theta'}) + C \sigma | \theta' |_1 \sqrt{\frac{\log(eM}{n}} + C \frac{\sigma^2 \log(eM)}{n} \Big\},
\end{align*}
which together with Theorem 3.4 yields the claim.

\end{proof}

Note that this last result holds for any estimator that satisfies an oracle inequality with respect to the $\ell_0$ norm as in Theorem~\ref{TH:BIC_mis}. In particular, this estimator need not be the BIC estimator. An example is the Exponential Screening estimator of~\cite{RigTsy11}.

Maurey's argument allows us to enjoy the best of both the $\ell_0$ and the $\ell_1$ world. The rate adapts to the sparsity of the problem and can be even generalized to $\ell_q$-sparsity (see Problem~\ref{EXO:maurey}). However, it is clear from the proof that this argument is limited to squared $\ell_2$ norms such as the one appearing in $\MSE$ and extension to other risk measures is non-trivial. Some work has been done for non-Hilbert spaces \cite{Pis81, DonDarGur97} using more sophisticated arguments.

\section{Nonparametric regression}

So far, the oracle inequalities that we have derived do not deal with the approximation error $\MSE(\varphi_\theta)$. We kept it arbitrary and simply hoped that it was small.  Note also that in the case of linear models, we simply assumed that the approximation error was zero. As we will see in this section, this error can be quantified under natural smoothness conditions if the dictionary of functions $\cH=\{\varphi_1, \ldots, \varphi_M\}$ is chosen appropriately. In what follows, we assume for simplicity that $d=1$ so that $f\,:\, \R \to \R$ and $\varphi_j\,:\, \R \to \R$. 

\subsection{Fourier decomposition}
Historically, nonparametric estimation was developed before high-dimensional statistics and most results hold for the case where the dictionary $\cH=\{\varphi_1, \ldots, \varphi_M\}$ forms an orthonormal system of $L_2([0,1])$:
$$
\int_0^1\varphi_j^2(x)\ud x =1\,, \quad \int_0^1 \varphi_j(x)\varphi_k(x)\ud x=0, \ \forall\, j\neq k\,.
$$
We will also deal with the case where $M=\infty$.

When $\cH$ is an orthonormal system, the coefficients $\theta_j^* \in \R$ defined by
$$
\theta_j^* =\int_0^1 f(x) \varphi_j(x)\ud x\,,
$$
are called \emph{Fourier coefficients} of $f$.

Assume now that the regression function $f$ admits the following decomposition
$$
f=\sum_{j=1}^\infty \theta_j^* \varphi_j\,.
$$

There exists many choices for the orthonormal system and we give only two as examples.

\begin{example}
	\label{ex:trigbasis}
\emph{Trigonometric basis}. This is an orthonormal basis of $L_2([0,1])$. It is defined by
\begin{eqnarray*}
\varphi_1&\equiv&1\\
\varphi_{2k}(x)&=&\sqrt{2} \cos(2\pi k x)\,,\\
\varphi_{2k+1}(x)&=&\sqrt{2} \sin(2\pi k x)\,,
\end{eqnarray*}
for $k=1,2, \dots$ and $x \in [0,1]$. The fact that it is indeed an orthonormal system can be easily check using trigonometric identities.
\end{example}

The next example has received a lot of attention in the signal (sound, image, \dots) processing community.

\begin{example}
\emph{Wavelets}. 
Let $\psi\,:\, \R \to \R$ be a sufficiently smooth and compactly supported function, called  ``\emph{mother wavelet}". Define the system of functions
$$
\psi_{jk}(x)=2^{j/2}\psi(2^jx-k)\,, \quad j,k \in \Z\,.
$$
It can be shown that for a suitable $\psi$, the dictionary $\{\psi_{j,k}, j,k \in \Z\}$ forms an orthonormal system of $L_2([0,1])$ and sometimes a basis. In the latter case, for any function $g \in L_2([0,1])$, it holds
$$
g=\sum_{j=-\infty}^\infty \sum_{k=-\infty}^\infty \theta_{jk}\psi_{jk}\,, \quad \theta_{jk}=\int_{0}^1 g(x)\psi_{jk}(x)\ud x\,.
$$
The coefficients $\theta_{jk}$ are called \emph{wavelet coefficients} of $g$.

The simplest example is given by the \emph{Haar system} obtained by taking $\psi$ to be the following piecewise constant function (see Figure~\ref{FIG:haar}). We will not give more details about wavelets here but simply point the interested reader to~\cite{Mal09}.
$$
\psi(x)=\left\{
\begin{array}{ll}
1& 0\le x<1/2\\
-1 & 1/2\le x \le 1\\
0 & \text{otherwise}
\end{array}
\right.
$$
\end{example}
\begin{figure}
  \centering
  \psfrag{y}[cb][bl]{\footnotesize \begin{turn}{-0}$\psi(x)$\end{turn}}
    \psfrag{x}[ct][b]{\footnotesize \begin{turn}{0}$x$\end{turn}}

\psfrag{Haar}[l]{\ }
 \includegraphics[width=.6\textwidth]{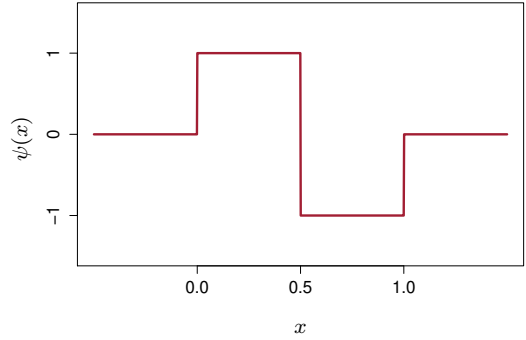}
\caption{The Haar mother wavelet}
\label{FIG:haar}
\end{figure}

\subsection{Sobolev classes and ellipsoids}

We begin by describing a class of smooth functions where smoothness is understood in terms of its number of derivatives. Recall that $f^{(k)}$ denotes the $k$-th derivative of $f$.
\begin{defn}
Fix parameters $\beta \in \{1, 2, \dots\}$ and $L>0$. The Sobolev class of functions $W(\beta,L)$ is defined by
\begin{align*}
W(\beta,L)=\Big\{f\,:\,& [0,1] \to \R\ :\ f\in L_2([0,1])\,, f^{(\beta-1)} \text{  is absolutely continuous and }\\
& \int_{0}^1[f^{(\beta)}]^2\le L^2\,, f^{(j)}(0)= f^{(j)}(1), j=0, \ldots, \beta-1\Big\}
\end{align*}
\end{defn}
Any function $f \in W(\beta, L)$ can represented\footnote{In the sense that $$\lim_{k \to \infty} \int_0^1|f(t)-\sum_{j=1}^k \theta_j \varphi_j(t)|^2\ud t=0$$} as its Fourier expansion along the trigonometric basis:
$$
f(x)=\theta_1^*\varphi_1(x) + \sum_{k=1}^\infty\big(\theta_{2k}^*\varphi_{2k}(x)+\theta_{2k+1}^*\varphi_{2k+1}(x)\big)\,, \quad \forall \ x \in [0,1]\,,
$$
where $\theta^*=\{\theta_j^*\}_{j \ge 1}$ is in the space of squared summable sequences $\ell_2(\N)$ defined by
$$
\ell_2(\N)=\Big\{\theta\,:\, \sum_{j=1}^\infty \theta_j^2 <\infty\Big\}\,.
$$

For any $\beta>0$, define the coefficients
\begin{equation}
\label{EQ:defaj}
a_j=\left\{
\begin{array}{cl}
j^\beta& \text{for $j$ even}\\
(j-1)^\beta & \text{for $j$ odd}
\end{array}
\right.
\end{equation}
With these coefficients, we can define the Sobolev class of functions in terms of Fourier coefficients.

\begin{thm}
Fix $\beta\ge 1$ and $L>0$ and let $\{\varphi_j\}_{j \ge 1}$ denote the trigonometric basis of $L_2([0,1])$. Moreover,  let $\{a_j\}_{j\ge 1}$ be defined as in~\eqref{EQ:defaj}. A function $f \in W(\beta,L)$ can be represented as 
$$
f=\sum_{j=1}^\infty \theta_j^* \varphi_j\,,
$$
where the sequence $\{\theta_j^*\}_{j \ge 1}$ belongs to Sobolev ellipsoid of $\ell_2(\N)$ defined by
$$
\Theta(\beta, Q)=\Big\{\theta \in \ell_2(\N)\,:\, \sum_{j=1}^\infty a_j^2\theta_j^2\le Q\Big\}
$$
for $Q=L^2/\pi^{2\beta}$.
\end{thm}
\begin{proof}
Let us first define the Fourier coefficients $\{s_k(j)\}_{k \ge 1}$ of the $j$th derivative $f^{(j)}$ of $f$ for $j=1, \ldots, \beta$:
\begin{align*}
s_1(j)&=\int_0^1f^{(j)}(t)\ud t=f^{(j-1)}(1)-f^{(j-1)}(0)=0\,,\\
s_{2k}(j)&=\sqrt{2}\int_0^1f^{(j)}(t)\cos(2\pi kt)\ud t\,,\\
s_{2k+1}(j)&=\sqrt{2}\int_0^1f^{(j)}(t)\sin(2\pi kt)\ud t\,,\\
\end{align*}
The Fourier coefficients of $f$ are given by $\theta_k=s_k(0)$.

Using integration by parts, we find that
\begin{align*}
s_{2k}(\beta)&=\sqrt{2}f^{(\beta-1)}(t) \cos(2\pi kt)\Big|_0^1 + (2\pi k)\sqrt{2}\int_0^1f^{(\beta-1)}(t) \sin(2\pi kt)\ud t\\
&= \sqrt{2}[f^{(\beta-1)}(1)-f^{(\beta-1)}(0)]+(2\pi k)\sqrt{2}\int_0^1f^{(\beta-1)}(t) \sin(2\pi kt)\ud t\\
&= (2\pi k)s_{2k+1}(\beta-1)\,.
\end{align*}
Moreover,
\begin{align*}
s_{2k+1}(\beta)&=\sqrt{2}f^{(\beta-1)}(t) \sin(2\pi kt)\Big|_0^1 - (2\pi k)\sqrt{2}\int_0^1f^{(\beta-1)}(t) \cos(2\pi kt)\ud t\\
&= -(2\pi k)s_{2k}(\beta-1)\,.
\end{align*}
In particular, it yields 
$$
s_{2k}(\beta)^2+s_{2k+1}(\beta)^2=(2\pi k)^2\big[s_{2k}(\beta-1)^2+s_{2k+1}(\beta-1)^2\big]
$$
By induction, we find that for any $k \ge 1$,
$$
s_{2k}(\beta)^2+s_{2k+1}(\beta)^2=(2\pi k)^{2\beta}\big(\theta_{2k}^2 + \theta_{2k+1}^2\big)
$$
Next, it follows for the definition~\eqref{EQ:defaj} of $a_j$ that
\begin{align*}
\sum_{k=1}^\infty(2\pi k)^{2\beta}\big(\theta_{2k}^2 + \theta_{2k+1}^2\big)&=\pi^{2\beta}\sum_{k=1}^\infty a_{2k}^2\theta_{2k}^2 + \pi^{2\beta}\sum_{k=1}^\infty a_{2k+1}^2\theta_{2k+1}^2\\
&=\pi^{2\beta}\sum_{j=1}^\infty a_{j}^2\theta_{j}^2\,.
\end{align*}
Together with the Parseval identity, it yields
$$
\int_0^1\big(f^{(\beta)}(t)\big)^2\ud t=\sum_{k=1}^\infty s_{2k}(\beta)^2+s_{2k+1}(\beta)^2=\pi^{2\beta}\sum_{j=1}^\infty a_{j}^2\theta_{j}^2\,.
$$
To conclude, observe that since $f \in W(\beta, L)$, we have
$$
\int_0^1\big(f^{(\beta)}(t)\big)^2\ud t \le L^2\,,
$$
so that $\theta \in \Theta(\beta, L^2/\pi^{2\beta})$\,.
\end{proof}
It can actually be shown that the reciprocal is true, that is, any function with Fourier coefficients in $\Theta(\beta, Q)$ belongs to  if $ W(\beta,L)$, but we will not be needing this.

In what follows, we will define smooth functions as functions with Fourier coefficients (with respect to the trigonometric basis) in a Sobolev ellipsoid. By extension, we write  $f \in \Theta(\beta, Q)$ in this case and consider any real value for $\beta$.

\begin{prop}
The Sobolev ellipsoids enjoy the following properties
\begin{itemize}
\item[(i)] For any $Q>0$, 
$$
0<\beta'<\beta \ \Rightarrow \ \Theta(\beta, Q) \subset \Theta(\beta',Q).
$$
\item[(ii)] For any $Q>0$, 
$$
\beta>\frac12  \ \Rightarrow \ f\ \text{is continuous}.
$$
\end{itemize}
\end{prop}

The proof is left as an exercise (Problem~\ref{PROB:beta12}).

It turns out that the first \( n \) functions in the trigonometric basis are not only orthonormal with respect to the inner product of $L_2$, but also with respect to the inner predictor associated with fixed design, $\langle f,g\rangle:=\frac{1}{n} \sum_{i=1}^n f(X_i)g(X_i)$, when the design is chosen to be \emph{regular}, i.e., $X_i=(i-1)/n$, $i=1, \ldots, n$.
\begin{lem}
\label{LEM:regdesign}
Assume that $\{X_1, \ldots, X_n\}$ is the regular design, i.e., $X_i=(i-1)/n$.  Then, for any $M\le n-1$, the design matrix $\Phi=\{\varphi_j(X_i)\}_{\substack{1\le i \le n\\1\le j \le M}}$ satisfies the \textsf{ORT} condition.
\end{lem}
\begin{proof}
	Fix $j,j'\in \{1, \ldots, n-1\}$, $j\neq j'$, and consider the inner product $\varphi_j^\top \varphi_{j'}$. Write $k_{j}=\lfloor j/2\rfloor$ for the integer part of $j/2$ and define the vectors $a, b, a', b' \in \R^n$ with coordinates such that $e^{\frac{\mathsf{i}2\pi k_j s}{n}}=a_{s+1}+\mathsf{i} b_{s+1}$ and $e^{\frac{\mathsf{i}2\pi k_{j'} s}{n}}=a'_{s+1}+\mathsf{i} b'_{s+1}$ for \( s \in \{0, \dots, n-1\} \). It holds that
$$
\frac{1}{2}\varphi_j^\top \varphi_{j'}\in \{a^\top a'\,,\, b^\top b'\,,\, b^\top a'\,,\, a^\top b'\}\,,
$$
depending on the parity of $j$ and $j'$. 

On the one hand, observe that  if $k_j \neq k_{j'}$, we have for any $\sigma \in \{-1, +1\}$,
$$
\sum_{s=0}^{n-1}e^{\frac{\mathsf{i}2\pi k_j s}{n}}e^{\sigma\frac{\mathsf{i}2\pi k_{j'} s}{n}}=\sum_{s=0}^{n-1}e^{\frac{\mathsf{i}2\pi (k_j+\sigma k_{j'}) s}{n}}=0\,.
$$
On the other hand
$$
\sum_{s=0}^{n-1}e^{\frac{\mathsf{i}2\pi k_j s}{n}}e^{\sigma\frac{\mathsf{i}2\pi k_{j'} s}{n}}=(a+\mathsf{i} b)^\top( a'+\sigma \mathsf{i} b')=a^\top a'-\sigma b^\top b' + \mathsf{i} \big[b^\top a' + \sigma a^\top b'  \big]
$$
so that  $a^\top a'=\pm b^\top b'=0$ and $b^\top a'=\pm a^\top b'=0$ whenever, $k_j \neq k_{j'}$. It yields $\varphi_j^\top \varphi_{j'}=0$.

Next, consider the case where $k_j=k_{j'}$ so that
$$
\sum_{s=0}^{n-1}e^{\frac{\mathsf{i}2\pi k_j s}{n}}e^{\sigma\frac{\mathsf{i}2\pi k_{j'} s}{n}}=\left\{
\begin{array}{ll}
0& \text{if $\sigma=1$}\\
n& \text{if $\sigma=-1$}\\
\end{array}\right.\,.
$$
On the one hand if $j \neq j'$, it can only be the case that $\varphi_j^\top \varphi_{j'} \in \{b^\top a'\,,\, a^\top b'\}$ but the same argument as above yields $b^\top a'=\pm a^\top b'=0$ since the imaginary part of the inner product is still \( 0  \). Hence, in that case, $\varphi_j^\top \varphi_{j'}=0$.
On the other hand, if $j=j'$, then $a=a'$ and $b=b'$ so that it yields  $a^\top a'=|a|_2^2=n$ and $b^\top b'=|b|_2^2=n$ which is equivalent to $\varphi_j^\top \varphi_j=|\varphi_j|_2^2=n$. 
Therefore, the design matrix $\Phi$ is such that
\begin{equation*}
\Phi^\top \Phi=nI_M\,. \qedhere
\end{equation*}
\end{proof}

\subsection{Integrated squared error}

As mentioned in the introduction of this chapter, the smoothness assumption allows us to control the approximation error. Before going into the details,  let us gain some insight. Note first that if $\theta \in \Theta(\beta, Q)$, then $a_j^2\theta_j^2 \to 0$ as $j \to \infty$ so that $|\theta_j|=o(j^{-\beta})$. Therefore, the $\theta_j$s decay polynomially to zero and it makes sense to approximate $f$ by its truncated Fourier series
$$
\sum_{j=1}^M \theta_j^* \varphi_j=:\varphi_{\theta^*}^M
$$
for any fixed $M$. This truncation leads to a systematic error that vanishes as $M \to \infty$. We are interested in understanding the rate at which this happens.

The Sobolev assumption allows us to control precisely this error as a function of the tunable parameter $M$ and the smoothness $\beta$. 

\begin{lem}
\label{TH:bias}
For any integer $M \ge 1$, and $f \in \Theta(\beta, Q)$, $\beta>1/2$, it holds
\begin{equation}
\label{EQ:bias1}
 \|\varphi_{\theta^*}^M -f\|_{L_2}^2=\sum_{j>M} |\theta^*_j|^2 \le QM^{-2\beta}\,.
\end{equation}
and for $M =n-1$, we have
\begin{equation}
\label{EQ:bias2}
|\varphi_{\theta^*}^{n-1} -f|_2^2\le 2n\Big(\sum_{j\ge n}|\theta^*_j|\Big)^2\lesssim Qn^{2-2\beta}\,.
\end{equation}
\end{lem}
\begin{proof}
Note that for any $\theta \in \Theta(\beta, Q)$, if $\beta>1/2$, then
\begin{align*}
\sum_{j=2}^\infty|\theta_j|&=\sum_{j=2}^\infty a_j|\theta_j|\frac1{a_j}\\
&\le \sqrt{\sum_{j=2}^\infty a_j^2\theta_j^2\sum_{j=2}^\infty\frac1{a_j^2}}\quad \text{by Cauchy-Schwarz}\\
&\le \sqrt{Q\sum_{j=1}^\infty\frac1{j^{2\beta}}}<\infty
\end{align*}
Since $\{\varphi_j\}_j$ forms an orthonormal system in $L_2([0,1])$, we have
$$
\min_{\theta \in \R^M} \|\varphi_\theta -f\|_{L_2}^2= \|\varphi_{\theta^*}^M -f\|_{L_2}^2 =\sum_{j>M} |\theta^*_j|^2\,.
$$
When $\theta^* \in \Theta(\beta,Q)$, we have
$$
\sum_{j>M}  |\theta_j^*|^2=\sum_{j>M}  a_j^2|\theta_j^*|^2\frac1{a_j^2}\le \frac1{a_{M+1}^2}Q\le \frac{Q}{M^{2\beta}}\,.
$$

To prove the second part of the lemma, observe that
$$
|\varphi_{\theta^*}^{n-1} -f|_2=\big|\sum_{j\ge n}\theta_j^* \varphi_j\big|_2 \le2\sqrt{2n}\sum_{j\ge n}|\theta_j^*|\,,
$$
where in the last inequality, we used the fact that for the trigonometric basis $|\varphi_j|_2 \le \sqrt{2n}, j \ge 1$ regardless of the choice of the design $X_1,\ldots, X_n$. 
When $\theta^* \in \Theta(\beta,Q)$, we have
\begin{equation*}
	\sum_{j\ge n} |\theta^*_j|=\sum_{j\ge n} a_j|\theta^*_j|\frac{1}{a_j}\le \sqrt{\sum_{j\ge n} a_j^2|\theta^*_j|^2}\sqrt{\sum_{j\ge n}\frac{1}{a_j^2}}\lesssim \sqrt{Q} n^{\frac{1}{2}-\beta}\,. \qedhere
\end{equation*}

\end{proof}
Note that the truncated Fourier series $\varphi_{\theta^*}^M$ is an oracle: this is what we see when we view $f$ through the lens of functions with only low frequency harmonics.

To estimate $\varphi_{\theta^*} = \varphi_{\theta^\ast}^M$, consider the  estimator $\varphi_{\thetals}$ where
$$
\thetals \in \argmin_{\theta \in \R^M} \sum_{i=1}^n \big(Y_i -\varphi_\theta(X_i)\big)^2\,,
$$
which should be such that $\varphi_{\thetals}$ is close to $\varphi_{\theta^*}$. For this estimator, we have proved (Theorem~\ref{TH:LS_mis}) an oracle inequality for the $\MSE$ that is of the form
$$
|\varphi_{\thetals}^M-f|_2^2\le\inf_{\theta \in \R^M}|\varphi_\theta^M-f|_2^2+C\sigma^2M\log(1/\delta)\,, \qquad C>0\,.
$$
It yields
\begin{align*}
|\varphi_{\thetals}^M-\varphi_{\theta^*}^M|_2^2&\le 2(\varphi_{\thetals}^M-\varphi_{\theta^*}^M)^\top(f-\varphi_{\theta^*}^M)+C\sigma^2M\log(1/\delta)\\
&= 2(\varphi_{\thetals}^M-\varphi_{\theta^*}^M)^\top(\sum_{j>M}\theta^*_j\varphi_j)+C\sigma^2M\log(1/\delta)\\
&=2(\varphi_{\thetals}^M-\varphi_{\theta^*}^M)^\top(\sum_{j\ge n}\theta^*_j\varphi_j)+C\sigma^2M\log(1/\delta)\,,
\end{align*}
where we used Lemma~\ref{LEM:regdesign} in the last equality. Together with~\eqref{EQ:bias2} and Young's inequality $2ab\le \alpha a^2+b^2/\alpha, a,b \ge 0$ for any $\alpha>0$, we get
$$
2(\varphi_{\thetals}^M-\varphi_{\theta^*}^M)^\top(\sum_{j\ge n}\theta^*_j\varphi_j)\le \alpha|\varphi_{\thetals}^M-\varphi_{\theta^*}^M
|_2^2 + \frac{C}{\alpha}Qn^{2-2\beta}\,,
$$
for some positive constant $C$ when $\theta^* \in \Theta(\beta, Q)$. As a result, 
\begin{equation}
\label{EQ:almostLSNP}
|\varphi_{\thetals}^M-\varphi_{\theta^*}^M|_2^2 \lesssim  \frac{1}{\alpha(1-\alpha)}Qn^{2-2\beta}+ \frac{\sigma^2M}{1-\alpha}\log(1/\delta)
\end{equation}
for any $\alpha \in (0,1)$. Since, Lemma~\ref{LEM:regdesign} implies, $|\varphi_{\thetals}^M-\varphi_{\theta^*}^M|_2^2=n\|\varphi_{\thetals}^M-\varphi_{\theta^*}^M\|_{L_2([0,1])}^2$, we have proved the following theorem.

\begin{thm}
\label{TH:UB_ls_NP}
Fix  $\beta\ge (1+\sqrt{5})/4\simeq 0.81, Q>0, \delta>0$ and assume the general regression model~\eqref{EQ:regmodgen} with $f \in \Theta(\beta,Q)$ and $\eps\sim\sg_n(\sigma^2), \sigma^2 \le 1$. Moreover, let $M=\lceil n^{\frac{1}{2\beta+1}}\rceil$ and $n$ be large enough so that $M \le n-1$. Then the least squares estimator $\thetals$ defined in~\eqref{EQ:defthetaLSmis} with $\{\varphi_j\}_{j=1}^M$ being the trigonometric basis, satisfies with probability $1-\delta$, for $n$ large enough,
$$
\|\varphi^M_{\thetals} -f\|_{L_2([0,1])}^2 \lesssim n^{-\frac{2\beta}{2\beta+1}}(1+\sigma^2\log(1/\delta))\,.
$$
where the constant factors may depend on $\beta, Q$ and $\sigma$. 
\end{thm}
\begin{proof}
Choosing $\alpha=1/2$ for example and absorbing $Q$ in the constants, we get from~\eqref{EQ:almostLSNP} and Lemma~\ref{LEM:regdesign} that
for $M \le n-1$,
$$
\|\varphi^M_{\thetals}-\varphi^M_{\theta^*}\|^2_{L_2([0,1])}\lesssim n^{1-2\beta}+\sigma^2\frac{M\log(1/\delta)}{n} \,.
$$
Using now Lemma~\ref{TH:bias} and  $\sigma^2 \le 1$, we get 
$$
\|\varphi^M_{\thetals}-f\|^2_{L_2([0,1])}\lesssim M^{-2\beta}+n^{1-2\beta}+\sigma^2\frac{M\log(1/\delta)}{n} \,.
$$
Taking $M=\lceil  n^{\frac{1}{2\beta+1}}\rceil\le n-1$ for $n$ large enough  yields
$$
\|\varphi^M_{\thetals}-f\|^2_{L_2([0,1])}\lesssim n^{-\frac{2\beta}{2\beta+1}}+n^{1-2\beta} \sigma^2\log(1/\delta)\,.
$$
To conclude the proof, simply note that for the prescribed $\beta$, we have $n^{1-2\beta}\le n^{-\frac{2\beta}{2\beta+1}}$\,.
\end{proof}

\subsection{Adaptive estimation}
The rate attained by the projection estimator $\varphi_{\thetals}$ with $M=\lceil n^{\frac{1}{2\beta+1}}\rceil$ is actually optimal so, in this sense, it is a good estimator. Unfortunately, its implementation requires the knowledge of the smoothness parameter $\beta$ which is typically unknown, to determine the level $M$ of truncation. The purpose of \emph{adaptive estimation} is precisely to adapt to the unknown $\beta$, that is to build an estimator that does not depend on $\beta$ and yet, attains a rate of the order of $Cn^{-\frac{2\beta}{2\beta+1}}$ (up to a logarithmic lowdown). To that end, we will use the oracle inequalities for the BIC and Lasso estimator defined in~\eqref{EQ:defBICmis} and~\eqref{EQ:defLassomis} respectively. In view of Lemma~\ref{LEM:regdesign}, the design matrix $\Phi$ actually satisfies the assumption~\textsf{ORT} when we work with the trigonometric basis. This has two useful implications:
\begin{enumerate}
\item Both estimators are actually thresholding estimators and can therefore be implemented efficiently
\item The condition \textsf{INC($k$)} is automatically satisfied for any $k \ge 1$.
\end{enumerate}

These observations lead to the following corollary.

\begin{cor}
\label{COR:adap}
Fix $\beta> (1+\sqrt{5})/4\simeq 0.81, Q>0, \delta>0$ and $n$ large enough to ensure $n-1\ge \lceil n^{\frac{1}{2\beta+1}}\rceil$ assume the general regression model~\eqref{EQ:regmodgen} with $f \in \Theta(\beta,Q)$ and $\eps\sim\sg_n(\sigma^2), \sigma^2 \le 1$. Let $\{\varphi_j\}_{j=1}^{n-1}$ be the trigonometric basis. Denote by $\varphi^{n-1}_{\thetabic}$ (resp. $\varphi^{n-1}_{\thetalasso}$) the BIC (resp. Lasso) estimator defined in~\eqref{EQ:defBICmis} (resp. \eqref{EQ:defLassomis}) over $\R^{n-1}$ with regularization parameter given by~\eqref{EQ:taubicmis} (resp. \eqref{EQ:taulassomis}). Then $\varphi^{n-1}_{\hat \theta}$, where $\hat \theta \in \{\thetabic, \thetalasso\}$ satisfies with probability $1-\delta$,
$$
\|\varphi^{n-1}_{\hat \theta} -f\|_{L_2([0,1])}^2 \lesssim (n/\log n)^{-\frac{2\beta}{2\beta+1}}(1+\sigma^2\log(1/\delta))\,,
$$
where the constants may depend on $\beta$ and $Q$.
\end{cor}
\begin{proof}
For $\hat \theta \in \{\thetabic, \thetalasso\}$, adapting the proofs of Theorem~\ref{TH:BIC_mis} for the BIC estimator and Theorem~\ref{TH:lasso_fast_mis} for the Lasso estimator, for any $\theta \in \R^{n-1}$, with probability $1-\delta$
$$
|\varphi^{n-1}_{\hat \theta} -f|_2^2 \le \frac{1+\alpha}{1-\alpha}|\varphi^{n-1}_\theta-f|_2^2 +  R(|\theta|_0)\,.
$$
where 
$$
R(|\theta|_0):=\frac{C\sigma^2}{\alpha(1-\alpha)}|\theta_0|\log(en/\delta) 
$$
It yields
\begin{align*}
|\varphi^{n-1}_{\hat \theta} -\varphi^{n-1}_\theta|_2^2& \le \frac{2\alpha}{1-\alpha}|\varphi^{n-1}_\theta-f|_2^2 +2(\varphi^{n-1}_{\hat \theta} -\varphi^{n-1}_\theta)^\top (f - \varphi^{n-1}_\theta)+R(|\theta|_0)\\
&\le \Big( \frac{2\alpha}{1-\alpha}+\frac{1}{\alpha}\Big)|\varphi^{n-1}_\theta-f|_2^2+\alpha|\varphi^{n-1}_{\hat \theta} -\varphi^{n-1}_\theta|_2^2 + R(|\theta|_0)\,,
\end{align*}
where we used Young's inequality once again. Let \( M \in \mathbb{N} \) be determined later and choose now $\alpha=1/2$ and $\theta=\theta^*_M$, where $\theta^*_M$ is equal to $\theta^*$ on its first $M$ coordinates and $0$ otherwise so that $\varphi^{n-1}_{\theta^*_M}=\varphi^M_{\theta^*}$. It yields
$$
|\varphi^{n-1}_{\hat \theta} -\varphi^{n-1}_{\theta^*_M}|_2^2\lesssim |\varphi^{n-1}_{\theta^*_M}-f|_2^2+ R(M)\lesssim|\varphi^{n-1}_{\theta^*_M}-\varphi^{n-1}_{\theta^*}|_2^2+|\varphi^{n-1}_{\theta^*}-f|_2^2+ R(M)
$$
Next, it follows from~\eqref{EQ:bias2} that $|\varphi^{n-1}_{\theta^*}-f|_2^2 \lesssim Qn^{2-2\beta}$. Together with   Lemma~\ref{LEM:regdesign}, it yields
$$
\|\varphi^{n-1}_{\hat \theta} -\varphi^{n-1}_{\theta^*_M}\|_{L_2([0,1])}^2\lesssim  \|\varphi^{n-1}_{\theta^*} -\varphi^{n-1}_{\theta^*_M}\|_{L_2([0,1])}^2+Qn^{1-2\beta}+ \frac{R(M)}{n}\,.
$$
Moreover, using~\eqref{EQ:bias1}, we find that
$$
\|\varphi^{n-1}_{\hat \theta} -f\|_{L_2([0,1])}^2\lesssim  M^{-2\beta}+Qn^{1-2\beta}+ \frac{M}{n}\sigma^2\log(en/\delta) \,.
$$
To conclude the proof, choose $M=\lceil (n/\log n)^{\frac{1}{2\beta+1}}\rceil$ and observe that the choice of $\beta$ ensures that $n^{1-2\beta} \lesssim M^{-2\beta}$.
\end{proof}

While there is sometimes a (logarithmic) price to pay for adaptation, it turns out that the extra logarithmic factor can be removed by a clever use of blocks (see \cite[Chapter~3]{Tsy09}). The reason why we get this extra logarithmic factor here is because we use a hammer that's too big. Indeed, BIC and Lasso allow for ``holes" in the Fourier decomposition, so we only get part of their potential benefits.

\newpage
\section{Problem Set}

\begin{exercise}\label{EXO:LS_mis}
Show that the least-squares estimator $\thetals$ defined in~\eqref{EQ:defthetaLSmis} satisfies the following \emph{exact} oracle inequality:
$$
\E\MSE(\varphi_{\thetals})\le \inf_{\theta \in \R^M}\MSE(\varphi_\theta) +C\sigma^2 {\frac{M}{n}}
$$
for some constant $C$ to be specified.
\end{exercise}

\begin{exercise}\label{EXO:lasso_mis_slow}
Assume that  $\eps\sim \sg_n(\sigma^2)$ and  the vectors $\varphi_j$ are normalized in such a way that $\max_j|\varphi_j|_2\le \sqrt{n}$. Show that there exists a choice of $\tau$ such that  the Lasso estimator $\thetalasso$ with regularization parameter $2\tau$ satisfies the following \emph{exact} oracle inequality:

$$
\MSE(\varphi_{\thetalasso})\le \inf_{\theta \in \R^M}\Big\{ \MSE(\varphi_\theta) +C\sigma|\theta|_1 \sqrt{\frac{\log M}{n}}\Big\}
$$
with probability at least $1-M^{-c}$ for some positive constants $C,c$.
\end{exercise}

\begin{exercise}\label{EXO:maurey}
Let $\{\varphi_1, \dots, \varphi_M\}$ be a dictionary normalized in such a way that $\max_j|\varphi_j|_2\le D \sqrt{n}$. Show that for any integer $k$ such that $1\le k \le M$, we have
$$
\min_{\substack{\theta \in \R^M\\ |\theta|_0\le 2k}}\MSE(\varphi_\theta) \le \min_{\substack{\theta \in \R^M\\ |\theta|_{w\ell_q}\le 1}} \MSE(\varphi_\theta) + C_qD^2 \frac{\big(M^\frac{1}{\bar q}-k^\frac{1}{\bar q}\big)^2}{k}\,,
$$
where $|\theta|_{w\ell_q}$ denotes the weak $\ell_q$ norm and $\bar q$ is such that $\frac1q+\frac1{\bar q}=1$, for \( q > 1 \).
\end{exercise}

\begin{exercise}\label{EXO:basis}
Show that the trigonometric basis and the Haar system indeed form an orthonormal system of $L_2([0,1])$.
\end{exercise}

\begin{exercise}\label{EXO:INCtrigo}
	Consider the $n \times d$ random matrix $\Phi=\{\varphi_j(X_i)\}_{\substack{1\le i \le n\\1\le j \le d}}$ where $X_1, \ldots, X_n$ are i.i.d uniform random variables on the interval $[0,1]$ and \( \phi_j \) is the trigonometric basis as defined in Example \ref{ex:trigbasis}. Show that $\Phi$ satisfies \textsf{INC$(k)$} with probability at least $.9$ as long as $n \ge Ck^2  \log (d)$ for some large enough constant $C>0$. 
\end{exercise}

\begin{exercise}\label{PROB:beta12}
If $f \in \Theta(\beta, Q)$ for $\beta>1/2$ and $Q>0$, then $f$ is continuous.
\end{exercise}

\chapter{Minimax Lower Bounds}
\label{chap:minimax}
\newcommand{\KL}{\mathsf{KL}}
\newcommand{\TV}{\mathsf{TV}}

In the previous chapters, we have proved several upper bounds and the goal of this chapter is to assess their optimality. Specifically, our goal is to answer the following questions: 
\begin{enumerate}
\item Can our analysis be improved? In other words: do the estimators that we have studied actually satisfy better bounds?
\item Can any estimator improve upon these bounds?
\end{enumerate}
Both questions ask about some form of \emph{optimality}. The first one is about optimality of an estimator, whereas the second one is about optimality of a bound.

The difficulty of these questions varies depending on whether we are looking for a positive or a negative answer. Indeed, a positive answer to these questions simply consists in finding a better proof for the estimator we have studied (question 1.) or simply finding a better estimator, together with a proof that it performs better (question 2.). A negative answer is much more arduous. For example, in question 2., it is a statement about \emph{all estimators}. How can this be done? The answer lies in information theory (see~\cite{CovTho06} for a nice introduction).

In this chapter, we will see how to give a negative answer to question 2. It will imply a negative answer to question 1.

\section{Optimality in a minimax sense}

Consider the  Gaussian Sequence Model (GSM) where we observe $\bY=(Y_1, \ldots, Y_d)^\top$, defined by
\begin{equation}
\label{EQ:GSMminimax}
Y_i=\theta^*_i+ \varepsilon_i\,, \qquad i=1, \ldots, d\,,
\end{equation}
where $\varepsilon=(\eps_1, \ldots, \eps_d)^\top \sim \cN_d(0,\frac{\sigma^2}{n}I_d)$, $\theta^*=(\theta^*_1, \ldots, \theta^*_d)^\top \in \Theta $ is the parameter of interest and $\Theta \subset \R^d$ is a given set of parameters. We will need a more precise notation for probabilities and expectations throughout this chapter. Denote by $\p_{\theta^*}$ and $\E_{\theta^*}$ the probability measure and corresponding expectation that are associated to the distribution of $\bY$ from the GSM~\eqref{EQ:GSMminimax}.

Recall that GSM is a special case of the linear regression model when the design matrix satisfies the ORT condition. In this case, we have proved several performance guarantees (\emph{upper bounds}) for various choices of $\Theta$ that can be expressed either in the form
\begin{equation}
\label{EQ:minimaxUB_E}
\E\big[|\hat \theta_n -\theta^*|_2^2\big] \le C\phi(\Theta)
\end{equation}
or the form
\begin{equation}
\label{EQ:minimaxUB_P}
|\hat \theta_n -\theta^*|_2^2 \le C\phi(\Theta)\,, \quad \text{with prob.} \ 1-d^{-2}
\end{equation}
For some constant $C$. The rates $\phi(\Theta)$ for different choices of $\Theta$ that we have obtained are gathered in Table~\ref{TAB:minimaxUB} together with the estimator (and the corresponding result from Chapter~\ref{chap:GSM}) that was employed to obtain this rate.
\renewcommand{\arraystretch}{2.5} 
\begin{table}[t]
\begin{center}
\begin{tabular}{c|c|c|c}
$\Theta$ & $\phi(\Theta)$ & Estimator & Result
\\
\hline
\hline
$\R^d$ & $\DS \frac{\sigma^2 d}{n}$ & $\thetals$ & Theorem~\ref{TH:lsOI}\\
$\cB_1$ & $\DS\sigma \sqrt{\frac{ \log d}{n}}$ & $\thetals_{\cB_1}$ & Theorem~\ref{TH:ell1const}\\
$\cB_0(k)$ & $\DS\frac{ \sigma^2 k}{n}\log (ed/k)$ & $\thetals_{\cB_0(k)}$ & Corollaries~\ref{COR:bss1}-\ref{COR:bss2} \\
\end{tabular}
\end{center}
\caption{Rate $\phi(\Theta)$ obtained for different choices of $\Theta$.}\label{TAB:minimaxUB}
\end{table}
\renewcommand{\arraystretch}{1} 
Can any of these results be improved? In other words, does there exists another estimator $\tilde \theta$ such that $\sup_{\theta^* \in \Theta}\E|\tilde \theta -\theta^*|_2^2\ll\phi(\Theta)$?

A first step in this direction is the Cram\'er-Rao lower bound \cite{Sha03} that allows us to prove lower bounds in terms of the Fisher information. Nevertheless, this notion of optimality is too stringent and  often leads to nonexistence of optimal estimators. Rather, we prefer here the notion of \emph{minimax optimality} that characterizes how fast $\theta^*$ can be estimated \emph{uniformly} over $\Theta$.

\begin{defn}
We say that an estimator $\hat \theta_n$ is  \emph{minimax optimal over $\Theta$} if it satisfies \eqref{EQ:minimaxUB_E} and there exists $C'>0$ such that
\begin{equation}
\label{EQ:minimaxLB1}
\inf_{\hat \theta}\sup_{\theta \in \Theta}\E_\theta|\hat \theta- \theta|_2^2 \ge C'\phi(\Theta)
\end{equation}
where the infimum is taker over all estimators (i.e., measurable functions of $\bY$). Moreover,  $\phi(\Theta)$ is called \emph{minimax rate of estimation over $\Theta$}.
\end{defn}
Note that minimax rates of convergence $\phi$ are defined up to multiplicative constants. We may then choose this constant such that the minimax rate has a simple form such as $\sigma^2d/n$ as opposed to $7\sigma^2d/n$ for example.

This definition can be adapted to rates that hold with high probability. As we saw in Chapter~\ref{chap:GSM} (Cf. Table~\ref{TAB:minimaxUB}), the upper bounds in expectation and those with high probability are of the same order of magnitude. It is also the case for lower bounds. Indeed, observe that it follows from the Markov inequality that for any $A>0$,
\begin{equation}
\label{EQ:markov_minimax}
\E_\theta\big[\phi^{-1}(\Theta)|\hat \theta- \theta|_2^2\big] \ge A \p_\theta\big[\phi^{-1}(\Theta)|\hat \theta- \theta|_2^2>A\big] 
\end{equation}
Therefore,~\eqref{EQ:minimaxLB1} follows if we prove that
$$
\inf_{\hat \theta}\sup_{\theta \in \Theta}\p_\theta\big[|\hat \theta- \theta|_2^2>A\phi(\Theta)\big]\ge C
$$
for some positive constants $A$ and $C"$. The above inequality also implies a lower bound with high probability. We can therefore employ the following alternate  definition for minimax optimality.
\begin{defn}
We say that an estimator $\hat \theta$ is  \emph{minimax optimal over $\Theta$} if it satisfies either \eqref{EQ:minimaxUB_E} or \eqref{EQ:minimaxUB_P} and there exists $C'>0$ such that
\begin{equation}
\label{EQ:minimaxLB}
\inf_{\hat \theta}\sup_{\theta \in \Theta}\p_\theta\big[|\hat \theta-\theta|_2^2>\phi(\Theta)\big] \ge C'
\end{equation}
where the infimum is taken over all estimators (i.e., measurable functions of $\bY$). Moreover,  $\phi(\Theta)$ is called \emph{minimax rate of estimation over $\Theta$}.
\end{defn}

\section{Reduction to finite hypothesis testing}
\label{SEC:reduc}
Minimax lower bounds rely on information theory and follow from a simple principle: if the number of observations is too small, it may be hard to distinguish between two probability distributions that are close to each other. For example, given $n$ i.i.d. observations, it is impossible to reliably decide whether they are drawn from $\cN(0,1)$ or $\cN(\frac1n,1)$. This simple argument can be made precise using the formalism of \emph{statistical hypothesis testing}. To do so, we reduce our estimation problem to a testing problem. The reduction consists of two steps.
\begin{enumerate}
\item \textbf{Reduction to a finite number of parameters.} In this step the goal is to find the largest possible number of parameters $\theta_1, \ldots, \theta_M \in \Theta$ under the constraint that
\begin{equation}
\label{EQ:dist_constraint}
|\theta_j -\theta_k|_2^2 \ge 4  \phi(\Theta)\,.
\end{equation}
This problem boils down to a \emph{packing}  of the set $\Theta$.

Then we can use the following trivial observations:
$$
\inf_{\hat \theta}\sup_{\theta \in \Theta}\p_\theta\big[|\hat \theta- \theta|_2^2>\phi(\Theta)\big] \ge \inf_{\hat \theta}\max_{1\le j \le M}\p_{\theta_j}\big[|\hat \theta- \theta_j|_2^2>\phi(\Theta)\big]\,.
$$

\item \textbf{Reduction to a hypothesis testing problem.} In this second step, the necessity of the constraint~\eqref{EQ:dist_constraint} becomes apparent.

For any estimator $\hat \theta$, define the minimum distance test $\psi(\hat \theta)$ that is associated to it by
$$
\psi(\hat \theta)=\argmin_{1\le j \le M} |\hat \theta -\theta_j|_2\,,
$$
with ties broken arbitrarily.

Next observe that if, for some $j=1, \ldots, M$, $\psi(\hat \theta)\neq j$, then there exists $k \neq j$ such that $|\hat \theta -\theta_k|_2 \le |\hat \theta -\theta_j|_2$. Together with the reverse triangle inequality it yields
$$
|\hat \theta - \theta_j|_2 \ge |\theta_j -\theta_k|_2- |\hat \theta - \theta_k|_2  \ge |\theta_j -\theta_k|_2- |\hat \theta - \theta_j|_2 
$$
so that
$$
|\hat \theta - \theta_j|_2 \ge \frac12 |\theta_j -\theta_k|_2
$$
Together with constraint~\eqref{EQ:dist_constraint}, it yields
$$
|\hat \theta - \theta_j|_2^2 \ge  \phi(\Theta)
$$
As a result,
\begin{align*}
\inf_{\hat \theta}\max_{1\le j \le M}\p_{\theta_j}\big[|\hat \theta- \theta_j|_2^2>\phi(\Theta)\big]&\ge \inf_{\hat \theta}\max_{1\le j \le M}\p_{\theta_j}\big[\psi(\hat \theta)\neq j\big]\\
&\ge \inf_{\psi}\max_{1\le j \le M}\p_{\theta_j}\big[\psi\neq j\big]
\end{align*}
where the infimum is taken over all tests $\psi$ based on $\bY$ and that take values in $\{1, \ldots, M\}$.
\end{enumerate}

\noindent \textsc{Conclusion:} it is sufficient for proving lower bounds to find $\theta_1, \ldots, \theta_M \in \Theta$ such that $|\theta_j -\theta_k|_2^2 \ge 4  \phi(\Theta)$ and
$$
\inf_{\psi}\max_{1\le j \le M}\p_{\theta_j}\big[\psi\neq j\big] \ge C'\,.
$$
The above quantity is called \emph{minimax probability of error}. In the next sections, we show how it can be bounded from below using arguments from information theory. For the purpose of illustration, we begin with the simple case where $M=2$ in the next section.

\section{Lower bounds based on two hypotheses}

\subsection{The Neyman-Pearson Lemma and the total variation distance}

Consider two probability measures $\p_0$ and $\p_1$ and observations $X$ drawn from either $\p_0$ or $\p_1$. We want to know which distribution $X$ comes from. It corresponds to the following statistical hypothesis problem:
\begin{eqnarray*}
H_0&:& Z\sim \p_0\\
H_1&:& Z \sim \p_1
\end{eqnarray*}
A test  $\psi=\psi(Z) \in \{0,1\}$  indicates which hypothesis should be true. Any test $\psi$ can make two types of errors. It can commit either  an error of type I ($\psi=1$ whereas $Z \sim \p_0$) or an error of type II ($\psi=0$ whereas $Z \sim \p_1$). Of course, the test may also be correct. The following fundamental result called the \emph{Neyman Pearson Lemma} indicates that any test $\psi$ is bound to commit one of these two types of error with positive probability unless $\p_0$ and $\p_1$ have essentially disjoint support.

Let $\nu$ be a sigma finite measure satisfying $\p_0 \ll \nu$ and $\p_1\ll \nu$. For example, we can take $\nu = \p_0+\p_1$. It follows from the Radon-Nikodym theorem \cite{Bil95} that both $\p_0$ and $\p_1$ admit probability densities with respect to $\nu$. We denote them by $p_0$ and $p_1$ respectively. For any function $f$, we write for simplicity
$$
\int f = \int f(x)\nu (\ud x)
$$
\begin{lem}[Neyman-Pearson]
\label{LEM:Neyman}
Let  $\p_0$ and $\p_1$ be two probability measures. Then for any test $\psi$, it holds
$$
\p_0(\psi=1)+\p_1(\psi=0) \ge \int \min (p_0, p_1)
$$
Moreover, equality holds for the \emph{Likelihood Ratio test} $\psi^\star=\1(p_1\ge p_0)$.
\end{lem}
\begin{proof}
Observe first that
\begin{align*}
\p_0(\psi^\star=1)+\p_1(\psi^\star=0) &=\int_{\psi^*=1}p_0 + \int_{\psi^*=0}p_1\\
&=\int_{p_1\ge p_0}p_0 + \int_{p_1< p_0}p_1\\
&=\int_{p_1\ge p_0}\min(p_0,p_1)+ \int_{p_1< p_0}\min(p_0,p_1)\\
&=\int\min(p_0,p_1)\,.
\end{align*}
Next for any test $\psi$, define its rejection region $R=\{\psi=1\}$. Let $R^\star=\{p_1\ge p_0\}$ denote the rejection region of the likelihood ratio test $\psi^\star$. It holds
\begin{align*}
\p_0(\psi=1)+\p_1(\psi=0) &=1+\p_0(R)-\p_1(R) \\
&=1+\int_Rp_0-p_1\\
&=1+\int_{R\cap R^\star}p_0-p_1 +\int_{R\cap (R^\star)^c}p_0-p_1\\
&=1-\int_{R\cap R^\star}|p_0-p_1|+\int_{R\cap (R^\star)^c}|p_0-p_1|\\
&=1+\int|p_0-p_1|\big[\1(R\cap (R^\star)^c)-\1(R\cap R^\star)\big]
\end{align*}
The above quantity is clearly minimized for $R=R^\star$.
\end{proof}
The lower bound in the Neyman-Pearson lemma is related to a well-known quantity: the total variation distance.

\begin{propdef}
\label{PROP:TV}
The \emph{total variation distance} between two probability measures $\p_0$ and $\p_1$ on a measurable space $(\cX, \cA)$ is defined by
\begin{align*}
\TV(\p_0,\p_1)&=\sup_{R \in \cA}|\p_0(R)-\p_1(R)|&(i)\\
&=\sup_{R \in \cA}\Big|\int_Rp_0-p_1\Big|&(ii)\\
&=\frac{1}{2}\int|p_0-p_1|&(iii)\\
&=1-\int\min(p_0, p_1)&(iv)\\
&=1-\inf_{\psi}\big[\p_0(\psi=1)+\p_1(\psi=0)\big]&(v)
\end{align*}
where the infimum above is taken over all tests.
\end{propdef}
\begin{proof}
Clearly $(i)=(ii)$ and the Neyman-Pearson Lemma gives $(iv)=(v)$. Moreover, by identifying a test $\psi$ to its rejection region, it is not hard to see that $(i)=(v)$. Therefore it remains only to show that $(iii)$ is equal to any of the other expressions. Hereafter, we show that $(iii)=(iv)$. To that end, observe that
\begin{align*}
\int|p_0-p_1|&=\int_{p_1\ge p_0}p_1 -p_0+\int_{p_1< p_0}p_0 -p_1\\
&=\int_{p_1\ge p_0}p_1+\int_{p_1< p_0}p_0-\int\min(p_0,p_1)\\
&=1-\int_{p_1< p_0}p_1+1-\int_{p_1\ge  p_0}p_0-\int\min(p_0,p_1)\\
&=2-2\int\min(p_0,p_1)
\end{align*}
\end{proof}
In view of the Neyman-Pearson lemma, it is clear that if we want to prove large lower bounds, we need to find probability distributions that are close in total variation. Yet, this conflicts with constraint~\eqref{EQ:dist_constraint} and a tradeoff needs to be achieved. To that end, in the Gaussian sequence model, we need to be able to compute the total variation distance between $\cN(\theta_0, \frac{\sigma^2}{n}I_d)$ and $\cN(\theta_1, \frac{\sigma^2}{n}I_d)$. None of the expression in Definition-Proposition~\ref{PROP:TV} gives an easy way to do so. The Kullback-Leibler divergence is much more convenient.

\subsection{The Kullback-Leibler divergence}

\begin{defn}
The Kullback-Leibler divergence  between   probability measures $\p_1$ and $\p_0$  is given by
$$
\KL(\p_1, \p_0)=\left\{
\begin{array}{ll}
\DS \int \log\Big(\frac{\ud \p_1}{\ud \p_0}\Big)\ud \p_1\,,& \text{if}\  \p_1 \ll \p_0\\
\infty\,, & \text{otherwise}\,.
\end{array}\right.
$$
\end{defn}
It can be shown \cite{Tsy09} that the integral is always well defined when  $\p_1 \ll \p_0$ (though it can be equal to $\infty$ even in this case). Unlike the total variation distance, the Kullback-Leibler divergence is not a distance. Actually, it is not even symmetric. Nevertheless, it enjoys properties that are very useful for our purposes.
\begin{prop}
\label{PROP:KL}
Let $\p$ and $\q$ be two probability measures. Then
\begin{enumerate}
\item $\KL(\p, \q)\ge 0$.
\item The function $(\p, \q) \mapsto \KL(\p,\q)$ is convex.
\item If $\p$ and $\q$ are product measures, i.e., 
$$
\p=\bigotimes_{i=1}^n \p_i\quad \text{and} \quad \q=\bigotimes_{i=1}^n \q_i
$$
then
$$
\KL(\p,\q)=\sum_{i=1}^n \KL(\p_i,\q_i)\,.
$$
\end{enumerate}
\end{prop}
\begin{proof}
If $\p$ is not absolutely continuous then the result is trivial. Next, assume that $\p \ll \q$ and let $X \sim \p$.

1. Observe that by Jensen's inequality, 
$$
\KL(\p,\q)=-\E\log\Big(\frac{\ud \q}{\ud \p}(X)\Big) \ge -\log\E\Big(\frac{\ud \q}{\ud \p}(X)\Big)=-\log(1)=0\,.
$$

2. Consider the function $f:(x,y) \mapsto x\log(x/y)$ and compute its Hessian:
$$
\frac{\partial^2f}{\partial x^2}=\frac{1}{x}\,, \qquad \frac{\partial^2f}{\partial y^2}=\frac{x}{y^2}\,, \qquad \frac{\partial^2f}{\partial x\partial y}=-\frac{1}{y}\,, 
$$
We check that the matrix $\DS H= \left(\begin{array}{cc}\frac{1}{x} & -\frac{1}{y} \\-\frac{1}{y} & \frac{x}{y^2}\end{array}\right)$ is positive semidefinite for all $x,y > 0$. To that end, compute its determinant and trace:
$$
\det(H)=\frac{1}{y^2}-\frac{1}{y^2}=0\,, \mathsf{Tr}(H)=\frac{1}{x}\frac{x}{y^2}>0\,.
$$
Therefore, $H$ has one null eigenvalue and one positive one and must therefore be positive semidefinite so that the function $f$ is convex on $(0,\infty)^2$. Since the KL divergence is a sum of such functions, it is also convex on the space of positive measures.

3. Note that if $X=(X_1, \ldots, X_n)$,
\begin{align*}
\KL(\p,\q)&=  \E\log\Big(\frac{\ud \p}{\ud \q}(X)\Big)\\
&=\sum_{i=1}^n \int \log\Big(\frac{\ud \p_i}{\ud \q_i}(X_i)\Big)\ud \p_1(X_1)\cdots \ud \p_n(X_n)\\
&=\sum_{i=1}^n \int \log\Big(\frac{\ud \p_i}{\ud \q_i}(X_i)\Big)\ud \p_i(X_i)\\
&=\sum_{i=1}^n \KL(\p_i,\q_i)
\end{align*}
\end{proof}
Point 2. in Proposition~\ref{PROP:KL} is particularly useful in statistics where observations typically consist of $n$ independent random variables.\begin{example}
\label{ex:KL_Gauss}
For any $\theta \in \R^d$, let $P_\theta$ denote the distribution of $\bY \sim \cN(\theta, \sigma^2I_d)$. Then
$$
\KL(P_\theta, P_{\theta'})=\sum_{i=1}^d \frac{(\theta_i -\theta'_i)^2}{2\sigma^2}=\frac{|\theta-\theta'|_2^2}{2\sigma^2}\,.
$$
The proof is left as an exercise (see Problem~\ref{EXO:KL_Gauss}).
\end{example}

The Kullback-Leibler divergence is easier to manipulate than the total variation distance but only the latter is related to the minimax probability of error. Fortunately, these two quantities can be compared using Pinsker's inequality. We prove here a slightly weaker version of Pinsker's inequality that will be sufficient for our purpose. For a  stronger statement, see \cite{Tsy09}, Lemma~2.5. 

\begin{lem}[Pinsker's inequality.]
\label{LEM:Pinsker}
Let $\p$ and $\q$ be two probability measures such that $\p \ll \q$. Then
$$
\TV(\p,\q) \le \sqrt{\KL(\p,\q)}\,.
$$
\end{lem}
\begin{proof}
Note that
\begin{align*}
\KL(\p,\q)&=\int_{pq>0}p\log\Big(\frac{p}{q}\Big)\\
&=-2\int_{pq>0}p\log\Big(\sqrt{\frac{q}{p}}\Big)\\
&=-2\int_{pq>0}p\log\Big(\Big[\sqrt{\frac{q}{p}}-1\Big]+1\Big)\\
&\ge -2\int_{pq>0}p\Big[\sqrt{\frac{q}{p}}-1\Big]\qquad \text{ (by Jensen)}\\
&=2-2\int\sqrt{pq}\\
\end{align*}
Next, note that
\begin{align*}
\Big(\int\sqrt{pq}\Big)^2&=\Big(\int\sqrt{\max(p,q)\min(p,q)}\Big)^2\\
&\le \int\max(p,q)\int\min(p,q) \qquad \text{(by Cauchy-Schwarz)}\\
&=\big[2-\int\min(p,q)\big]\int\min(p,q)\\
&=\big(1+ \TV(\p,\q)\big)\big(1- \TV(\p,\q)\big)\\
&=1- \TV(\p,\q)^2
\end{align*}
The two displays yield
$$
\KL(\p,\q)\ge 2-2\sqrt{1- \TV(\p,\q)^2}\ge\TV(\p,\q)^2\,,
$$
where we used the fact that $0\le \TV(\p,\q) \le 1$ and $\sqrt{1-x} \le 1-x/2$ for $x \in [0,1]$. 
\end{proof}
Pinsker's inequality yields the following theorem for the GSM.
\begin{thm}
\label{TH:LB2hyp}
Assume that $\Theta$ contains two hypotheses $\theta_0$ and $\theta_1$ such that $|\theta_0-\theta_1|_2^2 = 8\alpha^2\sigma^2/n$ for some $\alpha \in (0,1/2)$. Then
$$
\inf_{\hat \theta} \sup_{\theta \in \Theta}\p_\theta(|\hat \theta-\theta|_2^2 \ge \frac{2\alpha\sigma^2}{n})\ge \frac{1}{2}-\alpha\,.
$$
\end{thm}
\begin{proof}
Write for simplicity $\p_j=\p_{\theta_j}$, $j=0,1$.
Recall that it follows from the reduction to hypothesis testing that
\begin{align*}
\inf_{\hat \theta} \sup_{\theta \in \Theta}\p_\theta(|\hat \theta-\theta|_2^2 \ge  \frac{2\alpha\sigma^2}{n})&\ge \inf_{\psi}\max_{j=0,1}\p_j(\psi\neq j)&\\
&\ge \frac12\inf_{\psi}\Big(\p_0(\psi=1)+\p_1(\psi=0)\Big)&\\
&=\frac12\Big[1-\TV(\p_0,\p_1)\Big]&\hfill \text{(Prop.-def.~\ref{PROP:TV})}\\
&\ge \frac12\Big[1-\sqrt{\KL(\p_1, \p_0)}\Big]& \text{(Lemma~\ref{LEM:Pinsker})}\\
&= \frac12\Big[1-\sqrt{\frac{n|\theta_1-\theta_0|_2^2}{2\sigma^2}}\Big]&\text{(Example~\ref{ex:KL_Gauss})}\\
&= \frac12\Big[1-2\alpha\Big]& \qedhere
\end{align*}
\end{proof}
Clearly, the result of Theorem~\ref{TH:LB2hyp} matches the upper bound for $\Theta=\R^d$ only for $d=1$. How about larger $d$? A quick inspection of our proof shows that our technique, in its present state, cannot yield better results. Indeed, there are only two known candidates for the choice of $\theta^*$. With this knowledge, one can obtain upper bounds that do not depend on $d$ by simply projecting $Y$ onto the linear span of $\theta_0, \theta_1$ and then solving the GSM in two dimensions. To obtain larger lower bounds, we need to use more than two hypotheses. In particular, in view of the above discussion, we need a set of hypotheses that spans a linear space of dimension proportional to $d$. In principle, we should need at least order $d$ hypotheses but we will actually need much more.

\section{Lower bounds based on many hypotheses}

The reduction to hypothesis testing from Section~\ref{SEC:reduc} allows us to use more than two hypotheses. Specifically, we should find $\theta_1, \ldots, \theta_M$ such that 
$$
\inf_{\psi}\max_{1\le j \le M}\p_{\theta_j}\big[\psi\neq j\big] \ge C'\,,
$$
for some positive constant $C'$. Unfortunately, the Neyman-Pearson Lemma no longer exists for more than two hypotheses. Nevertheless, it is possible to relate the minimax probability of error directly to the Kullback-Leibler divergence, without involving the total variation distance. This is possible using a well known result from information theory called \emph{Fano's inequality}.

\begin{thm}[Fano's inequality]
	\label{thm:fano-ineq}
Let  $P_1, \ldots, P_M, M \ge 2$ be probability distributions such that $P_j \ll P_k$, $\forall \, j ,k$. Then
$$
\inf_{\psi}\max_{1\le j \le M}P_{j}\big[\psi(X)\neq j\big] \ge 1-\frac{\frac{1}{M^2}\sum_{j,k=1}^M\KL(P_j,P_k)+\log 2}{\log M}
$$
where the infimum is taken over all tests with values in $\{1, \ldots, M\}$.
\end{thm}
\begin{proof}
Define 
$$
p_j=P_j(\psi=j)\qquad \text{and} \qquad q_j=\frac{1}{M}\sum_{k=1}^M P_k(\psi=j)
$$
so that
$$
\bar p =\frac{1}{M}\sum_{j=1}^M p_j \in (0,1)\,,\qquad  \bar q = \frac1M\sum_{j=1}^M q_j\,.
$$
Moreover, define the KL divergence between two Bernoulli distributions:
$$
\mathsf{kl}(p, q)= p \log \big(\frac{p}{ q}\big)+(1- p) \log \big(\frac{1- p}{1- q}\big)\,.
$$
Note that $\bar p \log \bar p + (1-\bar p) \log(1-\bar p) \ge -\log 2$, which yields
$$
\bar p \le \frac{\mathsf{kl}(\bar p, \bar q)+ \log 2}{\log M}\,.
$$
Next, observe that by convexity (Proposition~\ref{PROP:KL}, 2.), it holds
$$
\mathsf{kl}(\bar p, \bar q) \le \frac{1}{M}\sum_{j=1}^M \mathsf{kl}( p_j, q_j) \le \frac{1}{M^2}\sum_{j.k=1}^M \mathsf{kl}( P_j(\psi=j), P_k(\psi=j)) \,.
$$
It remains to show that
$$
\mathsf{kl}( P_j(\psi=j), P_k(\psi=j)) \le \KL(P_j,P_k)\,.
$$
This result can be seen to follow directly from a well-known inequality often referred to as \emph{data processing inequality} but we are going to prove it directly.

Denote by $\ud P_k^{=j}$ (resp. $\ud P_k^{\neq j}$)  the conditional density of  $P_k$ given $\psi(X)=j$ (resp. $\psi(X)\neq j$) and recall that
\begin{align*}
&\KL(P_j,P_k)=\int\log\Big(\frac{\ud P_j}{\ud P_k}\Big) \ud P_j\\
&=\int_{\psi=j}\log\Big(\frac{\ud P_j}{\ud P_k}\Big) \ud P_j+\int_{\psi\neq j}\log\Big(\frac{\ud P_j}{\ud P_k}\Big) \ud P_j\\
&=P_j(\psi=j)\int\log\Big(\frac{\ud P_j^{=j}}{\ud P_k^{=j}}\frac{P_j(\psi=j)}{P_k(\psi=j)}\Big) \ud P_j^{=j}\\
&\qquad \qquad +P_j(\psi\neq j)\int\log\Big(\frac{\ud P^{\neq j}}{\ud P_k^{\neq j}}\frac{P_j(\psi\neq j)}{P_k(\psi\neq j)}\Big) \ud P_j^{\neq j}\\
&=P_j(\psi=j)\Big(\log\Big(\frac{P_j(\psi=j)}{P_k(\psi=j)} \Big)+ \KL(P_j^{=j},P_j^{=j})\Big)\\
&\qquad \qquad +P_j(\psi\neq j)\Big(\log\Big(\frac{P_j(\psi\neq j)}{P_k(\psi\neq j)} \Big)+ \KL(P_j^{\neq j},P_j^{\neq j})\Big)\\
& \ge \mathsf{kl}( P_j(\psi=j), P_k(\psi=j))\,.
\end{align*}
where we used Proposition~\ref{PROP:KL}, 1.

We have proved that
$$
\frac{1}{M}\sum_{j=1}^M P_j(\psi=j) \le \frac{\frac{1}{M^2}\sum_{j,k=1}^M\KL(P_j,P_k)+\log 2}{\log M}\,.
$$
Passing to the complemetary sets completes the proof of Fano's inequality.
\end{proof}
Fano's inequality leads to the following useful theorem.

\begin{thm}
\label{TH:mainLB}
Assume that $\Theta$ contains $M \ge 5$ hypotheses $\theta_1, \ldots, \theta_M$ such that for some constant $0< \alpha<1/4$, it holds
\begin{itemize}
\item[$(i)$] $\DS|\theta_j-\theta_k|_2^2 \ge 4\phi$
\item[$(ii)$] $\DS |\theta_j-\theta_k|_2^2 \le \frac{2\alpha\sigma^2 }{n}\log(M)$
\end{itemize}
Then
$$
\inf_{\hat \theta} \sup_{\theta \in \Theta}\p_\theta\big(|\hat \theta-\theta|_2^2 \ge \phi\big)\ge \frac{1}{2}-2\alpha\,.
$$
\end{thm}
\begin{proof}
	in view of (\emph{i}), it follows from the reduction to hypothesis testing that it is sufficient to prove that
$$
\inf_{\psi}\max_{1\le j \le M}\p_{\theta_j}\big[\psi\neq j\big] \ge \frac{1}{2}-2\alpha.
$$
If follows from (\emph{ii}) and Example~\ref{ex:KL_Gauss} that
$$
\KL(\p_j,\p_k)=\frac{n|\theta_j-\theta_k|_2^2}{2\sigma^2}\le \alpha\log(M)\,.
$$
Moreover, since $M \ge 5$, 
$$
\frac{\frac{1}{M^2}\sum_{j,k=1}^M\KL(\p_j,\p_k)+\log 2}{\log(M-1)}\le \frac{\alpha\log(M)+\log 2}{\log(M-1)}\le 2\alpha+\frac12\,.
$$
The proof then follows from Fano's inequality.
\end{proof}
Theorem~\ref{TH:mainLB} indicates that we must take $\phi\le \frac{\alpha\sigma^2 }{2n}\log(M)$. Therefore, the larger the $M$, the larger the lower bound can be. However, $M$ cannot be arbitrary larger because of the constraint $(i)$. We are therefore facing a \emph{packing} problem where the goal is to ``pack" as many Euclidean balls of radius proportional to  $\sigma\sqrt{\log(M)/n}$ in $\Theta$ under the constraint that their centers remain close together (constraint $(ii)$). If $\Theta=\R^d$, this the goal is to pack the Euclidean ball of radius $R=\sigma\sqrt{2\alpha\log(M)/n}$ with Euclidean balls of radius $R\sqrt{2\alpha/\gamma}$. This can be done using a volume argument (see Problem~\ref{EXO:packing}). However, we will use the more versatile lemma below. It gives a a lower bound on the size of a  packing of the discrete hypercube $\{0,1\}^d$ with respect to the \emph{Hamming distance} defined by
$$
\rho(\omega,\omega')=\sum_{i=1}^d\1(\omega_i \neq \omega'_j)\,, \qquad \forall\, \omega, \omega' \in \{0,1\}^d.
$$
\begin{lem}[Varshamov-Gilbert]
\label{LEM:VG}
For any $\gamma \in (0,1/2)$, there exist binary vectors $\omega_{1}, \ldots \omega_{M} \in \{0,1\}^d$ such that
\begin{itemize}
\item[$(i)$] $\DS\rho(\omega_{j}, \omega_{k})\ge \big(\frac12-\gamma\big)d$ for all $j \neq k$\,,
\item[$(ii)$] $\DS M =\lfloor e^{\gamma^2d}\rfloor\ge e^{\frac{\gamma^2d}{2}}$\,.
\end{itemize}
\end{lem}
\begin{proof}
Let $\omega_{j,i}$, $1\le i \le d, 1\le j \le M$ be \iid Bernoulli random variables with parameter $1/2$ and observe that
$$
d-\rho(\omega_{j}, \omega_{k})=X\sim \Bin(d,1/2)\,.
$$
Therefore it follows from a union bound that
$$
\p\big[\exists j\neq k\,, \rho(\omega_{j}, \omega_{k})< \big(\frac12-\gamma\big)d\big] \le \frac{M(M-1)}{2}\p\big(X-\frac{d}2> \gamma d\big)\,.
$$
Hoeffding's inequality then yields 
$$
\frac{M(M-1)}{2}\p\big(X-\frac{d}2> \gamma d\big)\le \exp\Big(-2\gamma^2d+\log\big( \frac{M(M-1)}{2}\big)\Big)<1
$$
as soon as
$$
M(M-1)< 2\exp\Big(2\gamma^2d\Big).
$$
A sufficient condition for the above inequality to hold is to take $M=\lfloor e^{\gamma^2d}\rfloor \ge  e^{\frac{\gamma^2d}{2}}$. For this value of $M$, we have
$$
\p\big(\forall j\neq k\,, \rho(\omega_{j}, \omega_{k})\ge \big(\frac12-\gamma\big)d\big)>0
$$
and by virtue of the probabilistic method, there exist $\omega_{1}, \ldots \omega_{M} \in \{0,1\}^d$ that satisfy $(i)$ and $(ii)$
\end{proof}

\section{Application to the Gaussian sequence model}
\label{sec:lb-gsm}

We are now in a position to apply Theorem~\ref{TH:mainLB} by choosing $\theta_1, \ldots, \theta_M$ based on $\omega_{1}, \ldots, \omega_{M}$ from the Varshamov-Gilbert Lemma. 

\subsection{Lower bounds for estimation}

Take $\gamma=1/4$ and apply the Varshamov-Gilbert Lemma to obtain $\omega_{1}, \ldots, \omega_{M}$ with $M=\lfloor e^{d/16}\rfloor\ge e^{d/32} $ and such that
$\rho(\omega_{j}, \omega_{k})\ge d/4$ for all $j \neq k$. Let $\theta_{1},\ldots, \theta_M$ be such that 
$$
\theta_j=\omega_j\frac{\beta\sigma}{\sqrt{n}}\,,
$$
for some $\beta>0$ to be chosen later. We can check the conditions of Theorem~\ref{TH:mainLB}:
\begin{itemize}
\item[$(i)$] $\DS|\theta_j-\theta_k|_2^2=\frac{\beta^2\sigma^2 }{n}\rho(\omega_{j}, \omega_{k}) \ge 4\frac{\beta^2\sigma^2 d}{16n}$;
\item[$(ii)$] $\DS |\theta_j-\theta_k|_2^2 =\frac{\beta^2\sigma^2 }{n}\rho(\omega_{j}, \omega_{k}) \le \frac{\beta^2\sigma^2 d}{n}\le \frac{32\beta^2\sigma^2 }{n}\log(M)=\frac{2\alpha \sigma^2}{n}\log(M)$\,,
\end{itemize}
for $\beta=\frac{\sqrt{\alpha}}{4}$. Applying now Theorem~\ref{TH:mainLB} yields
$$
\inf_{\hat \theta} \sup_{\theta \in \R^d}\p_\theta\big(|\hat \theta-\theta|_2^2 \ge \frac{\alpha}{256}\frac{\sigma^2 d}{n}\big)\ge \frac{1}{2}-2\alpha\,.
$$
It implies the following corollary.
\begin{cor}
The minimax rate of estimation of over $\R^d$ in the Gaussian sequence model is $\phi(\R^d)=\sigma^2d/n$. Moreover, it is attained by the least squares estimator $\thetals=\bY$. 
\end{cor}
Note that this rate is minimax over sets $\Theta$ that are strictly smaller than $\R^d$ (see Problem~\ref{EXO:minimax_other}). Indeed, it is minimax over any subset of $\R^d$ that contains $\theta_1, \ldots, \theta_M$.

\subsection{Lower bounds for sparse estimation}

It appears from Table~\ref{TAB:minimaxUB} that when estimating sparse vectors, we have to pay for an extra logarithmic term $\log(ed/k)$ for not knowing  the sparsity pattern of the unknown $\theta^*$. In this section, we show that this term is unavoidable as it appears in the minimax optimal rate of estimation of sparse vectors.

Note that the vectors $\theta_1, \ldots, \theta_M$ employed in the previous subsection are not guaranteed to be sparse because the vectors $\omega_1, \ldots, \omega_M$ obtained from the Varshamov-Gilbert Lemma may themselves not be sparse. To overcome this limitation, we need a sparse version of the Varhsamov-Gilbert lemma.
\begin{lem}[Sparse Varshamov-Gilbert]
\label{LEM:sVG}
There exist positive constants $C_1$ and $C_2$ such that the following holds for any two integers $k$ and $d$ such that $1\le k \le d/8$. 
There exist binary vectors $\omega_{1}, \ldots \omega_{M} \in \{0,1\}^d$ such that
\begin{itemize}
\item[$(i)$] $\DS\rho(\omega_{i}, \omega_{j})\ge \frac{k}2$ for all $i \neq j$\,,
\item[$(ii)$] $\DS \log(M) \ge \frac{k}{8}\log(1+\frac{d}{2k})$\,.
\item[$(iii)$] $\DS |\omega_j|_0=k$ for all $j$\,.
\end{itemize}
\end{lem}
\begin{proof}
Take $\omega_1, \ldots, \omega_M$ independently and uniformly at random from the set 
$$
C_0(k)=\{\omega\in \{0,1\}^d\,:\, |\omega|_0=k\}
$$
of $k$-sparse binary random vectors. Note that $C_0(k)$ has cardinality $\binom{d}{k}$. To choose $\omega_j$ uniformly from $C_0(k)$, we proceed as follows. Let $U_1, \ldots, U_k \in \{1, \ldots, d\}$ be $k$ random variables such that $U_1$ is drawn uniformly at random from $\{1, \ldots, d\}$ and for any $i=2, \ldots, k$, conditionally on $U_1, \ldots, U_{i-1}$, the random variable $U_i$ is drawn uniformly at random from $\{1, \ldots, d\}\setminus \{U_1, \ldots, U_{i-1}\}$. Then define

$$
\omega=\left\{
\begin{array}{ll}
1 & \text{if}\  i \in \{U_1, \ldots, U_k\} \\
0 & \text{otherwise}\,.
\end{array}\right.
$$
Clearly, all outcomes in $C_0(k)$ are equally likely under this distribution and therefore, $\omega$ is uniformly distributed on $C_0(k)$. Observe that
\begin{align*}
\p\big( \exists\, \omega_j \neq \omega_k\,:\, \rho(\omega_j, \omega_k) <k\big)&=\frac{1}{\binom{d}{k}} \sum_{\substack{x \in \{0,1\}^d\\ |x|_0=k}}\p\big( \exists\, \omega_j \neq x\,:\, \rho(\omega_j,x) <\frac{k}{2}\big)\\
&\le \frac{1}{\binom{d}{k}} \sum_{\substack{x \in \{0,1\}^d\\ |x|_0=k}}\sum_{j=1}^M\p\big(\omega_j \neq x\,:\, \rho(\omega_j,x) <\frac{k}{2}\big)\\
&=M\p\big(\omega \neq x_0\,:\, \rho(\omega,x_0) <\frac{k}{2}\big),
\end{align*}
where $\omega$ has the same distribution as $\omega_1$ and $x_0$ is any $k$-sparse vector in $\{0,1\}^d$.
The last equality holds by symmetry since (i) all the $\omega_j$s have the same distribution and (ii) all the outcomes of $\omega_j$ are equally likely.

Note that 
$$
\rho(\omega, x_0) \ge k-\sum_{i=1}^k Z_i\,,
$$
where $Z_i=\1(U_i \in \supp(x_0))$. Indeed the left hand side is the number of coordinates on which the vectors $\omega, x_0$ disagree and the right hand side is the number of coordinates in $\supp(x_0)$ on which the two vectors disagree.  In particular, we have that $Z_1 \sim\Bern(k/d)$ and for any $i=2, \ldots, d$, conditionally on $Z_1, \ldots, Z_{i-i}$, we have $Z_i\sim \Bern(Q_i)$, where
$$
Q_i=\frac{k-\sum_{l=1}^{i-1}Z_l}{p-(i-1)}\le \frac{k}{d-k}\le \frac{2k}{d},
$$
since $k \le d/2$.

Next we apply a Chernoff bound to get that for any $s>0$,
$$
\p\big(\omega \neq x_0\,:\, \rho(\omega,x_0) <\frac{k}{2}\big)\le \p\big(\sum_{i=1}^kZ_i>\frac{k}{2}\big)= \E\Big[\exp\big(s\sum_{i=1}^kZ_i\big)\Big]e^{-\frac{sk}{2}}
$$
The above MGF can be controlled by induction on $k$ as follows:
\begin{align*}
\E\Big[\exp\big(s\sum_{i=1}^kZ_i\big)\Big]&=\E\Big[\exp\big(s\sum_{i=1}^{k-1}Z_i\big)\E\exp\big(sZ_k\big| Z_1,\ldots, Z_{k=1}\big)\Big]\\
&=\E\Big[\exp\big(s\sum_{i=1}^{k-1}Z_i\big)(Q_k(e^s-1)+1)\Big]\\
&\le \E\Big[\exp\big(s\sum_{i=1}^{k-1}Z_i\big)\Big]\big(\frac{2k}{d}(e^s-1)+1\big)\\
&\quad \vdots\\
&\le \big(\frac{2k}{d}(e^s-1)+1\big)^k\\
&=2^k
\end{align*}
For $s=\log(1+\frac{d}{2k})$. Putting everything together, we get
\begin{align*}
\p\big( \exists\, \omega_j \neq \omega_k\,:\, \rho(\omega_j, \omega_k) <k\big)&\le\exp\Big(\log M+k\log 2-\frac{sk}{2}\Big)\\
&=\exp\Big(\log M+k\log 2-\frac{k}{2}\log(1+\frac{d}{2k})\Big)\\
&\le \exp\Big(\log M+k\log 2-\frac{k}{2}\log(1+\frac{d}{2k})\Big)\\
&\le \exp\Big(\log M-\frac{k}{4}\log(1+\frac{d}{2k})\Big)\qquad \text{(for $d\ge 8k$)}\\
&<1\,,
\end{align*}
if we take $M$ such that
\begin{equation*}
\log M<\frac{k}{4}\log(1+\frac{d}{2k}). \qedhere
\end{equation*}
\end{proof}
Apply the sparse Varshamov-Gilbert lemma to obtain $\omega_{1}, \ldots, \omega_{M}$ with $\log(M)\ge \frac{k}{8}\log(1+\frac{d}{2k})$ and such that
$\rho(\omega_{j}, \omega_{k})\ge k/2$ for all $j \neq k$. Let $\theta_{1},\ldots, \theta_M$ be such that 
$$
\theta_j=\omega_j\frac{\beta\sigma}{\sqrt{n}}\sqrt{\log(1+\frac{d}{2k})}\,,
$$
for some $\beta>0$ to be chosen later. We can check the conditions of Theorem~\ref{TH:mainLB}:
\begin{itemize}
\item[$(i)$] $\DS|\theta_j-\theta_k|_2^2=\frac{\beta^2\sigma^2 }{n}\rho(\omega_{j}, \omega_{k})\log(1+\frac{d}{2k}) \ge 4\frac{\beta^2\sigma^2 }{8n}k\log(1+\frac{d}{2k})$;
\item[$(ii)$] $\DS |\theta_j-\theta_k|_2^2 =\frac{\beta^2\sigma^2 }{n}\rho(\omega_{j}, \omega_{k})\log(1+\frac{d}{2k}) \le \frac{2k\beta^2\sigma^2 }{n}\log(1+\frac{d}{2k})\le \frac{2\alpha \sigma^2}{n}\log(M)$\,,
\end{itemize}
for $\beta=\sqrt{\frac{\alpha}{8}}$. Applying now Theorem~\ref{TH:mainLB} yields
$$
\inf_{\hat \theta} \sup_{\substack{\theta \in \R^d\\ |\theta|_0\le k}}\p_\theta\big(|\hat \theta-\theta|_2^2 \ge \frac{\alpha^2\sigma^2 }{64n}k\log(1+\frac{d}{2k})\big)\ge \frac{1}{2}-2\alpha\,.
$$
It implies the following corollary.
\begin{cor}
	\label{cor:minimax-sparse}
Recall that $\cB_{0}(k) \subset \R^d$ denotes the set of all $k$-sparse vectors of $\R^d$. 
The minimax rate of estimation over $\cB_{0}(k)$ in the Gaussian sequence model is $\phi(\cB_{0}(k))=\frac{\sigma^2k}{n}\log(ed/k)$. Moreover, it is attained by the constrained least squares estimator $\thetals_{\cB_0(k)}$. 
\end{cor}
Note that the modified BIC estimator of Problem~\ref{EXO:sparse:betterbic} and the Slope \eqref{eq:ad} are also minimax optimal over $\cB_{0}(k)$.
However, unlike \( \thetals_{\cB_0(k)} \) and the modified BIC, the Slope is also adaptive to \( k \).
For any $\eps>0$, the Lasso estimator and the BIC estimator are minimax optimal for sets of parameters such that $k\le d^{1-\eps}$.

\subsection{Lower bounds for estimating vectors in $\ell_1$ balls}

Recall that in Maurey's argument, we approximated a vector $\theta$ such that $|\theta|_1=R$ by a vector $\theta'$ such that $|\theta'|_0=\frac{R}{\sigma}\sqrt{\frac{n}{\log d}}$\,. We can essentially do the same for the lower bound.

Assume that $d \ge \sqrt{n}$ and let $\beta \in (0,1)$ be a parameter to be chosen later and define $k$ to be the smallest integer such that
$$
k\ge \frac{R}{\beta\sigma}\sqrt{\frac{n}{\log (ed/\sqrt{n})}}\,.
$$
Let $\omega_1, \ldots, \omega_M$ be obtained from the sparse Varshamov-Gilbert Lemma~\ref{LEM:sVG} with this choice of $k$ and define
$$
\theta_j=\omega_j\frac{R}{k}\,.
$$
Observe that $|\theta_j|_1=R$ for $j=1,\ldots, M$.
We can check the conditions of Theorem~\ref{TH:mainLB}:
\begin{itemize}
\item[$(i)$] $\DS|\theta_j-\theta_k|_2^2=\frac{R^2}{k^2}\rho(\omega_{j}, \omega_{k})\ge \frac{R^2}{2k}\ge 4R\min\big(\frac{R}8, \beta^2\sigma\frac{\log (ed/\sqrt{n})}{8n}\big)$\,.
\item[$(ii)$] $\DS |\theta_j-\theta_k|_2^2 \le \frac{2R^2}{k} \le 4R\beta\sigma\sqrt{\frac{\log(ed/\sqrt{n})}{n}}\le \frac{2\alpha \sigma^2}{n}\log(M)$\,,
\end{itemize}
for $\beta$ small enough if $d \ge Ck$ for some constant $C>0$ chosen large enough. Applying now Theorem~\ref{TH:mainLB} yields
$$
\inf_{\hat \theta} \sup_{\theta \in \R^d}\p_\theta\big(|\hat \theta-\theta|_2^2 \ge R\min\big(\frac{R}8, \beta^2\sigma^2\frac{\log (ed/\sqrt{n})}{8n}\big)\big)\ge \frac{1}{2}-2\alpha\,.
$$
It implies the following corollary.
\begin{cor}
Recall that $\cB_{1}(R) \subset \R^d$ denotes the set vectors $\theta \in \R^d$ such that $|\theta|_1 \le R$. Then there exist a constant $C>0$ such that if $d \ge n^{1/2+\eps}$, $\eps>0$, the minimax rate of estimation over $\cB_{1}(R)$ in the Gaussian sequence model is $$\phi(\cB_{0}(k))=\min(R^2,R\sigma\frac{\log d}{n}) \,.$$ 
Moreover, it is attained by the constrained least squares estimator $\thetals_{\cB_1(R)}$ if $R\ge \sigma\frac{\log d}{n}$ and  by the trivial estimator $\hat \theta =0$ otherwise.
\end{cor}
\begin{proof}
To complete the proof of the statement, we need to study the risk of the trivial estimator equal to zero for small $R$. Note that if $|\theta^*|_1\le R$, we have
\begin{equation*}
|0 -\theta^*|_2^2=|\theta^*|_2^2\le |\theta^*|_1^2=R^2\,. \qedhere
\end{equation*}
\end{proof}
\begin{rem}
Note that the inequality $|\theta^*|_2^2\le |\theta^*|_1^2$ appears to be quite loose. Nevertheless, it is tight up to a multiplicative constant for the vectors of the form $\theta_j=\omega_j\frac{R}{k}$ that are employed in the lower bound. Indeed, if $R\le \sigma\frac{\log d}{n}$, we have $k\le 2/\beta$ 
$$
|\theta_j|_2^2=\frac{R^2}{k}\ge \frac{\beta}{2}|\theta_j|_1^2\,.
$$
\end{rem}

\section{Lower bounds for sparse estimation via \texorpdfstring{\( \chi^2 \)}{χ2} divergence}

In this section, we will show how to derive lower bounds by directly controlling the distance between a simple and a composite hypothesis, which is useful when investigating lower rates for decision problems instead of estimation problems.

Define the \( \chi^2 \) divergence between two probability distributions \( \p, \q \) as
\begin{equation*}
	\chi^2( \p, \q) = \left\{
	\begin{aligned}
		\int \left( \frac{d\p}{d\q}-1 \right)^2 d\q \quad  &\text{if } \p \ll \q,\\
		\infty \quad &\text{otherwise.}
	\end{aligned}
	\right.
\end{equation*}
When we compare the expression in the case where \( \p \ll \q \), then we see that both the KL divergence and the \( \chi^2 \) divergence can be written as \( \int f \left( \frac{d\p}{d\q} \right) d\q \) with \( f(x) = x \log x \) and \( f(x) = (x-1)^2 \), respectively.
Also note that both of these functions are convex in \( x \) and fulfill \( f(1) = 0 \), which by Jensen’s inequality shows us that they are non-negative and equal to \( 0 \) if and only if \( \p = \q \).

Firstly, let us derive another useful expression for the \( \chi^2 \) divergence.
By expanding the square, we have
\begin{equation*}
	\chi^2(\p, \q) = \int \frac{(d\p)^2}{d\q} - 2 \int d\p + \int d\q = \int \left( \frac{d\p}{d\q} \right)^2 d\q - 1
\end{equation*}

Secondly, we note that we can bound the KL divergence from above by the \( \chi^2 \) divergence, via Jensen’s inequality, as
\begin{equation}
	\label{eq:ag}
	\KL(\p, \q) = \int \log \left( \frac{d\p}{d\q} \right) d\p \leq \log \int \frac{(d\p)^2}{d\q} = \log(1+\chi^2(\p, \q)) \leq \chi^2(\p, \q),
\end{equation}
where we used the concavity inequality \( \log(1+x) \leq x \) for \( x > -1 \).

Note that this estimate is a priori very rough and that we are transitioning from something at \( \log \) scale to something on a regular scale.
However, since we are mostly interested in showing that these distances are bounded by a small constant after adjusting the parameters of our models appropriately, it will suffice for our purposes.
In particular, we can combine \eqref{eq:ag} with the Neyman-Pearson Lemma \ref{LEM:Neyman} and Pinsker’s inequality, Lemma \ref{LEM:Pinsker}, to get
\begin{equation*}
	\inf_{\psi} \p_0(\psi = 1) + \p_1(\psi = 0) \geq \frac{1}{2} (1 - \sqrt{\varepsilon})
\end{equation*}
for distinguishing between two hypothesis with a decision rule \( \psi \) if \( \chi^2(\p_0, \p_1) \leq \varepsilon \).

In the following, we want to apply this observation to derive lower rates for distinguishing between
\begin{equation*}
	H_0 : Y \sim \p_0 \quad H_1 : Y \sim \p_u, \text{ for some } u \in B_0(k),
\end{equation*}
where \( \p_u \) denotes the probability distribution of a Gaussian sequence model \( Y = u + \frac{\sigma}{n} \xi \), with \( \xi \sim \cN(0, I_d) \).

\begin{thm}
	Consider the detection problem
	\begin{equation*}
		H_0: \theta^\ast = 0, \quad H_v : \theta^\ast  = \mu v, \; \text{for some } v \in \cB_0(k), \, \|v\|_2 = 1,
	\end{equation*}
	with data given by \( Y = \theta^\ast + \frac{\sigma}{\sqrt{n}}\xi \), \( \xi \sim \cN(0, I_d) \).

	There exist \( \varepsilon > 0 \) and \( c_\varepsilon > 0 \) such that if 
	\begin{equation*}
		\mu \leq \frac{\sigma}{2} \sqrt{\frac{k}{n}} \sqrt{\log \left( 1 + \frac{d \varepsilon}{k^2} \right)}
	\end{equation*}
	then,
	\begin{equation*}
		\inf_{\psi} \{ \p_0 (\psi = 1) \vee \max_{\substack{v \in \cB_0(k)\\\|v\|_2 = 1}} \p_v(\psi = 0) \} \geq \frac{1}{2}(1 - \sqrt{\varepsilon}).
	\end{equation*}
\end{thm}

\begin{proof}
We introduce the mixture hypothesis
\begin{equation*}
\bar{\p} = \frac{1}{\binom{d}{k}} \sum_{\substack{S \subseteq [d]\\|S| = k}} \p_S,
\end{equation*}
with \( \p_S \) being the distribution of \( Y = \mu \1_S/\sqrt{k} + \frac{\sigma}{\sqrt{n}} \xi \) and give lower bounds on distinguishing
\begin{equation*}
	H_0 : Y \sim \p_0, \quad H_1 : Y \sim \bar{\p},
\end{equation*}
by computing their \( \chi^2 \) distance,
\begin{align*}
	\chi^2(\bar{\p}, \p_0) = {} & \int \left( \frac{d \bar{p}}{d \p_0} \right)^2 d \p_0 - 1\\
	= {} & \frac{1}{\binom{d}{k} } \sum_{S, T} \int \left( \frac{d \p_S}{d \p_0} \frac{d \p_T}{d \p_0} \right) d \p_0 - 1.
\end{align*}
The first step is to compute \( d \p_S / (d \p_0) \).
Writing out the corresponding Gaussian densities yields
\begin{align*}
	\frac{d \p_S}{d \p_0}(X) = {} & \frac{\frac{1}{\sqrt{2 \pi}^d} \exp(-\frac{n}{2 \sigma^2} \| X - \mu \1_S /\sqrt{k}\|_2^2)}{\frac{1}{\sqrt{2 \pi}^d} \exp(-\frac{n}{2 \sigma^2} \| X - 0\|_2^2)}\\
	= {} & \exp \left( \frac{n}{2 \sigma^2} \left( \frac{2 \mu}{\sqrt{k}} \langle X , \1_S \rangle - \mu^2 \right) \right)
\end{align*}
For convenience, we introduce the notation \( X = \frac{\sigma}{\sqrt{n}} Z \) for a standard normal \( Z \sim \cN(0, I_d) \), \( \nu = \mu \sqrt{n}/(\sigma \sqrt{k}) \), and write \( Z_S \) for the restriction to the coordinates in the set \( S \).
Multiplying two of these densities and integrating with respect to \( \p_0 \) in turn then reduces to computing the mgf of a Gaussian and gives
\begin{align*}
	\leadeq{\E_{X \sim \p_0} \left[ \exp \left( \frac{n}{2 \sigma^2} \left( 2 \mu \frac{\sigma}{\sqrt{kn}} (Z_S + Z_T) - 2 \mu^2 \right) \right) \right]}\\
	= {} & \E_{Z \sim \cN(0, I_d)} \exp \left( \nu (Z_S + Z_T) - \nu^2 k \right).
\end{align*}
Decomposing \( Z_S + Z_T = 2 \sum_{i \in S \cap T} Z_i + \sum_{i \in S \Delta T} Z_i \) and noting that \( | S \Delta T | \leq 2 k \), we see that
\begin{align*}
	\E_{\p_0} \left( \frac{d \p_S}{d \p_0} \frac{d \p_T}{d \p_0} \right)
	= {} & \exp \left( \frac{4}{2} \nu^2 | S \cap T + \frac{\nu^2}{2} | S \Delta T | - \nu^2 k \right)\\
	\leq {} & \exp \left( 2 \nu^2 | S \cap T | \right).
\end{align*}

Now, we need to take the expectation over two uniform draws of support sets \( S \) and \( T \), which reduces via conditioning and exploiting the independence of the two draws to
\begin{align*}
	\chi^2( \bar{\p}, \p_0) \leq {} & \E_{S, T} [ \exp (2 \nu^2 | S \cap T |] -1\\
		= {} & \E_{T} \E_{S}[ [ \exp (2 \nu^2 | S \cap T | \,|\, T] ]  - 1\\
			= {} & \E_S [ \exp (2 \nu^2 | S \cap [k] | ) ] - 1.
\end{align*}
Similar to the proof of the sparse Varshamov-Gilbert bound, Lemma \ref{LEM:sVG}, the distribution of \( | S \cap [k] | \) is stochastically dominated by a binomial distribution \( \Bin(k, k/d) \), so that
\begin{align*}
	\chi^2( \bar{\p}, \p_0)
	\leq {} & \E[ \exp(2 \nu^2 \Bin(k, \frac{k}{d})] - 1\\
		= {} & \left( \E[ \exp (2 \nu^2 \Bern(k/d))] \right)^k - 1\\
		= {} & \left( \e^{2 \nu^2} \frac{k}{d} + \left( 1 - \frac{k}{d} \right) \right)^k - 1\\
		= {} & \left( 1 + \frac{k}{d} (\e^{2 \nu^2} - 1) \right)^k - 1
\end{align*}
Since \( (1 + x)^{k} - 1 \approx k x \) for \( x = o(1/k) \), we can take \( \nu^2 = \frac{1}{2} \log(1 + \frac{d}{k^2} \varepsilon) \) to get a bound of the order \( \varepsilon \).
Plugging this back into the definition of \( \nu \) yields
\begin{equation*}
	\mu = \frac{\sigma}{2} \sqrt{\frac{k}{n}} \sqrt{\log \left( 1 + \frac{d \varepsilon}{k^2} \right)}. 
\end{equation*}
The rate for detection for arbitrary \( v \) now follows from
\begin{align*}
		\p_0 (\psi = 1) \vee \max_{\substack{v \in \cB_0(k)\\\|v\|_2 = 1}} \p_v(\psi = 0) 
		\geq {} &  \p_0 (\psi = 1) \vee \bar{\p}(\psi = 0) \\
		\geq {} & \frac{1}{2}(1 - \sqrt{\varepsilon}). \qedhere
\end{align*}
\end{proof}

\begin{rem}
	If instead of the \( \chi^2 \) divergence, we try to compare the \( \KL \) divergence between \( \p_0 \) and \( \bar{\p} \), we are tempted to exploit its convexity properties to get
\begin{equation*}
	\KL(\bar{\p}, \p_0) \leq \frac{1}{\binom{d}{k} } \sum_{S} \KL(\p_S, \p_0) = \KL(\p_{[k]}, \p_0),
\end{equation*}
which is just the distance between two Gaussian distributions.
In turn, we do not see any increase in complexity brought about by increasing the number of competing hypotheses as we did in Section \ref{sec:lb-gsm}.
By using the convexity estimate, we have lost the granularity of controlling how much the different probability distributions overlap.
This is where the \( \chi^2 \) divergence helped.
\end{rem}

\begin{rem}
	This rate for detection implies lower rates for estimation of both \( \theta^\ast \) and \( \| \theta^\ast \|_2 \):
	If given an estimator \( \widehat{T}(X) \) for \( T(\theta) = \| \theta \|_2 \) that achieves an error less than \( \mu/2 \) for some \( X \), then we would get a correct decision rule for those \( X \) by setting
	\begin{equation*}
		\psi(X) = \left\{
		\begin{aligned}
			0, & \quad \text{if } \widehat{T}(X)  < \frac{\mu}{2},\\
			1, & \quad \text{if }  \widehat{T}(X)  \geq \frac{\mu}{2}.
		\end{aligned}
		\right.
	\end{equation*}
	Conversely, the lower bound above means that such an error rate can not be achieved with vanishing error probability for \( \mu \) smaller than the critical value.
	Moreover, by the triangle inequality, we can the same lower bound also for estimation of \( \theta^\ast \) itself.

	We note that the rate that we obtained is slightly worse in the scaling of the log factor (\( d/k^2 \) instead of \( d/k \)) than the one obtained in Corollary \ref{cor:minimax-sparse}.
	This is because it applies to estimating the \( \ell_2 \) norm as well, which is a strictly easier problem and for which estimators can be constructed that achieve the faster rate of \( \frac{k}{n} \log(d/k^2) \).
\end{rem}

\section{Lower bounds for estimating the \texorpdfstring{\( \ell_1 \)}{ℓ1} norm via moment matching}

So far, we have seen many examples of rates that scale with \( n^{-1} \) for the squared error.
In this section, we are going to see an example for estimating a functional that can only be estimated at a much slower rate of \( 1/\log n \), following \cite{CaiLow11}.

Consider the normalized \( \ell_1 \) norm,
\begin{equation}
	\label{eq:aw}
	T(\theta) = \frac{1}{n} \sum_{i = 1}^{n} | \theta_i |,
\end{equation}
as the target functional to be estimated, and measurements
\begin{equation}
	\label{eq:at}
	Y \sim N(\theta, I_n).
\end{equation}
We are going to show lower bounds for the estimation if restricted to an \( \ell_\infty \) ball, \( \Theta_n(M) = \{ \theta \in \R^n : | \theta_i | \leq M \} \), and for a general \( \theta \in \R^n \).

\subsection{A constrained risk inequality}
We start with a general principle that is in the same spirit as the Neyman-Pearson lemma, (Lemma \ref{LEM:Neyman}), but deals with expectations rather than probabilities and uses the \( \chi^2 \) distance to estimate the closeness of distributions.
Moreover, contrary to what we have seen so far, it allows us to compare two mixture distributions over candidate hypotheses, also known as priors.

In the following, we write \( X \) for a random observation coming from a distribution indexed by \( \theta \in \Theta \), \( \widehat{T} = \widehat{T}(X) \) an estimator for a function \( T(\theta) \), and write \( B(\theta) = \E_\theta \widehat{T} - T(\theta) \) for the bias of \( \widehat{T} \).

Let \( \mu_0 \) and \( \mu_1 \) be two prior distributions over \( \Theta \) and denote by \( m_i \), \( v_i^2 \) the means and variance of \( T(\theta) \) under the priors \( \mu_i \), \( i \in \{0, 1\} \),
\begin{equation*}
	\label{eq:au}
	m_i = \int T(\theta) d \mu_i(\theta), \quad \int v_i^2 = \int (T(\theta) - m_i)^2 d \mu_i (\theta).
\end{equation*}
Moreover, write \( F_i \) for the marginal distributions over the priors and denote their density with respect to the average of the two probability distributions by \( f_i \).
With this, we write the \( \chi^2 \) distance between \( f_0 \) and \( f_1 \) as
\begin{equation*}
	\label{eq:av}
	I^2 = \chi^2(\p_{f_0}, \p_{f_1}) = \E_{f_0} \left( \frac{f_1(X)}{f_0(X)} - 1 \right)^2 = \E_{f_0} \left( \frac{f_1(X)}{f_0(X)} \right)^2 - 1.
\end{equation*}

\begin{thm}
	\label{thm:constr_risk}
	(i) If \( \int \E_\theta (\widehat{T}(X) - T(\theta))^2 d\mu_0(\theta) \leq \varepsilon^2 \),
	\begin{equation}
		\label{eq:constr_risk}
		\left| \int B(\theta) d \mu_1(\theta) - \int B(\theta) d \mu_0(\theta) \right| \geq | m_1 - m_0| - (\varepsilon + v_0) I.
	\end{equation}

	(ii) If \( | m_1 - m_0 | > v_0 I\) and \( 0 \leq \lambda \leq 1 \),
	\begin{equation}
		\label{eq:mixture_risk}
		\int \E_\theta( \widehat{T}(X) - T(\theta))^2 d (\lambda \mu_0 + (1- \lambda) \mu_1)(\theta)
		\geq \frac{\lambda(1-\lambda)(|m_1 - m_0| - v_0 I)^2}{\lambda + (1-\lambda)(I + 1)^2},
	\end{equation}
	in particular
	\begin{equation}
		\label{eq:minimax_risk_mixture}
		\max_{i \in \{0, 1\}} \int \E_\theta( \widehat{T}(X) - T(\theta))^2 d \mu_i(\theta) \geq \frac{(|m_1 - m_0| - v_0 I)^2}{(I + 2)^2},
	\end{equation}
	and
	\begin{equation}
		\label{eq:minimax_risk}
		\sup_{\theta \in \Theta} \E_\theta( \widehat{T}(X) - T(\theta))^2 \geq \frac{(|m_1 - m_0| - v_0 I)^2}{(I + 2)^2}.
	\end{equation}
\end{thm}

\begin{proof}
  Without loss of generality, assume \( m_1 \geq m_0 \).
	We start by considering the term
	\begin{align*}
		\leadeq{\E_{f_0} \left[ ( \widehat{T}(X) - m_0 ) \left( \frac{f_1(X) - f_0(X)}{f_0(X)} \right) \right]}\\
		= {} & \E \left[ \widehat{T}(X)  \left( \frac{f_1(X) - f_0(X)}{f_0(X)} \right) \right]\\
		= {} & \E_{f_1} \left[ m_1 + \widehat{T}(X) - m_1 \right] - \E_{f_0} \left[ m_0 + \widehat{T}(X) - m_0 \right]\\
		= {} & m_1 + \int B(\theta) d \mu_1(\theta) - \left( m_0 + \int B(\theta) d \mu_0 (\theta) \right).
	\end{align*}
	Moreover,
	\begin{align*}
		\E_{f_0} (\widehat{T}(X) - m_0)^2 = {} & \int \E_\theta (\widehat{T}(X) -m_0)^2 d \mu_0(\theta)\\
		= {} & \int \E_\theta (\widehat{T}(X) - T(\theta) + T(\theta) - m_0)^2 d \mu_0 (\theta)\\
		= {} & \int \E_\theta (\widehat{T}(X) - T(\theta))^2 d \mu_0 ( \theta )\\
		{} & + 2 \int B(\theta) (T(\theta) - m_0) d \mu_0 (\theta)\\
		{} & + \int (T(\theta) - m_0)^2 d \mu_0 ( \theta )\\
		\leq {} & \varepsilon^2 + 2 \varepsilon v_0 + v_0^2 = (\varepsilon + v_0)^2.
	\end{align*}
	Therefore, by Cauchy-Schwarz,
	\begin{align*}
		\E_{f_0} \left[ ( \widehat{T}(X) - m_0 ) \left( \frac{f_1(X) - f_0(X)}{f_0(X)} \right) \right]
		\leq {} & (\varepsilon + v_0) I,
	\end{align*}
	and hence
	\begin{equation*}
		m_1 + \int B(\theta) d \mu_1(\theta) - \left( m_0 + \int B(\theta) d \mu_0 (\theta) \right) \leq (\varepsilon + v_0) I,
	\end{equation*}
	whence
	\begin{equation*}
		\label{eq:ac}
		\int B(\theta) d \mu_0(\theta) - \int B(\theta) d \mu_1 (\theta) \geq m_1 - m_0 - (\varepsilon + v_0) I,
	\end{equation*}
	which gives us \eqref{eq:constr_risk}.

	To estimate the risk under a mixture of \( \mu_0 \) and \( \mu_1 \), we note that from \eqref{eq:constr_risk}, we have an estimate for the risk under \( \mu_1 \) given an upper bound on the risk under \( \mu_0 \), which means we can reduce this problem to estimating a quadratic from below.
	Consider
	\begin{equation*}
		J(x) = \lambda x^2 + (1-\lambda) (a - bx)^2,
	\end{equation*}
	with \( 0 < \lambda < 1 \), \( a > 0 \) and \( b > 0 \).
	\( J \) is minimized by \( x = x_{\mathrm{min}} = \frac{a b (1- \lambda)}{\lambda + b^2 (1 - \lambda)} \) with \( a - b x_{\mathrm{min}} > 0 \) and \( J(x_\mathrm{min}) = \frac{a^2 \lambda (1-\lambda)}{\lambda + b^2(1 - \lambda)} \).
	Hence, if we add a cut-off at zero in the second term, we do not change the minimal value and obtain that
	\begin{equation*}
		\lambda x^2 + (1-\lambda)((a - bx)_{+})^2
	\end{equation*}
	has the same minimum.

	From \eqref{eq:constr_risk} and Jensen’s inequality, setting \( \varepsilon^2 = \int B(\theta)^2 d \mu_0(\theta) \), we get
	\begin{align*}
			\int B(\theta)^2 d \mu_1 ( \theta)
			\geq {} & \left( \int B(\theta) d \mu_1 (\theta) \right)^2\\
			\geq {} & ((m_1 - m_0 - (\varepsilon + v_0) I - \varepsilon)_{+})^2.
	\end{align*}
	Combining this with the estimate for the quadratic above yields
	\begin{align*}
		\leadeq{\int \E_\theta( \widehat{T}(X) - T(\theta))^2 d (\lambda \mu_0 + (1- \lambda) \mu_1)(\theta)} \\
		\geq {} & \lambda \varepsilon^2 + (1-\lambda)((m_1 - m_0 - (\varepsilon + v_0) I - \varepsilon)_{+})^2\\
		\geq {} & \frac{\lambda(1-\lambda) (|m_1 - m_0 | - v_0 I)^2}{\lambda + (1-\lambda) (I+1)^2},
	\end{align*}
	which is \eqref{eq:mixture_risk}.
	
	Finally, since the minimax risk is bigger than any mixture risk, we can bound it from below by the maximum value of this bound, which is obtained at \( \lambda = (I+1)/(I + 2) \) to get \eqref{eq:minimax_risk_mixture}, and by the same argument \eqref{eq:minimax_risk}.
\end{proof}

\subsection{Unfavorable priors via polynomial approximation}
In order to construct unfavorable priors for the \( \ell_1 \) norm \eqref{eq:aw}, we resort to polynomial approximation theory.
For a continuous function \( f \) on \( [-1, 1] \), denote the maximal approximation error by polynomials of degree \( k \), written as \( \mathcal{P}_k \), by
\begin{equation*}
	\delta_k(f) = \inf_{G \in \mathcal{P}_k} \max_{x \in [-1, 1]} | f(x) - G(x) |,
\end{equation*}
and the polynomial attaining this error by \( G_k^\ast \).
For \( f(t) = |t| \), it is known \cite{Riv90} that
\begin{equation*}
	\beta_\ast = \lim_{k \to \infty} 2 k \delta_{2k}(f) \in (0, \infty),
\end{equation*}
which is a very slow rate of convergence of the approximation, compared to smooth functions for which we would expect geometric convergence.

We can now construct our priors using some abstract measure theoretical results.

\begin{lem}
	\label{lem:minimax_priors}
	For an integer \( k > 0 \), there are two probability measures \( \nu_0 \) and \( \nu_1 \) on \( [-1, 1] \) that fulfill the following:
	\begin{enumerate}
		\item \( \nu_0 \) and \( \nu_1 \) are symmetric about 0;
		\item \( \int t^l d \nu_1 (t) = \int t^l d \nu_0(t) \) for all \( l \in \{0, 1, \dots, k \} \);
		\item \(\int |t| d \nu_1(t) - \int | t | d \nu_0(t) = 2 \delta_k \).
	\end{enumerate}
\end{lem}

\begin{proof}
	The idea is to construct the measures via the Hahn-Banach theorem and the Riesz representation theorem.

	First, consider \( f(t) = |t| \) as an element of the space \( C([-1, 1]) \) equipped with the supremum norm and define \( \mathcal{P}_k \) as the space of polynomials of degree up to \( k \), and \( \mathcal{F}_k := \operatorname{span}(\mathcal{P}_k \cup \{ f \}) \).
	On \( \mathcal{F}_k \), define the functional \( T(c f + p_k) = c \delta_k \), which is well-defined because \( f \) is not a polynomial.

	\textbf{Claim: } \( \| T \| = \sup \{ T(g) : g \in \mathcal{F}_k, \| g \|_\infty \leq 1\} = 1 \).

	\( \| T \| \geq 1 \): Let \( G_k^\ast \) be the best-approximating polynomial to \( f \) in \( \mathcal{P}_k \).
	Then, \( \| f - G_k^\ast \|_\infty = \delta_k \), \( \| (f - G^\ast_k)/\delta_k \|_\infty = 1 \), and \( T((f - G^\ast_k)/\delta_k) = 1 \).

	\( \| T \| \leq 1 \): Suppose \( g = cf + p_k \) with \( p_k \in \mathcal{P}_k \), \( \| g \|_\infty = 1 \) and \( T(g) > 1 \).
	Then \( c > 1/\delta_k \) and
	\begin{equation*}
		\| f - (-p_k/c) \|_\infty = \frac{1}{c} < \delta_k,
	\end{equation*}
	contradicting the definition of \( \delta_k \).

	Now, by the Hahn-Banach theorem, there is a norm-preserving extension of \( T \) to \( C([-1, 1]) \) which we again denote by \( T \).
	By the Riesz representation theorem, there is a Borel signed measure \( \tau \) with variation equal to \( 1 \) such that
	\begin{equation*}
		T(g) = \int_{-1}^{1} g(t) d \tau(t), \quad \text{for all } g \in C([-1, 1]).
	\end{equation*}
	The Hahn-Jordan decomposition gives two positive measures \( \tau_+ \) and \( \tau_- \) such that \( \tau = \tau_+ - \tau_- \) and
	\begin{align*}
		\int_{-1}^{1} |t| d ( \tau_+ - \tau_-)(t) = {} & \delta_k \text{ and}\\
		\int_{-1}^{1} t^l d \tau_+ (t) = {} & \int_{-1}^1 t^l d \tau_-(t), \quad l \in \{0, \dots, k\}.
	\end{align*}
	Since \( f \) is a symmetric function, we can symmetrize these measures and hence can assume that they are symmetric about \( 0 \).
	Finally, to get a probability measures with the desired properties, set \( \nu_1 = 2 \tau_+ \) and \( \nu_0 = 2 \tau_- \).
\end{proof}

\subsection{Minimax lower bounds}

\begin{thm}
	We have the following bounds on the minimax rate for \( T(\theta) = \frac{1}{n} \sum_{i = 1}^{n} | \theta_i| \):
	For \( \theta \in \Theta_n(M) = \{ \theta \in \R^n : | \theta_i | \leq M \} \),
	\begin{equation}
		\label{eq:ai}
		\inf_{\widehat{T}} \sup_{\theta \in \Theta_n(M)} \E_\theta( \widehat{T}(X) - T(\theta))^2 \geq \beta_\ast^2 M^2 \left( \frac{\log \log n}{\log n} \right)^2(1 + o(1)).
	\end{equation}
	For \( \theta \in \R^n \),
	\begin{equation}
		\label{eq:aj}
		\inf_{\widehat{T}} \sup_{\theta \in \R^n} \E_\theta( \widehat{T}(X) - T(\theta))^2 \geq \frac{\beta_\ast^2}{16 e^2 \log n}(1 + o(1)).
	\end{equation}
\end{thm}

The last remaining ingredient for the proof are the \emph{Hermite polynomials}, a family of orthogonal polynomials with respect to the Gaussian density
\begin{equation*}
	\label{eq:ae}
	\varphi(y) = \frac{1}{\sqrt{2 \pi}} \e^{-y^2/2}.
\end{equation*}
For our purposes, it is enough to define them by the derivatives of the density,
\begin{equation*}
	\frac{d^k}{dy^k} \varphi(y) = (-1)^k H_k(y) \varphi(y),
\end{equation*}
and to observe that they are orthogonal,
\begin{equation*}
	\int H_k^2(y) \varphi(y) dy = k!, \quad \int H_k(y) H_j(y) \varphi(y) dy = 0, \quad \text{for } k \neq j.
\end{equation*}

\begin{proof}
	We want to use Theorem \ref{thm:constr_risk}.
	To construct the prior measures, we scale the measures from Lemma \ref{lem:minimax_priors} appropriately:
	Let \( k_n \) be an even integer that is to be determined, and \( \nu_0 \), \( \nu_1 \) the two measures given by Lemma \ref{lem:minimax_priors}.
	Define \( g(x) = Mx \) and define the measures \( \mu_i \) by dilating \( \nu_i \), \( \mu_i(A) = \nu_i(g^{-1}(A)) \) for every Borel set \( A \subseteq [-M, M] \) and \( i \in \{0, 1\} \).
	Hence, 
	\begin{enumerate}
		\item \( \mu_0 \) and \( \mu_1 \) are symmetric about 0;
		\item \( \int t^l d \mu_1 (t) = \int t^l d \mu_0(t) \) for all \( l \in \{0, 1, \dots, k_n \} \);
		\item \(\int |t| d \mu_1(t) - \int | t | d \mu_0(t) = 2 M \delta_{k_n} \).
	\end{enumerate}
	To get priors for \( n \) \iid samples, consider the product priors \( \mu_i^n = \prod_{j=1}^n \mu_i \).
	With this, we have
	\begin{equation*}
		\E_{\mu_1^n} T(\theta) - \E_{\mu_0^n} T(\theta) = \E_{\mu_1} | \theta_1 | - \E_{\mu_0} | \theta_1 | = 2 M \delta_{k_n},
	\end{equation*}
	and
	\begin{equation*}
		\E_{\mu_0^n} (T(\theta) - \E_{\mu_0^n} T(\theta))^2 = 	\E_{\mu_0^n} (T(\theta) - \E_{\mu_0^n} T(\theta))^2 \leq \frac{M^2}{n},
	\end{equation*}
	since each \( \theta_i \in [-M, M] \).

	It remains to control the \( \chi^2 \) distance between the two marginal distributions.
	To this end, we will expand the Gaussian density in terms of Hermite polynomials.
	First, set \( f_i(y) = \int \varphi(y-t) d \mu_i(t) \).
	Since \( g(x) = \exp(-x) \) is convex and \( \mu_0 \) is symmetric,
	\begin{align*}
		f_0(y)
		\geq {} & \frac{1}{\sqrt{2 \pi}} \exp \left( - \int \frac{(y-t)^2}{2} d \mu_0 (t) \right)\\
		= {} & \varphi(y) \exp \left( -\frac{1}{2} M^2 \int t^2 d\nu_0 (t) \right)\\
		\geq {} & \varphi(y) \exp \left( -\frac{1}{2}M^2 \right).
	\end{align*}
	Expand \( \varphi \) as
	\begin{equation*}
		\varphi(y - t) = \sum_{k = 0}^{\infty} H_k(y) \varphi(y) \frac{t^k}{k!}.
	\end{equation*}
	Then,
	\begin{equation*}
		\int \frac{(f_1(y) - f_0(y))^2}{f_0(y)} dy \leq \e^{M^2/2} \sum_{k = k_n+1}^{\infty} \frac{1}{k!} M^{2k}.
	\end{equation*}
	The \( \chi^2 \) distance for \( n \) \iid observations can now be easily bounded by
	\begin{align}
		I_n^2 = {} & \int \frac{(\prod_{i=1}^n f_1(y_i) - \prod_{i=1}^n f_0(y_i))^2}{\prod_{i=1}^n f_0(y_i)} d y_1 \dotsi d y_n \nonumber \\
		= {} & \int \frac{(\prod_{i=1}^n f_1(y_i))^2}{\prod_{i=1}^n f_0(y_i)} d y_1 \dotsi d y_n - 1 \nonumber \\
		= {} & \left(\prod_{i=1}^n \int \frac{(f_1(y_i))^2}{f_0(y_i)} d y_i\right) - 1 \nonumber \\
		\leq {} & \left( 1 + \e^{M^2/2} \sum_{k = k_n+1}^{\infty} \frac{1}{k!} M^{2k}\right)^n - 1 \nonumber \\
		\leq {} & \left( 1 + \e^{3M^2/2} \frac{1}{k_n!} M^{2k_n}\right)^n - 1 \nonumber \\
		\leq {} & \left( 1 + \e^{3M^2/2} \left( \frac{\e M^2}{k_n} \right)^{k_n}\right)^n - 1, \label{eq:ax}
	\end{align}
	where the last step used a Stirling type estimate, \( k! > (k/e)^k \).

	Note that if \( x = o(1/n) \), then \( (1+x)^n - 1 = o(n x) \) by Taylor's formula, so we can choose \( k_n \geq \log n / (\log \log n) \) to guarantee that for \( n \) large enough, \( I_n \to 0 \).
	With Theorem \ref{thm:constr_risk}, we have
	\begin{align*}
		\inf_{\widehat{T}} \sup_{\theta \in \Theta_n(M)} \E ( \widehat{T} - T(\theta) )^2 
		\geq {} & \frac{(2 M \delta_{k_n} - (M/\sqrt{n}) I_n)^2}{(I_n + 2)^2}\\
		= {} & \beta_\ast^2 M^2 \left( \frac{\log \log n}{\log n} \right)^2 (1 + o(1)).
	\end{align*}

	To prove the lower bound over \( \R^n \), take \( M = \sqrt{\log n} \) and \( k_n \) to be the smallest integer such that \( k_n \geq 2 \e \log n \) and plug this into \eqref{eq:ax}, yielding
	\begin{align*}
		I_n^2 \leq \left( 1 + n^{3/2} \left(\frac{\e \log n}{2 \e \log n}\right)^{k_n} \right)^n - 1,
	\end{align*}
	which goes to zero.
	Hence, we can conclude just as before.
\end{proof}

Note that \cite{CaiLow11} also complement these lower bounds with upper bounds that are based on polynomial approximation, completing the picture of estimating \( T(\theta) \).

\newpage
\section{Problem Set}
\begin{exercise}\label{EXO:KL_Gauss}
\begin{enumerate}[label=(\alph*)]
\item Prove the statement of Example~\ref{ex:KL_Gauss}.
\item Let $P_\theta$ denote the distribution of $X \sim \Bern(\theta)$. Show that $$
	\KL(P_\theta, P_{\theta'})\ge C(\theta-\theta')^2\,.
$$
\end{enumerate}
\end{exercise}

\begin{exercise}\label{EXO:Sanov}
Let  $\p_0$ and $\p_1$ be two probability measures. Prove that for any test $\psi$, it holds
$$
\max_{j=0,1}\p_j(\psi\neq j) \ge \frac14e^{-\KL(\p_0,\p_1)}\,.
$$
\end{exercise}

\begin{exercise}\label{EXO:packing}
For any $R>0$, $\theta \in \R^d$, denote by $\cB_2(\theta, R)$ the (Euclidean) ball of radius $R$ and centered at $\theta$. For any $\eps>0$ let $N=N(\eps)$ be the largest integer such that there exist $\theta_1, \ldots, \theta_N \in \cB_2(0,1)$ for which
$$
|\theta_i-\theta_j| \ge 2\eps
$$
for all $i \neq j$.
We call the set $\{\theta_1, \ldots, \theta_N\}$ an $\eps$-packing of $\cB_2(0, 1)$\,.
\begin{enumerate}[label=(\alph*)]
\item Show that there exists a constant $C>0$ such that $N \le C_d/\eps^d$\,.
\item Show that for any $x \in \cB_2(0,1)$, there exists $i=1, \ldots, N$ such that $|x-\theta_i|_2\le 2\eps$.
\item Use (b) to conclude that  there exists a constant $C'>0$ such that $N \ge C'_d/\eps^d$\,.
\end{enumerate}
\end{exercise}

\begin{exercise}\label{EXO:minimax_other}
Show that the rate $\phi=\sigma^2d/n$ is the minimax rate of estimation over:
\begin{enumerate}[label=(\alph*)]
	\item The Euclidean Ball of $\R^d$ with radius $\sqrt{\sigma^2d/n}$.
\item The unit $\ell_\infty$ ball of $\R^d$ defined by
$$
\cB_\infty(1)=\{\theta \in \R^d\,:\, \max_j|\theta_j|\le 1\}
$$
as long as $\sigma^2 \le n$.
\item The set of nonnegative vectors of $\R^d$. 
\item The discrete hypercube $\frac{\sigma}{16\sqrt{n}}\{0,1\}^d$\,.
\end{enumerate}
\end{exercise}

\begin{exercise}\label{EXO:minimax_smooth}
Fix $\beta \ge 5/3, Q>0$ and prove that the minimax rate of estimation over $\Theta(\beta, Q)$ with the $\|\cdot\|_{L_2([0,1])}$-norm is given by $n^{-\frac{2\beta}{2\beta+1}}$.

\noindent \hint{Consider functions of the form $$f_j=\frac{C}{\sqrt{n}}\sum_{i=1}^N \omega_{ji} \varphi_i$$
where $C$ is a constant, $\omega_j \in \{0,1\}^N$ for some appropriately chosen $N$ and $\{\varphi_j\}_{j\ge 1}$ is the trigonometric basis.}
\end{exercise}

\chapter{Matrix estimation}
\label{chap:matrix}
\newcommand{\Thetasvt}{\hat \Theta^{\textsc{svt}}}
\newcommand{\Thetarank}{\hat \Theta^{\textsc{rk}}}

Over the past decade or so, matrices have entered the picture of high-dimensional statistics for several reasons. Perhaps the simplest explanation is that they are the most natural extension of vectors. While this is true, and we will see examples where the extension from vectors to matrices is straightforward, matrices have a much richer structure than vectors allowing ``interaction" between their rows and columns. In particular, while we have been describing simple vectors in terms of their sparsity, here we can measure the complexity of a matrix by its \emph{rank}. This feature was successfully employed in a variety of applications ranging from \emph{multi-task learning} to \emph{collaborative filtering}. This last application was made popular by the \textsc{Netflix} prize in particular. 

In this chapter, we study several statistical problems where the parameter of interest $\theta$ is a matrix rather than a vector. These problems include multivariate regression, covariance matrix estimation, and principal component analysis. Before getting to these topics, we begin with a quick reminder on matrices and linear algebra.

\section{Basic facts about matrices}
\label{SEC:matrixfacts}

Matrices are much more complicated objects than vectors. In particular, while vectors can be identified with linear operators from $\R^d$ to $\R$, matrices can be identified to linear operators from $\R^d$ to $\R^n$ for $n \ge 1$. This seemingly simple fact gives rise to a profusion of notions and properties as illustrated by Bernstein's book \cite{Ber09} that contains facts about matrices over more than a thousand pages. Fortunately, we will be needing only a small number of such properties, which can be found in the excellent book \cite{GolVan96}, which has become a standard reference on matrices and numerical linear algebra.

\subsection{Singular value decomposition}
Let $A=\{a_{ij}, 1\le i \le m, 1\le j \le n\}$ be a $m \times n$ real matrix of rank $r \le \min(m,n)$. The \emph{Singular Value Decomposition} (SVD) of $A$ is given by
$$
A=UDV^\top=\sum_{j=1}^r\lambda_j u_jv_j^\top\,,
$$
where $D$ is an $r\times r$ diagonal matrix with positive diagonal entries $\{\lambda_1, \ldots, \lambda_r\}$, $U$ is a matrix with columns $\{u_1, \ldots, u_r\} \in \R^m$ that are orthonormal, and $V$ is a matrix with columns $\{v_1, \ldots, v_r\} \in \R^n$ that are also orthonormal. Moreover, it holds that
$$
AA^\top u_j=\lambda_j^2 u_j\,, \qquad  \text{and} \qquad A^\top A v_j=\lambda_j^2v_j
$$
for $j =1, \ldots, r$. The values $\lambda_j>0$ are called \emph{singular values} of $A$ and are uniquely defined. If rank $r< \min(n,m)$ then the  singular values of $A$ are given by $\lambda=(\lambda_1, \ldots, \lambda_r, 0, \ldots, 0)^\top \in \R^{\min(n,m)}$ where there are $\min(n,m)-r$ zeros. This way, the vector $\lambda$ of singular values of a $n\times m$ matrix is a vector in $\R^{\min(n,m)}$.

In particular, if $A$ is an $n \times n$ symmetric positive semidefinite (PSD), i.e. $A^\top=A$ and $u^\top A u\ge 0$ for all $u \in \R^n$, then the singular values of $A$ are equal to its eigenvalues.

The largest singular value of $A$ denoted by $\lambda_{\max{}}(A)$ also satisfies the following variational formulation:
$$
\lambda_{\max{}}(A)=\max_{x \in \R^n}\frac{|Ax|_2}{|x|_2}=\max_{\substack{x \in \R^n\\ y \in \R^m}}\frac{y^\top A x}{|y|_2|x|_2}=\max_{\substack{x \in \cS^{n-1}\\ y \in \cS^{m-1}}}y^\top A x\,.
$$
In the case of a $n \times n$ PSD matrix $A$, we have
$$
\lambda_{\max{}}(A)=\max_{x \in \cS^{n-1}}x^\top A x\,.
$$

In these notes, we always assume that eigenvectors and singular vectors have unit norm.

\subsection{Norms and inner product}

Let $A=\{a_{ij}\}$ and $B=\{b_{ij}\}$ be two real matrices. Their size will be implicit in the following notation.

\subsubsection{Vector norms}

The simplest way to treat a matrix is to deal with it as if it were a vector. In particular, we can extend $\ell_q$ norms to matrices:
$$
|A|_q=\Big(\sum_{ij}|a_{ij}|^q\Big)^{1/q}\,, \quad q>0\,.
$$
The cases where $q \in \{0,\infty\}$ can also be extended matrices:
$$
|A|_0=\sum_{ij}\1(a_{ij}\neq 0)\,, \qquad |A|_\infty=\max_{ij}|a_{ij}|\,.
$$
The case $q=2$ plays a particular role for matrices and $|A|_2$ is called the \emph{Frobenius} norm of $A$ and is often denoted by $\|A\|_F$. It is also the Hilbert-Schmidt norm associated to the inner product:
$$
\langle A,B\rangle=\tr(A^\top B)=\tr(B^\top A)\,.
$$

\subsubsection{Spectral norms}

Let $\lambda=(\lambda_1, \ldots, \lambda_r, 0, \ldots, 0)$ be the singular values of a matrix $A$. We can define spectral norms on $A$ as vector norms on the vector $\lambda$. In particular, for any $q\in [1, \infty]$,
$$
\|A\|_q=|\lambda|_q\,,
$$
is called \emph{Schatten $q$-norm} of $A$. Here again, special cases have special names:
\begin{itemize}
\item $q=2$: $\|A\|_2=\|A\|_F$ is the Frobenius norm defined above.
\item $q=1$: $\|A\|_1=\|A\|_*$ is called the Nuclear norm (or trace norm) of $A$.
\item $q=\infty$: $\|A\|_\infty=\lambda_{\max{}}(A)=\|A\|_{\mathrm{op}}$ is called the operator norm (or spectral norm) of $A$.
\end{itemize}

We are going to employ these norms to assess the proximity to our matrix of interest. While the interpretation of vector norms is clear by extension from the vector case, the meaning of ``$\|A-B\|_{\mathrm{op}}$ is small" is not as transparent. The following subsection provides some inequalities (without proofs) that allow a better reading.

\subsection{Useful matrix inequalities}
\label{SUB:matrix_ineq}

Let $A$ and $B$ be two $m\times n$ matrices with singular values $ \lambda_1(A) \ge \lambda _2(A) \ldots \ge \lambda_{\min(m,n)}(A)$ and $ \lambda_1(B) \ge  \ldots \ge \lambda_{\min(m,n)}(B)$ respectively. Then the following inequalities hold:
\renewcommand*{\arraystretch}{2}
$$
\begin{array}{ll}
\DS \max_k\big|\lambda_k(A)-\lambda_k(B)\big|\le \|A-B\|_{\mathrm{op}}\,, & \text{Weyl (1912)}\\
\DS \sum_{k}\big|\lambda_k(A)-\lambda_k(B)\big|^2\le \|A-B\|_F^2\,,& \text{Hoffman-Weilandt (1953)}\\
\DS \langle A, B\rangle \le \|A\|_p \|B\|_q\,, \ \frac{1}{p}+\frac{1}{q}=1, p,q \in [1, \infty]\,,& \text{H\"older}
\end{array}
$$
\renewcommand*{\arraystretch}{1}

The singular value decomposition is associated to the following useful lemma, known as the \emph{Eckart-Young (or Eckart-Young-Mirsky) Theorem}. It states that for any given $k$ the closest matrix of rank $k$ to a given matrix $A$ in Frobenius norm is given by its truncated SVD.

\begin{lem}
\label{lem:eckart}
Let $A$ be a rank-$r$ matrix with singular value decomposition
$$
A = \sum_{i=1}^r \lambda_i u_i v_i^\top\,,
$$
where $\lambda_1 \ge \lambda_2 \ge \dots \ge \lambda_r > 0$ are the ordered singular values of $A$. 
For any $k<r$, let $A_k$ to be the truncated singular value decomposition of $A$ given by
$$
A_k= \sum_{i=1}^r \lambda_i u_i v_i^\top\,.
$$
Then for any matrix $B$ such that $\rank(B) \le k$, it holds
$$
\|A-A_k\|_F\le \|A-B\|_F\,.
$$
Moreover, 
$$
\|A-A_k\|_F^2=\sum_{j=k+1}^r \lambda_j^2\,.
$$
\end{lem}
\begin{proof}
Note  that the last equality of the lemma is obvious since 
$$
A-A_k=\sum_{j=k+1}^r \lambda_ju_jv_j^\top
$$
and the matrices in the sum above are orthogonal to each other.

Thus, it is sufficient to prove that for any matrix $B$ such that $\rank(B) \le k$, it holds
$$
\|A-B\|_F^2 \ge \sum_{j=k+1}^r \lambda_j^2\,.
$$
To that end, denote by $\sigma_1 \ge \sigma_2 \ge\dots \ge \sigma_k \ge 0$ the ordered singular values of $B$---some may be equal to zero if $\rank(B)<k$---and observe that it follows from the Hoffman-Weilandt inequality that
\begin{equation*}
\|A-B\|_F^2 \ge \sum_{j=1}^r\big(\lambda_j -\sigma_j\big)^2=\sum_{j=1}^k\big(\lambda_j -\sigma_j\big)^2+ \sum_{j=k+1}^r \lambda_j^2\ge \sum_{j=k+1}^r \lambda_j^2\,.\qedhere
\end{equation*}
\end{proof}

\section{Multivariate regression}
\label{SEC:MVR}
In the traditional regression setup, the response variable $Y$ is a scalar. In several applications, the goal is not to predict a variable but rather a vector $Y \in \R^T$, still from a covariate $X \in \R^d$. A standard example arises in genomics data where $Y$ contains $T$ physical measurements of a patient and $X$ contains the expression levels for $d$ genes. As a result the regression function in this case $f(x)=\E[Y|X=x]$ is a function from $\R^d$ to $\R^T$. Clearly, $f$ can be estimated independently for each coordinate, using the tools that we have developed in the previous chapter. However, we will see that in several interesting scenarios, some structure is shared across coordinates and this information can be leveraged to yield better prediction bounds.

\subsection{The model}

Throughout this section, we consider the following multivariate linear regression model:
\begin{equation}
\label{EQ:MVRmodel}
\Y=\X\Theta^* + E\,,
\end{equation}
where $\Y \in \R^{n\times T}$ is the matrix of observed responses, $\X$ is the $n \times d$ observed design matrix (as before), $\Theta \in \R^{d\times T}$ is the matrix of unknown parameters and $E\sim \sg_{n\times T}(\sigma^2)$ is the noise matrix. In this chapter, we will focus on the prediction task, which consists in estimating $\X\Theta^*$. 

As mentioned in the foreword of this chapter, we can view this problem as $T$ (univariate) linear regression problems $Y^{(j)}=\X\theta^{*,(j)}+ \eps^{(j)}, j=1, \ldots, T$, where $Y^{(j)}, \theta^{*,(j)}$ and $\eps^{(j)}$ are the $j$th column of $\Y, \Theta^*$ and $E$ respectively. In particular, an estimator for $\X\Theta^*$ can be obtained by concatenating the estimators for each of the $T$ problems. This approach is the subject of Problem~\ref{EXO:concat}.

The columns of $\Theta^*$ correspond to $T$ different regression tasks. Consider the following example as a motivation. Assume that the  \textsc{Subway} headquarters want to evaluate the effect of $d$ variables (promotions, day of the week, TV ads,\dots) on their sales. To that end, they ask each of their $T=40,000$ restaurants to report their sales numbers for the past $n=200$ days. As a result, franchise $j$ returns to  headquarters a vector $\Y^{(j)}  \in \R^n$. The $d$ variables for each of the $n$ days are already known to headquarters and are stored in a  matrix $\X \in \R^{n\times d}$. In this case, it may be reasonable to assume that the same subset of variables has an impact of the sales for each of the franchises, though the magnitude of this impact may differ from franchise to franchise. As a result, one may assume that the matrix $\Theta^*$ has each of its $T$ columns that is row sparse and that they \emph{share the same sparsity pattern}, i.e., $\Theta^*$ is of the form:
$$
\Theta^*=\left(\begin{array}{cccc}0 & 0 & 0 & 0 \\\bullet & \bullet & \bullet & \bullet \\\bullet & \bullet & \bullet & \bullet \\0 & 0 & 0 & 0 \\\vdots & \vdots & \vdots & \vdots \\0 & 0 & 0 & 0 \\\bullet & \bullet & \bullet & \bullet\end{array}\right)\,,
$$
where $\bullet$ indicates a potentially nonzero entry.

It follows from the result of Problem~\ref{EXO:concat} that if each task is performed individually, one may find an estimator $\hat \Theta$ such that
$$
\frac{1}{n}\E\|\X\hat \Theta -\X \Theta^*\|_F^2 \lesssim \sigma^2\frac{kT\log(ed)}{n}\,,
$$
where $k$ is the number of nonzero coordinates in each column of $\Theta^*$. We remember that the term $\log(ed)$ corresponds to the additional price to pay for not knowing where the nonzero components are. However, in this case, when the number of tasks grows, this should become easier. This fact was proved in~\cite{LouPonTsy11}. We will see that we can recover a similar phenomenon when the number of tasks becomes large, though larger than in \cite{LouPonTsy11}. Indeed, rather than exploiting sparsity, observe that such a matrix $\Theta^*$ has rank $k$. This is the kind of structure that we will be predominantly using in this chapter.

Rather than assuming that the columns of $\Theta^*$ share the same sparsity pattern, it may be more appropriate to assume that the matrix $\Theta^*$ is low rank or approximately so. As a result, while the matrix may not be sparse at all, the fact that it is low rank still materializes the idea that some structure is shared across different tasks. In this more general setup, it is assumed that the columns of $\Theta^*$ live in a lower dimensional space. Going back to the \textsc{Subway} example this amounts to assuming that while there are 40,000 franchises, there are only a few canonical profiles for these franchises and that all franchises are linear combinations of these profiles.

\subsection{Sub-Gaussian matrix model}

Recall that under the assumption \textsf{ORT} for the design matrix, i.e., $\X^\top \X=nI_d$, then the univariate regression model can be reduced to the sub-Gaussian sequence model. Here we investigate the effect of this assumption on the multivariate regression model~\eqref{EQ:MVRmodel}.

Observe that under  assumption \textsf{ORT},
$$
\frac{1}{n}\X^\top \Y=\Theta^* + \frac{1}{n}\X^\top E\,.
$$
Which can be written as an equation in $\R^{d\times T}$ called the \emph{sub-Gaussian matrix model (sGMM)}:
\begin{equation}
\label{EQ:sGMM}
y=\Theta^* + F\,,
\end{equation}
where $y=\frac{1}{n}\X^\top \Y$ and $F=\frac{1}{n}\X^\top E\sim \sg_{d\times T}(\sigma^2/n)$. 

Indeed, for any $u \in \cS^{d-1}, v \in \cS^{T-1}$, it holds
$$
u^\top Fv=\frac{1}{n}(\X u)^\top E v=\frac{1}{\sqrt{n}} w^\top E v \sim \sg(\sigma^2/n)\,,
$$
where $w=\X u/\sqrt{n}$ has unit norm: $|w|_2^2=u^\top \frac{\X^\top \X}{n}u=|u|_2^2=1$.

Akin to the sub-Gaussian sequence model, we have a \emph{direct} observation model where we observe the parameter of interest with additive noise. This enables us to use thresholding methods for estimating $\Theta^*$ when $|\Theta^*|_0$ is small. However, this also follows from Problem~\ref{EXO:concat}. The reduction to the vector case in the sGMM is just as straightforward. The interesting analysis begins when $\Theta^*$ is low-rank, which is equivalent to sparsity in its unknown eigenbasis.

Consider the SVD of $\Theta^*$:
$$
\Theta^*=\sum_{j} \lambda_j u_j v_j^\top\,.
$$
and recall that $\|\Theta^*\|_0=|\lambda|_0$. Therefore, if we knew $u_j$ and $v_j$, we could simply estimate the $\lambda_j$s by hard thresholding. It turns out that estimating these eigenvectors by the eigenvectors of $y$ is sufficient.

Consider the SVD of the observed matrix $y$:
$$
y=\sum_{j} \hat \lambda_j \hat u_j \hat v_j^\top\,.
$$
\begin{defn}
\label{DEF:svt}
The \textbf{singular value thresholding} estimator with threshold $2\tau\ge 0$ is defined by
$$
\Thetasvt=\sum_{j} \hat \lambda_j\1(|\hat \lambda_j|>2\tau) \hat u_j \hat v_j^\top\,.
$$
\end{defn}
Recall that the threshold for the hard thresholding estimator was chosen to be the level of the noise with high probability. The singular value thresholding estimator obeys the same rule, except that the norm in which the magnitude of the noise is measured is adapted to the matrix case. Specifically, the following lemma will allow us to control the operator norm of the matrix $F$.
\begin{lem}
\label{LEM:matrixopnorm}
Let $A$ be a $d\times T$ random matrix such that $A \sim \sg_{d\times T}(\sigma^2)$. Then
$$
\|A\|_{\mathrm{op}} \le 4\sigma\sqrt{\log(12)(d\vee T)}+2\sigma\sqrt{2\log (1/\delta)}
$$
with probability $1-\delta$.
\end{lem}
\begin{proof}
This proof follows the same steps as Problem~\ref{EXO:randmat}. Let $\cN_1$ be a $1/4$-net for $\cS^{d-1}$ and $\cN_2$ be a $1/4$-net for $\cS^{T-1}$. It follows from Lemma~\ref{lem:coveringell2} that we can always choose $|\cN_1|\le 12^d$ and $|\cN_2|\le 12^T$. Moreover, for any $u \in \cS^{d-1}, v \in \cS^{T-1}$, it holds
\begin{align*}
u^\top A v &\le \max_{x \in \cN_1}x^\top A v + \frac{1}{4} \max_{u \in \cS^{d-1}}u^\top A v\\
& \le \max_{x \in \cN_1}\max_{y \in \cN_2}x^\top A y+ \frac{1}{4} \max_{x \in \cN_1}\max_{v \in \cS^{T-1}}x^\top A v + \frac{1}{4} \max_{u \in \cS^{d-1}}u^\top A v\\
&\le \max_{x \in \cN_1}\max_{y \in \cN_2}x^\top A y +\frac{1}{2} \max_{u \in \cS^{d-1}}\max_{v \in \cS^{T-1}}u^\top A v
\end{align*}
It yields
$$
\|A\|_{\textrm{op}} \le 2\max_{x \in \cN_1}\max_{y \in \cN_2}x^\top A y
$$
So that for any $t \ge 0$, by a union bound,
$$
\p\big(\|A\|_{\textrm{op}}> t\big) \le \sum_{\substack{x \in \cN_1\\y \in \cN_2}}\p\big(x^\top A y> t/2\big)
$$
Next, since $A \sim \sg_{d\times T}(\sigma^2)$, it holds that $x^\top A y\sim \sg(\sigma^2)$ for any $x \in \cN_1, y \in \cN_2$.Together with the above display, it yields
$$
\p\big(\|A\|_{\textrm{op}}> t\big) \le 12^{d+T}\exp\big(-\frac{t^2}{8\sigma^2}\big) \le \delta
$$
for
$$
t\ge 4\sigma\sqrt{\log(12)(d\vee T)}+2\sigma\sqrt{2\log (1/\delta)}\,.
$$
\end{proof}

The following theorem holds.
\begin{thm}
\label{TH:svt}
Consider the multivariate linear regression model~\eqref{EQ:MVRmodel} under the assumption \textsf{ORT} or, equivalently, the sub-Gaussian matrix model~\eqref{EQ:sGMM}. Then, the singular value thresholding estimator $\Thetasvt$ with threshold
\begin{equation}
\label{EQ:threshSVT}
2\tau=8\sigma\sqrt{\frac{\log(12)(d\vee T)}{n}}+4\sigma\sqrt{\frac{2\log (1/\delta)}{n}}\,,
\end{equation}
satisfies
\begin{align*}
\frac{1}{n}\|\X\Thetasvt-\X\Theta^*\|_F^2=\|\Thetasvt-\Theta^*\|_F^2& \le 144\rank(\Theta^*)\tau^2\\
&\lesssim \frac{\sigma^2\rank(\Theta^*)}{n}\Big(d\vee T + \log(1/\delta)\Big)\,.
\end{align*}
with probability $1-\delta$. 
\end{thm}
\begin{proof}
Assume without loss of generality that the singular values of $\Theta^*$ and $y$ are arranged in a non increasing order: $\lambda_1 \ge \lambda_2 \ge \dots$ and $\hat \lambda_1 \ge \hat \lambda_2\ge \dots$. Define the set $S=\{j\,:\, |\hat \lambda_j| >2 \tau\}$.

Observe first that it follows from Lemma~\ref{LEM:matrixopnorm} that $\|F\|_{\mathrm{op}} \le \tau$ for $\tau$ chosen as in~\eqref{EQ:threshSVT} on an event $\cA$ such that $\p(\cA)\ge 1-\delta$. The rest of the proof assumes that the event $\cA$ occurred. 

Note that it follows from Weyl's inequality that $|\hat \lambda_j -  \lambda_j|\le \|F\|_{\mathrm{op}}\le \tau$. It implies that $S \subset \{j\,:\, |\lambda_j|>\tau\}$ and $S^c \subset \{j\,:\, |\lambda_j| \le 3 \tau\}$.

Next define the oracle $\bar \Theta=\sum_{j \in S} \lambda_j u_jv_j^\top$ and note that
\begin{equation}
\label{EQ:prSVT1}
\|\Thetasvt -\Theta^*\|_F^2 \le 2\|\Thetasvt -\bar \Theta\|_F^2 +2\|\bar \Theta -\Theta^*\|_F^2 
\end{equation}
Using the H\"older inequality, we control the first term as follows
$$
\|\Thetasvt -\bar \Theta\|_F^2\le \rank(\Thetasvt - \bar \Theta)\|\Thetasvt -\bar \Theta\|_{\mathrm{op}}^2 \le 2|S|\|\Thetasvt -\bar \Theta\|_{\mathrm{op}}^2
$$
Moreover, 
\begin{align*}
\|\Thetasvt -\bar \Theta\|_{\mathrm{op}}&\le \|\Thetasvt -y\|_{\mathrm{op}}+\|y - \Theta^*\|_{\mathrm{op}}+\|\Theta^* -\bar \Theta\|_{\mathrm{op}}\\
&\le \max_{j \in S^c}|\hat \lambda_j|+\tau + \max_{j \in S^c}|\lambda_j| \le 6\tau\,.
\end{align*}
Therefore,
$$
\|\Thetasvt -\bar \Theta\|_F^2\le 72|S|\tau^2=72 \sum_{j\in S} \tau^2\,.
$$
The second term in~\eqref{EQ:prSVT1} can be written as
$$
\|\bar \Theta -\Theta^*\|_F^2 =\sum_{j \in S^c}|\lambda_j|^2\,.
$$
Plugging the above two displays in~\eqref{EQ:prSVT1}, we get
$$
\|\Thetasvt -\Theta^*\|_F^2 \le 144\sum_{j \in S}\tau^2+\sum_{j \in S^c}|\lambda_j|^2
$$
Since on $S$, $\tau^2=\min(\tau^2, |\lambda_j|^2)$ and on $S^c$, $|\lambda_j|^2\le 9\min(\tau^2, |\lambda_j|^2)$, it yields,
\begin{align*}
\|\Thetasvt -\Theta^*\|_F^2 &\le 144\sum_{j}\min(\tau^2,|\lambda_j|^2)\\
&\le 144\sum_{j=1}^{\rank(\Theta^*)}\tau^2\\
&=144\rank(\Theta^*)\tau^2\,.\qedhere
\end{align*}
\end{proof}
In the next subsection, we extend our analysis to the case where $\X$ does not necessarily satisfy the assumption \textsf{ORT}.

\subsection{Penalization by rank}

The estimator from this section is the counterpart of the BIC estimator in the spectral domain. However, we will see that unlike BIC, it can be computed efficiently.

Let $\Thetarank$ be any solution to the following minimization problem:
$$
\min_{\Theta \in \R^{d \times T}}\Big\{ \frac{1}{n}\|\Y -\X \Theta\|_F^2 + 2\tau^2\rank(\Theta)\Big\}\,.
$$
This estimator is called \emph{estimator by rank penalization with regularization parameter $\tau^2$}. It enjoys the following property.

\begin{thm}
\label{TH:matrixBIC}
Consider the multivariate linear regression model~\eqref{EQ:MVRmodel}. Then, the estimator by rank penalization $\Thetarank$ with regularization parameter $\tau^2$, where $\tau$ is defined in~\eqref{EQ:threshSVT}
satisfies
$$
\frac{1}{n}\|\X\Thetarank-\X\Theta^*\|_F^2 \le  8\rank(\Theta^*)\tau^2\lesssim \frac{\sigma^2\rank(\Theta^*)}{n}\Big(d\vee T + \log(1/\delta)\Big)\,.
$$
with probability $1-\delta$. 
\end{thm}
\begin{proof}
We begin as usual by noting that
$$
\|\Y -\X \Thetarank\|_F^2 + 2n\tau^2\rank(\Thetarank)\le \|\Y -\X \Theta^*\|_F^2 + 2n\tau^2\rank(\Theta^*)\,,
$$
which is equivalent to
$$
\|\X\Thetarank -\X\Theta^*\|_F^2 \le 2\langle E, \X\Thetarank -\X\Theta^*\rangle -2n\tau^2\rank(\Thetarank)+2n\tau^2\rank(\Theta^*)\,.
$$
Next, by Young's inequality, we have
$$
2\langle E, \X\Thetarank -\X\Theta^*\rangle=2\langle E, U\rangle^2 + \frac{1}{2}\|\X\Thetarank -\X\Theta^*\|_F^2\,,
$$
where 
$$
U=\frac{\X\Thetarank -\X\Theta^*}{\|\X\Thetarank -\X\Theta^*\|_F}\,.
$$
Write
$$
\X\Thetarank -\X\Theta^*=\Phi N\,,
$$
where $\Phi$ is a $n \times r, r \le d$ matrix whose columns form an orthonormal basis of the column span of $\X$. The matrix $\Phi$ can come from the SVD of $\X$ for example: $\X=\Phi \Lambda \Psi^\top$. 
It yields
$$
U=\frac{\Phi N}{\|N\|_F}
$$
and 
\begin{equation}
\label{EQ:prmatrixBIC1}
\|\X\Thetarank -\X\Theta^*\|_F^2 \le 4\langle \Phi^\top E, N/\|N\|_F\rangle^2 -4n\tau^2\rank(\Thetarank)+4n\tau^2\rank(\Theta^*)\,.
\end{equation}

Note that $\rank(N)\le \rank(\Thetarank)+\rank(\Theta^*)$. Therefore, by H\"older's inequality, we get
\begin{align*}
\langle E, U\rangle^2&=\langle \Phi^\top E, N/\|N\|_F\rangle^2\\
&\le \|\Phi^\top E\|_{\mathrm{op}}^2\frac{\|N\|_1^2}{\|N\|_F^2} \\
&\le \rank(N)\|\Phi^\top E\|_{\mathrm{op}}^2\\
&\le \|\Phi^\top E\|_{\mathrm{op}}^2\big[\rank(\Thetarank)+\rank(\Theta^*)\big]\,.
\end{align*}
Next, note that Lemma~\ref{LEM:matrixopnorm} yields  $\|\Phi^\top E\|_{\mathrm{op}}^2 \le n\tau^2$ so that
$$
\langle E, U\rangle^2\le n\tau^2 \big[\rank(\Thetarank)+\rank(\Theta^*)\big]\,.
$$
Together with~\eqref{EQ:prmatrixBIC1}, this completes the proof.
\end{proof}

It follows from Theorem~\ref{TH:matrixBIC} that the estimator by rank penalization enjoys the same properties as the singular value thresholding estimator even when $\X$ does not satisfy the \textsf{ORT} condition. This is reminiscent of the BIC estimator which enjoys the same properties as the hard thresholding estimator. However this analogy does not extend to computational questions. Indeed, while the rank penalty, just like the sparsity penalty, is not convex, it turns out that $\X \Thetarank$ can be computed efficiently. 

Note first that
$$
\min_{\Theta \in \R^{d \times T}} \frac{1}{n}\|\Y -\X \Theta\|_F^2 + 2\tau^2\rank(\Theta)=
\min_k\Big\{ \frac{1}{n}\min_{\substack{\Theta \in \R^{d \times T}\\ \rank(\Theta)\le k}} \|\Y -\X \Theta\|_F^2 + 2\tau^2k\Big\}\,.
$$
Therefore, it remains to show that 
$$
\min_{\substack{\Theta \in \R^{d \times T}\\ \rank(\Theta)\le k}} \|\Y -\X \Theta\|_F^2
$$
can be solved efficiently. To that end, let $\bar \Y=\X (\X^\top \X)^\dagger\X^\top\Y$ denote the orthogonal projection of $\Y$ onto the image space of $\X$: this is a linear operator from $\R^{d \times T}$ into $\R^{n \times T}$. By the Pythagorean theorem, we get for any $\Theta \in \R^{d\times T}$, 
$$
\|\Y -\X\Theta\|_F^2 = \|\Y-\bar \Y\|_F^2 + \|\bar \Y -\X\Theta\|_F^2\,.
$$
Next consider the SVD of $\bar \Y$:
$$
\bar \Y=\sum_j \lambda_j u_j v_j^\top
$$
where $\lambda_1 \ge \lambda_2 \ge \ldots$ and define $\tilde \Y$ by
$$
\tilde \Y=\sum_{j=1}^k \lambda_j u_j v_j^\top\,.
$$
By Lemma~\ref{lem:eckart}, it holds
$$
\|\bar \Y -\tilde \Y\|_F^2 =\min_{Z:\rank(Z) \le k}\|\bar \Y - Z\|_F^2\,.
$$
Therefore, any minimizer of $\X\Theta\mapsto  \|\Y -\X \Theta\|_F^2$ over matrices of rank at most $k$ can be obtained by truncating the SVD of $\bar \Y$ at order $k$. 

Once $\X\Thetarank$ has been found, one may obtain a corresponding $\Thetarank$ by least squares but this is not necessary for our results.

\begin{rem}
While the rank penalized estimator can be computed efficiently, it is worth pointing out that a convex relaxation for the rank penalty can also be used.  The estimator by nuclear norm penalization $\hat \Theta$ is defined to be any solution to the minimization problem
$$
\min_{\Theta \in \R^{d \times T}} \Big\{ \frac{1}{n}\|\Y-\X\Theta\|_F^2 + \tau \|\Theta\|_1\Big\}
$$
Clearly, this criterion is convex and it can actually be  implemented efficiently using semi-definite programming. It has been popularized by matrix completion problems. Let $\X$ have the following SVD:
$$
\X=\sum_{j=1}^r \lambda_j u_jv_j^\top\,,
$$
with $\lambda_1 \ge \lambda_2 \ge \ldots \ge \lambda_r>0$. It can be shown that for some appropriate choice of $\tau$, it holds 
$$
\frac{1}{n}\|\X\hat \Theta-\X\Theta^*\|_F^2 \lesssim \frac{\lambda_1}{\lambda_r}\frac{\sigma^2\rank(\Theta^*)}{n}d\vee T
$$
with probability .99. However, the proof of this result is far more involved than a simple adaption of the proof for the Lasso estimator to the matrix case (the readers are invited to see that for themselves). For one thing, there is no assumption on the design matrix (such as \textsf{INC} for example). This result can be found in \cite{KolLouTsy11}.
\end{rem}

\section{Covariance matrix estimation}
\label{sec:cov-est}

\subsection{Empirical covariance matrix}

Let $X_1, \ldots, X_n$ be $n$ \iid copies of a random vector $X \in \R^d$ such that $\E\big[X X^\top\big] =\Sigma$ for some unknown matrix $\Sigma \succ 0$ called \emph{covariance matrix}. This matrix contains information about the moments of order 2 of the random vector $X$. A natural candidate to estimate $\Sigma$ is the \emph{empirical covariance matrix} $\hat \Sigma$ defined by
$$
\hat \Sigma=\frac{1}{n} \sum_{i=1}^n X_i X_i^\top\,.
$$

Using the tools of Chapter~\ref{chap:subGauss}, we can prove the following result.

\begin{thm}
\label{TH:covmatop}
Let $Y \in \R^d$ be a random vector such that $\E[Y]=0, \E[YY^\top]=I_d$ and $Y \sim \sg_d(1)$. Let $X_1, \ldots, X_n$ be $n$ independent copies of sub-Gaussian random vector $X =\Sigma^{1/2}Y$. Then $\E[X]=0, \E[XX^\top]=\Sigma$ and  $X \sim \sg_d(\|\Sigma\|_{\mathrm{op}})$. Moreover, 
$$
\|\hat \Sigma -\Sigma \|_{\mathrm{op}} \lesssim \|\Sigma\|_{\mathrm{op}}\Big( \sqrt{\frac{d+\log(1/\delta)}{n}} \vee \frac{d+\log(1/\delta)}{n}\Big)\,,
$$
with probability $1-\delta$.
\end{thm}
\begin{proof}
It follows from elementary computations that $\E[X]=0, \E[XX^\top]=\Sigma$ and  $X \sim \sg_d(\|\Sigma\|_{\mathrm{op}})$. To prove the deviation inequality, observe first that without loss of generality, we can assume that $\Sigma=I_d$. Indeed,
\begin{align*}
\frac{\|\hat \Sigma -\Sigma \|_{\mathrm{op}}}{ \|\Sigma\|_{\mathrm{op}}}&=\frac{\|\frac1n\sum_{i=1}^nX_iX_i^\top -\Sigma \|_{\mathrm{op}} }{ \|\Sigma\|_{\mathrm{op}}}\\
&\le \frac{\|\Sigma^{1/2}\|_{\mathrm{op}}\|\frac1n\sum_{i=1}^nY_iY_i^\top -I_d \|_{\mathrm{op}}\|\Sigma^{1/2}\|_{\mathrm{op}} }{ \|\Sigma\|_{\mathrm{op}}}\\
&=\|\frac1n\sum_{i=1}^nY_iY_i^\top -I_d \|_{\mathrm{op}}\,.
\end{align*}

Thus, in the rest of the proof, we assume that $\Sigma=I_d$.
Let $\cN$ be a $1/4$-net for $\cS^{d-1}$ such that $|\cN|\le 12^d$. It follows from the proof of Lemma~\ref{LEM:matrixopnorm} that
$$
\|\hat \Sigma -I_d\|_{\textrm{op}} \le 2\max_{x,y \in \cN}x^\top (\hat \Sigma -I_d) y.
$$
Hence, for any $t \ge 0$, by a union bound,
\begin{equation}
\label{EQ:pr:opnormcov1}
\p\big(\|\hat \Sigma -I_d\|_{\textrm{op}}> t\big) \le \sum_{x,y \in \cN}\p\big(x^\top (\hat \Sigma -I_d) y> t/2\big)\,.
\end{equation}
It holds,
$$
x^\top (\hat \Sigma -I_d) y=\frac{1}{n}\sum_{i=1}^n\big\{(X_i^\top x)(X_i^\top y)-\E\big[(X_i^\top x)(X_i^\top y)\big]\big\}\,.
$$
Using polarization, we also have
$$
(X_i^\top x)(X_i^\top y)=\frac{Z_{+}^2-Z_{-}^2}{4}\,,
$$
where $Z_{+}=X_i^\top (x+y)$ and $Z_{-}=X_i^\top (x-y)$. It yields 
\begin{align*}
\E&\big[\exp\big(s\big((X_i^\top x)(X_i^\top y)-\E\big[(X_i^\top x)(X_i^\top y)\big]\big)\big)\Big]\\
&= \E\big[\exp\big(\frac{s}{4}\big(Z_{+}^2-\E[Z_{+}^2]\big)-\frac{s}{4}\big(Z_{-}^2-\E[Z_{-}^2])\big)\big)\Big]\\
&\le \Big(\E\big[\exp\big(\frac{s}{2}\big(Z_{+}^2-\E[Z_{+}^2]\big)\big)\big]\E\big[\exp\big(-\frac{s}{2}\big(Z_{-}^2-\E[Z_{-}^2]\big)\big)\big]\Big)^{1/2}\,,
\end{align*}
where in the last inequality, we used Cauchy-Schwarz. Next,  since $X \sim \sg_d(1)$, we have $Z_{+},Z_{-}\sim \sg(2)$, and it follows from Lemma~\ref{LEM:squaredsubG} that 
$$
Z_{+}^2-\E[Z_{+}^2]\sim \subE(32)\,, \qquad \text{and} \qquad Z_{-}^2-\E[Z_{-}^2]\sim \subE(32).
$$
Therefore, for any $s \le 1/16$, we have for any $Z \in \{Z_+, Z_-\}$ that
$$
\E\big[\exp\big(\frac{s}{2}\big(Z^2-\E[Z^2]\big)\big)\big]\le e^{128s^2}\,.
$$
It yields that
\begin{align*}
(X_i^\top x)(X_i^\top y)-\E\big[(X_i^\top x)(X_i^\top y)\big]\sim \subE(16)\,.
\end{align*}
Applying now Bernstein's inequality (Theorem~\ref{TH:Bernstein}), we get
\begin{equation}
	\label{eq:chi2-control}
\p\big(x^\top (\hat \Sigma -I_d) y>t/2\big)\le \exp\Big[-\frac{n}{2}(\Big(\frac{t}{32}\Big)^2\wedge \frac{t}{32})\Big]\,.
\end{equation}
Together with~\eqref{EQ:pr:opnormcov1}, this yields
\begin{equation}
\label{EQ:pr:opnormcov2}
\p\big(\|\hat \Sigma -I_d\|_{\textrm{op}}>t\big)\le  144^d\exp\Big[-\frac{n}{2}(\Big(\frac{t}{32}\Big)^2\wedge \frac{t}{32})\Big]\,.
\end{equation}
In particular, the right hand side of the above inequality is at most $\delta \in (0,1)$ if 
$$
\frac{t}{32} \ge \Big(\frac{2d}{n}\log(144)+\frac{2}{n}\log(1/\delta)\Big) \vee \Big(\frac{2d}{n}\log(144)+\frac{2}{n}\log(1/\delta)\Big)^{1/2}.
$$
This concludes our proof.
\end{proof}
Theorem~\ref{TH:covmatop} indicates that for fixed $d$, the empirical covariance matrix is a consistent estimator of $\Sigma$ (in any norm as they are all equivalent in finite dimension). However, the bound that we got is not satisfactory in high-dimensions when $d \gg n$. To overcome this limitation, we can introduce sparsity as we have done in the case of regression. The most obvious way to do so is to assume that few of the entries of $\Sigma$ are non zero and it turns out that in this case thresholding is optimal. There is a long line of work on this subject (see for example \cite{CaiZhaZho10} and \cite{CaiZho12}).

Once we have a good estimator of $\Sigma$, what can we do with it? The key insight is that $\Sigma$ contains information about the projection of the vector $X$ onto \emph{any} direction $u \in \cS^{d-1}$. Indeed, we have that $\var(X^\top u)=u^\top \Sigma u$, which can be readily estimated by $\widehat \Var(X^\top u)=u^\top \hat \Sigma u$. Observe that it follows from Theorem~\ref{TH:covmatop} that
\begin{align*}
\big|\widehat \Var(X^\top u)- \Var(X^\top u)\big|&=\big|u^\top(\hat \Sigma-\Sigma)u\big|\\
&\le \|\hat \Sigma -\Sigma \|_{\mathrm{op}}\\
& \lesssim \|\Sigma\|_{\mathrm{op}}\Big( \sqrt{\frac{d+\log(1/\delta)}{n}} \vee \frac{d+\log(1/\delta)}{n}\Big)
\end{align*}
with probability $1-\delta$.

The above fact is useful in the Markowitz theory of portfolio section for example \cite{Mar52}, where a portfolio of assets is a vector $u \in \R^d$ such that $|u|_1=1$ and the risk of a portfolio is given by the variance $ \Var(X^\top u)$. The goal is then to maximize reward subject to risk constraints. In most instances, the empirical covariance matrix is plugged into the formula in place of $\Sigma$.
(See Problem~\ref{EXO:portfolio}).

\section{Principal component analysis}

\subsection{Spiked covariance model}
Estimating the variance in all directions is also useful for \emph{dimension reduction}. In \emph{Principal Component Analysis (PCA)}, the goal is to find one (or more) directions onto which the data $X_1, \ldots, X_n$ can be projected without loosing much of its properties. There are several goals for doing this but perhaps the most prominent ones are data visualization (in few dimensions, one can plot and visualize the cloud of $n$ points) and clustering (clustering is a hard computational problem and it is therefore preferable to carry it out in lower dimensions). An example of the output of a principal component analysis is given in Figure~\ref{FIG:genmap}. In this figure, the data has been projected onto two orthogonal directions \textsf{PC1} and \textsf{PC2}, that were estimated to have the most variance (among all such orthogonal pairs). The idea is that when projected onto such directions, points will remain far apart and a clustering pattern will still emerge. This is the case in Figure~\ref{FIG:genmap} where the original data is given by $d=500,000$ gene expression levels measured on $n \simeq 1,387$ people. Depicted are the projections of these $1,387$ points in two dimension. This image has become quite popular as it shows that gene expression levels can recover the structure induced by geographic clustering.
\begin{figure}[t]
\centering
\includegraphics[width=\textwidth]{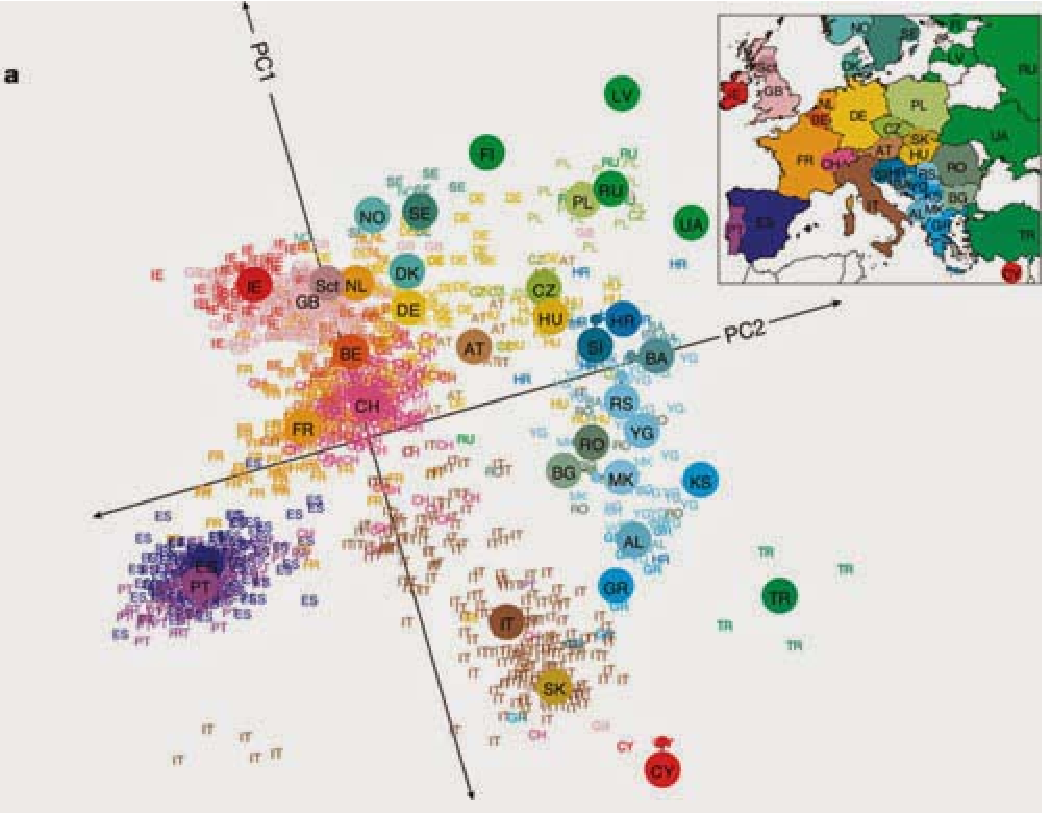}
\caption{Projection onto two dimensions of $1,387$ points from gene expression data. Source: Gene expression blog.} 
\label{FIG:genmap}
\end{figure}
How is it possible to ``compress" half a million dimensions into only two? The answer is that the data is intrinsically low dimensional. In this case, a plausible assumption is that all the $1,387$ points live close to a two-dimensional linear subspace. To see how this assumption (in one dimension instead of two for simplicity) translates into the structure of the covariance matrix $\Sigma$, assume that $X_1, \ldots, X_n$ are Gaussian random variables generated as follows. Fix a direction $v \in \cS^{d-1}$ and  let $Y_1, \ldots, Y_n \sim \cN_d(0,I_d)$ so that $v^\top Y_i$ are i.i.d. $\cN(0,1)$. In particular, the vectors $(v^\top Y_1)v, \ldots, (v^\top Y_n)v$ live in the one-dimensional space spanned by $v$. If one would observe such data the problem would be easy as only two observations would suffice to recover $v$. Instead, we observe $X_1, \ldots, X_n \in \R^d$ where $X_i=(v^\top Y_i)v+Z_i$, and $Z_i \sim \cN_d(0, \sigma^2I_d)$ are \iid and independent of the $Y_i$s, that is we add isotropic noise to every point. If the $\sigma$ is small enough, we can hope to recover the direction $v$ (See Figure~\ref{FIG:lowdim}). The covariance matrix of $X_i$ generated as such is given by
$$
\Sigma=\E\big[XX^\top\big]=\E\big[((v^\top Y)v+Z)((v^\top Y)v+Z)^\top\big]=vv^\top + \sigma^2I_d\,.
$$
This model is often called the \emph{spiked covariance model}. By a simple rescaling, it is equivalent to the following definition.
\begin{defn}
A covariance matrix $\Sigma \in \R^{d\times d}$ is said to satisfy the spiked covariance model if it is of the form
$$
\Sigma=\theta vv^\top +I_d\,,
$$
where $\theta>0$ and $v \in \cS^{d-1}$. The vector $v$ is called the \emph{spike}.
\end{defn}
This model can be extended to more than one spike but this extension is beyond the scope of these notes.

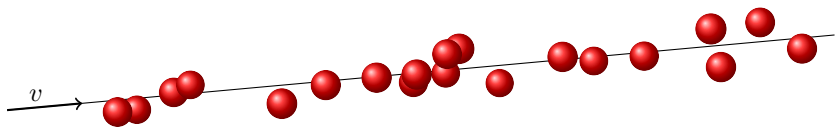
\begin{figure}[t]
\begin{tikzpicture}
\draw (0,0) -- (11,1);
\draw [->, thick](0,0) -- (1,1/11);
\draw (.6, .2) node[left] {$v$};
\pgfmathsetseed{10}

  \foreach \i in {1,2,...,20}{
    \pgfmathsetmacro{\x}{(rand+1)*5+1.2}
    \pgfmathsetmacro{\y}{\x/11+0.3*rand}
      \pgfmathsetmacro{\rad}{0.2+rand*0.01}
    \pgfmathsetmacro{\opacVal}{.1}
    \shade [ball color = red] (\x,\y) circle [radius=\rad];
  }
\end{tikzpicture}
\caption{Points are close to a one-dimensional space spanned by $v$.} 
\label{FIG:lowdim}
\end{figure}

Clearly, under the spiked covariance model, $v$ is the eigenvector of the matrix $\Sigma$ that is associated to its largest eigenvalue $1+\theta$. We will refer to this vector simply as \emph{largest eigenvector}.  To estimate it, a natural candidate is the largest eigenvector $\hat v$ of $\tilde \Sigma$, where $\tilde \Sigma$ is any estimator of $\Sigma$. There is a caveat: by symmetry, if $u$ is an eigenvector, of a symmetric matrix, then $-u$ is also an eigenvector associated to the same eigenvalue. Therefore, we may only estimate $v$ up to a sign flip. To overcome this limitation, it is often useful to describe proximity between two vectors $u$ and $v$ in terms of the principal angle between their linear span. Let us recall that for two unit vectors the principal angle between their linear spans is denoted by $\angle(u,v)$ and defined as 
$$
\angle(u,v)=\arccos(|u^\top v|)\,.
$$
The following result from perturbation theory is known as the Davis-Kahan $\sin(\theta)$ theorem as it bounds the sin of the principal angle between eigenspaces. This theorem exists in much more general versions that extend beyond one-dimensional eigenspaces. 

\begin{thm}[Davis-Kahan $\sin(\theta)$ theorem]
\label{TH:sintheta}
Let $A, B$ be two $d \times d$ PSD matrices. Let $(\lambda_1, u_1), \ldots, (\lambda_d, u_d)$ (resp. $(\mu_1, v_1), \ldots, (\mu_d, v_d)$) denote the pairs of eigenvalues--eigenvectors  of $A$ (resp. B) ordered such that $\lambda_1 \ge \lambda_2 \ge \dots$ (resp. $\mu_1 \ge \mu_2 \ge \dots$). Then 
$$
\sin\big(\angle(u_1, v_1)\big)\le \frac{2}{\max(\lambda_1-\lambda_2, \mu_1-\mu_2)} \|A-B\|_{\mathrm{op}}\,.
$$
Moreover,
$$
\min_{\eps \in \{\pm1\}}|\eps u_1 -v_1|_2^2 \le 2 \sin^2\big(\angle(u_1, v_1)\big)\,.
$$

\end{thm}
\begin{proof}
Note that $u_1^\top A u_1=\lambda_1$ and for any $x \in \cS^{d-1}$, it holds
\begin{align*}
x^\top A x &=\sum_{j=1}^d\lambda_j(u_j^\top x)^2\le \lambda_1(u_1^\top x)^2 + \lambda_2(1-(u_1^\top x)^2)\\\
&=\lambda_1\cos^2\big(\angle(u_1, x)\big) + \lambda_2 \sin^2\big(\angle(u_1, x)\big)\,.
\end{align*}
Therefore, taking $x=v_1$, we get that on the one hand,
\begin{align*}
u_1^\top A u_1-v_1^\top A v_1&\ge \lambda_1 -\lambda_1\cos^2\big(\angle(u_1, x)\big) - \lambda_2 \sin^2\big(\angle(u_1, x)\big)\\
&=(\lambda_1-\lambda_2)\sin^2\big(\angle(u_1, x)\big)\,.
\end{align*}
On the other hand,
\begin{align*}
u_1^\top A u_1-v_1^\top A v_1&=u_1^\top B u_1-v_1^\top A v_1+u_1^\top(A-B)u_1&\\
&\le v_1^\top B v_1 -v_1^\top A v_1+u_1^\top(A-B)u_1& \text{($v_1$ is leading eigenvector of $B$)}\\
&=\langle A-B, u_1 u_1^\top -v_1 v_1^\top\rangle&\\
&\le \|A-B\|_\mathrm{op}\|u_1 u_1^\top -v_1 v_1^\top\|_1 & \text{(H\"older)}\\
& \le \|A-B\|_\mathrm{op}\sqrt{2}\|u_1 u_1^\top -v_1 v_1^\top\|_2 & \text{(Cauchy-Schwarz)}\\
\end{align*}
where in the last inequality, we used the fact that $\rank(u_1 u_1^\top -v_1 v_1^\top)=2$. 

It is straightforward to check that
$$
\|u_1 u_1^\top -v_1 v_1^\top\|_2^2=\|u_1 u_1^\top -v_1 v_1^\top\|_F^2=2-2(u_1^\top v_1)^2=2\sin^2\big(\angle(u_1, v_1)\big)\,.
$$
We have proved that
$$
(\lambda_1-\lambda_2)\sin^2\big(\angle(u_1, x)\big) \le 2  \|A-B\|_\mathrm{op}\sin\big(\angle(u_1, x)\big)\,,
$$
which concludes the first part of the lemma. Note that we can replace $\lambda_1-\lambda_2$ with $\mu_1 -\mu_2$ since the result is completely symmetric in $A$ and $B$. 

It remains to show the second part of the lemma. To that end, observe that
\begin{equation*}
\min_{\eps \in \{\pm1\}}|\eps u_1 -v_1|_2^2=2-2|\tilde u_1^\top v_1|\le 2-2(u_1^\top v_1)^2=2\sin^2\big(\angle(u_1, x)\big)\,.\qedhere
\end{equation*}
\end{proof}
Combined with Theorem~\ref{TH:covmatop}, we immediately get the following corollary.
\begin{cor}
\label{COR:sintheta}
Let $Y \in \R^d$ be a random vector such that $\E[Y]=0, \E[YY^\top]=I_d$ and $Y \sim \sg_d(1)$. Let $X_1, \ldots, X_n$ be $n$ independent copies of sub-Gaussian random vector $X =\Sigma^{1/2}Y$ so that $\E[X]=0, \E[XX^\top]=\Sigma$ and  $X \sim \sg_d(\|\Sigma\|_{\mathrm{op}})$.  Assume further that $\Sigma=\theta vv^\top + I_d$ satisfies the spiked covariance model. Then, the largest eigenvector $\hat v$ of the empirical covariance matrix $\hat \Sigma$ satisfies,
$$
\min_{\eps \in \{\pm1\}}|\eps\hat v -v|_2\lesssim \frac{1+\theta}{\theta}\Big( \sqrt{\frac{d+\log(1/\delta)}{n}} \vee \frac{d+\log(1/\delta)}{n}\Big)
$$
with probability $1-\delta$. 
\end{cor}
This result justifies the use of the empirical covariance matrix $\hat \Sigma$ as a replacement for the true covariance matrix $\Sigma$ 
when performing PCA in low dimensions, that is when $d \ll n$. In the high-dimensional case, where $d \gg n$, the above result is uninformative. As before, we resort to sparsity to overcome this limitation.

\subsection{Sparse PCA}
In the example of Figure~\ref{FIG:genmap}, it may be desirable to interpret the meaning of the two directions denoted by \textsf{PC1} and \textsf{PC2}. We know that they are linear combinations of the original 500,000 gene expression levels. A natural question to ask is whether only a subset of these genes could suffice to obtain similar results. Such a discovery could have potential interesting scientific applications as it would point to a few genes responsible for disparities between European populations.

In the case of the spiked covariance model this amounts to having a sparse $v$. Beyond interpretability as we just discussed, sparsity should also lead to statistical stability as in the case of sparse linear regression for example. To enforce sparsity, we will assume that $v$ in the spiked covariance model is $k$-sparse: $|v|_0=k$. Therefore, a natural candidate to estimate $v$ is given by $\hat v$ defined by
$$
\hat v^\top \hat \Sigma \hat v =\max_{\substack{u \in \cS^{d-1}\\|u|_0=k}}u^\top \hat \Sigma u\,.
$$
It is easy to check that $\lambda^k_{\max}(\hat \Sigma)=\hat v^\top \hat \Sigma \hat v$ is the largest of all leading eigenvalues among all $k\times k$ sub-matrices of $\hat \Sigma$ so that the maximum is indeed attained, though there my be several maximizers. We call $\lambda^k_{\max}(\hat \Sigma)$ the $k$-sparse leading eigenvalue of $\hat \Sigma$ and $\hat v$ a $k$-sparse leading eigenvector. 

\begin{thm}
\label{TH:sparsePCA}
Let $Y \in \R^d$ be a random vector such that $\E[Y]=0, \E[YY^\top]=I_d$ and $Y \sim \sg_d(1)$. Let $X_1, \ldots, X_n$ be $n$ independent copies of sub-Gaussian random vector $X =\Sigma^{1/2}Y$ so that $\E[X]=0, \E[XX^\top]=\Sigma$ and  $X \sim \sg_d(\|\Sigma\|_{\mathrm{op}})$.  Assume further that $\Sigma=\theta vv^\top + I_d$ satisfies the spiked covariance model for $v$ such that $|v|_0=k \le d/2$. Then, the $k$-sparse largest eigenvector $\hat v$ of the empirical covariance matrix satisfies,
$$
\min_{\eps \in \{\pm1\}}|\eps\hat v -v|_2\lesssim \frac{1+\theta}{\theta}\Big( \sqrt{\frac{k\log(ed/k)+\log(1/\delta)}{n}} \vee \frac{k\log(ed/k)+\log(1/\delta)}{n}\Big)\,.
$$
with probability $1-\delta$. 
\end{thm}
\begin{proof}
We begin by obtaining an intermediate result of the Davis-Kahan $\sin(\theta)$ theorem. Using the same steps as in the proof of Theorem~\ref{TH:sintheta}, we get
$$
v^\top \Sigma v -\hat v^\top \Sigma \hat v\le \langle \hat \Sigma - \Sigma, \hat v \hat v^\top -vv^\top\rangle
$$
Since both $\hat v$ and $v$ are $k$ sparse, there exists a (random) set $S \subset \{1, \ldots, d\}$ such that $|S|\le 2k$ and $\{\hat v \hat v^\top -vv^\top\}_{ij} =0$ if $(i,j) \notin S^2$. It yields
$$
\langle \hat \Sigma - \Sigma, \hat v \hat v^\top -vv^\top\rangle=\langle \hat \Sigma(S) - \Sigma(S), \hat v(S) \hat v(S)^\top -v(S)v(S)^\top\rangle
$$
Where for any $d\times d$ matrix  $M$, we defined the matrix $M(S)$ to be the $|S|\times |S|$ sub-matrix of $M$ with rows and columns indexed by $S$ and for any vector $x \in \R^d$, $x(S)\in \R^{|S|}$ denotes the sub-vector of $x$ with coordinates indexed by $S$. It yields by H\"older's inequality that
$$
v^\top \Sigma v -\hat v^\top \Sigma \hat v\le \|\hat \Sigma(S) - \Sigma(S)\|_{\mathrm{op}}\| \hat v(S) \hat v(S)^\top -v(S)v(S)^\top\|_1\,.
$$
Following the same steps as in the proof of Theorem~\ref{TH:sintheta}, we get now that
$$
\sin\big(\angle(\hat v, v)\big)\le \frac{2}{\theta}\sup_{S\,:\,|S|=2k} \|\hat \Sigma(S) - \Sigma(S)\|_{\mathrm{op}}\,.
$$
To conclude the proof, it remains to control $\sup_{S\,:\,|S|=2k} \|\hat \Sigma(S) - \Sigma(S)\|_{\mathrm{op}}$. To that end, observe that
\begin{align*}
\p\Big[\sup_{S\,:\,|S|=2k}&  \|\hat \Sigma(S) - \Sigma(S)\|_{\mathrm{op}}>t\|\Sigma\|_{\mathrm{op}}\Big]\\
&\le \sum_{{S\,:\,|S|=2k} }\p\Big[\sup_{S\,:\,|S|=2k}  \|\hat \Sigma(S) - \Sigma(S)\|_{\mathrm{op}}>t\|\Sigma(S)\|_{\mathrm{op}}\Big]\\
&\le \binom{d}{2k} 144^{2k}\exp\Big[-\frac{n}{2}(\Big(\frac{t}{32}\Big)^2\wedge \frac{t}{32})\Big]\,.
\end{align*}
where we used~\eqref{EQ:pr:opnormcov2} in the second inequality. Using Lemma~\ref{lem:nchoosek}, we get that the right-hand side above is further bounded by
$$
 \exp\Big[-\frac{n}{2}(\Big(\frac{t}{32}\Big)^2\wedge \frac{t}{32})+2k\log(144)+k\log\big(\frac{ed}{2k}\big)\Big]\,
$$
Choosing now $t$ such that
$$
t\ge C\sqrt{\frac{k\log(ed/k)+\log(1/\delta)}{n}} \vee \frac{k\log(ed/k)+\log(1/\delta)}{n}\,,
$$
for large enough $C$ ensures that the desired bound holds with probability at least $1-\delta$.
\end{proof}

\subsection{Minimax lower bound}

The last section has established that the spike $v$ in the spiked covariance model $\Sigma =\theta v v^\top$ may be estimated at the rate $\theta^{-1} \sqrt{k\log (d/k)/n}$, assuming that $\theta \le 1$ and $k\log (d/k)\le n$, which corresponds to the interesting regime. Using the tools from Chapter~\ref{chap:minimax}, we can show that this rate is in fact optimal already in the Gaussian case.

\begin{thm}
\label{thm:lbspca}
Let $X_1, \ldots, X_n$ be $n$ independent copies of the Gaussian random vector $X \sim\cN_d(0, \Sigma)$ where $\Sigma=\theta vv^\top + I_d$ satisfies the spiked covariance model for $v$ such that $|v|_0=k$ and write $\p_v$ the distribution of $X$ under this assumption and by $\E_v$ the associated expectation. Then, there exists constants $C,c>0$ such that for any $n\ge 2$,  $d \ge 9$, $2 \le k \le (d+7)/8$, it holds
$$
\inf_{\hat v} \sup_{\substack{v \in \cS^{d-1}\\ |v|_0 =k}} \p_v^n\Big[ \min_{\eps \in \{\pm 1\}} |\eps \hat v -v|_2\ge C \frac{1}{\theta} \sqrt{\frac{k}{n} \log (ed/k)}\Big] \ge c\,.
$$ 
\end{thm}
\begin{proof}
We use the general technique developed in  Chapter~\ref{chap:minimax}. Using this technique, we need to provide the a set $v_1, \ldots, v_M \in \cS^{d-1}$ of $k$ sparse vectors such that
\begin{itemize}
\item[(i)] $\DS  \min_{\eps \in \{-1,1\}} |\eps v_j -v_k|_2> 2\psi$ for all $1\le j <k \le M$\,,
\item[(ii)] $\DS \KL(\p_{v_j}, \p_{v_k}) \le c\log(M)/n$\,,
\end{itemize}
where $\psi=C\theta^{-1} \sqrt{k\log (d/k)/n}$. Akin to the Gaussian sequence model, our goal is to employ the sparse Varshamov-Gilbert lemma to construct the $v_j$'s. However, the $v_j$'s have the extra constraint that they need to be of unit norm so we cannot simply rescale the sparse binary vectors output by the sparse Varshamov-Gilbert lemma. Instead, we use the last coordinate to make sure that they have unit norm.

Recall that it follows from the sparse Varshamov-Gilbert Lemma~\ref{LEM:sVG} that there exits $\omega_1, \ldots, \omega_M \in \{-1, 1\}^{d-1}$ such that $|\omega_j|_0=k-1$,  $\rho(\omega_i, \omega_j) \ge (k-1)/2$ for all $i \neq j$ and $\log(M) \ge \frac{k-1}{8}\log(1+ \frac{d-1}{2k-2})$. Define $v_j=(\gamma\omega_j, \ell)$ where $\gamma<1/\sqrt{k-1}$ is to be defined later and $\ell =\sqrt{1-\gamma^2(k-1)}$. It is straightforward to check that $|v_j|_0=k$, $|v_j|_2=1$ and $v_j$ has only nonnegative coordinates.  Moreover, for any $i \neq j$, 
\begin{equation}
\label{EQ:prlbspca1}
\min_{\eps \in {\pm 1}} |\eps v_i -v_j|^2_2= |v_i -v_j|^2_2=\gamma^2\rho(v_i, v_j)\ge \gamma^2 \frac{k-1}2\,.
\end{equation}
We choose 
$$
\gamma^2=\frac{C_\gamma}{\theta^2 n} \log\big(\frac{d}{k}\big)\,,
$$
so that $\gamma^2(k-1)=\psi$ and (i) is verified. It remains to check (ii). To that end, note that
\begin{align*}
\KL(\p_u, \p_v)&=\frac{1}{2}\E_u\Big(X^\top \big[(I_d+\theta vv^\top)^{-1}-(I_d+\theta vv^\top)^{-1} \big]X\Big)\\
&=\frac{1}{2}\E_u\tr\Big(X^\top \big[(I_d+\theta vv^\top)^{-1}-(I_d+\theta vv^\top)^{-1} \big]X\Big)\\
&=\frac{1}{2}\tr\Big(\big[(I_d+\theta vv^\top)^{-1}-(I_d+\theta uu^\top)^{-1} \big](I_d+\theta uu^\top)\Big)
\end{align*} 

Next, we use the Sherman-Morrison formula\footnote{To see this, check that $(I_d+\theta u u^\top)(I_d-\frac{\theta}{1+\theta})=I_d$ and that the two matrices on the left-hand side commute.}:
$$
(I_d+\theta u u^\top)^{-1}=I_d-\frac{\theta}{1+\theta} u u^\top\,, \qquad \forall\, u \in \cS^{d-1}\,.
$$
It yields
\begin{align*}
\KL(\p_u, \p_v)&=\frac{\theta}{2(1+\theta)}\tr\Big((uu^\top-vv^\top)(I_d+\theta uu^\top)\Big)\\
&=\frac{\theta^2}{2(1+\theta)}(1-(u^\top v)^2)\\
&=\frac{\theta^2}{2(1+\theta)}\sin^2\big(\angle(u,v)\big)
\end{align*}

It yields
$$
\KL(\p_{v_i}, \p_{v_j})=\frac{\theta^2}{2(1+\theta)}\sin^2\big(\angle(v_i,v_j)\big)\,.
$$
Next, note that
\begin{align*}
\sin^2\big(\angle(v_i,v_j)\big)&=1-\gamma^2 \omega_i^\top \omega_j -\ell^2\le \gamma^2(k-1)\le C_\gamma\frac{k}{\theta^2 n} \log\big(\frac{d}{k}\big)\le \frac{\log M}{\theta^2 n}\,,
\end{align*}
for $C_\gamma$ small enough. We conclude that (ii) holds, which completes our proof.
\end{proof}

Together with the upper bound of Theorem~\ref{TH:sparsePCA}, we get the following corollary.

\begin{cor}
Assume that  $\theta \le 1$ and $k\log (d/k)\le n$. Then  $\theta^{-1} \sqrt{k\log (d/k)/n}$ is the minimax rate of estimation over $\cB_0(k)$ in the spiked covariance model.
\end{cor}

\section{Graphical models}

\subsection{Gaussian graphical models}

In Section \ref{sec:cov-est}, we noted that thresholding is an appropriate way to gain an advantage from sparsity in the covariance matrix when trying to estimate it.
However, in some cases, it is more natural to assume sparsity in the inverse of the covariance matrix, \( \Theta = \Sigma^{-1} \), which is sometimes called \emph{concentration} or \emph{precision matrix}.
In this chapter, we will develop an approach that allows us to get guarantees similar to the lasso for the error between the estimate and the ground truth matrix in Frobenius norm.

We can understand why sparsity in \( \Theta \) can be appropriate in the context of \emph{undirected graphical models} \cite{Lau96}.
These are models that serve to simplify dependence relations between high-dimensional distributions according to a graph.

\begin{defn}
Let \( G = (V, E) \) be a graph on \( d \) nodes and to each node \( v \), associate a random variable \( X_v \).
An undirected graphical model or \emph{Markov random field} is a collection of probability distributions over \( X_v \) that factorize according to
\begin{equation}
	\label{eq:factor-graph}
	p(x_1, \dots, x_d) = \frac{1}{Z} \prod_{C \in \mathcal{C}} \psi_C(x_C),
\end{equation}
where \( \mathcal{C} \) is the collection of all cliques (completely connected subsets) in \( G \), \( x_C \) is the restriction of \( (x_1, \dots, x_d) \) to the subset \( C \), and \( \psi_C \) are non-negative potential functions.
\end{defn}

This definition implies certain conditional independence relations, such as the following.

\begin{prop}
	A graphical model as in \eqref{eq:factor-graph} fulfills the \emph{global Markov property}.
	That is, for any triple \( (A, B, S) \) of disjoint subsets such that \( S \) separates \( A \) and \( B \),
	\begin{equation*}
		A \independent B \mid S.
	\end{equation*}
\end{prop}

\begin{proof}
	Let \( (A, B, S) \) a triple as in the proposition statement and define \( \tilde{A} \) to be the connected component in the induced subgraph \( G[V \setminus S] \), as well as \( \tilde{B} = G[V \setminus (\tilde{A} \cup S)] \).
	By assumption, \( A \) and \( B \) have to be in different connected components of \( G[V \setminus S] \), hence any clique in \( G \) has to be a subset of \( \tilde{A} \cup S \) or \( \tilde{B} \cup S \).
	Denoting the former cliques by \( \mathcal{C}_A \), we get
	\begin{equation*}
		f(x) = \prod_{C \in \mathcal{C}} \psi_C(x) = \prod_{C \in \mathcal{C}_A} \psi_C(x) \prod_{c \in \mathcal{C} \setminus \mathcal{C}_A} \psi_C(x) = h(x_{\tilde{A} \cup S}) k(x_{\tilde{B} \cup S}),
	\end{equation*}
	which implies the desired conditional independence.
\end{proof}

A weaker form of the Markov property is the so-called \emph{pairwise Markov property}, which says that the above holds only for singleton sets of the form \( A = \{a\} \), \( B = \{ b \} \) and \( S = V \setminus \{a, b\} \) if \( (a, b) \notin E \).

In the following, we will focus on Gaussian graphical models for \( n \) \iid samples \( X_i \sim \cN(0, (\Theta^\ast)^{-1}) \).
By contrast, there is a large body of work on discrete graphical models such as the Ising model (see for example \cite{Bre15}), which we are not going to discuss here, although links to the following can be established as well \cite{LohWaioth12}.

For a Gaussian model with mean zero, by considering the density in terms of the concentration matrix \( \Theta \),
\begin{equation*}
	f_\Theta(x) \propto \exp(-x^\top \Theta x/2),
\end{equation*}
it is immediate that the pairs \( a, b \) for which the pairwise Markov property holds are exactly those for which \( \Theta_{i, j} = 0 \).
Conversely, it can be shown that for a family of distributions with non-negative density, this actually implies the factorization \eqref{eq:factor-graph} according to the graph given by the edges \( \{ (i, j) : \Theta_{i, j} \neq 0 \} \).
This is the content of the Hammersley-Clifford theorem.

\begin{thm}[Hammersley-Clifford]
	\label{thm:hammersley-clifford}
	Let \( P \) be a probability distribution with positive density \( f \) with respect to a product measure over \( V \).
	Then \( P \) factorizes over \( G \) as in \eqref{eq:factor-graph} if and only if it fulfills the pairwise Markov property.
\end{thm}

We show the if direction. 
The requirement that \( f \) be positive allows us to take logarithms.
Pick a fixed element \( x^\ast \).
The idea of the proof is to rewrite the density in a way that allows us to make use of the pairwise Markov property.
The following combinatorial lemma will be used for this purpose.

\begin{lem}[Moebius Inversion]
	\label{lem:moebius-inversion}
	Let \( \Psi, \Phi : \cP(V) \to \R \) be functions defined for all subsets of a set \( V \).
	Then, the following are equivalent:
	\begin{enumerate}
		\item For all \( A \subseteq V \): \( \Psi(A) = \sum_{B \subseteq A} \Phi(B) \);
		\item For all \( A \subseteq V \): \( \Phi(A) = \sum_{B \subseteq A} (-1)^{|A \setminus B}\Psi(B) \).
	\end{enumerate}
\end{lem}

\begin{proof}
	\( 1. \implies 2. \):
	\begin{align*}
		\sum_{B \subseteq A} \Phi(B)
		= {} & \sum_{B \subseteq A} \sum_{C \subseteq B} (-1)^{|B \setminus C|} \Psi(C)\\
		= {} & \sum_{C \subseteq A} \Psi(C) \left( \sum_{C \subseteq B \subseteq A} (-1)^{|B \setminus C|}\right) \\
		= {} & \sum_{C \subseteq A} \Psi(C) \left( \sum_{H \subseteq A \setminus C} (-1)^{|H|} \right).
	\end{align*}
	Now, \( \sum_{H \subseteq A \setminus C} (-1)^{|H|} = 0 \) unless \( A \setminus C = \emptyset \) because any non-empty set has the same number of subsets with odd and even cardinality, which can be seen by induction.
\end{proof}

\begin{proof}
	Define
	\begin{equation*}
		H_A(x) = \log f(x_A, x^\ast_{A^c})
	\end{equation*}
	and
	\begin{equation*}
		\Phi_A(x) = \sum_{B \subseteq A} (-1)^{|A \setminus B|} H_B(x).
	\end{equation*}
	Note that both \( H_A \) and \( \Phi_A \) depend on \( x \) only through \( x_A \).
	By Lemma \ref{lem:moebius-inversion}, we get
	\begin{equation*}
		\log f(x) = H_V(x) = \sum_{A \subseteq V} \Phi_A(x).
	\end{equation*}

	It remains to show that \( \Phi_A = 0 \) if \( A \) is not a clique.
	Assume \( A \subseteq V \) with \( v, w \in A \), \( (v, w) \notin E \), and set \( C = V \setminus \{v, w\} \).
	By adding all possible subsets of \( \{v, w\} \) to every subset of \( C \), we get every subset of \( A \), and hence
	\begin{equation*}
		\Phi_A(x) = \sum_{B \subseteq C} (-1)^{|C \setminus B|} ( H_B - H_{B \cup \{v\}} - H_{B \cup \{w\}} + H_{B \cup \{v, w\}}).
	\end{equation*}
	Writing \( D = V \setminus \{v, w\} \), by the pairwise Markov property, we have
	\begin{align*}
		H_{B \cup \{v, w\}}(x) - H_{B \cup \{v\}}(x)
		= {} & \log \frac{f(x_v, x_w, x_B, x^\ast_{D \setminus B}}{f(x_v, x_w^\ast, x_B, x^\ast_{D \setminus B}}\\
				= {} & \log \frac{f(x_v | x_{B}, x^\ast_{D \setminus B}) f(x_w, x_B, x^\ast_{D \setminus B})}{f(x_v | x_{B}, x^\ast_{D \setminus B}) f(x^\ast_w, x_B, x^\ast_{D \setminus B})}\\
				= {} & \log \frac{f(x^\ast_v | x_{B}, x^\ast_{D \setminus B}) f(x_w, x_B, x^\ast_{D \setminus B})}{f(x^\ast_v | x_{B}, x^\ast_{D \setminus B}) f(x^\ast_w, x_B, x^\ast_{D \setminus B})}\\
				= {} & \log \frac{f(x^\ast_v, x_w, x_B, x^\ast_{D \setminus B})}{f(x^\ast_v, x_w^\ast, x_B, x^\ast_{D \setminus B})}\\
				= {} & H_{B \cup \{ w \}}(x) - H_B(x).
	\end{align*}
	Thus \( \Psi_A \) vanishes if \( A \) is not a clique.
\end{proof}

In the following, we will show how to exploit this kind of sparsity when trying to estimate \( \Theta \).
The key insight is that we can set up an optimization problem that has a similar error sensitivity as the lasso, but with respect to \( \Theta \).
In fact, it is the penalized version of the maximum likelihood estimator for the Gaussian distribution.
We set
\begin{equation}
	\label{eq:af}
	\hat{\Theta}_\lambda = \argmin_{\Theta \succ 0} \tr(\hat{\Sigma} \Theta) - \log \det (\Theta) + \lambda \| \Theta_{D^c} \|_1,
\end{equation}
where we wrote \( \Theta_{D^c} \) for the off-diagonal restriction of \( \Theta \), \( \| A \|_1 = \sum_{i, j} | A_{i, j} | \) for the element-wise \( \ell_1 \) norm, and \( \hat{\Sigma} = \frac{1}{n} \sum_{i} X_i X_i^\top \) for the empirical covariance matrix of \( n \) \iid observations.
We will also make use of the notation
\begin{equation*}
	\| A \|_\infty = \max_{i, j \in [d]} | A_{i, j} |
\end{equation*}
for the element-wise \( \infty \)-norm of a matrix \( A \in \R^{d \times d} \) and
\begin{equation*}
	\| A \|_{\infty, \infty} = \max_{i} \sum_{j} | A_{i, j} |
\end{equation*}
for the \( \infty \) operator norm.

We note that the function \( \Theta \mapsto - \log \operatorname{det}(\Theta) \) has first derivative \( - \Theta^{-1} \) and second derivative \( \Theta^{-1} \otimes \Theta^{-1} \), which means it is convex.
Derivations for this can be found in \cite[Appendix A.4]{BoyVan04}.
This means that in the population case and for \( \lambda = 0 \), the minimizer in \eqref{eq:af} would coincide with \( \Theta^\ast \).

First, let us derive a bound on the error in \( \hat{\Sigma} \) that is similar to Theorem \ref{TH:finitemax}, with the difference that we are dealing with sub-Exponential distributions.

\begin{lem}
	\label{lem:chi2-bounds}
	If \( X_i \sim \cN(0, \Sigma) \) are \iid and \( \hat{\Sigma} = \frac{1}{n} \sum_{i=1}^{n} X_i X_i^\top \) denotes the empirical covariance matrix, then
	\begin{equation*}
		\| \hat{\Sigma} - \Sigma \|_{\infty} \lesssim \| \Sigma \|_{\mathrm{op}}\sqrt{\frac{\log(d/\delta)}{n}},
	\end{equation*}
	with probability at least \( 1 - \delta \), if \( n \gtrsim \log(d/\delta) \).
\end{lem}

\begin{proof}
	Start as in the proof of Theorem \ref{TH:covmatop} by writing \( Y = \Sigma^{-1/2} X \sim \cN(0, I_d) \) and notice that
	\begin{align*}
		| e_i^\top (\hat{\Sigma} - \Sigma) e_j |
		= {} & | ( \Sigma^{1/2} e_i)^\top \left( \frac{1}{n} \sum_{i=1}^{n} Y_i Y_i^\top - I_d \right) (\Sigma^{1/2} e_j) |\\
		\leq {} & \| \Sigma^{1/2} \|_{\mathrm{op}}^2 | v^\top \left(\frac{1}{n} \sum_{i=1}^{n} Y_i Y_i^\top - I_d \right) w |
	\end{align*}
	with \( v, w \in \cS^{d-1} \).
	Hence, without loss of generality, we can assume \( \Sigma = I_d \).
	The remaining term is sub-Exponential and can be controlled as in \eqref{eq:chi2-control}
	\begin{equation*}
		\p\big(e_i^\top (\hat \Sigma -I_d) e_j>t\big)\le \exp\Big[-\frac{n}{2}(\Big(\frac{t}{16}\Big)^2\wedge \frac{t}{16})\Big]\,.
	\end{equation*}
	By a union bound, we get
	\begin{equation*}
		\p\big( \| \hat \Sigma -I_d \|_\infty > t\big)\le d^2 \exp\Big[-\frac{n}{2}(\Big(\frac{t}{16}\Big)^2\wedge \frac{t}{16})\Big]\,.
	\end{equation*}
	Hence, if \( n \gtrsim \log(d/\delta) \), then the right-hand side above is at most \( \delta \in (0, 1) \) if
	\begin{equation*}
		\frac{t}{16} \geq \left( \frac{2 \log(d/\delta)}{n}\right)^{1/2}. \qedhere
	\end{equation*}
\end{proof}

\begin{thm}
	\label{thm:graphical-lasso}
	Let \( X_i \sim \cN(0, \Sigma) \) be \iid and \( \hat{\Sigma} = \frac{1}{n} \sum_{i=1}^{n} X_i X_i^\top \) denote the empirical covariance matrix.
	Assume that \( | S | = k \), where \( S = \{(i,j) : i \neq j, \, \Theta_{i, j} \neq 0 \} \).
	There exists a constant \( c = c(\| \Sigma \|_{\mathrm{op}}) \) such that if
	\begin{equation*}
		\lambda = c \sqrt{\frac{\log(d/\delta)}{n}},
	\end{equation*}
	and \( n \gtrsim \log(d/\delta) \), then the minimizer in \eqref{eq:af} fulfills
	\begin{equation*}
		\| \hat{\Theta} - \Theta^\ast \|_F^2 \lesssim \frac{(p+k) \log(d/\delta)}{n}
	\end{equation*}
	with probability at least \( 1 - \delta \).
\end{thm}

\begin{proof}
We start by considering the likelihood
\begin{equation*}
	l(\Theta, \Sigma) = \tr(\Sigma \Theta) - \log \det (\Theta).
\end{equation*}
Taylor expanding it and the mean value theorem yield
\begin{align*}
	l(\Theta, \Sigma_n) - l(\Theta^\ast, \Sigma_n)
	= {} & \tr(\Sigma_n (\Theta - \Theta^\ast)) - \tr(\Sigma^\ast (\Theta - \Theta^\ast))\\
	{} &+ \frac{1}{2} \tr(\tilde{\Theta}^{-1} (\Theta - \Theta^\ast) \tilde{\Theta}^{-1} (\Theta - \Theta^\ast)),
\end{align*}
for some \( \tilde{\Theta} = \Theta^\ast + t (\Theta - \Theta^\ast) \), \( t \in [0, 1] \).
Note that essentially by the convexity of \( \log \det \), we have 
\begin{equation*}
	\tr(\tilde{\Theta}^{-1} (\Theta - \Theta^\ast) \tilde{\Theta}^{-1} (\Theta - \Theta^\ast)) = \| \tilde{\Theta}^{-1} (\Theta - \Theta^\ast) \|_F^2 \geq \lambda_{\mathrm{min}}(\tilde{\Theta}^{-1})^2 \| \Theta - \Theta^\ast \|_F^2
\end{equation*}
and
\begin{equation*}
	\lambda_{\mathrm{min}}(\tilde{\Theta}^{-1}) = (\lambda_{\mathrm{max}}(\tilde{\Theta}))^{-1}.
\end{equation*}
If we write \( \Delta = \Theta - \Theta^\ast \), then for \( \| \Delta \|_F \leq 1 \), \begin{equation*}
	\lambda_{\mathrm{max}}(\tilde{\Theta}) = \| \tilde{\Theta} \|_\mathrm{op} = \| \Theta^\ast + \Delta \|_{\mathrm{op}}
	\leq \| \Theta^\ast \|_{\mathrm{op}} + \| \Delta \|_{\mathrm{op}}
	\leq \| \Theta^\ast \|_{\mathrm{op}} + \| \Delta \|_{F}
	\leq \| \Theta^\ast \|_\mathrm{op} + 1,
\end{equation*}
and therefore
\begin{equation*}
	l(\Theta, \Sigma_n) - l(\Theta^\ast, \Sigma_n) \geq \tr((\Sigma_n - \Sigma^\ast) (\Theta - \Theta^\ast)) + c \| \Theta - \Theta^\ast \|_F^2,
\end{equation*}
for \( c = (\| \Theta^\ast \|_\mathrm{op} + 1)^{-2}/2 \), if \( \| \Delta \|_F \leq 1 \).

This takes care of the case where \( \Delta \) is small.
To handle the case where it is large, define \( g(t) = l(\Theta^\ast + t \Delta, \Sigma_n) - l(\Theta^\ast, \Sigma_n) \).
By the convexity of \( l \) in \( \Theta \),
\begin{equation*}
	\frac{g(1) - g(0)}{1} \geq \frac{g(t) - g(0)}{t},
\end{equation*}
so that plugging in \( t = \| \Delta \|_F^{-1} \) gives
\begin{align*}
	l(\Theta, \Sigma_n) - l(\Theta^\ast, \Sigma_n)
	\ge {} & \| \Delta \|_F \left( l(\Theta^\ast +\frac{1}{\| \Delta \|_F} \Delta , \Sigma_n) - l(\Theta^\ast, \Sigma_n) \right)\\
	\geq {} & \| \Delta \|_F \left( \tr((\Sigma_n - \Sigma^\ast) \frac{1}{\| \Delta \|_F} \Delta) + c \right)\\
	= {} & \tr((\Sigma_n - \Sigma^\ast) \Delta) + c \| \Delta \|_F,
\end{align*}
for \( \| \Delta \|_F \geq 1 \).

If we now write \( \Delta = \hat{\Theta} - \Theta^\ast \) and assume \( \| \Delta \|_F \geq 1 \), then by optimality of \( \hat{\Theta} \),
\begin{align*}
	\| \Delta \|_F
	\leq {} & C \left[ \tr((\Sigma^\ast - \Sigma_n) \Delta) + \lambda (\| \Theta^\ast_{D^c} \|_1 - \| \Theta_{D^c} \|_1) \right]\\
	\leq {} & C \left[ \tr((\Sigma^\ast - \Sigma_n) \Delta) + \lambda (\| \Delta_S \|_1 - \| \Delta_{S^c} \|_1) \right],
\end{align*}
where \( S = \{ (i, j) \in D^c : \Theta^\ast_{i, j} \neq 0 \} \) by triangle inequality.
Now, split the error contributions for the diagonal and off-diagonal elements,
\begin{align*}
	\leadeq{\tr((\Sigma^\ast - \Sigma_n) \Delta) + \lambda (\| \Delta_S \|_1 - \| \Delta_{S^c} \|_1)}\\
	\leq {} & \| (\Sigma^\ast - \Sigma_n)_D \|_F \| \Delta_D \|_F + \| (\Sigma^\ast - \Sigma_n)_{D^c} \|_\infty \| \Delta_{D^c} \|_1\\
	{} & + \lambda (\| \Delta_S \|_1 - \| \Delta_{S^c} \|_1).
\end{align*}
By H\"older inequality, \( \| (\Sigma^\ast - \Sigma_n)_D \|_F \leq \sqrt{d} \| \Sigma^\ast - \Sigma_n \|_\infty \), and by Lemma \ref{lem:chi2-bounds}, \( \| \Sigma^\ast - \Sigma_n \|_\infty \leq \sqrt{\log(d/\delta)}/\sqrt{n} \) for \( n \gtrsim \log (d/\delta) \).
Combining these two estimates,
\begin{align*}
	\leadeq{\tr((\Sigma^\ast - \Sigma_n) \Delta) + \lambda (\| \Delta_S \|_1 - \| \Delta_{S^c} \|_1)}\\
	\lesssim {} & \sqrt{\frac{d \log (d/\delta)}{n}} \| \Delta_D \|_F + \sqrt{\frac{\log (d/\delta)}{n}} \| \Delta_{D^c} \|_1 + \lambda (\| \Delta_S \|_1 - \| \Delta_{S^c} \|_1)
\end{align*}
Setting \( \lambda = C \sqrt{\log(ep/\delta)/n} \) and splitting \( \| \Delta_{D^c} \|_1 = \| \Delta_S \|_1 + \| \Delta_{S^c} \|_1 \) yields 
\begin{align*}
	\leadeq{\sqrt{\frac{d \log (d/\delta)}{n}} \| \Delta_D \|_F + \sqrt{\frac{\log (d/\delta)}{n}} \| \Delta_{D^c} \|_1 + \lambda (\| \Delta_S \|_1 - \| \Delta_{S^c} \|_1)}\\
	\leq {} & \sqrt{\frac{d \log (d/\delta)}{n}} \| \Delta_D \|_F + \sqrt{\frac{\log (d/\delta)}{n}} \| \Delta_{S} \|_1\\
	\leq {} & \sqrt{\frac{d \log (d/\delta)}{n}} \| \Delta_D \|_F + \sqrt{\frac{k \log (d/\delta)}{n}} \| \Delta_{S} \|_F\\
	\leq {} & \sqrt{\frac{(d + k)\log (d/\delta)}{n}} \| \Delta \|_F
\end{align*}
Combining this with a left-hand side of \( \| \Delta \|_F \) yields \( \| \Delta \|_F = 0 \) for \( n \gtrsim (d + k) \log(d/\delta) \), a contradiction to \( \| \Delta \|_F \geq 1 \) where this bound is effective.
Combining it with a left-hand side of \( \| \Delta \|_F^2 \) gives us
\begin{equation*}
	\| \Delta \|_F^2 \lesssim \frac{(d+k) \log (d/\delta)}{n},
\end{equation*}
as desired.
\end{proof}

Write \( \Sigma^\ast = W^\ast \Gamma^\ast W^\ast \), where \( W^2 = \Sigma^\ast_D \) and similarly define \( \hat{W} \) by \( \hat{W}^2 = \hat{\Sigma}_D \).
Define a slightly modified estimator by replacing \( \hat{\Sigma} \) in \eqref{eq:af} by \( \hat{\Gamma} \) to get an estimator \( \hat{K} \) for \( K = (\Gamma^\ast)^{-1} \).

\begin{align}
	\| \hat{\Gamma} - \Gamma \|_\infty
	= {} & \| \hat{W}^{-1} \hat{\Sigma} \hat{W}^{-1} - (W^\ast)^{-1} \Sigma^\ast (W^\ast)^{-1} \|_\infty\\
	\leq {} & \| \hat{W}^{-1} - (W^\ast)^{-1} \|_{\infty, \infty} \| \hat{\Sigma} - \Sigma^\ast \|_\infty  \| \hat{W}^{-1} - (W^\ast)^{-1} \|_{\infty, \infty} \nonumber \\
	{} & + \| \hat{W}^{-1} - (W^\ast)^{-1} \|_{\infty, \infty} \left( \| \hat{\Sigma} \|_\infty \| (W^\ast)^{-1} \|_\infty + \| \Sigma^\ast \|_\infty \| \hat{W}^{-1} \|_{\infty, \infty} \right) \nonumber \\
	{} & + \| \hat{W}^{-1} \|_{\infty, \infty} \| \hat{\Sigma} - \Sigma^\ast \|_\infty \| (W^\ast)^{-1} \|_{\infty, \infty} \label{eq:aa}
\end{align}
Note that \( \| A \|_{\infty} \leq \| A \|_{\mathrm{op}} \), so that if \( n \gtrsim \log (d/\delta) \), then \( \| \hat{\Gamma} - \Gamma \|_\infty \lesssim \sqrt{\log (d/\delta)}/\sqrt{n} \).

From the above arguments, conclude \( \| \hat{K} - K^\ast \|_F \lesssim \sqrt{k \log(d/\delta)/n} \) and with a calculation similar to \eqref{eq:aa} that
\begin{equation*}
	\| \hat{\Theta}' - \Theta^\ast \|_\mathrm{op} \lesssim \sqrt{\frac{k \log(d/\delta)}{n}}.
\end{equation*}
This will not work in Frobenius norm since there, we only have \( \| \hat{W}^2 - W^2 \|_F^2 \lesssim \sqrt{p \log(d/\delta)/n}\).

\subsubsection{Lower bounds}
Here, we will show that the bounds in Theorem \ref{thm:graphical-lasso} are optimal up to log factors.
It will again be based on applying Fano's inequality, Theorem \ref{thm:fano-ineq}.

In order to calculate the KL divergence between two Gaussian distributions with densities \( f_{\Sigma_1} \) and \( f_{\Sigma_2} \), we first observe that
\begin{align*}
	\log \frac{f_{\Sigma_1}(x)}{f_{\Sigma_2}(x)}
	= {} & -\frac{1}{2} x^\top \Theta_1 x + \frac{1}{2}\log \operatorname{det}(\Theta_1) + \frac{1}{2} x^\top \Theta_2 x - \frac{1}{2} \log \operatorname{det}(\Theta_2)
\end{align*},
so that
\begin{align*}
	\KL(\p_{\Sigma_1}, \p_{\Sigma_2})
	= {} & \E_{\Sigma_1} \log \left( \frac{f_{\Sigma_1}}{f_{\Sigma_2}}(X) \right)\\
	= {} & \frac{1}{2} \left[\log \operatorname{det} (\Theta_1) - \log \operatorname{det} (\Theta_2) + \E[ \tr((\Theta_2 - \Theta_1) X X^\top)] \right]\\
	= {} & \frac{1}{2} \left[ \log \operatorname{det} (\Theta_1) - \log \operatorname{det} (\Theta_2) + \tr((\Theta_2 - \Theta_1) \Sigma_1) \right] \\
	= {} & \frac{1}{2} \left[\log \operatorname{det} (\Theta_1) - \log \operatorname{det} (\Theta_2) + \tr(\Theta_2 \Sigma_1) - p\right].
\end{align*}
writing \( \Delta = \Theta_2 - \Theta_1 \), \( \tilde{\Theta} = \Theta_1 + t \Delta  \), \( t \in [0,1] \), Taylor expansion then yields
\begin{align*}
	\KL(\p_{\Sigma_1}, \p_{\Sigma_2})
	= {} & \frac{1}{2} \left[ - \tr(\Delta \Sigma_1) + \frac{1}{2} \tr(\tilde{\Theta}^{-1} \Delta \tilde{\Theta}^{-1}\Delta) + \tr(\Delta \Sigma_1) \right]\\
	\leq {} & \frac{1}{4} \lambda_{\mathrm{max}}(\tilde{\Theta}^{-1}) \| \Delta \|_F^2\\
	= {} & \frac{1}{4} \lambda_\mathrm{min}^{-1}(\tilde{\Theta}) \| \Delta \|_F^2.
\end{align*}

This means we are in good shape in applying our standard tricks if we can make sure that \( \lambda_\mathrm{min}(\tilde{\Theta}) \) stays large enough among our hypotheses.
To this end, assume \( k \leq p^2/16 \) and define the collection of matrices \( (B_{j})_{j \in [M_1]} \) with entries in \( \{0, 1\} \) such that they are symmetric, have zero diagonal and only have non-zero entries in a band of size \( \sqrt{k} \) around the diagonal, and whose entries are given by a flattening of the vectors obtained by the sparse Varshamov-Gilbert Lemma \ref{LEM:sVG}.
Moreover, let \( (C_j)_{j \in [M_2]} \) be diagonal matrices with entries in \( \{0, 1\} \) given by the regular Varshamov-Gilbert Lemma \ref{LEM:VG}.
Then, define
\begin{equation*}
	\Theta_j = \Theta_{j_1, j_2} = I + \frac{\beta}{\sqrt{n}} (\sqrt{\log(1+d/(2k))} B_{j_1} + C_{j_2}), \quad j_1 \in [M_1], \, j_2 \in [M_2].
\end{equation*}
We can ensure that \( \lambda_\mathrm{min}(\tilde{\Theta}) \) is small by the Gershgorin circle theorem.
For each row of \( \tilde{\Theta} - I \),
\begin{align*}
	\sum_{l} |\tilde{\Theta}_{il}| \leq \frac{2\beta (\sqrt{k} + 1)}{\sqrt{n}} \sqrt{\log(1+d/(2k)} < \frac{1}{2}
\end{align*}
if \( n \gtrsim k/\log(1 + d/(2k)) \).

Hence,
\begin{align*}
	\| \Theta_j - \Theta_l \|_F^2
	= {} & \frac{\beta^2}{n} (\rho(B_{j_1}, B_{l_1}) \log (1+d/(2k)) + \rho(C_{j_2}, C_{l_2}))\\
	\gtrsim {} & \frac{\beta^2 k}{n} (k \log(1+d/(2k)) + d),
\end{align*}
and
\begin{align*}
	\| \Theta_j - \Theta_l \|_F^2
	\lesssim {} & \frac{\beta^2}{n} (k \log(1+d/(2k)) + d)\\
	\leq {} & \frac{\beta^2}{n} \log(M_1 M_2),
\end{align*}
which yields the following theorem.

\begin{thm}
	Denote by \( \mathcal{B}_k \) the set of positive definite matrices with at most \( k \) non-zero off-diagonal entries and assume \( n \gtrsim k \log(d) \).
	Then, if \( \p_{\Theta} \) denotes a Gaussian distribution with inverse covariance matrix \( \Theta \) and mean \( 0 \), there exists a constant \( c > 0 \) such that
	\begin{equation*}
		\inf_{\hat{\Theta}} \sup_{\Theta^\ast \in \mathcal{B}_k} \p^{\otimes n}_{\Theta^\ast} \left( \| \hat{\Theta} - \Theta^\ast \|_F^2 \geq c \frac{\alpha^2}{n} (k \log(1+d/(2k)) + d) \right) \geq \frac{1}{2} - 2 \alpha.
	\end{equation*}
\end{thm}

\subsection{Ising model}

Like the Gaussian graphical model, the Ising model is also a model of pairwise interactions but for random variables that take values in $\{-1,1\}^d$.

Before developing our general approach, we describe the main ideas in the simplest Markov random field: an Ising model without\footnote{The presence of an external field does not change our method. It merely introduces an intercept in the logistic regression problem and comes at the cost of more cumbersome notation. All arguments below follow, potentially after minor modifications and explicit computations are left to the reader.} external field~\cite{VufMisLok16}.  Such models specify the distribution of a random vector $Z=(Z^{(1)}, \ldots, Z^{(d)})\in \dHyp$ as follows
\begin{equation}
\label{EQ:defIsing}
\p(Z=z)=\exp\big(z^\top W z   -\Phi(W)\big)\,,
\end{equation}
where $W \in \R^{pd \times d}$ is an unknown matrix of interest with null diagonal elements and 
$$
\Phi(W)=\log \sum_{z \in \dHyp}\exp\big(z^\top W z  \big)\,,
$$
is a normalization term known as \emph{log-partition function}. 

Fix $\beta,\lambda>0$. Our goal in this section is to estimate the parameter $W$ subject to the constraint that the $j$th row $e_j^\top W$ of $W$ satisfies $|e_j^\top W|_1\le \lambda$.
%
To that end, we observe $n$ independent copies $Z_1, \ldots Z_n$ of $Z\sim \p$. To that end, estimate the matrix $W$ row by row using constrained likelihood maximization. Recall from \cite[Eq.~(4)]{RavWaiLaf10} that for any $j \in [d]$, it holds for any $z^{(j)}\in \{-1, 1\}$, that 
\begin{equation}
\label{EQ:logistic}
\p(Z^{(j)}=z^{(j)} | Z^{(\lnot j)}=z^{(\lnot j)})=\frac{\exp(2z^{(j)} e_j^\top W z)}{1+\exp(2z^{(j)} e_j^\top W z)}\,,
\end{equation}
where we used the fact that the diagonal elements $W$ are equal to zero. 

Therefore, the $j$th row $e_j^\top W$ of $W$ may be estimated by performing a logistic regression of $Z^{(j)}$ onto $Z^{(\lnot j)}$ subject to the constraint that $|e_j^\top W|_1 \le \lambda$ and $\|e_j^\top W\|_\infty \le \beta$. To simplify notation, fix $j$ and define $Y=Z^{(j)}, X=Z^{(\lnot j)}$ and $w^*=(e_j^\top W)^{(\lnot j)}$. Let $(X_1, Y_1), \ldots, (X_n, Y_n)$ be $n$ independent copies of $(X,Y)$. Then the (rescaled) log-likelihood of a candidate parameter $w \in \R^{d-1}$ is denoted by $\bar \ell_n(w)$ and defined by
\begin{align*}
\bar \ell_n(w)&=\frac{1}{n}\sum_{i=1}^n\log\Big( \frac{\exp(2Y_i X_i^\top w)}{1+\exp(2Y_i X_i^\top w)}\Big)\\
&=\frac{1}{n}\sum_{i=1}^nY_iX_i^\top w - \frac{1}{n}\sum_{i=1}^n\log\big(1+ \exp(2Y_i X_i^\top w)\big)
\end{align*}
With this representation, it is easy to see that $w \mapsto \bar \ell_n(w)$ is a concave function. The (constrained) maximum likelihood estimator (MLE) $\hat w \in \R^{d-1}$ is defined to be any solution of the following convex optimization problem:
\begin{equation}
\label{EQ:MLE-Ising}
\hat w \in  \argmax_{w \in \cB_1(\lambda)} \bar \ell_n(w)
\end{equation}
This problem can be solved very efficiently using a variety of methods similar to the Lasso.

The following lemma follows from~\cite{Rig12}.

\begin{lem}
\label{LEM:MLE-logistic}
Fix $\delta \in (0,1)$. Conditionally on $(X_1, \ldots, X_n)$, the constrained {MLE} $\hat w$ defined in~\eqref{EQ:MLE-Ising} satisfies  with probability at least $1-\delta$ that
$$
\frac{1}{n}\sum_{i=1}^n (X_i^\top (\hat w - w))^2 \le 2\lambda e^\lambda\sqrt{\frac{\log (2(p-1)/\delta)}{2n}}
$$
\end{lem}
\begin{proof}
We begin with some standard manipulations of the log-likelihood. Note that
\begin{align*}
\log\Big( \frac{\exp(2Y_i X_i^\top w)}{1+\exp(2Y_i X_i^\top w)}\Big)&=\log\Big( \frac{\exp(Y_i X_i^\top w)}{\exp(-Y_i X_i^\top w)+\exp(Y_i X_i^\top w)}\Big)\\
&=\log\Big( \frac{\exp(Y_i X_i^\top w)}{\exp(-X_i^\top w)+\exp(X_i^\top w)}\Big)\\
&=(Y_i+1) X_i^\top w - \log\big(1+\exp(2X_i^\top w) \big)
\end{align*}
Therefore, writing $\tilde Y_i =(Y_i+1)/2 \in \{0,1\}$, we get that $\hat w$ is the solution to the following minimization problem:
$$
\hat w = \argmin_{w \in \cB_1(\lambda)} \bar \kappa_n(w)
$$
where 
$$
\bar \kappa_n(w)=\frac{1}{n}\sum_{i=1}^n\big\{ -2\tilde Y_i  X_i^\top w + \log\big(1+\exp(2X_i^\top w) \big)\big\}
$$
For any $w \in \R^{d-1}$, write $\kappa(w)=\E[\bar \kappa_n(w)]$ where here and throughout this proof, all expectations are implicitly taken conditionally on $X_1, \ldots, X_n$. Observe that
$$
\kappa(w)=-\E[\tilde Y_1 ] X_1^\top w + \log\big(1+\exp(2X_i^\top w) \big)\,.
$$

Next, we get from the basic inequality $\bar \kappa_n(\hat w) \le \bar \kappa_n(w)$ for any $w \in \cB_1(\lambda)$ that
\begin{align*}
\kappa(\hat w) -\kappa(w) &\le \frac{1}{n} \sum_{i=1}^n (\tilde Y_i -\E[\tilde Y_i]) X_i^\top(\hat w - w) \\
&\le \frac{2\lambda}{n} \max_{1\le j \le p-1}\Big|\sum_{i=1}^n (\tilde Y_i -\E[\tilde Y_i]) X_i^\top e_j\Big|\,,
\end{align*}
where in the second inequality, we used  H\"older's inequality and the fact $|X_i|_\infty \le 1$ for all $i \in [n]$.
Together with  Hoeffding's inequality and a union bound, it yields 
$$
\kappa(\hat w) -\kappa(w) \le 2\lambda \sqrt{\frac{\log (4(p-1)/\delta)}{2n}}\,, \quad \text{with probability $1-\frac{\delta}{2}$.}
$$

Note that the Hessian of $\kappa$ is given by
$$
\nabla^2\kappa(w)=\frac{1}{n} \sum_{i=1}^n \frac{4\exp(2X_i^\top w)}{(1+\exp(2X_i^\top w))^2} X_i X_i^\top\,.
$$
Moreover, for any $w \in \cB_1(\lambda)$, since $\|X_i\|_\infty \le 1$,  it holds
$$
\frac{4\exp(2X_i^\top w)}{(1+\exp(2X_i^\top w))^2}\ge \frac{4e^{2\lambda}}{(1+e^{2\lambda})^2} \ge e^{-2\lambda}
$$
Next, since $w^*$ is a minimizer of $w \mapsto \kappa(w)$, we get from a second order Taylor expansion that
$$
\kappa(\hat w) -\kappa(w^*) \ge e^{-2\lambda} \frac{1}{n}\sum_{i=1}^n (X_i^\top (\hat w - w))^2\,.
$$
This completes our proof.
\end{proof}

Next, we bound from below the quantity 
$$
\frac{1}{n}\sum_{i=1}^n (X_i^\top (\hat w - w))^2\,.
$$
To that end, we must exploit the covariance structure of the Ising model. The following lemma is similar to the combination of Lemmas~6 and~7 in~\cite{VufMisLok16}. 
\begin{lem}
\label{LEM:LBempcov}
Fix $R>0$, $\delta\in (0,1)$ and let $Z \in \R^{d}$ be distributed according to~\eqref{EQ:defIsing}. The following holds with probability $1-\delta/2$, uniformly in $u \in \cB_1(2\lambda)$:
$$
\frac{1}{n}\sum_{i=1}^n (Z_i^\top u)^2 \ge \frac12|u|_\infty^2 \exp(-2\lambda)- 4\lambda\sqrt{\frac{\log (2p(p-1)/\delta)}{2n}}
$$
\end{lem}
\begin{proof}
Note that 
$$
\frac{1}{n}\sum_{i=1}^n (Z_i^\top u)^2 =u^\top S_n u\,,
$$
where $S_n \in \R^{d\times d}$ denotes the sample covariance matrix of $Z_1, \ldots, Z_n$, defined by
$$
S_n=\frac{1}{n}\sum_{i=1}^n X_iX_i^\top\,.
$$
Observe that
\begin{equation}
\label{EQ:LBempcov:pr1}
u^\top S_n u \ge u^\top \Sigma u - 2\lambda  |S_n -S|_\infty\,,
\end{equation}
where $S=\E[S_n]$. 
Using Hoeffding's inequality together with a union bound, we get that
$$
|S_n -S|_\infty \le 2\sqrt{\frac{\log (2p(p-1)/\delta)}{2n}}
$$
with probability $1-\delta/2$.

Moreover, $u^\top S u= \var(X^\top u)$. Assume without loss of generality that $|u^{(1)}|=|u|_\infty$ and observe that
$$
\var(Z^\top u)=\E[(Z^\top u)^2]=(u^{(1)})^2 + \E\big[\big(\sum_{j=2}^{p}u^{(j)}Z^{(j)}\big)^2\big]+2\E\big[u^{(1)}Z^{(1)}\sum_{j=2}^{p}u^{(j)}Z^{(j)}\big]\,.
$$
To control the cross term, let us condition on the neighborhood $Z^{(\lnot 1)}$ to get
\begin{align*}
2\Big|\E\big[u^{(1)}Z^{(1)}\sum_{j=2}^{p}u^{(j)}Z^{(j)}\big]\Big|&=2\Big|\E\Big(\E\big[u^{(1)}Z^{(1)}|Z^{(\lnot 1)}\big]\sum_{j=2}^{p}u^{(j)}Z^{(j)}\big]\Big)\Big|\\
&\le \E\Big(\big(\E\big[u^{(1)}Z^{(1)}|Z^{(\lnot 1)}\big]\big)^2\Big)+ \E\Big(\big(\sum_{j=2}^{p}u^{(j)}Z^{(j)}\big)^2\Big)\,,
\end{align*}
where we used the fact that $2|ab|\le a^2 + b^2$ for all $a,b \in \R$. The above two displays together yield
$$
\var(Z^\top u)\ge |u|_\infty^2 (1- \sup_{z \in \{-1,1\}^{p-1}}\big( \E[Z^{(1)}|Z^{(\lnot 1)}=z]\big)^2\big)
$$
To control the above supremum, recall that from~\eqref{EQ:logistic}, we have
\begin{align*}
\sup_{z \in \{-1,1\}^{p-1}}\E[Z^{(1)}|Z^{(\lnot 1)}=z]&=\sup_{z \in \{-1,1\}^{p}}\frac{\exp(2 e_1^\top W z)-1}{\exp(2e_1^\top W z)+1}=\frac{\exp(2 \lambda)-1}{\exp(2 \lambda)+1}\,.
\end{align*}
Therefore, 
$$
\var(Z^\top u)\ge \frac{|u|_\infty^2 }{1+\exp(2\lambda)}\ge \frac12 |u|_\infty^2\exp(-\lambda)\,.
$$
Together with~\eqref{EQ:LBempcov:pr1}, it yields the desired result.
\end{proof}

Combining the above two lemmas immediately yields the following theorem.

\begin{thm}
\label{TH:main-Ising} Let $\hat w$ be the constrained maximum likelihood estimator~\eqref{EQ:MLE-Ising} and let $w^*=(e_j^\top W)^{(\lnot j)}$ be the $j$th row of $W$ with the $j$th entry removed. Then with probability $1-\delta$ we have
$$
|\hat w - w^*|_\infty^2 \le   9\lambda \exp(3\lambda)\sqrt{\frac{\log (2p^2/\delta)}{n}}\,.
$$
\end{thm}

\newpage

\section{Problem set}

\medskip

\begin{exercise}
\label{EXO:concat}
Using the results of Chapter~\ref{chap:GSM}, show that the following holds for the multivariate regression model~\eqref{EQ:MVRmodel}.
\begin{enumerate}
\item There exists an estimator $\hat \Theta \in \R^{d \times T}$ such that
$$
\frac{1}{n}\|\X\hat \Theta -\X \Theta^*\|_F^2 \lesssim \sigma^2\frac{rT}{n}
$$
with probability .99, where $r$ denotes the rank of $\X$\,.
\item There exists an estimator $\hat \Theta \in \R^{d \times T}$ such that
$$
\frac{1}{n}\|\X\hat \Theta -\X \Theta^*\|_F^2 \lesssim \sigma^2\frac{|\Theta^*|_0\log(ed)}{n}\,.
$$
with probability .99.
\end{enumerate}
\end{exercise}

\medskip

\begin{exercise}
\label{EXO:svtMRV}
Consider the multivariate regression model~\eqref{EQ:MVRmodel} where $\Y$ has SVD:
$$
\Y=\sum_j\hat \lambda_j \hat u_j \hat v_j^\top\,.
$$
Let $M$ be defined by
$$
\hat M=\sum_j\hat \lambda_j \1(|\hat \lambda_j|>2\tau)\hat u_j \hat v_j^\top\,, \tau>0\,.
$$
\begin{enumerate}
\item Show that there exists a choice of $\tau$ such that
$$
\frac{1}{n}\|\hat M -\X \Theta^*\|_F^2 \lesssim \frac{\sigma^2\rank(\Theta^*)}{n}(d\vee T)
$$
with probability .99.
\item Show that there exists a matrix $n \times n$ matrix $P$ such that $P\hat M=\X\hat \Theta$ for some estimator $\hat \Theta$ and 
$$
\frac{1}{n}\|\X\hat \Theta -\X \Theta^*\|_F^2 \lesssim \frac{\sigma^2\rank(\Theta^*)}{n}(d\vee T)
$$
with probability .99.
\item Comment on the above results in light of the results obtain in Section~\ref{SEC:MVR}.
\end{enumerate}
\end{exercise}

\medskip

\begin{exercise}
\label{EXO:nuclear1}
Consider the multivariate regression model~\eqref{EQ:MVRmodel} and define $\hat \Theta$ be the any solution to the minimization problem
$$
\min_{\Theta \in \R^{d \times T}} \Big\{ \frac{1}{n}\|\Y-\X\Theta\|_F^2 + \tau \|\X\Theta\|_1\Big\}
$$
\begin{enumerate}
\item Show that there exists a choice of $\tau$ such that
$$
\frac{1}{n}\|\X \hat \Theta -\X \Theta^*\|_F^2 \lesssim \frac{\sigma^2\rank(\Theta^*)}{n}(d\vee T)
$$
with probability .99. 

\texttt{[Hint:Consider the matrix 
$$
\sum_j \frac{\hat\lambda_j + \lambda_j^*}{2}\hat u_j \hat v_j^\top
$$
where $\lambda^*_1\ge \lambda^*_2 \ge \dots$ and $\hat\lambda_1\ge \hat\lambda_2\ge \dots$ are the singular values of $\X\Theta^*$ and $\Y$ respectively and the SVD of $\Y$ is given by
$$
\Y=\sum_j\hat \lambda_j \hat u_j \hat v_j^\top
$$}
\item Find a closed form for $\X\hat\Theta$.
\end{enumerate}
\end{exercise}

\medskip

\begin{exercise}
\label{EXO:portfolio}
In the Markowitz theory of portfolio selection~\cite{Mar52}, a portfolio may be identified to a vector $u \in \R^d$ such that $u_j \ge 0$ and $\sum_{j=1}^d u_j=1$. In this case, $u_j$ represents the proportion of the portfolio invested in asset $j$. The vector of (random) returns of $d$ assets is denoted by $X \in \R^d$ and we assume that $X\sim \sg_d(1)$ and $\E[XX^\top]=\Sigma$ unknown.t

In this theory, the two key characteristics of a portfolio $u$ are it's \emph{reward} $\mu(u)=\E[u^\top X]$ and its \emph{risk} $R(u)=\var{X^\top u}$. According to this theory one should fix a minimum reward $\lambda>0$ and choose the optimal portfolio
$$
u^*=\argmin_{u\,: \mu(u) \ge \lambda}  R(u)
$$
when a solution exists for a given. It is the portfolio that has minimum risk among all portfolios with reward at least $\lambda$, provided such portfolios exist.

In practice, the distribution of $X$ is unknown. Assume that we observe $n$ independent copies $X_1, \ldots, X_n$ of $X$ and use them to compute the following estimators of $\mu(u)$ and $R(u)$ respectively:
\begin{align*}
\hat \mu(u)&=\bar X^\top u =\frac{1}{n}\sum_{i=1}^n X_i^\top u\,, \\
\hat R(u)&= u^\top \hat \Sigma u, \quad \hat \Sigma = \frac{1}{n-1}\sum_{i=1}^n (X_i-\bar X)(X_i -\bar X)^\top\,.
\end{align*}
We use the following estimated portfolio:
$$
\hat u=\argmin_{u\,: \hat \mu(u) \ge \lambda}  \hat R(u)
$$

We assume throughout that $\log d \ll n \ll d$\,.

\begin{enumerate}
\item Show that for any portfolio $u$, it holds 
$$
\big|\hat \mu(u) - \mu(u)\big|  \lesssim \frac{1}{\sqrt{n}}\,,
$$
and 
$$
\big|\hat R(u) - R(u)\big|  \lesssim \frac{1}{\sqrt{n}}\,.
$$
\item Show that 
$$
\hat R(\hat u) -  R(\hat u) \lesssim \sqrt{\frac{\log d}{n}}\,,
$$
with probability .99.
\item Define the estimator $\tilde u$  by:
$$
\tilde u=\argmin_{u\,: \hat \mu(u) \ge \lambda- \eps}  \hat R(u)
$$
find the smallest $\eps>0$ (up to multiplicative constant) such that we have $R(\tilde u) \le R(u^*)$ with probability .99.
\end{enumerate}
\end{exercise}

\backmatter

\bibliographystyle{alpha}
\bibliography{hds_bibliography}

\end{document}